\documentclass{amsart}
\pdfoutput=1
\usepackage{amscd, amsmath, amssymb, tikz}
\usetikzlibrary{matrix,arrows}
\usepackage{enumitem}
\usepackage{datatool}
\usepackage{glossaries}
\newglossaryentry{lin}{name={\ensuremath{\textnormal{Lin}_M (A)}}, text={\textnormal{Lin}}, description={}}

\newglossaryentry{gicone}{name={\ensuremath{\mathcal{A}}}, user1={\ensuremath{\mathcal{A}}}, user2={\textnormal{Lin}_{\mathbb{N}} (\mathcal{A})}, description={}}

\newglossaryentry{suphyp}{name={\ensuremath{\bar{Q}}}, text={}, description={}}

\newglossaryentry{Adet}{name={\ensuremath{E_A}}, text={}, description={}}

\newglossaryentry{Adetsect}{name={\ensuremath{E^s_A}}, text={}, description={}}

\newglossaryentry{convhl}{name={\ensuremath{\textnormal{Conv} (A)}}, text={}, description={}}

\newglossaryentry{Adethyp}{name={\ensuremath{\mathcal{E}_A}}, text={}, description={}}

\newglossaryentry{Adisc}{name={\ensuremath{\Delta_A}}, text={}, description={}}

\newglossaryentry{seccone}{name={\ensuremath{C_\mathbb{R} ( S)}}, text={}, description={}}

\newglossaryentry{colimit}{name={\ensuremath{\mathbf{\Sigma}^\to_g, \mathcal{X}_{g}^{\to}}}, text={}, description={}}

\newglossaryentry{fundseq}{name={\ensuremath{\alpha_B, \beta_B}}, text={}, description={}}

\newglossaryentry{lafcone}{name={\ensuremath{C_\mathbb{R} ( S, A_p)}}, text={}, description={}}

\newglossaryentry{markpoly}{name={\ensuremath{(Q,A)}}, text={}, description={}}

\newglossaryentry{subdiv}{name={\ensuremath{S}}, text={}, description={}}

\newglossaryentry{triang}{name={\ensuremath{T}}, text={}, description={}}

\newglossaryentry{psubdiv}{name={\ensuremath{(S, A_p)}}, text={}, description={}}

\newglossaryentry{stfan}{name={\ensuremath{(\Lambda_1,\Lambda_2, \beta , \Sigma)}}, text={}, description={}}

\newglossaryentry{spv}{name={\ensuremath{\varphi_T}}, text={}, description={}}

\newglossaryentry{sp}{name={\ensuremath{\Sigma (A)}}, text={}, description={}}

\newglossaryentry{secfan}{name={\ensuremath{\mathcal{F}_{\Sigma (A)}}}, text={}, description={}}

\newglossaryentry{tsp}{name={\ensuremath{\Sigma_v (A)}}, text={}, description={}}

\newglossaryentry{lp}{name={\ensuremath{\Theta (A)}}, text={}, description={}}

\newglossaryentry{lph}{name={\ensuremath{\Theta_p (A)}}, text={}, description={}}

\newglossaryentry{lfan}{name={\ensuremath{\mathcal{F}_{\Theta_p (A)}}}, text={}, description={}}

\newglossaryentry{ls}{name={\ensuremath{\mathcal{X}_{\Theta (A)}}}, text={}, description={}}

\newglossaryentry{prestack}{name={\ensuremath{X_{\mathbf{\Sigma}}}}, text={}, description={}}

\newglossaryentry{hgroup}{name={\ensuremath{\mathbb{H}_{\mathbf{\Sigma}}}}, text={}, description={}}

\newglossaryentry{torstack}{name={\ensuremath{\mathcal{X}_{\mathbf{\Sigma}}}}, text={}, description={}}

\newglossaryentry{ggroup}{name={\ensuremath{\mathbb{G}_{\mathbf{\Sigma}}}}, text={}, description={}}

\newglossaryentry{lb}{name={\ensuremath{L_{B}}}, text={}, description={}}

\newglossaryentry{kb}{name={\ensuremath{K_B}}, text={}, description={}}

\newglossaryentry{nsubb}{name={\ensuremath{n_b}}, text={}, description={}}

\newglossaryentry{polsf}{name={\ensuremath{\mathbf{\Sigma}_Q}}, text={}, description={}}

\newglossaryentry{polst}{name={\ensuremath{\mathcal{X}_Q}}, text={}, description={}}

\newglossaryentry{polbound}{name={\ensuremath{\partial \mathcal{X}_Q}}, text={}, description={}}

\newglossaryentry{pollb}{name={\ensuremath{\mathcal{O}_A (1)}}, text={}, description={}}

\newglossaryentry{polls}{name={\ensuremath{\mathcal{L}_A}}, text={}, description={}}

\newglossaryentry{polgroup}{name={\ensuremath{\mathbb{G}_Q}}, text={}, description={}}

\newglossaryentry{avert}{name={\ensuremath{A_v}}, text={}, description={}}

\newglossaryentry{anvert}{name={\ensuremath{A_{nv}}}, text={}, description={}}

\newglossaryentry{ucirc}{name={\ensuremath{\mathbb{T}, \mathbb{T}_\Gamma}}, text={}, description={}}

\newglossaryentry{moment}{name={\ensuremath{\mu}}, text={}, description={}}

\newglossaryentry{symred}{name={\ensuremath{\rho_\omega}}, text={}, description={}}

\newglossaryentry{eta}{name={\ensuremath{\eta, \tilde{\eta}}}, text={}, description={}}

\newglossaryentry{qeta}{name={\ensuremath{Q_\eta}}, text={}, description={}}

\newglossaryentry{ceta}{name={\ensuremath{c_{\eta, i}}}, text={}, description={}}

\newglossaryentry{degfam}{name={\ensuremath{(\mathcal{X}_\eta , \mathcal{Y}_{\iota_\eta (s^\prime)})}}, text={}, description={}}

\newglossaryentry{tordeg}{name={\ensuremath{F_\eta}}, text={}, description={}}

\newglossaryentry{lafhv}{name={\ensuremath{\overline{\Theta (A)}^h, \overline{\Theta (A)}^v}}, text={}, description={}}

\newglossaryentry{sflaf}{name={\ensuremath{\widetilde{\mathbf{\Sigma}}_{\Theta_p (A)}}}, text={}, description={}}

\newglossaryentry{uuniv}{name={\ensuremath{\mathbf{O}_A (1)}}, text={}, description={}}

\newglossaryentry{unsec}{name={\ensuremath{s_A}}, text={}, description={}}

\newglossaryentry{tothyp}{name={\ensuremath{\tilde{\mathcal{Y}}_A}}, text={}, description={}}

\newglossaryentry{lafst}{name={\ensuremath{\mathcal{X}_{\Theta (A)}}}, text={}, description={}} 

\newglossaryentry{unhyp}{name={\ensuremath{\mathcal{Y}_A}}, text={}, description={}}

\newglossaryentry{vrho}{name={\ensuremath{\varrho_A}}, text={}, description={}}

\newglossaryentry{tlfst}{name={\ensuremath{\mathcal{X}_{\Theta_p (A)}}}, text={}, description={}}

\newglossaryentry{psvsf}{name={\ensuremath{\widetilde{\mathbf{\Sigma}}_{\Sigma_v (A)}}},text={}, description={}}

\newglossaryentry{secstack}{name={\ensuremath{\mathcal{X}_{\Sigma (A)}}}, text={}, description={}}

\newglossaryentry{walat}{name={\ensuremath{(\Lambda \oplus \mathbb{Z})_{wall}}}, text={}, description={}}

\newglossaryentry{xilat}{name={\ensuremath{\Xi_{\mathcal{A}}}}, text={}, description={}}

\newglossaryentry{lamgdual}{name={\ensuremath{\Lambda_{B^\vee}}}, text={}, description={}}

\newglossaryentry{secstfan}{name={\ensuremath{\widetilde{\mathbf{\Sigma}}_{\Sigma (A)}}}, text={}, description={}}

\newglossaryentry{mcvA}{name={\ensuremath{\mathcal{V}_A}}, text={}, description={}}

\newglossaryentry{mcuA}{name={\ensuremath{\mathcal{U}_A}}, text={}, description={}}

\newglossaryentry{mcwA}{name={\ensuremath{\mathcal{W}_A}}, text={}, description={}}

\newglossaryentry{circsign}{name={\ensuremath{\sigma (A)}}, text={}, description={}}

\newglossaryentry{apm}{name={\ensuremath{A_-, A_0, A_+}}, text={}, description={}}

\newglossaryentry{vola}{name={\ensuremath{v_A}}, text={}, description={}}

\newglossaryentry{corea}{name={\ensuremath{\textnormal{Core}(A)}}, text={}, description={}}

\newglossaryentry{tpm}{name={\ensuremath{T_\pm}}, text={}, description={}}

\newglossaryentry{hordiv}{name={\ensuremath{\mathcal{D}^h}}, text={}, description={}}

\newglossaryentry{critvala}{name={\ensuremath{c_A}}, text={}, description={}}

\newglossaryentry{lpm}{name={\ensuremath{\ell_\pm}}, text={}, description={}}

\newglossaryentry{tlpm}{name={\ensuremath{\tilde{\ell}_\pm}}, text={}, description={}}

\newglossaryentry{tilci}{name={\ensuremath{\tilde{c}_i}}, text={}, description={}}

\newglossaryentry{restaprime}{name={\ensuremath{s||_{A^\prime}}}, text={}, description={}}

\newglossaryentry{lgmod}{name={\ensuremath{\mathbf{w}}}, text={}, description={}}

\newglossaryentry{tilgap}{name={\ensuremath{\tilde{G}_{A^\prime}, G_{A^\prime}}}, text={}, description={}}

\newglossaryentry{lgmoduli}{name={\ensuremath{\mathcal{M}_{A,A^\prime}}}, text={}, description={}}

\newglossaryentry{vthfth}{name={\ensuremath{V_\theta,\mathcal{F}_\theta, H_\theta}}, text={}, description={}}

\newglossaryentry{mppol}{name={\ensuremath{\Sigma_\gamma (P)}}, text={}, description={}}

\newglossaryentry{gvgv}{name={\ensuremath{G^\vee, G^\wedge}}, text={}, description={}}

\newglossaryentry{rhoaprime}{name={\ensuremath{\rho_{A^\prime}}}, text={}, description={}}

\newglossaryentry{txi}{name={\ensuremath{T_\xi}}, text={}, description={}}

\newglossaryentry{dmj}{name={\ensuremath{(d_j, m_j)}}, text={}, description={}}

\newglossaryentry{ddatxi}{name={\ensuremath{\mathbf{M}_\xi}}, text={}, description={}}

\newglossaryentry{fukcat}{name={\ensuremath{\textnormal{Fuk}^\rightharpoonup}}, text={}, description={}}

\newglossaryentry{sigt}{name={\ensuremath{\Sigma_T, \Lambda_T}}, text={}, description={}}

\newglossaryentry{start}{name={\ensuremath{\textnormal{St}_T}}, text={}, description={}}

\newglossaryentry{mirseq}{name={\ensuremath{\mathbf{S}_\xi}}, text={}, description={}}

\newglossaryentry{circon}{name={\ensuremath{C(w), C_\pm (w), C_0 (w)}}, text={}, description={}}

\newglossaryentry{Upspm}{name={\ensuremath{{\Upsilon},\overline{\Upsilon}}}, text={}, description={}}

\newglossaryentry{sympf}{name={\ensuremath{\textnormal{Simp}(\Sigma)}}, text={}, description={}}

\newglossaryentry{Upsneg}{name={\ensuremath{{\Upsilon}^-,\overline{\Upsilon}^-}}, text={}, description={}}

\newglossaryentry{feb}{name={\ensuremath{\Lambda_E, \Lambda_F, \Lambda_B, \mathcal{E}, \mathcal{F}, \mathcal{B}}}, text={}, description={}}

\newglossaryentry{multip}{name={\ensuremath{\textnormal{Mult}(\sigma )}}, text={}, description={}}

\newglossaryentry{fibeta}{name={\ensuremath{\mathcal{Z}_\eta (t)}}, text={}, description={}}

\newglossaryentry{fibtot}{name={\ensuremath{\mathcal{Z}_A (t)}}, text={}, description={}}

\newglossaryentry{atlas}{name={\ensuremath{(U_\beta, G_\beta , \pi_\beta)_{\beta \in \mathcal{B}}}}, text={}, description={}}

\newglossaryentry{symporb}{name={\ensuremath{\textnormal{Symp} (\mathcal{Y}), \textnormal{Symp} (L / \mathcal{Y})}}, text={}, description={}}

\newglossaryentry{bolT}{name={\ensuremath{\mathbf{T}_\varepsilon}}, text={}, description={}}

\newglossaryentry{sympdv}{name={\ensuremath{\textnormal{Symp}(\mathcal{Y}, D)}}, text={}, description={}}

\newglossaryentry{frsymp}{name={\ensuremath{\textnormal{Symp}^{\mathbf{F}} (\mathcal{Y}, D)}}, text={}, description={}}

\newglossaryentry{bframe}{name={\ensuremath{\mathbf{F}}}, text={}, description={}}

\newglossaryentry{bframered}{name={\ensuremath{\mathbf{F}^{red}}}, text={}, description={}}

\newglossaryentry{jshrp}{name={\ensuremath{j^\#}}, text={}, description={}}

\newglossaryentry{kppa}{name={\ensuremath{\kappa}}, text={}, description={}}

\newglossaryentry{nueta}{name={\ensuremath{\nu_\eta}}, text={}, description={}}

\newglossaryentry{feta}{name={\ensuremath{f_\eta}}, text={}, description={}}

\newglossaryentry{prtrns}{name={\ensuremath{\mathbf{P}_\infty}}, text={}, description={}}

\newglossaryentry{partrns}{name={\ensuremath{\mathbf{P}}}, text={}, description={}}

\newglossaryentry{mongroup}{name={\ensuremath{\mathbf{G}_p}}, text={}, description={}}

\newglossaryentry{idim}{name={\ensuremath{I_d,I_m}}, text={}, description={}}

\newglossaryentry{rketa}{name={\ensuremath{\mathbb{R}^k_{\eta}}}, text={}, description={}}

\newglossaryentry{relfg}{name={\ensuremath{\mathbf{F}^{rel}}}, text={}, description={}}

\newglossaryentry{frmpi}{name={\ensuremath{\mathbf{T}_\pi}}, text={}, description={}}

\newglossaryentry{gampi}{name={\ensuremath{\gamma}}, text={}, description={}}

\newglossaryentry{grcurve}{name={\ensuremath{\mathbf{G}_\mathcal{C}}}, text={}, description={}}

\newtheorem{thm}{Theorem}[section]
\newtheorem{cor}[thm]{Corollary}
\newtheorem{lem}[thm]{Lemma}

\newtheorem{prop}[thm]{Proposition}
\newtheorem{conj}[thm]{Conjecture}
\theoremstyle{definition}
\newtheorem{defn}[thm]{Definition}
\newtheorem{eg}[thm]{Example}
\theoremstyle{remark}

\newcommand{\mlg}[2]{{\mathcal{M}_{#1 , #2}}}

\newcommand{\tgroup}[1]{\mathbb{G}_{#1}} 
\newcommand{\stackmap}[2]{\tilde{#1}_{#2}} 
\newcommand{\ver}[1]{{{#1}_v}}
\newcommand{\nver}[1]{{{#1}_{nv}}}

\newcommand{\univ}[1]{\mathbf{O}_{#1}(1)}
\newcommand{\linsys}[1]{{\mcl_{#1}}} 
\newcommand{\simplex}[2]{\Delta^{#1}_{#2}}
\newcommand{\dul}[1]{{\overline{#1}}}
\newcommand{\du}[1]{{\bar{#1}}} 

\newcommand{\orb}[1]{\textnormal{orb}_{#1}} 

\newcommand{\psec}[1]{{\Sigma (#1 )}}
\newcommand{\psecv}[2]{{\Sigma_{#2} ( #1)}} 
\newcommand{\plaf}[1]{{\Theta_p ( #1 )}} 
\newcommand{\ptlaf}[1]{{\Theta (#1 )}} 
\newcommand{\fans}[1]{{\mcf_{\Sigma (#1 )}}} 
\newcommand{\fanl}[1]{{\mcf_{\plaf{#1}}}} 
\newcommand{\fanlt}[1]{{\mcf_{\ptlaf{#1}}}} 

\newcommand{\fan}{{\Sigma}} 
\newcommand{\sfan}{{\mathbf{\Sigma }}} 

\newcommand{\laf}[1]{{\mathcal{X}}_{\ptlaf{ #1 }}} 
\newcommand{\tlaf}[1]{{\mathcal{X}_{\plaf{#1}}}} 
\newcommand{\secon}[1]{\mathcal{X}_{\Sigma (#1 )}} 

\newcommand{\hyp}[1]{\mathcal{Y}_{#1}} 
\newcommand{\fib}[2]{\mathcal{Z}_{#1} (#2 )} 

\newcommand{\detl}[1]{\textnormal{Det} (#1 )}

\newcommand{\agroup}{\mathbf{G}}

\newcommand{\one}{\mathbf{1}} 
\newcommand{\reduce}[1]{{{#1}^{red}}}
\newcommand{\relaxed}[1]{{{#1}^{rel}}}
\newcommand{\partrans}{{\mathbf{P}}}
\newcommand{\partran}{{\mathbf{P}_\infty}}

\newcommand{\frmee}{\mathbf{T}_\varepsilon}
\newcommand{\frme}{\mathbf{T}}

\newcommand{\maxint}{D_{[m]}} 
\newcommand{\disc}{B} 
\newcommand{\divi}{\tilde{D}} 
\newcommand{\curve}{\mcc} 

\newcommand{\vc}{L}
\newcommand{\vt}{T}
\newcommand{\rotat}{\rho }
\newcommand{\core}[1]{\textnormal{Core} (#1 )}

\newcommand{\interior}[2]{\mci^{#2}_{#1}}
\newcommand{\boundflow}{{\tau}}

\newcommand{\gdual}[1]{{#1}^{\vee}}

\newcommand{\ddata}[1]{\mathbf{M}_{#1}} 

\newcommand{\simp}[1]{\textnormal{Simp} (#1 )}

\newcommand{\aff}{{\textnormal{Aff}}}
\newcommand{\lin}{{\textnormal{Lin}}}
\newcommand{\convhull}{{\textnormal{Conv}}}
\newcommand{\Symp}{{\textnormal{Symp}}}
\newcommand{\ext}{{\textnormal{Ext}}}
\newcommand{\rank}{{\textnormal{rk}}}
\newcommand{\Vol}{{\textnormal{Vol}}}
\newcommand{\Hom}{{\textnormal{Hom}}}
\newcommand{\Pic}{{\textnormal{Pic}}}
\newcommand{\Tor}{{\textnormal{Tor}}}
\newcommand{\Divisor}{{\textnormal{Div}}}
\newcommand{\image}{{\textnormal{im}}}
\newcommand{\Core}{{\textnormal{Core}}}
\newcommand{\Cone}{{\textnormal{Cone}}}

\newcommand{\Hess}{{\textnormal{Hess}}}

\newcommand{\R}{\mathbb{R}}
\newcommand{\C}{\mathbb{C}}
\newcommand{\Q}{\mathbb{Q}}
\newcommand{\N}{\mathbb{N}}
\newcommand{\p}{\mathbb{P}}

\newcommand{\mbbT}{\mathbb{T}}

\newcommand{\Z}{\mathbb{Z}}

\newcommand{\mca}{\mathcal{A}}
\newcommand{\mcb}{\mathcal{B}}
\newcommand{\mcc}{\mathcal{C}}
\newcommand{\mcd}{\mathcal{D}}
\newcommand{\mce}{\mathcal{E}}
\newcommand{\mcf}{\mathcal{F}}
\newcommand{\mcg}{\mathcal{G}}

\newcommand{\mci}{\mathcal{I}}
\newcommand{\mcj}{\mathcal{J}}

\newcommand{\mcl}{\mathcal{L}}
\newcommand{\mco}{\mathcal{O}}

\newcommand{\mcr}{\mathcal{R}}
\newcommand{\mcs}{\mathcal{S}}
\newcommand{\mct}{\mathcal{T}}
\newcommand{\mcu}{\mathcal{U}}
\newcommand{\mcv}{\mathcal{V}}
\newcommand{\mcw}{\mathcal{W}}
\newcommand{\mcx}{\mathcal{X}}
\newcommand{\mcy}{\mathcal{Y}}
\newcommand{\mcz}{\mathcal{Z}}

\newcommand{\mbG}{\mathbf{G}}

\newcommand{\mbF}{\mathbf{F}}

\newcommand{\mft}{\mathfrak{t}}

\begin{document}

\title[Symplectic relations and degenerations of LG models]{Symplectomorphism group relations and degenerations of Landau-Ginzburg models}

\author[C. Diemer]{Colin Diemer}
\address{Department of Mathematics, University of Miami, Coral Gables, FL, 33146, USA
\indent Fakult\"at f\"ur Mathematik , Universit\"at Wien, 1090 Wien, Austria}
\email{diemer@math.miami.edu}

\author[L. Katzarkov]{Ludmil Katzarkov}
\address{Department of Mathematics, University of Miami, Coral Gables, FL, 33146, USA 
\indent Fakult\"at f\"ur Mathematik , Universit\"at Wien, 1090 Wien, Austria}
\email{l.katzarkov@math.miami.edu}


\author[G. Kerr]{Gabriel Kerr}
\address{Department of Mathematics, Kansas State University, Manhattan, KS, 66502, USA
\indent Fakult\"at f\"ur Mathematik , Universit\"at Wien, 1090 Wien, Austria}
\email{gdkerr@math.miami.edu}

\subjclass[2010]{Primary 53D37; Secondary 53D05}




\begin{abstract}
We describe explicit relations in the symplectomorphism groups of hypersurfaces in toric stacks. To define the elements involved, we construct a proper stack of these hypersurfaces whose boundary represents stable pair degenerations. Our relations arise through the study of the one dimensional strata of this stack. The results are then examined from the perspective of homological mirror symmetry where we view sequences of relations as maximal degenerations of Landau-Ginzburg models. We then study the $B$-model mirror to these degenerations, which gives a new mirror symmetry approach to the minimal model program. 
\end{abstract}

\maketitle

\section{Introduction}


The mapping class groups of punctured Riemann surfaces have been studied from a variety of perspectives for many years. Following the ideas of Hatcher, Thurston and others, Wajnryb gave a finite presentation for these groups \cite{Wajnryb}. Generalizations of these results to diffeomorphism groups in higher dimensions is much less tractable; moreover, if the manifold is equipped with a symplectic structure, there exist subtle distinctions between the group of diffeomorphisms and the group of symplectomorphisms \cite{seidelsymp}. However, by considering symplectic manifolds in the context of toric or tropical geometry, structures which produce meaningful relations in the symplectomorphism group arise. This paper aims to introduce a systematic approach for studying such generators and relations in appropriate symplectomorphism groups, valid in all dimensions.

Let $\mathcal{Y}$ denote a suitably generic hypersurface in a toric variety (or toric orbifold) $\mathcal{X}$. Note that $\mathcal{Y}$ has a  boundary divisor $\partial \mathcal{Y}$ obtained by the intersection with the toric boundary, and $\mathcal{Y}$ may be viewed as a symplectic orbifold $(\mathcal{Y} ,\omega )$ if $\mathcal{X}$ is itself equipped with a symplectic form. We then introduce generators and relations for a subgroup $\agroup$ of the symplectic mapping class group $\pi_0 (\Symp (\mathcal{Y} , \partial \mathcal{Y}))$. 
Our method is to consider a stack $\mathcal{V}$ whose points correspond to (orbifold) smooth hypersurfaces $\mcy$ moving in a fixed linear system on $\mcx$ and which obey appropriate transversality conditions with respect to the toric boundary of $\mcx$. 
We find a symplectic connection on the universal hypersurface  over $\mathcal{V}$ and employ symplectic parallel transport:
\begin{equation} \label{eq:paralleltransport}
\partrans : \Omega_* ( \mcv ) \to \Symp (\mathcal{Y} , \partial \mathcal{Y}).
\end{equation}
where $\Omega_* $ is the based loop space.

Taking the group $\agroup = \pi_0 (\image (\partrans ))$, we find generators and relations by studying them in $\pi_1 (\mathcal{V})$. The moduli space $\mathcal{V}$ is constructed following the techniques of Alexeev  \cite{alexeev02}, and is studied via the combinatorial methods of Gelfand, Kapranov, Zelevinsky \cite{GKZ}. 

Let $Q$ be an integral polytope and $A\subset Q$ a subset of lattice points such that the convex hull of $A$ is $Q$. This data defines a linear system on the toric variety $X_Q$. We construct a toric stack $\secon{A}$ which has the affine toric DM substack $\mcv_A$ containing $\mcv$. In fact, $\mathcal{V}$ arises as the complement of a particular discriminant locus in $\mcv_A \subset \secon{A}$. Unfortunately, there are precious few cases where the fundamental groups of complements of discriminant loci are completely understood (see e.g. \cite{DolgachevLibgober, Lonne}). We bypass this difficulty by considering only the one-dimensional strata of the toric boundary of $\secon{A}$. Combinatorially, these strata correspond to the circuits of $A$ \cite{GKZ}. 

In the case of curves, the generators of the mapping class group can be taken to be Dehn twists and braids. One expects a more complicated set of generators to occur in higher dimensions. Indeed, the generators we obtain fall into two different classes: hypersurface degeneration monodromy and stratified Morse function monodromy. The former refers to monodromy around the points in $\partial \secon{A}$. Combinatorially, this means monodromy obtained from degenerations of hypersurfaces which correspond to subdivisions of $Q$.  This monodromy was studied in the case of a maximal triangulation by Abouzaid \cite{abouzaid} in terms of tropical geometry. The geometric description of these symplectomorphisms is obtained by first breaking the hypersurface up into its degenerated components and then convolving along the degenerating vanishing cycle to obtain a global map. For curves, this amounts to a combination of a Dehn twist and a finite order map. The other generators corresponding to stratified Morse function monodromy arise from monodromy around the discriminant locus in $\secon{A}$. 
The local model for monodromy here is a generalization of the usual monodromy around a Morse singularity to that around a stratified Morse singularity as defined in Goresky and Macpherson's work \cite{goresky}. 
Its description is that of a generalized braid about a Lagrangian submanifold which is a join of a sphere and simplex. This gives twists about Lagrangian discs and balls, as well as other interesting 
joins, and thus are generalizations of Seidel's symplectic Dehn twists about Lagrangian spheres \cite{seideldehn}. We emphasize that these generalized spherical twists come from vanishing loci which are not topological spheres, but which actually appear to be quite natural from the contributing toric geometry.

We summarize the above discussion with the following abridged version of Theorem~\ref{thm:circuitrelation} in Section~\ref{sec:circuit1}.

\begin{thm}
  Let $A$ be a circuit affinely spanning $\Z^d$, $X_A$ the associated toric stack, and $\mathcal{Y}\subset X_A$  a general hypersurface in the the linear system given by $A$. Then there are symplectomorphisms $T_0, T_1, T_\infty \in \Symp (\mathcal{Y} , \partial \mathcal{Y} )$ with $T_0$ and $T_\infty$ hypersurface degeneration monodromy maps and $T_1$ the monodromy about a stratified Morse singularity. In the mapping class group $\pi_0 (\Symp (\mathcal{Y}, \partial \mathcal{Y}))$, these satisfy the relation:
\begin{equation} T_0 T_1 T_\infty = \boundflow (\mathbf{t} ) \end{equation}
where $\boundflow (\mathbf{t})$ is a rotation about the boundary $\partial \mathcal{Y}$.
\end{thm}

For brevity, the above theorem was only stated for the case of a circuit itself; this is a very small class of toric varieties. In order to put the generators and relations into a symplectomorphism group of any smooth hypersurface in a toric variety, we address the process of regeneration of circuits. This allows us to import relations obtained from the one-dimensional boundary strata of $\secon{A}$ into the interior, and thus to study the topology of general hypersurfaces in toric varieties. 
In this way, we obtain a host of geometrically meaningful relations between generators in the subgroup $\agroup$. More specifically, by taking a general map $\phi : \p^1 \to \secon{A}$ and pulling back the universal hypersurface gives a framed Lefschetz pencil over $\p^1$. We describe a presentation of the monodromy group associated to such pencils by performing an isotopy near the boundary of $\secon{A}$ and relating the bubbled components to circuits. This gives a combinatorial description not only of the groups involved, but their action on the hypersurface.

A supplemental goal of this work is to study these ideas in the context of homological mirror symmetry, and more specifically, to give applications to the study of Landau-Ginzburg (short: LG) models and their $A$-model Fukaya-Seidel categories. The mirrors of Fano toric varieties are open subsets of certain pencils of hypersurfaces in toric varieties \cite{givental94, horivafa}. Our perspective takes a fiberwise compactification of such a LG model as a curve $i: \mathcal{C} \to \secon{A}$.  More precisely, the mirrors of Fano toric varieties which arise from the Hori-Vafa construction are obtained as compactifications of one parameter torus orbits in $\secon{A}$. Following results of \cite{KSZ}, we observe that the coarse moduli space of these LG models has a natural compactification as a toric variety whose moment polytope is the monotone path polytope of $\psec{A}$ \cite{BS1, BS2}. 
The vertices of the monotone path polytope of $\Sigma (A)$ correspond to particular sequences of circuits on $A$ or equivalently to sequence of edges on $\psec{A}$. One main application of our work is to use any such  sequence to describe an associated semi-orthogonal decomposition of the Fukaya-Seidel category of the LG model. 

For the mirror description, i.e. the corresponding structure on a mirror toric variety, this semi-orthogonal decomposition complements recent developments in the study of derived categories of toric varieties. Work of Bondal-Orlov \cite{bondal} and Kawamata \cite{kawamataderived} demonstrated relations between birational transformations coming from the minimal model program and semi-orthogonal decompositions. One goal of this paper is thus to supply a mirror $A$-model interpretation of Kawamata's work.  In the toric case, the equivariant birational geometry is well-understood combinatorially, going back to the work of Reid \cite{Reid}, and is also dictated by the combinatorics of secondary polytopes. We show concretely that degenerations of LG models mirror to a toric variety $\mcx_\Sigma$ correspond bijectively to certain runs of the minimal model program for $\mcx_\Sigma$. The particular runs are those given by running the minimal model program with scaling. As a consequence we obtain a concise description of the mirrors of toric flips and toric divisorial contractions in terms of circuits. We conjecture that there is an equivalence of categories which restricts to this identification of semiorthogonal components, giving a clear picture of the geometry underlying homological mirror symmetry for toric DM stacks. We give evidence for this conjecture by computing ranks in $K$-theory, extending results of Borisov-Horja \cite{borisovhorja}. 

We summarize the relationship between the  minimal model sequences of $\mcx_\Sigma$ and the mirror $A$-model LG degenerations on $\mcx_\Sigma^{mir}$ in Theorem~\ref{thm:mmps}, which in simple cases reduces to the following statement.

\begin{thm} The set of regular minimal model sequences for a Fano toric stack $\mcx_\Sigma$ are in bijective correspondence with the set of maximal degenerations of the LG models on the mirror stack $\mcx_\Sigma^{mir}$. Both are in bijective correspondence with vertices of a monotone path polytope $\Sigma_{\rho} (\psec{A})$.
\end{thm}


{\em Acknowledgements:} Support was provided by NSF Grant DMS-0901330, NSF FRG Grant DMS-0854977, FWF
Grant 24572-N25, and an ERC GEMIS Grant. We would like to thank Denis Auroux, Matthew Ballard, Alessio Corti, David Favero, Paul Horja, Maxim Kontsevich, Tony Pantev and Paul Seidel for helpful comments and conversation during the preparation of this work.

\section{The circuit relation} 

This section will give one main result of this paper which is a
detailed description of a class of relations that occur in the
symplectic mapping class groups of hypersurfaces in a toric stacks. These relations involve a combination of stable pair degeneration monodromy and twisting about a stratified Morse singularity, both of whose local models are investigated in Appendix~\ref{sec:pfsymp}. After stating the relation, we work through three examples in dimension $1$. Finally, we conclude with a brief investigation of regenerations which incorporate various relations into a finite presentation.

\subsection{\label{sec:circuit1} Circuit stacks}
This section will be concerned with establishing a relation between certain elements of the mapping class group of a hypersurface $\mcz$ in the toric stack $\mcx_Q$ where the $Q$ is the convex hull of what is known as an affine circuit $A$. The elements in this relation arise as monodromy transformations around singular values of a function $\pi$. This function appears naturally as the universal hypersurface over a moduli stack of hypersurfaces. In particular, in Appendix~\ref{sec:toricstack} we define a compactified moduli space $\secon{A}$ of hypersurfaces in $\mcx_Q$ and a total space $\laf{A}$ with a universal hypersurface $\hyp{A}$. The stacks $\secon{A}$ and $\laf{A}$ are both toric and are referred to as the secondary and Lafforgue stacks respectively. There is an equivariant toric morphism $\pi : \laf{A} \to \secon{A}$ which restricts to the universal hypersurface $\pi : \hyp{A} \to \secon{A}$. The fibers of this map represent hypersurfaces associated to sections of a natural line bundle $\mco_A (1)$ over $\mcx_Q$, and their degenerations. As we will see, there are three critical points around which symplectic parallel transport yields interesting symplectomorphisms of the fiber.  We will refer to Appendices~\ref{sec:toric} and \ref{sec:pfsymp} for the important details concerning the construction and properties of toric moduli stacks and symplectic mapping class groups respectively.

We begin by recalling the definition and basic properties of a circuit
from \cite[Chapter~7.1.B]{GKZ} and detailing the Lafforgue and secondary stack of
a circuit. The map $\pi : \laf{A} \to \secon{A}$ will be also be
reexpressed in concrete terms and its monodromy will be studied. For what follows, we will assume that $\Lambda \cong \Z^d$ is a
rank $d$ affine lattice.

\begin{defn} A circuit $A \subset \Lambda$ is an affinely dependent set such that every proper subset is affinely independent. \end{defn}

We will say that a subset $A \subset \Lambda$ has rank $r$ if $\rank (\aff_\Z (A)) = r$ where $\aff_\Z (A)$ is the integral affine span of $A$. A circuit is non-degenerate if its rank equals that of $\Lambda$. In what follows, we will consider both non-degenerate and degenerate circuits.

\begin{figure}[t]
\begin{picture}(0,0)%
\includegraphics{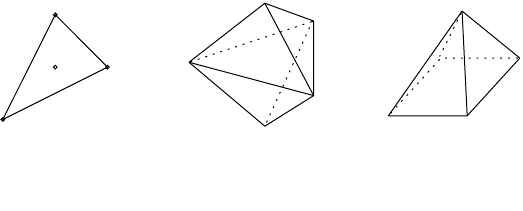}%
\end{picture}%
\setlength{\unitlength}{4144sp}%
\begingroup\makeatletter\ifx\SetFigFont\undefined%
\gdef\SetFigFont#1#2#3#4#5{%
  \reset@font\fontsize{#1}{#2pt}%
  \fontfamily{#3}\fontseries{#4}\fontshape{#5}%
  \selectfont}%
\fi\endgroup%
\begin{picture}(3974,1591)(967,-1244)
\put(789,-950){\makebox(0,0)[lb]{\smash{{\SetFigFont{10}{12.3}{\rmdefault}{\mddefault}{\updefault}{\color[rgb]{0,0,0}$\sigma_A=(1,3)$}%
}}}}
\put(789,-1178){\makebox(0,0)[lb]{\smash{{\SetFigFont{10}{12.3}{\rmdefault}{\mddefault}{\updefault}{\color[rgb]{0,0,0}$\mathbf{c}=(3,-1,-1,-1)$}%
}}}}
\put(2276,-954){\makebox(0,0)[lb]{\smash{{\SetFigFont{10}{12.3}{\rmdefault}{\mddefault}{\updefault}{\color[rgb]{0,0,0}$\sigma_A=(2,3)$}%
}}}}
\put(3847,-954){\makebox(0,0)[lb]{\smash{{\SetFigFont{10}{12.3}{\rmdefault}{\mddefault}{\updefault}{\color[rgb]{0,0,0}$\sigma_A=(2, 2; 1)$}%
}}}}
\put(3847,-1169){\makebox(0,0)[lb]{\smash{{\SetFigFont{10}{12.3}{\rmdefault}{\mddefault}{\updefault}{\color[rgb]{0,0,0}$\mathbf{c}=(-1,-1,1,1,0)$}%
}}}}
\put(2284,-1177){\makebox(0,0)[lb]{\smash{{\SetFigFont{10}{12.3}{\rmdefault}{\mddefault}{\updefault}{\color[rgb]{0,0,0}$\mathbf{c}=(3,3,-2,-2,-2)$}%
}}}}
\end{picture}%
\caption{\label{fig:circ0} Examples of extended circuits.}
\end{figure}
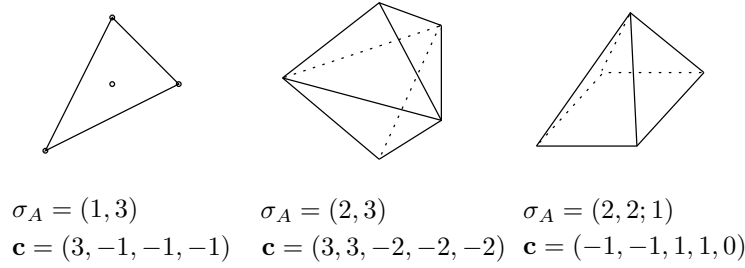

\begin{defn} An extended circuit is a subset $A \subset \Lambda$ such that $|A| = d + 2$ and $\rank (\aff (A)) = d$.
\end{defn}
Alternatively, an extended circuit is an affinely spanning subset
$A = \{a_0, \ldots, a_{d + 1} \}$ whose lattice of
affine relations has rank $1$, generated by $\mathbf{c} = (c_0,
\ldots, c_{d + 1} ) \in \Z^{d + 2}$ where
\begin{align} \label{eq:affrelation} 
\begin{split}
\sum_{i = 0}^{d + 1} c_i a_i & = 0 , \\
\sum_{i = 0}^{d + 1} c_i & =  0 .
\end{split}
\end{align}
Given the relation \eqref{eq:affrelation}, we may write $A$ as the disjoint union\gls{apm} $A =
A_- \cup A_0 \cup A_+$ where $a_i \in A_\pm$ if and only if $\pm c_i >
0$ and $a_i \in A_0$ if and only if $c_i = 0$. The signature of an
extended circuit is defined to be\gls{circsign} $\sigma (A) = (|A_+|, |A_-|; |A_0| )$. When $A$ is a circuit, it is clear
that $|A_0| = 0$ and we then write $\sigma (A) = (|A_+| , |A_-
|)$. The signature does depend on the sign of $\mathbf{c}$
up to transposing $|A_+|$ and $|A_-|$.

We will call a marked polytope $(Q, A)$ a circuit, or an extended circuit, if $A$ is one. Our convention is not to take $\mathbf{c}$ as a primitive element, but to force the greatest common divisor of the $c_i$ to be $|K_{{\mathcal{A}}}|$, where $K_{{\mathcal{A}}}$ is defined in equation \eqref{eq:Aex}. This then implies the volume of $Q$ is 
\begin{equation*}\gls{vola} v_A := \Vol (Q) = \pm \sum_{a_i \in A_\pm } c_i. \end{equation*}
where we normalize the volume of the standard simplex to $1$.

We note that an extended circuit is not generally a circuit unless $A_0 = \emptyset$. This motivates the following definition.
\begin{defn} \label{defn:core}\gls{corea}
	The core of an extended circuit $A$ is the circuit 
	\begin{align*} \Core (A) :=  A_+ \cup A_-. \end{align*}
\end{defn}

To any extended circuit $A$, there are precisely two regular triangulations $T_\pm$ of $(Q, A)$ as in Definition~\ref{defn:regsub}. These are given by
\begin{equation}\label{eq:triang}\gls{tpm} T_\pm = \{ \convhull (A - \{a_i\} )\}_{a_i \in A_\pm} . \end{equation}
The union of the vertices of the simplices in $T_\pm$ equals $A$ unless $|A_+| = 1$ or $|A_- | = 1$, in which case the respective triangulation is marked by
$A - A_\pm$. 

While we will deal with the geometry of extended circuits $(Q, A)$ in isolation for most of this and the next section, the primary reason for us to investigate them is how they relate to a larger marked polytope $(\mathbf{Q} , \mathbf{A})$ containing $(Q, A)$. The key fact in this regard is that every edge of the secondary polytope $\psec{\mathbf{A}}$ from equation~\eqref{eq:secondef} corresponds to a circuit modification. We recall this theorem and the necessary definitions from \cite{GKZ}. 

\begin{defn} \label{defn:cirsupp} Let $\mathbf{T}$ be a triangulation of $(\mathbf{Q} , \mathbf{A})$ and $A \subset \mathbf{A}$ a circuit with triangulations $T_\pm$. We say that $\mathbf{T}$ is positively (resp. negatively) supported on $A$ if the following conditions hold:
\begin{enumerate}[label=(\roman*), ref=\thedefn(\roman*)]
\item \label{defn:cirsupp:1} $T_+$ (resp. $T_-$ ) consists of faces of simplices in $\mathbf{T}$. 
\item \label{defn:cirsupp:2} For every $J \subset  \mathbf{A}$, if $\sigma \in T_+$ (resp. $T_-$ ) with $J \cap \sigma = \emptyset$ and $J \cup \sigma$  a maximal simplex of $\mathbf{T}$ then $J \cup \sigma^\prime \in \mathbf{T}$ for every $\sigma^\prime \in T_+$ (resp. $T_-$ ).
\end{enumerate}
For any $J$ satisfying (ii), we say that $J \cup A$ is a separating extended circuit of $\mathbf{T}$. 
\end{defn}

If $\mathbf{T}$ is positively supported on $A$, then one may define a new triangulation $m_A (\mathbf{T} ) := \mathbf{T}^\prime$ which is negatively supported on $A$ by changing the triangulations of every separating extended circuit. Such a change is referred to as a circuit modification along $A$.

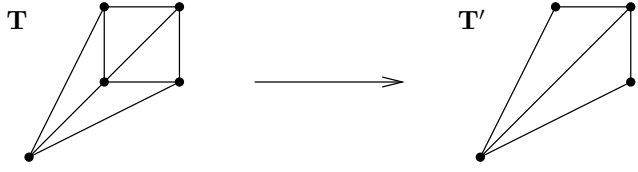
\begin{figure}[t]
\begin{picture}(0,0)%
\includegraphics{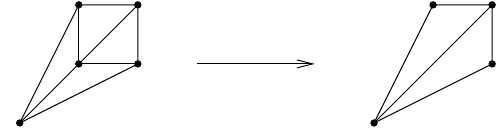}%
\end{picture}%
\setlength{\unitlength}{4144sp}%
\begin{picture}(3782,962)(2551,-1442)
\put(2566,-646){\makebox(0,0)[lb]{\smash{$\mathbf{T}$}}}
\put(5266,-646){\makebox(0,0)[lb]{\smash{$\mathbf{T}^\prime$}}}
\end{picture}%
\caption{\label{fig:triangulations} Circuit modification along $A$.}
\end{figure}

\begin{thm}\cite[Theorem~7.2.10]{GKZ}\label{thm:gkz3} Let $\mathbf{T}$ and $\mathbf{T}^\prime$ be two regular triangulations of $(\mathbf{Q} , \mathbf{A})$. The vertices $\varphi_\mathbf{T} , \varphi_{\mathbf{T}^\prime} \in \psec{\mathbf{A}}$ are joined by an edge if and only if there is a circuit $A \subset \mathbf{A}$ such that $\mathbf{T}$ is supported on $A$ and $\mathbf{T}^\prime = m_A (\mathbf{T} )$. 
\end{thm}

\begin{eg} \label{eg:dplus3}
Let 
\begin{align*} \mathbf{A} =  \{ (1, 0), (0, 1), (1, 1) , (-1, -1), (0, 0) \} \end{align*} and $\mathbf{Q}$ be its convex hull. As an extended circuit in $\Z^2$ must have $4$ elements, we see that $\mathbf{A}$ contains five extended circuits, namely the $4$ element subsets of $\mathbf{A}$. However, $\mathbf{A}$ only contains four circuits, as $A := \{(1, 1) , (-1, -1), (0, 0)\}$ is the core of both $A_1 := \{ (1, 0),  (1, 1) , (-1, -1), (0, 0) \}$ and $A_2 := \{ (0, 1), (1, 1) , (-1, -1), (0, 0) \}$. Choosing $\mathbf{c} = (-1, -1, 2)$ for the affine relation of $A$, the two triangulations of the interval $A$ are given by $T_-$ which breaks $A$ into two intervals and $T_+$ which is all of $A$, but with the marked set $\{  (-1, -1), (1, 1) \}$.
Consider the regular triangulations $\mathbf{T}$ and $\mathbf{T}^\prime$ illustrated in  Figure~\ref{fig:triangulations}. The triangulation $\mathbf{T}$ is negatively supported on $A$, while the triangulation $\mathbf{T}^\prime$ is positively supported on $A$. Clearly $J_1 := \{(1, 0)\}$ and $J_2 := \{(0, 1)\}$ satisfy Definition \ref{defn:cirsupp:2} so that both extended circuits $A_1$ and $A_2$ are separating. To see the secondary polytope of $\mathbf{A}$, the remaining circuits and their modifications, we refer the reader to  Figure~\ref{fig:secF1}.
\end{eg}
In equation~\eqref{eq:princadet}, we defined the principal $A$-determinant $E_A$ on the linear system of sections $\linsys{A} \subset \Gamma (\mcx_Q, \mco_A (1))$. This polynomial vanishes on elements of $\linsys{A}$ whose zero locus intersects an orbit $\orb{F}$ non-transversely for some face $F$ of $Q$. In Definition~\ref{defn:secstack}, we extend $E_A$ to a section $E^s_A$ of a line bundle over the secondary stack $\secon{A}$ with zero loci $\mce_A$. Now, the edges of the secondary polytope correspond to one dimensional orbits of $\secon{A}$. Thus Theorem~\ref{thm:gkz3} indicates that if we aim to understand the symplectic monodromy of a hypersurface as we loop around $\mce_\mathbf{A} = \{E^s_\mathbf{A} = 0\}$, it is a reasonable first step to understand the monodromy around circuits, extended circuits and, more generally, circuit modifications. 

We now take a moment to study basic properties of the toric stack $\mcx_Q$ associated to an extended circuit by investigating the normal fan to $Q$.
\begin{defn} Suppose $\Gamma =
\Gamma_1 \oplus \Gamma_2$ where the rank of $\Gamma_i$ is $d_i$
and $(Q_i , A_i)$ are marked polytopes in $\Gamma_i \otimes \R$. If $(Q_1, A_1)$ is a $(d_1 -1)$-dimensional simplex which does not contain $0$, we say
that 
\begin{align*} \left(\convhull\left( ( Q_1 \times \{0\}) \cup (\{0 \} \times Q_2 ) \right) , (A_1 \times \{0\} )\cup (\{0 \} \times A_2 ) \right)
\end{align*} is a $d_1$-simplicial extension of $(Q_2, A_2)$. 
\end{defn}
Combinatorially, a $d_1$-simplicial extension of $(Q_2, A_2)$ is the same as the join of $Q_2$ with a $(d_1 -1)$-simplex. Note that an extended circuit $A$ of signature $(p, q; r)$ is an $r$-simplicial extension of its core. 

\begin{eg}
Take $A_1 = \{-1, 0, 1\}$ and $A_2 = \{(1,0), (0,1)\}$ with $Q_1$ and $Q_2$ their respective convex hulls. Then the tetrahedron illustrated in Figure~\ref{fig:simpext} is a simplicial extension of the interval $(Q_1, A_1)$.
\end{eg}
\begin{figure}[t]
\begin{picture}(0,0)%
\includegraphics{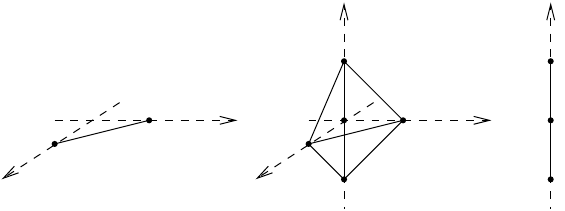}%
\end{picture}%
\setlength{\unitlength}{4144sp}%
\begin{picture}(4302,1599)(1204,-1873)
\put(1756,-916){\makebox(0,0)[lb]{\smash{$(Q_1, A_1)$}}}
\put(4096,-646){\makebox(0,0)[lb]{\smash{Simplical}}}
\put(4096,-871){\makebox(0,0)[lb]{\smash{Extension}}}
\put(5491,-916){\makebox(0,0)[lb]{\smash{$(Q_2, A_2)$}}}
\end{picture}%
\caption{\label{fig:simpext} A $2$-simplicial extension of $Q_2$.}
\end{figure}
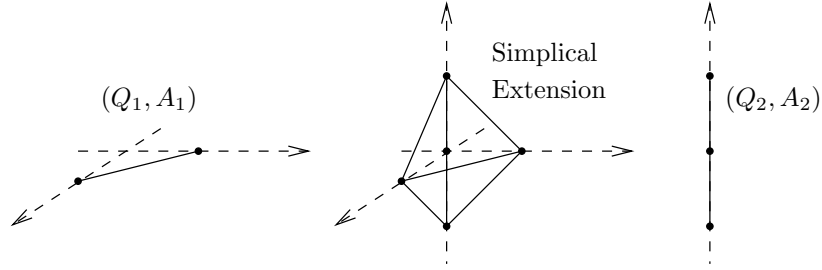
If both $(Q_1, A_1)$ and $(Q_2, A_2)$ are $d_1$ and $d_2$-dimensional simplices in $\Gamma_\R$ which span complementary affine subspaces, we say $(Q_1 + Q_2 , A_1 + A_2)$ is a $(d_1, d_2)$-prism. We also recall some terminology from convex polytopes. Given a polytope $P \subset \Gamma_\R$ which contains $0$ in its interior, its polar dual is the polytope 
\begin{align*}
P^\circ := \left\{ u \in \Gamma_\R^\vee : \left< u, v \right> \geq -1 \text{ for all } v \in P \right\}.
\end{align*}

\begin{prop} \label{prop:dualfan1} 
 If $(Q,A)$ is an extended circuit with signature $(p, q; r)$, there exists $u_A \in \Gamma_\R$ such that the polar dual polytope $(Q - u_A)^\circ$ to the translation $(Q - u_A)$ of $Q$ is an $r$-simplicial extension of a $(p - 1, q - 1)$-prism.
\end{prop}
\begin{proof} We begin with the case of a non-degenerate circuit $A$. We claim that, in this case, any facet  of $Q$ arises as the convex hull $F_{ij} = \convhull (A - \{a_i , a_j\})$ where $a_i \in A_+$ and $a_j \in A_-$. To see this, first observe that every such $F_{ij}$ is a facet which is clear from the description of the triangulations $T_\pm$ in equation~\eqref{eq:triang}. Conversely, observe that the element
\begin{equation*} u_A := \frac{1}{v_A} \sum_{a_i \in A_+} c_i a_i =  - \frac{1}{v_A} \sum_{a_j \in A_-} c_j a_j \end{equation*}
lies in the convex hull of both $A_+$ and $A_-$ and the interior of $Q$. Thus no facet $F$ can contain $A_+$ or $A_-$ which implies that there exists an $i$ and $j$ with $F_{ij} \subset F$. As every boundary facet of $Q$ is a simplex with vertices in $A$, this implies $F = F_{ij}$ for some $a_i \in A_+$ and $A_j \in A_-$.

Let $\tilde{A}_\pm = \{a - u_A : a \in A_\pm\}$ and $\Lambda_\pm = \lin_\R (\tilde{A}_\pm )$. It is obvious that the convex hulls $(\tilde{Q}_+, \tilde{A}_+)$ is a $(p - 1)$-simplex and $(\tilde{Q}_-, \tilde{A}_-)$ is a $(q - 1)$-simplex. We write $B_\pm = \{v : v (w) \geq -1 \text{ for } w \in \tilde{A}_\pm \}\subset \Lambda_\pm^\vee \otimes \R$ for their polar duals. Since $A$ affinely spans $\Lambda_\R$, we have that $\tilde{A}_\pm$ affinely span $\Lambda_\pm$ and that $\Lambda_+ + \Lambda_- = \Lambda_\R$. If $u \in \Lambda_+ \cap \Lambda_-$,  we have that there exists coefficients $r_i \in \R$ for $a_i \in A$ such that $\sum_{a_i \in A_\pm} r_i = 1$ and
\begin{align*}
\sum_{a_i \in A_+} r_i (a_i - u_A) & = u = \sum_{a_j \in A_-} r_{j} (a_j - u_A).
\end{align*} 
This implies
\begin{align*}
\sum_{a_i \in A_+} r_i  a_i - \sum_{a_j \in A_-} r_j  a_j = 0
\end{align*}
where the coefficients can be seen to add to zero. Since the affine relations of $A$ are generated by those  in equation \eqref{eq:affrelation}, this implies that there exists $\lambda \in \R$ such that $r_i = \pm \lambda c_i$ for every $a_i \in A_\pm$. Furthermore, the fact that $\sum_{a_i \in A_\pm} r_i = 1$ implies that $\lambda = 1/ v_A$.  But then 
\begin{align*} u  =  \sum_{a_i \in A_+}  \frac{c_i}{v_A} a_i - \left( \sum_{a_i \in A_+} r_i \right)u_A = 0. \end{align*}
Thus $\Lambda_\R = \Lambda_+ \oplus \Lambda_-$.

For $a_i \in A_+$ and $a_j \in A_-$, take $\tilde{F}_i^+$ and $\tilde{F}_j^-$ for the convex hulls of $\tilde{A}_+ - \{a_i - u_A\}$ and $\tilde{A}_- - \{a_j - u_A\}$ respectively. These form the facets of $\tilde{Q}_\pm$ and, from the description of the facets of $Q$ as $F_{ij}$, we see that  
\begin{align*}
F_{ij} - u_A = \{r v + s w : v \in \tilde{F}_i^+, w \in \tilde{F}_j^-, r + s = 1  \}.
\end{align*}
Now, if $b_i \in B_+$ and $b_j \in B_-$ are vertices dual to $\tilde{F}_i^+$ and $\tilde{F}_j^+$ respectively, then one easily sees that $b_i + b_j$ is constantly equal to $-1$ on $F_{ij} - u_A$. Thus the vertices of $B_- \oplus B_+$ are contained in the set of those of the polar polytope of $Q - u_A$, but as these define all facets of $Q - u_A$, their convex hull must equal $(Q - u_A)^\circ$.

Now, if $A$ has signature $(p, q; r)$ with $r \ne 0$, we take $\tilde{A}_0 = \{ a - u_A : a \in A_0\}$ and $\Lambda_0 = \lin_\R \{ a - u_A : a \in A_0\}$. Since $\tilde{A}_0$ is not full dimensional, there is no dual polytope in $\Lambda_0^\vee$, but we still have that $\Lambda \otimes \R = \Lambda_+ \oplus \Lambda_- \oplus \Lambda_0$ and we have that $\tilde{A}_0$ is a basis for $\Lambda_0$. If $B_0 \subset \Lambda_0^\vee$ denotes the negatives of the linear duals to $\tilde{A}_0$, then  the polar dual for $A$ is the simplicial extension $(B_+ \oplus B_- ) + B_0$.
\end{proof}

For later reference, we utilize the previous proposition to index the boundary facets of $Q$.

\begin{cor} If $(Q, A)$ is an extended circuit of signature $(p, q; r)$, then it has $pq + r$ facets $\du{Q} = \{b_{ij} : \alpha_i \in A_- , \alpha_j \in A_+\} \cup \{b_k : \alpha_k \in A_0 \}$.
\end{cor}
\begin{proof}
Suppose $A^\prime \subset \Gamma_1 \oplus \Gamma_2$ and $(Q^\prime, A^\prime)$ is a $d_1$-simplicial extension of $(Q_2, A_2)$ by $(Q_1, A_1)$. The vertices of $Q^\prime$ consist of the $d_1$ points in $A_1$ along with the vertices of $Q_2$. Thus the number of vertices of $Q^\prime$ equals $d_1$ plus the number of vertices of $Q_2$. Of course, as a $(p,q)$ prism is the Minkowski sum of a $(p - 1)$-simplex and $(q - 1)$-simplex on complementary subspaces, it has precisely $pq$ vertices. As the vertices of the polar polytope $Q^\circ$ index the facets of $Q$, we have the result.
\end{proof}

One important consequence of the proposition is that $\mcx_Q$ fails to be smooth as a stack unless the signature of $A$ has $p = 1$, $q = 1$ or is $(2, 2; r)$. Indeed, for a circuit $(Q, A)$, the maximal normal cones to $Q$ are cones over products of simplices, and are therefore not simplicial. Nevertheless, as $\mcx_Q$ is toric, the normal fan of $Q$ has a simplicial refinement. This follows from the elementary fact that any rational convex polyhedral cone supports a simplicial fan. Indeed, intersecting the cone with a hyperplane to obtain a codimension $1$ polytope, one can triangulate this polytope and take the fan which consists of cones over the simplices in this triangulation. This implies that $(\mcx_Q , \partial \mcx_Q)$ is a standard symplectic stack as discussed in  Definition~\ref{defn:resolvcol}. 

We now examine the secondary and Lafforgue stacks associated to
$A$ as defined in Appendix~\ref{sec:moduli}. The key ingredient leading to the definition of these stacks is the fundamental sequence \eqref{eq:Aex}. For the circuit $A \subset \Lambda$ and $\mathcal{A} = \{(a, 1): a \in A\}$, this is the exact sequence
\begin{align}\label{eq:circex} 0 \to \Z \stackrel{\alpha_{{\mathcal{A}}}}{\longrightarrow} \Z^{{\mathcal{A}}} \stackrel{\beta_{{\mathcal{A}}}}{\longrightarrow} \Lambda \oplus \Z \to K_{{\mathcal{A}}} \to 0 .
\end{align}
Here we have that $\beta_{\mathcal{A}} \left( e_{(a,1)} \right) = (a,1)$ for $a \in A$ and  
\begin{align} \alpha_{\mathcal{A}} (1) = \frac{1}{|K_{\mathcal{A}}|} \sum_{a_i \in A} c_i e_{(a_i,1)}. \end{align}

In concert with this sequence, we must examine the polyhedra $\psec{A}$, $\ptlaf{A}$ and $\plaf{A}$, all of which lie in $\mathbb{R}^{\mathcal{A}}$. Applying equation \eqref{eq:secvert}, the triangulations $T_\pm$ correspond to the vertices $\varphi_\pm \in \psec{A}$ given by
\begin{equation} \label{eq:circtrian}
\varphi_\pm  = v_A \sum_{i = 0}^{d + 1} e_i \mp \sum_{a_i \in A_\pm}  c_i e_i .
\end{equation}
So that, by equation \eqref{eq:secondef}, $\psec{A} = \convhull (\{\varphi_- , \varphi_+ \} )$. Thus the coarse space $\mathcal{X}_{\psec{A}}^r$ is isomorphic to $\p^1$. To obtain the secondary stack $\secon{A}$, we first study $\laf{A}$ and $\tlaf{A}$.
From Definition \ref{defn:moduli} the stacks $\laf{A}$ and $\tlaf{A}$ arise as modifications of the toric stacks defined from the polyhedra $\ptlaf{A}$ and $\plaf{A}$. In particular, $\tlaf{A}$ is given by the stacky fan
\begin{align} \widetilde{\sfan}_\plaf{A} = \left( \Z^{\dul{\plaf{A}}} , (\Z^A)^\vee , \tilde{\beta}_{\dul{\plaf{A}}} , \fan_{\plaf{A}} \right) \end{align}
where $\tilde{\beta}_{\dul{\plaf{A}}}$ is defined in equation \eqref{eq:deftildbeta}.

By equation \eqref{eq:morphtoproj}  and Definition \ref{defn:moduli}, the Lafforgue stack $\laf{A}$ comes equipped with a map to $\p^{|A| - 1}$ and the universal hyperplane section $\hyp{A} \subset \laf{A}$ is the pullback of $s = \sum_{i = 0}^{|A|} Z_i$. In the next proposition, we will see that this morphism can be thought of as the coarsening map from a weighted projective space along with a blow-down along codimension $2$-planes. To state the proposition, we first introduce some notation. Let 
\begin{align} \gls{lpm}
\ell_\pm & = \textnormal{lcm} \{c_i : a_i \in A_\pm \}, \\
\ell & = \textnormal{lcm} \left\{\frac{\ell_\pm}{c_i} : a_i \in A_\pm \right\} ,
\end{align}
and define the constants
\begin{align}\gls{tilci}
\tilde{c}_i & := \begin{cases} \frac{|c_i| \ell}{\ell_{\pm}} & \textnormal{ if } a_i \in A_\pm, \\ \ell & \textnormal{ if } a_i \in A_0. \end{cases}
\end{align}
To simplify our exposition, we will assume $K_{{\mathcal{A}}} = 0$ for the remainder of the section. For convenience, we also index the elements of $A$ so that $A = \{a_0, \ldots, a_{d + 1} \}$. 

\begin{prop} \label{prop:circuitlaf} Given an extended circuit
$A \subset \Lambda$ of signature $(p, q; r)$ for which $K_{{\mathcal{A}}} = 0$, $\laf{A}$ is a stacky blow-up of $\p (\tilde{c}_0, \ldots, \tilde{c}_{d + 1})$ along $pq$ codimension $2$ projective subspaces.
The universal line bundle $\univ{A}$ and section are the pullbacks of $\mco (\ell )$ and $s_A = \sum_{a_i \in A_\pm} Z_i^{\ell_\pm / |c_i|}  + \sum_{a_j \in A_0} Z_j$. 
\end{prop}
\begin{proof} We recall that $\plaf{A} \subset \R^A$ is a polyhedron of dimension $|A|$.
By Lemma \ref{lem:lafhyppart}, the supporting primitives defining the facets of $\plaf{A}$ can be partitioned
\begin{align*} \dul{\plaf{A}} = \{\varrho_A\} \cup \dul{\plaf{A}}^v \cup \dul{\plaf{A}}^h .\end{align*} \gls{vrho} 
Here $\varrho_A = \sum e_a^\vee$, the elements of $\dul{\plaf{A}}^v$ correspond to vertical hyperplanes and those of $\dul{\plaf{A}}^h$ correspond to horizontal hyperplanes. The former are indexed by pointed subdivisions $(S, A_p)$ for which $S$ is a coarse subdivision of $(Q,A)$. Since $A$ is an extended circuit, these are given by $\{(T_\pm, A - \{a\}): a \in A_\pm\}$. By Lemma \ref{lem:lafnormal:2} the primitive $\eta_{(T_\pm , A- \{a\})}$ in $\dul{\plaf{A}}^v $ corresponding to $(T_\pm, A - \{a\})$ must vanish on $A_\pm - \{a\}$. It is then simple to see that $\eta_{(T_\pm , A- \{a\})} = e_a^\vee$ and $\{e_a^\vee \}_{a \in A_+ \cup A_-} = \dul{\plaf{A}}^v$. 

While this gives the vertical primitives, equation \eqref{eq:deftildbeta} for $\tilde{\beta}_{\dul{\plaf{A}}}$ takes the basis element corresponding to $\eta_{(T_\pm , A- \{a_i\})}$ and sends it to $\bar{\eta}_{(T_\pm , A- \{a_i\})}$. The latter element can be expressed as $m_i e_i^\vee$ and must be a primitive $\Lambda$-defining function for the triangulation $T_\pm$ as defined in equation~\eqref{eq:lambdadef} (here we denote $e_{a_i} \in \Z^{{\mathcal{A}}}$ by $e_i$).  To obtain the coefficient $m_i$, first note that since $K_{{\mathcal{A}}} = 0$, ${\mathcal{A}}$ spans $\Lambda$. Thus if $a_j \in A_+ \cup A_-$ is not equal to $a_i$ and $\Lambda_{i,j} = \textnormal{Lin}_\Z \{(a,1) : a \in A - \{a_i, a_j\}\}$, then $(a_i, 1)$ and $(a_j, 1)$ generate $(\Lambda \oplus \Z) / \Lambda_{i,j} \cong \Z$. Using the isomorphism with $\Z$, denote the equivalence class $[(a_i, 1)]$  and $[(a_j , 1)]$ by $t_{i}$ and $t_{j}$ respectively and note that the $\Lambda$-defining function $\bar{\eta}_{(T_\pm , A- \{a_i\})} = m_i e_i^\vee$ for $T_\pm$ must satisfy $t_i \mid m_i$. Since they generate $\Z$, we have that $\gcd (t_i, t_{j}) = 1$. Also, since $|c_i|$ and $|c_j|$ are the normalized volumes of $\convhull (A - \{a_i\})$ and $\convhull (A - \{a_j\})$ respectively, it follows that the volume of $\convhull (A - \{a_i, a_j\})$ in $\Lambda_{i,j}$ is $d_{i,j} := \textnormal{gcd} (c_i, c_j)$. Consequently, $t_i = \pm c_i / d_{i, j} = \pm \textnormal{lcm} (c_i, c_{j}) /c_i$ and, as $m_i e_i^\vee$ is a primitive $\Lambda$-defining function for $T_\pm$, we have that $m_i = \ell_\pm / |c_i|$.

By Proposition \ref{prop:dualfan1}, we have that the primitive hyperplane supporting functions in $\du{Q}$ correspond to the facets $F_{ij} := \convhull (A - \{a_i, a_j\})$ where $a_i \in A_+$ and $a_j \in A_-$ along with the facets $F_k := \convhull (A - \{a_k\})$ where $a_k \in A_0$. Writing $b_{ij}$ and $b_k$ for the corresponding hyperplane primitives and appealing to Proposition \ref{lem:lafnormal:1} gives \begin{align*} \dul{\plaf{A}}^h =  \left\{c_{b_{ij}}^{-1} \beta_{\mathcal{A}}^\vee (b_{ij} , n_{b_{ij}} )  : a_i \in A_- , a_j \in A_+ \right\}  \cup \left\{c_{b_{k}}^{-1} \beta_{\mathcal{A}}^\vee  (b_k , n_{b_k}) : a_k \in A_0 \right\}. \end{align*} 
To compute $\tilde{\beta}_{\dul{\plaf{A}}}$, it suffices to find $c_{b_{ij}}$ and $c_{b_k}$. We first evaluate $c_{b_k}$ where $a_k \in A_0$. Let $\Lambda_k = \textnormal{Lin}_{\Z} ({\mathcal{A}} - \{(a_k , 1)\})$ and note that, due to the fact that $K_{{\mathcal{A}}} = 0$, $[(a_k, 1)]$ generates $\Lambda \oplus \Z / \Lambda_k$.  This implies that, while $b_k |_{F_k} = -n_k$ by definition, $b_k (a_k) = 1 - n_k$ so that the evaluation pairing $\left<(b_k , n_{b_k}) , (a_k, 1)\right> = 1$ and $c_{b_k} = 1$. Moreover, note that $\left< (b_k, n_{b_k}), (a, 1) \right> = 0$ for all $a \in A$ not equal to $a_k$ so that $\beta_{{\mathcal{A}}}^\vee (b_k, n_{b_k}) = e_k^\vee \in (\Z^{{\mathcal{A}}})^\vee$. 

Before proceeding to the constants $c_{b_{ij}}$, we observe that the morphism $\tilde{G} : \tlaf{A} \to \mco_{\p^{|A| - 1}} (-1)$ in equation \eqref{eq:morphtoproj} factors through a morphism to the equivariant line bundle $\mco (-1)$ over $\p (\tilde{c}_0, \ldots, \tilde{c}_{d + 1})$. Indeed, coarsening the Lafforgue fan by considering only  
\begin{align} \label{eq:raygens} B = \{\varrho_A \} \cup \{\ell_{\pm}/ |c_i| e_i^\vee : a_i \in A_\pm \} \cup \{e_k^\vee : a_k \in A_0\} \subset \dul{\plaf{A}}, \end{align}
gives the stacky fan 
\begin{align*}
\left( \Z^B, (\Z^{{\mathcal{A}}})^\vee , \tilde{\beta}_{\dul{\plaf{A}}}|_{\Z^B} , \Sigma_{\mco (-1)} \right),
\end{align*}
where $\Sigma_{\mco (-1)}$ is the same fan in $\R^B$ as that for $\mco_{\p^{|A| - 1}} (-1)$. Note that the stack associated to this fan is $\mco_{\p (\tilde{c}_0, \ldots, \tilde{c}_{d + 1})} (-1)$.  Quotienting by $\varrho_A$ leads to the factorization $G : \laf{A} \to \p^{|A| - 1}$ via
\begin{align} \label{eq:f1f2}
\begin{CD}
 \laf{A} @>{f_1}>> \p (\tilde{c}_0, \ldots, \tilde{c}_{d + 1} ) @>{f_2}>> \p^{|A| -1} .
\end{CD}
\end{align}
If $b = \ell_\pm / |c_i| e^\vee_i$, map $f_2$ takes $Z_b$ to $Z_b^{|c_i| / \ell_\pm}$ implying that 
\begin{align} \label{eq:f2pullback} f_2^* (\mco_{\p^{|A| - 1}} (1)) = \mco_{\p (\tilde{c}_0, \ldots, \tilde{c}_{d + 1})} (\ell). \end{align}

We now interpret the map $f_1$ as a weighted blowdown by considering the elements $c_{b_{ij}}^{-1} \beta_{\mathcal{A}}^\vee (b_{ij} , n_{b_{ij}} )  \in \dul{\plaf{A}}^h$ where $b_{ij} \in \bar{Q}$ is the supporting primitive for $F_{ij}$.  By definition, for any  $a \in A - \{a_i, a_j\}$ with $a_i \in A_+$ and $a_j \in A_-$ we have $\left< b_{ij} , a \right> = - n_{b_{ij}}$, or $\left< (b_{ij}, n_{b_{ij}}) , (a, 1) \right> = 0$. Taking $s_i = \left< (b_{ij}, n_{b_{ij}}) , (a_i, 1) \right>$ and $s_j = \left< (b_{ij}, n_{b_{ij}}) , (a_j, 1) \right>$  we then have $\beta^\vee_{{\mathcal{A}}} (b_{ij}, n_{b_{ij}}) = s_i e_i^\vee + s_j e_j^\vee$. Letting $r_{ij}$ be the volume of $F_{ij}$, we have that $c_i = \Vol (A - \{a_i\}) = r_{ij} s_j$ and $- c_j = \Vol (A - \{a_j\}) = r_{ij} s_i$ so that 
\begin{equation*} \overline{b}_{ij} := \beta^\vee (b_{ij}  , n_{b_k}) = \frac{1}{r_{{ij}}} \left(   c_i e_{a_j}^\vee  - c_j e_{a_i}^\vee\right) .
\end{equation*}
Thus the stacky fan for $\laf{A}$ is obtained by refining the fan for $\p (\tilde{c}_0, \ldots, \tilde{c}_{d + 1})$ by subdividing it along one-cones contained in the two-cones $\lin_{\R_{\geq 0}} (e_i^\vee,  e_j^\vee)$ for every $a_i \in A_+$ and $a_j \in A_-$. This implies that the divisor corresponding to $b_{ij}$ contracts to 
\begin{align} \label{eq:vijdefn} V_{ij} := \{Z_i = 0 = Z_j\} \end{align}
under $f_1$.  From the factorization of $G$ through $f_1$ and $f_2$ and equation~\eqref{eq:f2pullback}, we see that  $\univ{A}$ is $\mco (\ell )$ and $s_A = \sum_{a_i \in A_\pm} Z_i^{\ell_\pm / |c_i|}  + \sum_{a_j \in A_0} Z_j$. 
\end{proof}
Recall that the hypersurface $\hyp{A} \subset \laf{A}$ is defined as the zero
locus of $s_A \in H^0 (\laf{A} , \univ{A} )$ which implies that $\hyp{A}$ is the proper transform of the zero locus
\begin{equation*} Z_0 + \cdots + Z_{d + 1} = 0 \end{equation*} on $\p^{d + 1}$ along $G : \laf{A} \to \p^{d + 1}$. Using the previous proposition, we easily obtain the secondary stack associated to an extended circuit. For this, let $r = \textnormal{gcd} (\ell_+, \ell_-)$ and\gls{tlpm} $\tilde{\ell}_\pm = \ell_\pm / r$. 

\begin{prop} \label{prop:circuitsec} Assume $A \subset \Lambda$ is an extended circuit and $K_{{\mathcal{A}}} = 0$. Then  
	\begin{align*} \secon{A} \cong \frac{\p \left( \tilde{\ell}_+ , \tilde{\ell}_- \right)}{\Z / r \Z} . \end{align*} \end{prop}

\begin{proof} By Lemma \ref{lem:KtoXi} and the assumption that $K_{{\mathcal{A}}} = 0$, we have that $\Xi_{{\mathcal{A}}} = \Lambda_{{\mathcal{A}}^\vee} = L_{\mathcal{A}}^\vee \cong \Z$. From Lemma \ref{lem:stfansecon}, a stacky fan for $\secon{A}$ is given by
	\begin{align}
	\widetilde{\sfan}_{\psec{A}} = \left( \Z^{\dul{\psec{A}}}, \Xi_{{\mathcal{A}}}, \tilde{\beta}_{\dul{\psec{A}}}, \fan_{\mcb}  \right).
	\end{align}
Since $\Xi_{{\mathcal{A}}} = L_{\mathcal{A}}^\vee$, diagram~\eqref{diag:secondef}
is a colimit diagram and $\tilde{\beta}_{\dul{\psec{A}}}$ can be identified with $\tilde{\beta}_{\dul{\psecv{A}{v}}}$. In the case of a circuit, this reduces to
\begin{align}\label{diag:circpi}
	\begin{CD}
		\Z^{\dul{\plaf{A}}} @>{\tilde{\beta}_{\dul{\plaf{A}}}}>> (\Z^A)^\vee \\
		@V{p_1}VV  @V{\alpha^\vee_A}VV \\
		\Z^2 @>{\tilde{\beta}_{\dul{\psec{A}}}}>> \Z .
	\end{CD}
\end{align}
where $\alpha^\vee_A (e^\vee_{a_i} ) = c_i$ and, by equation \eqref{eq:defp1},
\begin{align*}
p_1 (e_b) & = \begin{cases}
	e_{1} & \text{ if } b = \eta_{(T_+,A - \{a_i\} )} , \text{ for } a_i \in A_+ , \\
	e_{2} & \text{ if } b = \eta_{(T_-,A - \{a_i\} )} , \text{ for } a_i \in A_-
	, \\
	0 & \text{ otherwise}. \end{cases} 
\end{align*} 	
By the second paragraph of the proof of Proposition \ref{prop:circuitlaf}, we have that 
\begin{align*}\tilde{\beta}_{\dul{\plaf{A}}} \left( \eta_{(T_\pm,A - \{a_i\} )} \right) = \ell_\pm / |c_i|. \end{align*} 
Thus, using the commutativity of diagram \ref{diag:circpi}, we conclude $\tilde{\beta}_{\dul{\psec{A}}}  = \ell_+ e_1^\vee - \ell_- e_2^\vee$ and $\widetilde{\sfan}_{\psec{A}}= (\Z^2, \Z, \ell_+ e_1^\vee - \ell_- e_2^\vee, \Sigma_{\mcb})$. This is the stacky fan for the toric stack $\p (\tilde{\ell}_+, \tilde{\ell}_- ) / (\Z / r \Z )$.
\end{proof}

We now give an explicit description of the map $\pi : \laf{A} \to \secon{A}$ from Definition \ref{defn:secstack}. Write\gls{hordiv} $\mcd^h \subset \laf{A}$ for the union of horizontal divisors in $\laf{A}$, $\laf{A}^\circ = \laf{A} - \mcd^h$ and $\hyp{A}^\circ = \hyp{A} - (\hyp{A} \cap \mcd^h)$. From Lemma \ref{lem:lafhyppart} the components of $\mcd^h$ are indexed by the facets of $Q$ which are in bijection with the set $(A_- \times A_+) \cup A_0$. By Proposition \ref{prop:circuitlaf}, restricting $f_1$ in equation \eqref{eq:f1f2} gives an isomorphism 
\begin{align} \laf{A}^\circ = \p (\tilde{c}_0, \ldots, \tilde{c}_{d + 1}) - \left[ \left( \bigcup_{a_i \in A_-, a_j \in A_+} V_{ij}\right) \cup \left( \bigcup_{a_k \in A_0 } \{Z_k = 0\} \right) \right]. 
\end{align} where $V_{ij}$ is defined in equation \eqref{eq:vijdefn}. Now, the map $p_1$ in diagram \eqref{diag:circpi} yields the expression for $\pi : \laf{A} \to \secon{A}$ from the homogeneous coordinates of $\laf{A}$ to those of $\secon{A}$. Including only those coordinates associated to the vertical divisors then gives $\pi^\circ : \laf{A}^\circ \to \secon{A}$ as a weighted pencil on $\p (\tilde{c}_0, \ldots, \tilde{c}_{d + 1})$ given by
\begin{equation*} \left[\prod_{a_i \in A_+} Z_i : \prod_{a_i \in A_-} Z_i \right].
\end{equation*}
The base locus of the pencil is the union $\cup V_{ij}$ of cycles
that are blown up in Proposition \ref{prop:circuitlaf}, which give some of the components of $\mcd^h$ (and all of them when $A$ is a circuit). Passing to coarse spaces, $\p^{d +1}$ for $\laf{A}^\circ$ and $\p^1$ for $\secon{A}$, leads to the diagram 
\begin{align}
\begin{CD}
\laf{A}^\circ @>{f_2}>> \p^{d + 1} - \cup V_{ij} , \\
@V{\pi}VV @V{\bar{\pi}}VV \\
\secon{A} @>>> \p^1. 
\end{CD}
\end{align}
Here the map $\bar{\pi}$ has the especially simple form as the pencil 
\begin{equation} \label{eq:circuitpencil} [s_0 : s_\infty] = \left[\prod_{a_j \in A_+}  Z_j^{c_j} :  \prod_{a_i \in A_-} Z_i^{-c_i}  \right].
\end{equation}
As neither $\laf{A}$ and $\secon{A}$ have generic stabilizers, this pencil describes the map $\pi$ up to isomorphism on the maximal torus. Moreover, from the description of the components of $\mcd^h$ indexed by $A_0$ in Proposition \ref{prop:circuitlaf}, $\bar{\pi}$ is isomorphic to $\pi$ when we include these divisors as well. This pencil also describes $\pi$ restricted to the universal hypersurface $\hyp{A}^\circ$ away from its degenerations at $0$ and $\infty$. We note that these fibers of the pencil give toric degenerations of $\mcx_Q$ corresponding to $T_+$ and  $T_-$. These are both singular as stacks unless $p = 1$ or $q = 1$. 

Our main interest is not in the morphism $\pi$, but rather its restriction to $\hyp{A} =  \left\{ s_A = 0 \right\} \subset \laf{A}$. Abusing notation, we will also denote this restriction as $\pi$. We note that, off of $\partial \hyp{A} = \hyp{A} \cap \partial \laf{A}$, the map $\pi$ is described by the pencil in equation \eqref{eq:circuitpencil}. We let\gls{critvala} $c_A \in \secon{A}$ be the point whose coarse point is represented by $  [ \prod_{j \in A_+} c_j^{c_j} : \prod_{i \in A_-} c_i^{-c_i}] \in \p^1$. Using notation established in Definitions~\ref{defn:degfamily} and \ref{defn:framedpencil}, we establish the following proposition.

\begin{prop} \label{prop:circflp} Let $A$ be an extended circuit of signature $(p, q; r)$ with $K_{{\mathcal{A}}} = 0$. The morphism $\pi : (\hyp{A} , \partial \hyp{A}) \to \secon{A}$ is a $\partial$-framed pencil. The critical values of $\pi$ consist of a unique stratified Morse singularity  over $c_A$  and 
\begin{enumerate}
	\item if $p > 1$ the fiber over $0$ is a stable pair degeneration, 
	\item if $q > 1$ the fiber over $\infty$ is a stable pair  degeneration.
\end{enumerate} 
\end{prop}

\begin{proof} We first address the statements concerning the critical values of $\pi$. If $p > 1$ (resp. $q > 1$), then $T_+$ (resp. $T_-$)is a triangulation of $(Q,A)$ with more than one simplex. Then $[0:1] \in \partial \secon{A}$ (resp. $[1:0]$) does not represent a full section implying it is contained in the compactifying divisor of the moduli of full sections $\mcv_A \subset \secon{A}$. By Theorem \ref{thm:toric2}, it then represents a stable pair  degeneration. 
	
Now, let $\hyp{A}^\prime = \hyp{A}^\circ - (F_0 \cup F_\infty)$ be the universal hypersurface away from the fibers over $0$ and $\infty$. The function $\pi : \hyp{A}^\prime \to \C^*$ is represented by the pencil in equation \eqref{eq:circuitpencil} restricted to $\hyp{A}^\prime := \{\sum_{i = 0}^{d + 1} Z_i  = 0\}$. The critical points of this function then can be calculated to be $\C^*$-orbits in the zero locus of $\lambda := d (\sum Z_i) \wedge d (s_0 / s_\infty)$. Writing $f = s_0 / s_\infty$ and computing, we obtain
\begin{align*}
 \lambda = d \left(\sum Z_i \right) \wedge d (s_0 / s_\infty) & = \left( \sum_{i = 0}^{d + 1} Z_i \right)  \wedge f \left( \sum_{i = 0}^{d + 1} c_i Z_i^{-1} \, dZ_i \right) , \\ & = f \sum_{i < j} \left( c_i Z_{i}^{-1} - c_j Z_j^{-1} \right) dZ_i \wedge dZ_j.
\end{align*} 
Note that the functions $Z_i^{-1}$ are well defined on $\hyp{A}^\prime$ for $a_i \in A_\pm$ while when $a_i \in A_0$, the coefficients $c_i = 0$ renders a zero term for $c_i Z_i^{-1}$. This two-form is zero if and only if $c_i Z_{i}^{-1} = c_j Z_j^{-1} $ for all $0 \leq i , j \leq d + 1$. If $r \ne 0$, then there are no zeros of $\lambda$. Indeed, if $A_0 = \{a_{d + 1 - r}, \ldots, a_{d + 1} \}$, then the $c_0 Z_0^{-1} dZ_0 \wedge dZ_{d + 1}$ will always be a non-zero summand of $\lambda$. One checks that for any proper subset $I \subsetneq \{d + 1 - r, \ldots, d + 1\}$, taking $C_I = \cap_{i \in I} \{Z_i = 0\}$ and restricting $\lambda|_{C_I}$, we still obtain a non-zero two-form. However, when $I = \{d + 1 - r, \ldots, d + 1\}$, 
\begin{align*} \lambda|_{C_I} = f \sum_{0 \leq i < j \leq d + 1 - r} \left( c_i Z_{i}^{-1} - c_j Z_j^{-1} \right).
\end{align*} 
This is zero if and only if $c_i Z_j = c_j Z_i$ for all $0 \leq i < j \leq d + 1 - r$ which holds precisely when $[Z_0: \cdots :Z_{d + 1}] = [c_0: \cdots :c_{d + 1}]$. Evaluating $f$ at this point gives $c_A$.

To see that this is a stratified Morse singularity, we restrict the Hessian of $f$ at $[c_0: \cdots :c_{d + 1}]$ to $\{ \sum Z_i = 0 \} \cap C_{d + 1 - r, \ldots , d + 1}$. One computes \begin{align*} \Hess_{(c_0, \ldots, c_{d + 1})} (f) = f (c_0, \ldots, c_{d + 1}) \left(h_{i,j} \right)_{i,j}\end{align*} where $h_{i,j} = c_i c_j$ if $i \ne j$ and $c_i^2 - c_i$ otherwise. As we are restricting to $C_I$, we may assume that $r = 0$ so that $c_i \ne 0$ for all $i$. From the expression for $\Hess (f)$, we see that it can be written as $H_1 - H_2$ where $H_1$ is a rank $1$ matrix with image $\textnormal{Lin}_\R \{(c_0, \ldots, c_{d + 1}) \}$ and $H_2$ is the diagonal matrix $\textnormal{Diag} (c_0, \ldots, c_{d + 1} )$. As $H_2$ is invertible, a tangent vector $v \in T \hyp{A}$ will be in the kernel of this difference only if $H_1 (v) \in \image (H_2)$. This implies $v$ is a multiple of $\sum_{i =0}^{d + 1} \partial_{Z_i}$ and, as this vector does not pair with  $\sum d Z_i $ to equal zero, it is not tangent to $\hyp{A}$ and we must have $v = 0$. Thus $\Hess_{(c_0, \ldots, c_{d + 1})} (f)$ restricted to $\{ \sum Z_i = 0 \} \cap C_{d + 1 - r, \ldots , d + 1}$ is non-degenerate and, by Proposition \ref{prop:str}, $\pi$ has a stratified Morse singularity at $[c_0: \cdots : c_{d + 1}]$. The statement that $\pi$ is a $\partial$-framed pencil then follows immediately from Definition \ref{defn:framedpencil}.
\end{proof}

\begin{figure}[t]
\begin{picture}(0,0)%
\includegraphics{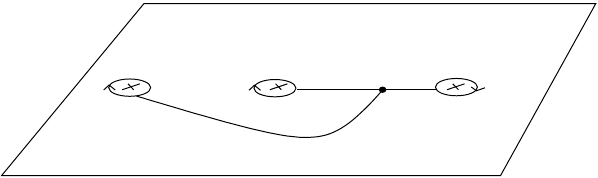}%
\end{picture}%
\setlength{\unitlength}{4144sp}%
\begin{picture}(4552,1335)(730,-949)
\put(2873,175){\makebox(0,0)[lb]{\smash{$\secon{A}$}}}
\put(2801,-101){\makebox(0,0)[lb]{\smash{$c_A$}}}
\put(1721,-101){\makebox(0,0)[lb]{\smash{$0$}}}
\put(4151,-101){\makebox(0,0)[lb]{\smash{$\infty$}}}
\put(3826,-466){\makebox(0,0)[lb]{\smash{$\delta_\infty$}}}
\put(1981,-691){\makebox(0,0)[lb]{\smash{$\delta_0$}}}
\put(3126,-466){\makebox(0,0)[lb]{\smash{$\delta_1$}}}
\put(3626,-196){\makebox(0,0)[lb]{\smash{$t_0$}}}
\end{picture}%
\caption{\label{fig:example0} Distinguished basis on $\secon{A}$}
\end{figure}

We write $\pi_0 : (\C^*)^{d + 1} \to \C^*$ for the restriction of $\pi$ to the complement of the coordinate divisors on $\p^{d + 1}$. We now fix a point $t_0 \in \secon{A} (\R )$ near infinity and let $\delta_0, \delta_1$ and $\delta_\infty$ be paths, based at $t_0$, around $0$, $c_A$ and $\infty$. Here $\delta_1$ and $\delta_\infty$ are straight line paths and $\delta_0$ is a concatenation of a straight line path to an $\varepsilon$ neighborhood of $c_A$, a clockwise semicircle around $c_A$ and a straight line path to $0$. These are pictured in Figure \ref{fig:example0}. 

Our main theorem now appears as a consequence of the Proposition~\ref{prop:frmechnge}.

\begin{thm} \label{thm:circuitrelation} Let $(Q, A)$ be an extended circuit with $K_{\mathcal{A}} = 0$, $T_i = \partrans (\delta_i )$ and \begin{equation*} \mathbf{x} = \left( -  \frac{ 2 \pi \textnormal{ gcd} (c_i, c_j)}{\textnormal{lcm} (c_i,  c_j)} : c_i > 0 , c_j < 0 \right) . \end{equation*} 
	Then
	\begin{equation*} T_0 T_1 T_\infty = \boundflow (\mathbf{x} ) \end{equation*}
	in $\pi_0 (\Symp^\mathbf{F} (\fib{A}{t_0}, \partial \fib{A}{t_0}))$.
\end{thm}

\begin{proof} By Proposition~\ref{prop:frmechnge}, the only result needed for the theorem is the computation of the Chern numbers for the rigid boundary divisors associated to $\overline{b}_{ij}$. In the proof of Proposition~\ref{prop:circuitlaf}, we saw that $b_{ij} = \frac{1}{r_{{ij}}} \left(   c_i e_{a_j}^\vee  - c_j e_{a_i}^\vee\right)$. By Proposition~\ref{prop:rigid},  the divisor $D_{ij} \subset \laf{A}$ corresponding to $\overline{b}_{ij}$ is isomorphic to the products of $\secon{A}$ and the boundary divisor $\tilde{D}_{ij} \subset \mcx_{Q}$ corresponding to the facet $F_{ij} = \convhull ( A - \{a_i, a_j\})$. Let $\sfan_{ij}$ be the stacky fan in $(\Z^3, \Z^2, \beta_{ij}, \Sigma_{ij})$ where $\{e_i, e_j, e_D\}$ is the standard basis for $\Z^3$. Define $\beta_{ij}$ to be the map $\beta_{ij} (e_i) = \left( \ell_+ / c_i \right) e_1$, $\beta_{ij} (e_j) = - \left( \ell_- / c_j \right) e_2$ and $\beta_{ij} (e_D ) = \frac{1}{r_{{ij}}} \left(   c_i e_2  - c_j e_1 \right)$. Take the fan $\Sigma_{ij}$ to consist of two maximal cones $\text{Lin}_{\R_{\geq 0}}\{e_i, e_D\}$ and  $\text{Lin}_{\R_{\geq 0}}\{e_j, e_D\}$ whose image under $\beta_{ij}$ gives the fan pictured in Figure~\ref{fig:sfij}.

\begin{figure}
\begin{picture}(0,0)%
\includegraphics{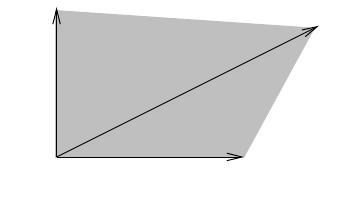}%
\end{picture}%
\setlength{\unitlength}{4144sp}%
\begin{picture}(2595,1533)(3856,-3505)
\put(3521,-2131){\makebox(0,0)[lb]{\smash{$\left(0, -\frac{\ell_-}{c_j} \right) $}}}
\put(6436,-2401){\makebox(0,0)[lb]{\smash{$\left(-\frac{c_j}{r_{ij}}, \frac{c_i}{r_{ij}} \right) $}}}
\put(5491,-3436){\makebox(0,0)[lb]{\smash{$\left(\frac{\ell_+}{c_i},0 \right) $}}}
\end{picture}%
	\caption{\label{fig:sfij}The stacky fan $\sfan_{ij}$}
\end{figure}

Using Proposition~\ref{prop:circuitlaf} and equation~\eqref{eq:raygens}, the star of $\overline{b}_{ij}$ is the product of $\sfan_{ij}$ and the fan for $\tilde{D}_{ij}$. Thus the toric stack associated to the fan $\sfan_{ij}$ is isomorphic to the normal bundle of a section of $\pi$ lying on $D_{ij}$. The Chern number of the normal bundle of the divisor $D$ corresponding to $e_D$ is then computed as $D  \cdot D = - \frac{r_{{ij}}^2}{c_i c_j}$ which equals the indicated factor under the assumption $K_{\mathcal{A}} = 0$.
\end{proof}

\subsection{Examples in dimension $1$}

In this section we explore three examples in dimension $1$ of the circuit relation in full detail. These circuits are illustrated in Figure \ref{fig:exampleset}. The first relation is known as the lantern relation for mapping class groups of marked curves and, to a large degree, is the case that inspired this paper. The next example yields the star relation. We observe that the circuit stack in this example, as well as its higher dimensional generalizations, arises naturally in the context of homological mirror symmetry. We refer to \cite[Chapter~2]{farb} for general background on the mapping class groups of marked curves and classical proofs of these relations.

For every example, we take a fiber $t_0 \in \R_{> 1}$ near $\infty$ and choose the distinguished basis of paths $\delta_0$, $\delta_1$ and $\delta_\infty$ on $\secon{A}$ as in Theorem \ref{thm:circuitrelation} and Figure \ref{fig:example0}.

\subsubsection{Circuit of signature $(2, 2)$}

\begin{figure}[t]
\includegraphics{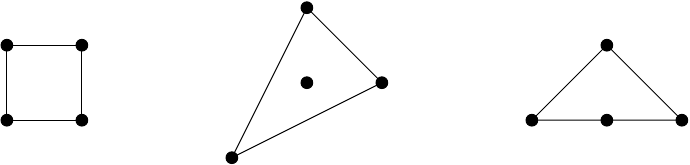}%
\caption{\label{fig:exampleset} Examples in dimension $1$}
\end{figure}
Here we take the set $A = \{ (0, 0), (1, 0), (1, 1), (0, 1) \}$ and fix the orientation of $A$ as $\mathbf{c} = (1, -1, 1, -1)$. We have that $\hyp{A} = \{ Z_0 + Z_1 + Z_2 + Z_3 = 0 \} \subset \p^3$ and $\pi$ is defined as the pencil $[Z_0 Z_2 : Z_1 Z_3]$. Taking the coordinate $t$ for the point $[t:1] \in \p^1$ we utilize equation \eqref{eq:circuitpencil} to find $t = \pi ([Z_0:Z_1:Z_2:Z_3]) = \frac{Z_0 Z_2}{Z_1 Z_3}$ so that every fiber $\fib{A}{t} = \hyp{A} \cap \pi^{-1} (t)$ for $t \in \C^* - \{1\}$ is isomorphic to $\p^1$. The boundary divisor $\partial \fib{A}{t}$ consists of four points given as an intersection with $\cup_{i = 0}^3 D_i$ where $D_i = \{Z_i = 0 = Z_{i + 1} \}$ using an index in $\Z / 4 \Z$. Thus, using a M\"obius transformation, we can find a coordinate $x$ for each fiber so that 
\begin{align} \label{eq:divisorcond}
\begin{split}  q_1 = D_1 \cap \fib{A}{t} & = \{x  = 0 \} , \\ q_2 = D_2 \cap \fib{A}{t} & = \{x  =  1\} , \\ q_3 = D_3 \cap \fib{A}{t} & = \{x  = \infty \}.
\end{split} \end{align}
Parameterizing $\fib{A}{t}$ so that $Z_i$ is at most quadratic in $x$ and which satisfies equations \eqref{eq:divisorcond} gives
\begin{equation*}  \fib{A}{t} = \{ [1 - tx : ( tx - 1) x : tx (1 - x): x - 1 ] : x \in \C  \} \end{equation*}
with remaining boundary divisor component 
\begin{align*}
q_0 = D_0 \cap \fib{A}{t} & = \{x  = t^{-1} \}.
\end{align*}
Over the limiting degeneration values of $t = 0$ and $\infty$, one sees that this converges to give parameterizations of the intersections $\{Z_2 = 0 \} \cap \hyp{A}$ and $\{Z_3 = 0 \} \cap \hyp{A}$, respectively.  

As $t_0 > 0$ was chosen close to $\infty$, we have that $q_0 > 0$ is close to zero and indeed tends to $q_1$ as $t$ tends to $\infty$. This reflects the bubbling of the intersection $\hyp{A} \cap \{Z_0 = 0 \}$ off in the limit and we see that the vanishing cycle of $\delta_\infty$ is a loop $\gamma_\infty$ encircling $q_0$ and $0$ in the $x$-plane. In a similar vein, we may follow the path $\delta_1$ from $t_0$ to $1$ and observe that the point $q_0$ follows the straight line path to $q_2$. Thus the vanishing cycle associated to $\delta_1$ is isotopic to $\gamma_1$ illustrated in Figure \ref{fig:example1}. Finally, as $t$ tends from $t_0$ to $0$ along the path $\delta_0$, $q_0$ passes above $q_2$ and towards $q_3$. The vanishing cycle may be pulled back along this path and is seen to be equivalent to $\gamma_0$ which, up to isotopy, is illustrated in Figure \ref{fig:example1}. 

\begin{figure}[b]
\begin{picture}(0,0)%
\includegraphics{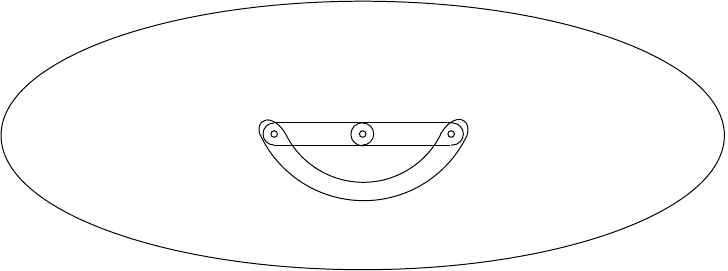}%
\end{picture}%
\setlength{\unitlength}{4144sp}
\begin{picture}(5528,2062)(-7,-1575)
\put(5118,-477){\makebox(0,0)[lb]{\smash{$q_3$}}}
\put(1676,-561){\makebox(0,0)[lb]{\smash{$q_1$}}}
\put(3628,-561){\makebox(0,0)[lb]{\smash{$q_2$}}}
\put(2658,-756){\makebox(0,0)[lb]{\smash{$q_0$}}}	\put(2216,-376){\makebox(0,0)[lb]{\smash{$\gamma_\infty$}}}	\put(2971,-376){\makebox(0,0)[lb]{\smash{$\gamma_1$}}}
\put(2656,-1186){\makebox(0,0)[lb]{\smash{$\gamma_0$}}}
\end{picture}%
\caption{\label{fig:example1} The $(2, 2)$ circuit relation or the lantern relation}
\end{figure}
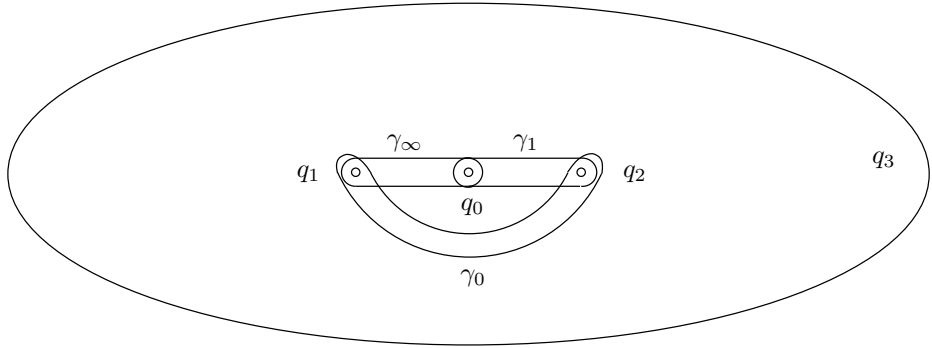

Applying Theorem \ref{thm:circuitrelation} in this example yields the well known lantern relation arising in mapping class groups. 

\subsubsection{Circuit of signature $(1, 3)$}

In our example of a $(1, 3)$ circuit, we take the set $A = \{(0, 0), (1, 0), (0, 1), (-1, -1)\}$ and fix $\mathbf{c} = (3, -1, -1, -1)$. We have the same hypersurface $\hyp{A} \subset \p^3$ as before, but with $\pi ([Z_0: Z_1: Z_2: Z_3]) = [Z_0^3 : Z_1 Z_2 Z_3 ]$. The smooth fibers $\fib{A}{t}$ of $\pi$ are elliptic curves with boundary points indexed by the divisors $D_i = \{Z_i = 0 = Z_0\} = \{q_i\}$ for $i = 1, 2, 3$. Near $t = \infty$, we have $\fib{A}{t}$ approaches an intersection of $\hyp{A}$ with the three divisors in $\{Z_1 = 0\}$, $\{Z_2 = 0 \}$ and $\{Z_3 = 0 \}$ which subdivides it into three pairs of pants. On the other hand, at $t = 0$, $\fib{A}{t} - \partial \fib{A}{t}$ is the quotient of the elliptic curve $\{(x,y): x + y + x^{-1}y^{-1} = 0 \} \subset (\C^*)^2$ by a $\Z / 3 \Z$ action. We note that there is $1$ component of $\fib{A}{0}$ and $3$ components of $\fib{A}{\infty}$. This occurs generally as the signature corresponds to $(|A_+|, |A_-|)$ and the number of simplices in the triangulation $T_\pm$ is $|A_\pm|$. The fiber over $0$ (resp. $\infty$) is a stable pair degeneration corresponding to the triangulation $T_+$ (resp. $T_-$) and the number of components of this degeneration equals the number of simplices in the triangulation.

For the moment, we consider the case of a more general signature $(1, d + 1)$ circuit with $A_+ = \{a_0\}$ and let
\begin{equation*} \hyp{A}^+ (\R ) := \{[r_0: \cdots :r_{d + 1} ] \in \hyp{A} : r_i \in \R^*,  r_i r_j < 0 \text{ with }i < j \text{ iff } i = 0 \}. \end{equation*}
In particular, $\hyp{A}^+ (\R)$ is isomorphic to the positive simplex in $\R^{d + 1}_{> 0}$ using the coordinates $\{[-1 : r_1: \cdots :r_{d +1}] : \sum r_i = 1\} \in \hyp{A}^+ (\R )$ . One checks that the assumption on the signature of $A$ gives $[c_0: \cdots : c_{d + 1}] \in \hyp{A}^+ (\R )$. Furthermore, following the computations of the critical points and Hessian of $\pi$ in the proof of Proposition \ref{prop:circflp}, which do not rely on whether we work over $\R$ or $\C$, shows that $\pi|_{\hyp{A}^+ (\R )} : \hyp{A}^+ (\R ) \to \p_\R^1$ has a unique Morse singularity at $[c_0: \cdots : c_{d + 1}]$ with critical value $c_A \in \p_\R^1$. Furthermore, along the boundary of the closure of $\hyp{A}^+ (\R)$ (where one of the coordinates equals zero), $\pi$ evaluates to $\infty = [1:0]$. Finally, since $\pi$ does not take the value of $[0:1]$ on $\hyp{A}^+ (\R )$, we can conclude that the unique critical point is a maximum (resp. a minimum ) if $a_0$ is odd (resp. even) and that $\hyp{A}^+ (\R)$ is the stable (resp. unstable) manifold associated to $c_A$. As such a manifold is obtained by gradient flow using the Hermitian metric, this flow equals that of the symplectic parallel transport map along the real line. Thus $\hyp{A}^+ (\R)$ is contained in the vanishing thimble of $\delta_1$ and, as it is a smooth manifold of the correct dimension, it must equal the vanishing thimble. Alternatively, one could observe this fact by considering $\hyp{A}^+ (\R)$ as the fixed locus of an anti-holomorphic involution which is equivariant with respect to $\pi$.

The boundary of $\hyp{A}^+ (\R)$ is the union of three arcs contained in the three components $\fib{A}{\infty}$. Taking symplectic transport to $t_0$ near $\infty$, these arcs lie in three pairs of pants which converge to the degenerated components giving $\gamma_1$ in Figure \ref{fig:example2}. The three circles depicting  $\gamma_\infty$ are the vanishing cycles associated to the degeneration. The circuit relation in this example is known as the star relation.

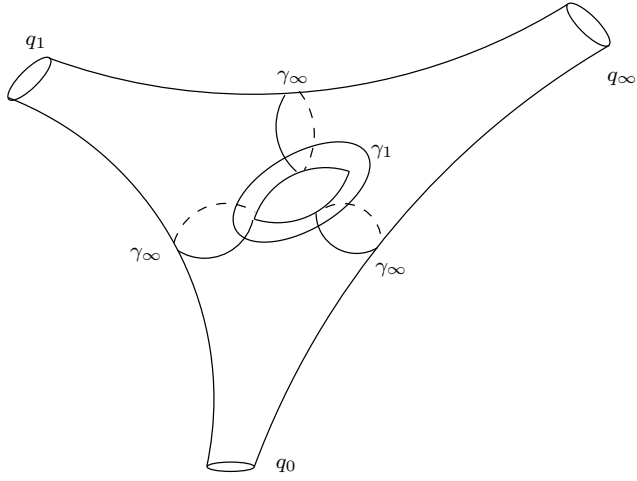
\begin{figure}
\begin{picture}(0,0)%
\includegraphics{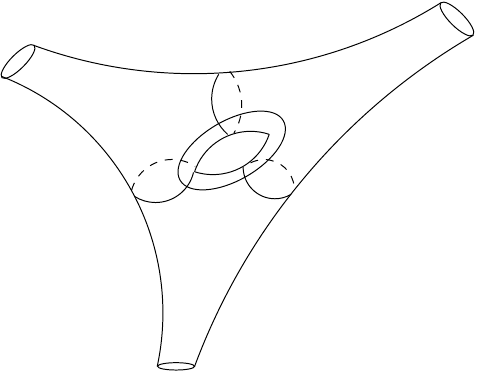}%
\end{picture}%
\setlength{\unitlength}{4144sp}%
\begingroup\makeatletter\ifx\SetFigFont\undefined%
\gdef\SetFigFont#1#2#3#4#5{%
  \reset@font\fontsize{#1}{#2pt}%
  \fontfamily{#3}\fontseries{#4}\fontshape{#5}%
  \selectfont}%
\fi\endgroup%
\begin{picture}(3620,2850)(1050,-3492)
\put(2670,-1110){\makebox(0,0)[lb]{\smash{{\SetFigFont{8}{9.6}{\rmdefault}{\mddefault}{\updefault}{\color[rgb]{0,0,0}$\gamma_\infty$}%
}}}}
\put(3231,-1554){\makebox(0,0)[lb]{\smash{{\SetFigFont{8}{9.6}{\rmdefault}{\mddefault}{\updefault}{\color[rgb]{0,0,0}$\gamma_1$}%
}}}}
\put(2662,-3436){\makebox(0,0)[lb]{\smash{{\SetFigFont{8}{9.6}{\rmdefault}{\mddefault}{\updefault}{\color[rgb]{0,0,0}$q_0$}%
}}}}
\put(1165,-904){\makebox(0,0)[lb]{\smash{{\SetFigFont{8}{9.6}{\rmdefault}{\mddefault}{\updefault}{\color[rgb]{0,0,0}$q_1$}%
}}}}
\put(4643,-1115){\makebox(0,0)[lb]{\smash{{\SetFigFont{8}{9.6}{\rmdefault}{\mddefault}{\updefault}{\color[rgb]{0,0,0}$q_\infty$}%
}}}}
\put(1791,-2173){\makebox(0,0)[lb]{\smash{{\SetFigFont{8}{9.6}{\rmdefault}{\mddefault}{\updefault}{\color[rgb]{0,0,0}$\gamma_\infty$}%
}}}}
\put(3257,-2258){\makebox(0,0)[lb]{\smash{{\SetFigFont{8}{9.6}{\rmdefault}{\mddefault}{\updefault}{\color[rgb]{0,0,0}$\gamma_\infty$}%
}}}}
\end{picture}%

\caption{\label{fig:example2} The $(1, 3)$ circuit relation or the star relation}
\end{figure}

In higher dimensions, we may consider the signature $(1, d + 1)$ case with $\mathbf{c} = (c_0 ,  c_1 , \ldots,  c_{d + 1} )$ where $c_0 = v_A > 0$. Again we have that the hypersurface $\hyp{A} $ is $ \{\sum_{i = 0}^{d + 1} Z_i = 0 \}$ in $\p^{d + 1}$ and
\begin{equation*}
\pi ([Z_0: \cdots : Z_{d + 1}] ) = [ Z_0^{v_A} : Z_1^{-c_1} \ldots Z_{d + 1}^{-c_{d + 1}} ] .
\end{equation*}
If $K_{\mathcal{A}} = 0$, the secondary stack $\secon{A}$ is  $ \p (v_A / r , \tilde{\ell}_-) / (\Z / r \Z)$ and we may take an orbifold chart around zero to be the map $z^{a_0}$. Pulling $\pi$ back along this chart we obtain a map $w : (\C^*)^{d} \to \C$. Indeed, taking $\pi = [z^{a_0} : 1]$ and restricting to $Z_1^{a_1} \ldots, Z_{d + 1}^{c_{d + 1}} = 1$ yields $Z_0 = t$ so that we may express $w$ as
\begin{equation*} w (Z_1, \ldots, Z_d) = Z_0 = - \sum_{i = 1}^d Z_i - \frac{1}{Z_1^{c_1/c_{d + 1}} \cdots Z_d^{c_d / c_{d + 1}}} . \end{equation*}
Referring to \cite[Equation~1.4]{horivafa} we have that, up to a scale, the map $\pi$ is the equivariant quotient of the homological mirror LG model of the weighted projective space $\p (c_1, \ldots, c_{d + 1})$. This will appear again as one piece of a general conjectural program for homological mirror symmetry  in the final section.

\subsubsection{\label{sec:exampledeg} Circuit of signature $(1, 2; 1)$}

In our only degenerate example, we observe a relation between braids and Dehn twists. We take $A  = \{ (0, 0), (1, 0), (-1, 0), (0, 1)\}$, $\mathbf{c} = (2, -1, -1, 0)$. Here $\laf{A}$ is the blow-up of $\p^3$ along the two coordinate lines $L_1 = \{Z_0 = 0 = Z_1\}$ and $L_2 = \{Z_0 = 0 = Z_2\}$ which are the base locus of the pencil $\tilde{\pi}$ given as
\begin{equation*}
\tilde{\pi} ([Z_0: Z_1:Z_2:Z_3] ) = [Z_0^2 : Z_1 Z_2].
\end{equation*}
The secondary stack of $A$ is $\secon{A} = \p (2, 1)$. 

Since $A$ is a degenerate circuit, the divisor $\{Z_3 = 0 \}$ is not contained in a fiber over $0$ or infinity, but rather intersects $\fib{A}{t}$ in two points everywhere except over the degenerate point $[2:-1:-1:0]$ with value $c_A = 4$.  We give $\fib{A}{t}$ coordinates,
\begin{equation*}
\fib{A}{t} = \{[tx:x^2:t :-tx - x^2 - t] : x \in \C \} .
\end{equation*}
The boundary points on $\fib{A}{t}$ are then,
\begin{align*} q_1 & =  \fib{A}{t} \cap \{Z_1 = 0 \} = \{x= 0 \}  ,\\
q_2  & = \fib{A}{t} \cap \{Z_2 = 0 \} = \{t x = \infty\} ,\\
q_{3, \pm} & = \{x = -t \pm \sqrt{t^2 - 4t}/2 \} .
\end{align*}
As $t$ tends from $c_A$ to $t_0$, we see that $q_{3, \pm}$ splits along the real axis. The vanishing cycle $\gamma_1$ for $\delta_1$ thus forms an interval stretching between $q_{3, \pm}$. This can be seen from the local description of vanishing cycles for stratified Morse singularities given in Proposition \ref{prop:vcvt} and its proof.  Tending from $t_0$ to $\infty$, one observes $q_{3,+}$ converging to $-1$ and $q_{3, -}$ bubbling off with $\infty$. This parameterization converges to the component $\hyp{A} \cap \{Z_1 = 0\}$.  Thus we may draw a vanishing cycle $\gamma_\infty$ around $\infty$ and $q_{3, -}$ corresponding to $\delta_\infty$.

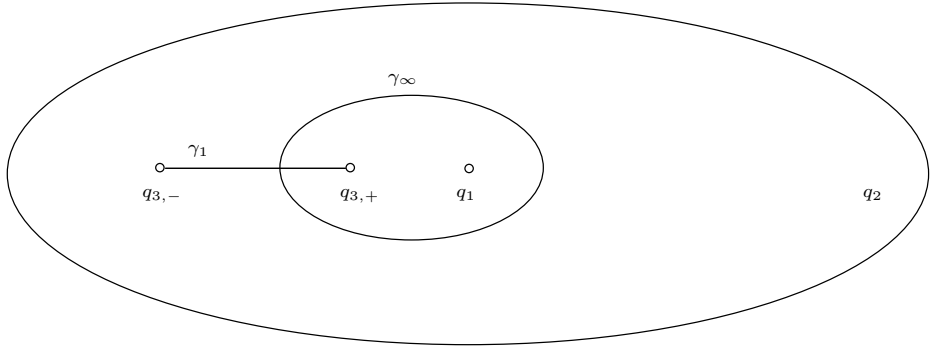
\begin{figure}
\begin{picture}(0,0)%
\includegraphics{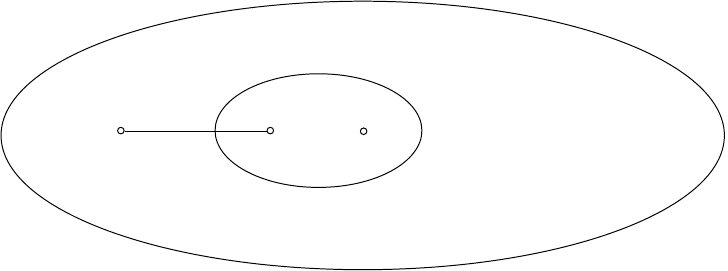}%
\end{picture}%
\setlength{\unitlength}{4144sp}%
\begingroup\makeatletter\ifx\SetFigFont\undefined%
\gdef\SetFigFont#1#2#3#4#5{%
  \reset@font\fontsize{#1}{#2pt}%
  \fontfamily{#3}\fontseries{#4}\fontshape{#5}%
  \selectfont}%
\fi\endgroup%
\begin{picture}(5528,2062)(-7,-1575)
\put(5118,-685){\makebox(0,0)[lb]{\smash{{\SetFigFont{7}{8.4}{\rmdefault}{\mddefault}{\updefault}{\color[rgb]{0,0,0}$q_2$}%
}}}}
\put(1081,-421){\makebox(0,0)[lb]{\smash{{\SetFigFont{7}{8.4}{\rmdefault}{\mddefault}{\updefault}{\color[rgb]{0,0,0}$\gamma_1$}%
}}}}
\put(2277,  8){\makebox(0,0)[lb]{\smash{{\SetFigFont{7}{8.4}{\rmdefault}{\mddefault}{\updefault}{\color[rgb]{0,0,0}$\gamma_\infty$}%
}}}}
\put(808,-685){\makebox(0,0)[lb]{\smash{{\SetFigFont{7}{8.4}{\rmdefault}{\mddefault}{\updefault}{\color[rgb]{0,0,0}$q_{3,-}$}%
}}}}
\put(1991,-685){\makebox(0,0)[lb]{\smash{{\SetFigFont{7}{8.4}{\rmdefault}{\mddefault}{\updefault}{\color[rgb]{0,0,0}$q_{3,+}$}%
}}}}
\put(2686,-685){\makebox(0,0)[lb]{\smash{{\SetFigFont{7}{8.4}{\rmdefault}{\mddefault}{\updefault}{\color[rgb]{0,0,0}$q_1$}%
}}}}
\end{picture}%

\caption{\label{fig:example3} The $(1, 2; 1)$ circuit relation}
\end{figure}

At $t = 0$ we have a $\Z / 2 \Z$ orbifold point where, in the coordinates given by $x$, we have quotiented by the action. This implies that the monodromy map $T_0^2 = 1$. As $T_{\partial \fib{A}{t_0}}$ is supported near the boundary $\{q_1, q_2, q_{3, \pm}\}$, it commutes with $T_0, T_1$ and $T_\infty$.  In fact, ignoring the framing on the endpoints $q_{3, \pm}$ of the braid, $T_{\partial \fib{A}{t_0}}$ is a half twist about $q_1$ and $q_2$. Thus we may take the relation $T_0 T_1 T_\infty = T_{\partial \fib{A}{t_0}}$ from Theorem \ref{thm:circuitrelation} and rewrite it as $T_1 T_\infty =  T_0^{-1} T_{\partial \fib{A}{t_0}} $. Squaring both sides gives the relation $(T_1 T_\infty )^2 = T_{\partial \fib{A}{t_0}}^2$. This does not seem to have a direct analog in the literature, but can be thought of as a hyperelliptic relation for a braid and a loop.

\subsection{\label{sec:applications1} Regeneration}
In contrast to the topology of discriminant complements (see \cite{DolgachevLibgober}), the geometry of the principal $A$-determinant complement seems relatively unexplored.  For extended circuits, we have completed the project of understanding $\secon{A} - \mce_A$ in Proposition~\ref{prop:circflp} as the once punctured quotient of a weighted projective line. On the other hand, as one considers more complicated sets $A$, the complexity of the topology of their determinant complements grows rapidly. In order to retain the information obtained from more basic cases of $A^\prime \subset A$ such as circuits, we require a method of regeneration. In large measure, the toric and symplectic preliminaries in Appendices~\ref{sec:toric} and \ref{sec:pfsymp} are designed to make such a method possible and accessible. 

Let $A \subset \Z^d$ and $A^\prime \subset A$ be finite subsets and $S = \{(Q_i, A_i): i \in I \}$ a regular subdivision of $Q$ such that $(Q_i , A_i )$ is a marked simplex for all $A_i$ not containing $A^\prime$ and $A_i$ is a simplicial extension of $A^\prime$ otherwise.  Call such a subdivision a \emph{triangular extension} of $A^\prime$. Such a subdivision induces an map of affine polytopes $\psec{A^\prime} \to \psec{A}$ which is obtained by taking the vertex $\varphi_{T^\prime}$ corresponding to the regular triangulation $T^\prime$ of $(Q^\prime , A^\prime)$ to the vertex $\varphi_{\bar{T}^\prime}$ where $\bar{T}^\prime$ is the unique refinement of $S$ which restricts to $T^\prime$ on $(Q^\prime, A^\prime)$. This map of secondary polytopes induces a natural inclusion $i_{S} : \secon{A^\prime} \to \secon{A}$ of secondary stacks. By the definition of triangular extensions and  \cite[Theorem~10.1.12]{GKZ}, we have that $i_S (\mce_{A^\prime}) = \mce_A \cap i_S (\secon{A^\prime})$. Let $\secon{A}^\circ$ be the maximal torus orbit of $\secon{A}$ and $\mce_A^\circ$ be the intersection $\mce_A \cap \secon{A}^\circ$.  Given $\varepsilon > 0$, let $\interior{A^\prime}{\varepsilon} \subset \secon{A^\prime}^\circ - \mce^\circ_{A^\prime}$ be the complement of an $\varepsilon$ neighborhood of $\mce_{A^\prime}^\circ$. For sufficiently small $\varepsilon$, $\interior{A^\prime}{\varepsilon}$ is diffeomorphic to $\secon{A^\prime}^\circ - \mce^\circ_{A^\prime}$.

\begin{defn} Let $\disc \subset \C$ be a disc around the origin and $\mci$ a complex manifold. A regeneration of ${A^\prime}$ relative to ${A}$ is a pair $(\mci , \psi )$ where $\psi_{\_} : \disc \times \mci \to \secon{A}$ is holomorphic with $\psi_0 : \mci \to i_S (\interior{A^\prime}{\varepsilon})$ a covering map onto its image and $\psi_t : \mci \to \secon{A}^\circ$ injective for all $t \ne 0$.
\end{defn}

The following proposition shows that there exists many distinct regenerations of $A^\prime$ relative to $A$.

\begin{prop} Let $S$ be a triangular extension of $A^\prime \subset A$ and $n \in \Z$. There exists a regeneration $(\mci , \psi)$ of ${A^\prime}$ with $\psi_0$ an $(\Z / n \Z)$-cover. 
\end{prop}
\begin{proof}
This result follows from general facts about stacky fans. In particular, suppose $\sfan = \left( \Z^r, \Lambda , \beta , \Sigma \right)$ is a canonical stacky fan for a complete toric stack where the rank of $\Lambda$ is $d$. Let $\sigma = \textnormal{Lin}_{\R_{\geq 0}} (e_{1}, \ldots, e_s)$ be an $s$-dimensional cone in $\Sigma$ where $s < d$ and $e_1, \ldots, e_r$ is the standard basis of $\Z^r$. The stacky subfan $\sfan_\sigma = \left( \Z^r, \Lambda, \beta, \sigma \right)$ gives the normal neighborhood of the orbit corresponding to $\sigma$. 

Now suppose $\tau = \textnormal{Lin}_{\R_{\geq 0}} (e_{1}, \ldots, e_s, e_{s + 1}) \in \Sigma$ and let $\Gamma = \textnormal{Lin}_{\Z} (e_1, \ldots, e_{s + 1})$. Take $(e_s)$ to be the ray in $\R^r$ generated by $e_s$ and define the stacky fan $\sfan_{s + 1} = \left( \Z^{s + 1} , \Gamma, \beta, (e_{s + 1}) \right)$. The stack associated to $\sfan_{s + 1}$ is clearly isomorphic to $\C \times (\C^*)^{s}$. Define $(g_1,g_2) : \sfan_{s + 1} \to \sfan_{\sigma}$ by
\begin{align} g_1 (e_i) & = \begin{cases}  e_i & \text{ if } i < s, \\ n e_s & \text{ if } i = s, \\ e_s + e_{s + 1}, \text{ if } i = {s + 1}. \end{cases} \end{align}
and take $g_2$ to be the unique map satisfying $\beta \circ g_1 = g_2 \circ \beta$. The associated map on stacks $g : \C \times (\C^*)^s \to \mcx_{\sfan_{\sigma}}$ is an $n$-fold cover at $0 \times (\C^*)^s$ and is injective otherwise. Composing with the inclusion $\mcx_{\sfan_\sigma} \hookrightarrow \mcx_\sfan$ and taking $\sfan$ to be the stacky secondary fan $\widetilde{\sfan}_{\psec{A}}$ from Lemma~\ref{lem:stfansecon} with $\sigma$ the cone $C_S$  gives the result.
\end{proof}

The next proposition gives a functorial viewpoint on symplectic parallel transport and regeneration. As in Appendix~\ref{subsec:flp}, we take $\Pi (\mcx)$ to be the path category of the stack $\mcx$ and $\mathbf{Symp}$ to be the category of symplectic manifolds.  If $\mcx \subset \secon{A} - \left( \mce_A \cup \partial \secon{A} \right)$, then we take $\mathbf{P}_\mcx : \Pi (\mcx) \to \mathbf{Symp}$ to be the parallel transport functor taking $p$ to $\fib{A}{p} - \partial \fib{A}{p} = \pi|_{\hyp{A} - \partial \hyp{A}}^{-1} (p)$ and a path to symplectic parallel transport. Denote the essential image of $\mathbf{P}_\mcx$ by $\mathbf{C} (\mcx )$. 

\begin{prop}  \label{prop:regrel} Assume $A^\prime$ affinely spans $\R^d$ and let $(\mci , \psi )$ be a regeneration of $A^\prime$ relative to $A$ and $\mcx = i_S^{-1}(\psi_0 (\mci)) \subset \secon{A^\prime}$. Then for any $t \ne 0$, there is a functor $F_{A^\prime} : \mathbf{C} (\psi_t (\mci) ) \to \mathbf{C} (\mcx )$ which the completes diagram
\begin{equation*}
\begin{CD} \Pi (\mci ) @>{\partrans}>> \mathbf{C} (\psi_t (\mci) )  \\ @V{i_S^{-1} \circ \psi_0}VV  @V{F_{A^\prime}}VV 
\\ \Pi (\mcx ) @>{\partrans}>> \mathbf{C} (\mcx)
\end{CD}
\end{equation*}
Furthermore, this diagram commutes up to isotopy.
\end{prop}
\begin{proof} Let $S = \{(Q_i, A_i) : i \in I \}$  and 
consider the singular symplectic fiber bundle $\mcf = \psi^* (\hyp{A} - \partial \hyp{A})$ over $B \times \mci$. Note that $\mcf$ is smooth over $(B - \{0\}) \times \mci$, but that over $\{0\} \times \mci$, the fibers of $\mcf$ are singular unions $\cup_{i \in I} \mcz_i$ where $\mcz_i \subset \mcx_{Q_i}$. By Proposition \ref{prop:hypdeg}, these fibers are stable pair degenerations. After excising the intersections $\mcz_i \cap \mcz_j$, this decomposition can be made global on $\mcf$ by taking symplectic parallel transport along rays in $B$ to the origin and removing the vanishing cycle $W$. From Proposition \ref{prop:hypdeg}, we have that $W$ is the singular coisotropic hypersurface consisting of all points that flow into the critical locus  of $\psi_0^* \hyp{A}$. Taking $\mcf^\prime = \mcf - W$, we have that $\mcf^\prime = \sqcup_{i \in I} \mcf_i^\prime$ is a smooth symplectic bundle over $B \times \mci$ whose connected components are indexed by the polytopes $(Q_i , A_i)$ in $S$. Using symplectic parallel transport along rays in $B$, the fiber of $\mcf^\prime_i$ over $(t, p)$ for any $p \in \mci$ is symplectomorphic to $\mcz_i - \partial \mcz_i$ over $(0,p)$.  As $A^\prime$ affinely spans $\R^d$, we have that $(Q^\prime, A^\prime) =(Q_{i_0}, A_{i_0})$ for some $i_0 \in I$. For $p \in \mci$, take $F_{A^\prime} (p)$ to be the fiber of $\mcf^\prime_{i_0}$ over $(0,p)$. While parallel transport along $\mcf^\prime / \{t\} \times \mci$ may not strictly commute with the parallel transport along the rays $[0, t] \times p$ for $p \in \mci$, they do commute up to isotopy yielding the homotopy commutative diagram in the proposition.
\end{proof}

Proposition~\ref{prop:regrel} suggest a general method of approaching the symplectomorphism group of a a hypersurface in a toric stack through an analysis of the groups on degenerated pieces. Of course, the general case of $A$ is exceptionally complex as it requires an understanding of groups for all smaller sets $A^\prime \subset A$. In this section we will see to what extent this approach is accessible in an example where $A$ is minimally more complicated, namely $A$ contains $(d + 3)$ points. 

The general case of $(d + 3)$ points has been studied and explicit formulas for $E_A$ are known \cite{SD}. At this level of generality, the formulas do not immediately render the geometry of the principal $A$-determinant or its complement accessible. However, it is worth mentioning that the  $A$-discriminant component is always a rational curve in a $\secon{A}$, usually with complicated singularities \cite{horn}. 

\begin{eg} We continue to explore Example~\ref{eg:dplus3} and take
\begin{align} \label{eq:regex} A & =  \{ (1, 0), (0, 1), (1, 1) , (-1, -1), (0, 0) \}. \end{align}
Any non-degenerate hypersurface $\fib{A}{p}$ is an elliptic curve with $4$ boundary points. By writing out the set of regular triangulations of $(Q,A)$ and applying equation \eqref{eq:secvert}, one obtains the vertices of $\psec{A}$ in $\R^{{\mathcal{A}}}$. Translating and pulling back to $L_{{\mathcal{A}}}$ via $\alpha_{{\mathcal{A}}}$ gives the secondary polytope $\psecv{A}{v}$ on the right hand side of Figure \ref{fig:secF1}. 
To obtain the stacky fan of the secondary stack, first observe that, for each coarse subdivision $S = \{(Q_i, A_i) : i \in I\}$ of $(Q,A)$ and pointing set $A_i$, the unique primitive function defining $S$ and zero on $A_i$ is a $\Lambda$-defining function so that  $\du{\eta}_{(S,A_i)} = \eta_{(S,A_i)}$. Thus, using equation \eqref{eq:deftildbeta},  $\tilde{\beta}_{\dul{\plaf{A}}} (e_{\eta_{(S,A_i)}}) = \eta_{(S,A_i)} $. Also, since $K_{{\mathcal{A}}} = 0$, Lemma \ref{lem:KtoXi} implies that $\Xi_{{\mathcal{A}}} = L_{{\mathcal{A}}}^\vee$. Finally, applying the Lemma \ref{lem:stfansecon} shows that the stacky fan for $\secon{A}$ equals the normal fan of $\psecv{A}{v}$ which has one-cone generators 
\begin{equation*} \fans{A} =  \{v_1, \ldots, v_4\} = \{(1, 1), (0, 1), (-2, -3), (1, 0)\} \subset \Z^2 .\end{equation*}
The secondary fan and polytope are illustrated in Figure \ref{fig:secF1}. 

To simplify the cumbersome notation, we order $A$ as in equation \eqref{eq:regex} and write $y_i$ for the monomial which evaluates the $i$-th coefficient. For example, $y_4 = x_{(-1,-1)}$ is regarded as the projection $\linsys{A} = \C^A$ to the $(-1,-1)$-coordinate. Then, utilizing \cite[Theorem~10.1.2]{GKZ}, one can compute the $A$-discriminant and the principal $A$-determinant to be
\begin{align*}
\Delta_A & = y_1 y_2 y_4 y_5^3 - y_3 y_4 y_5^4 + 27 y_1^2 y_2^2 y_4^2 - 36 y_1 y_2 y_3 y_4^2 y_5 + 8 y_3^2 y_4^2 y_5^2 - 16 y_3^3 y_4^3, \\
E_A & = y_1^2 y_2^2 y_3 y_4 \Delta_A.
\end{align*}
From Definition~\ref{defn:secstack}, the principal $A$-determinant induces a section of $\mco_{\secon{A}} (1)$ denoted $E^s_A$. By taking the unique interior point of $\psecv{A}{v}$ to be zero, and using the right side of Figure~\ref{fig:secF1} as the exponents of the Laurent monomials, we obtain coordinates $(u_1, u_2)$ of the maximal $(\C^*)^2$-orbit of $\secon{A}$ over which $\mco_{\secon{A}} (1)$ is trivialized.  Then the principal $A$-determinant restricts to the Laurent polynomial
\begin{align*}
u_2^{-1} - u_1^{-1} + 27 u_1^{-1} u_2 - 36 + 8u_1 u_2^{-1} - 16 u_1^2 u_2^{-1}.
\end{align*}

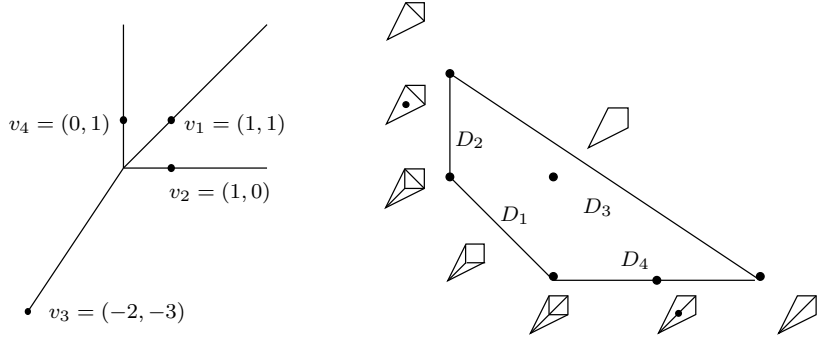
\begin{figure}[t]
\begin{picture}(0,0)%
\includegraphics{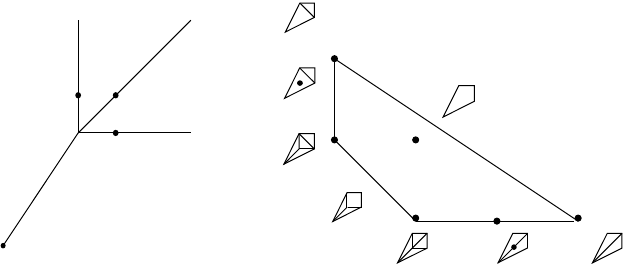}%
\end{picture}%
\setlength{\unitlength}{4144sp}%
\begingroup\makeatletter\ifx\SetFigFont\undefined%
\gdef\SetFigFont#1#2#3#4#5{%
  \reset@font\fontsize{#1}{#2pt}%
  \fontfamily{#3}\fontseries{#4}\fontshape{#5}%
  \selectfont}%
\fi\endgroup%
\begin{picture}(4751,2005)(797,-1739)
\put(3646,-1051){\makebox(0,0)[lb]{\smash{{\SetFigFont{8}{9.6}{\rmdefault}{\mddefault}{\updefault}{\color[rgb]{0,0,0}$D_1$}%
}}}}
\put(3376,-601){\makebox(0,0)[lb]{\smash{{\SetFigFont{8}{9.6}{\rmdefault}{\mddefault}{\updefault}{\color[rgb]{0,0,0}$D_2$}%
}}}}
\put(4141,-1006){\makebox(0,0)[lb]{\smash{{\SetFigFont{8}{9.6}{\rmdefault}{\mddefault}{\updefault}{\color[rgb]{0,0,0}$D_3$}%
}}}}
\put(4366,-1321){\makebox(0,0)[lb]{\smash{{\SetFigFont{8}{9.6}{\rmdefault}{\mddefault}{\updefault}{\color[rgb]{0,0,0}$D_4$}%
}}}}
\put(1756,-511){\makebox(0,0)[lb]{\smash{{\SetFigFont{8}{9.6}{\rmdefault}{\mddefault}{\updefault}{\color[rgb]{0,0,0}$v_1 = (1,1)$}%
}}}}
\put(1666,-916){\makebox(0,0)[lb]{\smash{{\SetFigFont{8}{9.6}{\rmdefault}{\mddefault}{\updefault}{\color[rgb]{0,0,0}$v_2 = (1,0)$}%
}}}}
\put(946,-1636){\makebox(0,0)[lb]{\smash{{\SetFigFont{8}{9.6}{\rmdefault}{\mddefault}{\updefault}{\color[rgb]{0,0,0}$v_3 = (-2, -3)$}%
}}}}
\put(706,-511){\makebox(0,0)[lb]{\smash{{\SetFigFont{8}{9.6}{\rmdefault}{\mddefault}{\updefault}{\color[rgb]{0,0,0}$v_4 = (0,1)$}%
}}}}
\end{picture}%
\caption{\label{fig:secF1} The secondary fan and polytope of $A$}
\end{figure}

As was pointed out in Example~\ref{eg:dplus3}, there are five extended circuits contained in $A$ and four circuits $\{C_1, C_2, C_3, C_4\}$, in this case they correspond bijectively to the four boundary divisors of $\secon{A}$. Denote the facet of $\Sigma (A)$ corresponding to $v_i$ by $F_i$, the subdivision defining the facet by $S_i$ and the divisor in $\secon{A}$ by $D_i$. The divisor $D_4$ corresponds to the degenerate circuit $C_4 = \{ (-1,-1), (0,0), (1,1)\}$ supporting two extended circuits. To each circuit $C_i$ there is a unique triangular extension given by the subdivision associated to the facet defined by $v_i$. Let us first examine regenerations of $\secon{C_1} \cong D_1$ and $\secon{C_2} \cong D_2$. 
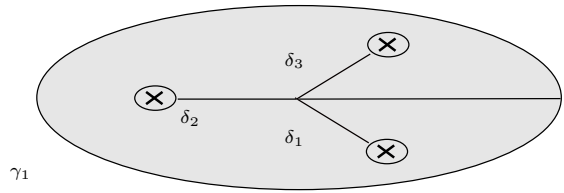
\begin{figure}[b]
\begin{picture}(0,0)%
\includegraphics{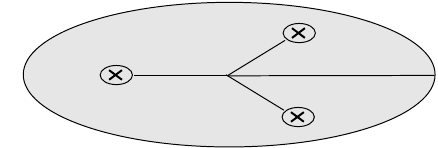}%
\end{picture}%
\setlength{\unitlength}{4144sp}%
\begingroup\makeatletter\ifx\SetFigFont\undefined%
\gdef\SetFigFont#1#2#3#4#5{%
  \reset@font\fontsize{#1}{#2pt}%
  \fontfamily{#3}\fontseries{#4}\fontshape{#5}%
  \selectfont}%
\fi\endgroup%
\begin{picture}(3324,1118)(910,-846)
\put(2570,-97){\makebox(0,0)[lb]{\smash{{\SetFigFont{7}{8.4}{\rmdefault}{\mddefault}{\updefault}{\color[rgb]{0,0,0}$\delta_3$}%
}}}}
\put(1945,-441){\makebox(0,0)[lb]{\smash{{\SetFigFont{7}{8.4}{\rmdefault}{\mddefault}{\updefault}{\color[rgb]{0,0,0}$\delta_2$}%
}}}}
\put(2570,-561){\makebox(0,0)[lb]{\smash{{\SetFigFont{7}{8.4}{\rmdefault}{\mddefault}{\updefault}{\color[rgb]{0,0,0}$\delta_1$}%
}}}}
\put(925,-758){\makebox(0,0)[lb]{\smash{{\SetFigFont{7}{8.4}{\rmdefault}{\mddefault}{\updefault}{\color[rgb]{0,0,0}$\gamma_1$}%
}}}}
\end{picture}%
\caption{\label{fig:regd2} Paths for the regenerated circuit of $C_2$}
\end{figure}

We first will find the intersection numbers $\mce_A \cdot D_1$ and $\mce_A \cdot D_2$. For this, we compute in the homogeneous coordinate ring $\C [x_1, x_2, x_3, x_4]$ of $\secon{A}$ given in equation \eqref{eq:coxring} which is graded by $\Pic (\secon{A}) \cong L_{\dul{\psecv{A}{v}}}^\vee \cong \Z^2$. To obtain the degree of the monomial $x_i$ which defines $D_i$, apply $\alpha_{\dul{\psecv{A}{v}}}^\vee$ to $e_i^\vee \in \Z^{\dul{\psecv{A}{v}}}$. After a choice of basis, we obtain $\deg (x_1) = (1, 2) $, $\deg (x_2 ) = (1, 0)$, $\deg (x_3) = (1, 1)$ and $\deg (x_4) = (2, 1)$. With this choice of basis, a straightforward computation in intersection theory of toric varieties (see \cite[Section 5.1]{fulton}) gives the intersection pairing
\begin{align*}
\left[ \begin{matrix}
- \frac{1}{3} & \frac{2}{3} \\[6pt] \frac{2}{3} & -\frac{5}{6}
\end{matrix} \right] .
\end{align*}
We also calculate that $\mco_{\secon{A}} (1) = \mco (D_1 + D_2 + D_3 + D_4)$ which corresponds to $(5, 4)$ so that  $\mce_A \cdot D_1 = 1 = \mce_A \cdot D_2$. 

Starting with $C_1$ we observe that $N_{\secon{A}} \secon{C_1}$ is isomorphic to $\mco (-1 )$ over $ \p^1$ and trivializes over $\orb{F_1}$ where $\orb{F}$ is the maximal torus orbit associated to a face of $Q$. The circuit $C_1 = \{(0,0), (1, 0), (0,1), (1,1)\}$ affinely generates $\Z^2$ which, by \cite[Theorem 1.12]{GKZ} and the computation $\mce_A \cdot D_1 = 1$ implies that the restriction of $E^s_A$ to $\orb{F_1}$ equals $E^s_{C_1}$ with multiplicity $1$. Since $\mce_{C_1}$ is a point, this implies that in a tubular neighborhood $U$ of $\orb{F_1}$, $\mce_A \cap U$ is a disc transversely intersecting $\orb{F_1}$. Removing an $\varepsilon$-neighborhood of this disc gives a regeneration $\mci$ of $C_1$ in $A$. Utilizing $i_{S_1} : \secon{C_1} \to \secon{A}$, we choose a basepoint $p_1 \in \mci - \orb{F_1}$ close to $i_{S_1} (t_0)$ where $t_0$ appears in Theorem~\ref{thm:circuitrelation}. Then Proposition \ref{prop:regrel} gives us that the regenerated circuit relation is simply the $(2, 2)$ circuit relation restricted to the region $V_1 \subset \fib{A}{p_1}$. Here $V_1$ is the open subset in $\fib{A}{p_1}$ which converges via symplectic parallel transport to the degenerated component of $\mcx_Q$ corresponding to $C_1$ as $p_1 \in U$ tends towards the boundary $D_1$.

The divisor $D_2$ is $\p (1, 3)$ with normal bundle $\mco_{\p (1, 3)} (- 1)$ where $\mco_{\p (1, 3)} (d)$ corresponds to the equivariant line bundle over $\C^2 - \{(0,0)\}$ with character $z^d \in \Hom (\C^* , \C^*)$. Even after deleting the point at infinity, we can not regenerate $D_1$ using sections of this bundle because of the stacky point at the origin, so we must consider a covering. There is only one non-trivial covering in this case, namely the \'etale cover $z^3$ of $\interior{C_2}{\varepsilon} \subset \p (1, 3) - D_1 \cap D_2 \approx \C / \mu_3$. To find the regeneration which extends this cover, one simply takes the stacky chart of a neighborhood $U$ of the point $D_2 \cap D_3$ which is $\C^2 / \mu_3$ where $\zeta (t, x) = (\zeta^{-1} t , \zeta x)$. The map $\psi : \C^2 \to U$ is obviously \'etale  and at $t = 0$ gives the covering above, so restricting $\psi$ to $\psi^{-1} (U - V)$ where $V$ is an $\varepsilon$ neighborhood of $E_A$ gives a regeneration of $C_2$. Applying Proposition \ref{prop:regrel} to this situation, we observe that $\psi^{-1} (U - V) \cap \{t \} \times \C$ is a disc with three discs removed near the third roots of unity as in Figure \ref{fig:regd2}. Using the proposition and Theorem \ref{thm:circuitrelation}, composing the parallel transport $T_i$ along the three paths $\delta_i$ gives the cube of parallel transport along $\gamma$ as well as a full boundary twist. Taking the composition of these two operations as $T_4$ we write simply $T_1 T_2 T_3 = T_4$ and observe this as a relation in $\fib{A}{p_2}$ where we choose $p_2$ in the interior of $\secon{A}$ and close to $i_{S_2} (t_0)$.
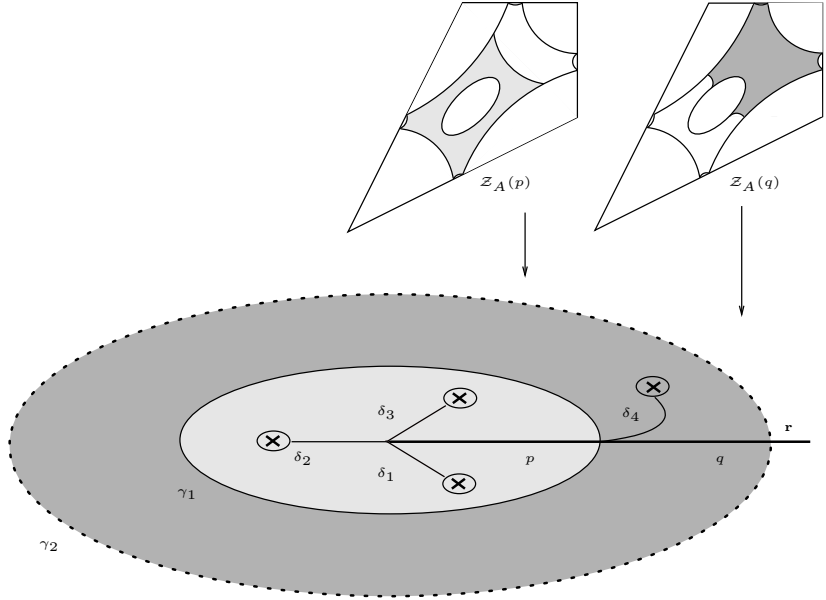
\begin{figure}[b]
\begin{picture}(0,0)%
\includegraphics{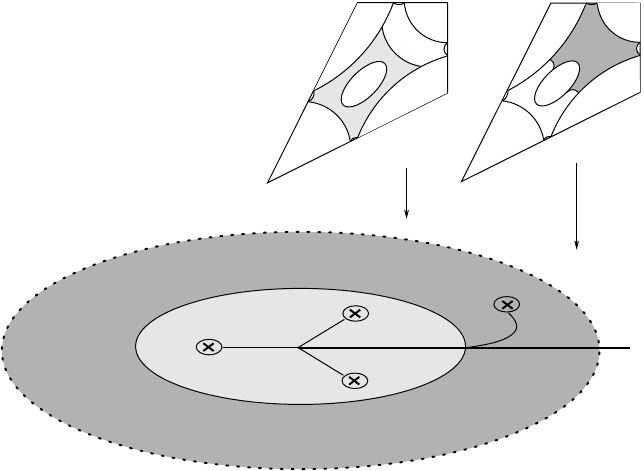}%
\end{picture}%
\setlength{\unitlength}{4144sp}%
\begingroup\makeatletter\ifx\SetFigFont\undefined%
\gdef\SetFigFont#1#2#3#4#5{%
  \reset@font\fontsize{#1}{#2pt}%
  \fontfamily{#3}\fontseries{#4}\fontshape{#5}%
  \selectfont}%
\fi\endgroup%
\begin{picture}(4891,3579)(945,-2998)
\put(3159,-1901){\makebox(0,0)[lb]{\smash{{\SetFigFont{5}{6.0}{\rmdefault}{\mddefault}{\updefault}{\color[rgb]{0,0,0}$\delta_3$}%
}}}}
\put(2658,-2177){\makebox(0,0)[lb]{\smash{{\SetFigFont{5}{6.0}{\rmdefault}{\mddefault}{\updefault}{\color[rgb]{0,0,0}$\delta_2$}%
}}}}
\put(4622,-1914){\makebox(0,0)[lb]{\smash{{\SetFigFont{5}{6.0}{\rmdefault}{\mddefault}{\updefault}{\color[rgb]{0,0,0}$\delta_4$}%
}}}}
\put(3159,-2273){\makebox(0,0)[lb]{\smash{{\SetFigFont{5}{6.0}{\rmdefault}{\mddefault}{\updefault}{\color[rgb]{0,0,0}$\delta_1$}%
}}}}
\put(5595,-1998){\makebox(0,0)[lb]{\smash{{\SetFigFont{5}{6.0}{\rmdefault}{\mddefault}{\updefault}{\color[rgb]{0,0,0}$\mathbf{r}$}%
}}}}
\put(1962,-2377){\makebox(0,0)[lb]{\smash{{\SetFigFont{5}{6.0}{\rmdefault}{\mddefault}{\updefault}{\color[rgb]{0,0,0}$\gamma_1$}%
}}}}
\put(1137,-2691){\makebox(0,0)[lb]{\smash{{\SetFigFont{5}{6.0}{\rmdefault}{\mddefault}{\updefault}{\color[rgb]{0,0,0}$\gamma_2$}%
}}}}
\put(3769,-531){\makebox(0,0)[lb]{\smash{{\SetFigFont{5}{6.0}{\rmdefault}{\mddefault}{\updefault}{\color[rgb]{0,0,0}$\fib{A}{p}$}%
}}}}
\put(5262,-531){\makebox(0,0)[lb]{\smash{{\SetFigFont{5}{6.0}{\rmdefault}{\mddefault}{\updefault}{\color[rgb]{0,0,0}$\fib{A}{q}$}%
}}}}
\put(5183,-2180){\makebox(0,0)[lb]{\smash{{\SetFigFont{5}{6.0}{\rmdefault}{\mddefault}{\updefault}{\color[rgb]{0,0,0}$q$}%
}}}}
\put(4044,-2180){\makebox(0,0)[lb]{\smash{{\SetFigFont{5}{6.0}{\rmdefault}{\mddefault}{\updefault}{\color[rgb]{0,0,0}$p$}%
}}}}
\end{picture}%
\caption{\label{fig:regpath}Generating paths for $\mathbf{G}_{\curve_t}$ with trivialized fiber over $\mathbf{r}$} 
\end{figure}

One can often regenerate several subsets of $A$ simultaneously, thereby incorporating the symplectomorphisms of the regenerated pieces into those of the hypersurface $\fib{A}{t}$. We give a more systematic account of this method in the next section for extended circuits, but for now we consider sections of the ample line bundle $\mcl = \mco (D_1 + 3 D_2)$ on $\secon{A}$. Consider the pencil 
\begin{align*} f(x_1, x_2, x_3, x_4) = [s_0 : s_\infty] := [x_1 x_2^3 : x_4^2]. \end{align*} Taking $\curve_t = \{s_0 - t s_\infty \}$, one observes that for small $t$, we obtain a smooth curve which approximates $D_1 + 3D_2$. We wish to understand the $\curve_t$ subgroup $\mathbf{G}_{\curve_t} \subset \Symp (\fib{A}{p} , \partial \fib{A}{p} )$ from Definition \ref{defn:curvesubgroup} by viewing $\curve_t$ as a simultaneous regeneration of $C_1$ and $C_2$. We trivialize the fibers $\fib{A}{p}$ along the ray $\mathbf{r} = \R_{\geq 0} \subset \C$ and consider parallel transport $\{T_1, \ldots, T_4 , \tilde{T}_1, \tilde{T}_2 \}$ along the paths $\{ \delta_1, \delta_2, \delta_3, \delta_4, \gamma_1, \gamma_2\}$ as in Figure \ref{fig:regpath}.

 Utilizing Proposition \ref{prop:regrel}, the monodromy symplectomorphisms $T = \mathbf{P} (\delta)$ on the degenerated hypersurfaces can be regenerated to monodromy transformations on the smooth hypersurfaces. These are the compositions of disjoint Dehn twists 
\begin{eqnarray*} T_1 & = & T_{ k_1 } ,\\
T_2 & = & T_{k_2} ,\\
T_3 & = & T_{k_3} , \\
T_4 & = & T_{a } , \\
\tilde{T}_1 & = & T_b T_d^3 T_e^3 T_f^3 ,  \\
\tilde{T}_2 & = & T_c T_d^2 T_e^2 T_f^3 .
\end{eqnarray*}
 
Those associated to $\gamma_1$ and $\gamma_2$ correspond to monodromy around the hypersurface degeneration associated to the points $D_1 \cap D_2$ and $D_1 \cap D_4$. The vanishing cycles for the twists $T_i$ are given in Figure \ref{fig:regvc}.
 
One can calculate that $E^s_A$ has precisely one cusp in the interior of $\secon{A}$. This cusp yields the braid relations between $T_4$ and $T_i$ for $i = 1, 2, 3$. Adding these to the circuit relations, we obtain a finite presentation of $\mathbf{G}_{\curve_t}$.
\begin{equation*} \mathbf{R} \to <T_1, \ldots, T_4 , \tilde{T}_1, \tilde{T}_2 > \to \mathbf{G}_{\curve_t} \to 1 \end{equation*}
\begin{figure}
\begin{picture}(0,0)%
\includegraphics{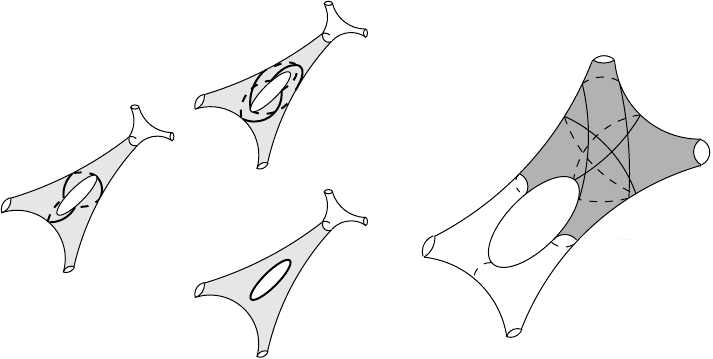}%
\end{picture}%
\setlength{\unitlength}{4144sp}%
\begin{picture}(5417,2728)(133,-2190)
\put(4632,-227){\makebox(0,0)[lb]{\smash{$a$}}}
\put(4306,-377){\makebox(0,0)[lb]{\smash{$b$}}}
\put(5072,-359){\makebox(0,0)[lb]{\smash{$c$}}}
\put(3951,-821){\makebox(0,0)[lb]{\smash{$d$}}}
\put(4540,-1411){\makebox(0,0)[lb]{\smash{$e$}}}
\put(3624,-1683){\makebox(0,0)[lb]{\smash{$f$}}}
\put(4650,162){\makebox(0,0)[lb]{\smash{$\partial_1$}}}
\put(3195,-1283){\makebox(0,0)[lb]{\smash{$\partial_2$}}}
\put(4152,-2009){\makebox(0,0)[lb]{\smash{$\partial_3$}}}
\put(5411,-878){\makebox(0,0)[lb]{\smash{$\partial_4$}}}
\put(2409,-1727){\makebox(0,0)[lb]{\smash{$k_1$}}}
\put(923,-1067){\makebox(0,0)[lb]{\smash{$k_2$}}}
\put(2453,-212){\makebox(0,0)[lb]{\smash{$k_3$}}}
\end{picture}%
\caption{\label{fig:regvc} }
\end{figure}
One can use this method for higher dimensions as well, but understanding the singularities of $E_A^s$ for $(d + 3)$-sets is necessary to the completion of this project, as these generate additional relations.

As a final remark, we observe that near $\curve_\infty$, we obtain a regeneration of the circuit $C_4$. Observing that $D_4 \cdot E_A^s = 2$, we see that the critical value in $\curve_\infty$ splits into two values for each of the branches of the $2$-fold \'etale cover, yielding a total of four critical values. To see the effect on the vanishing cycles, observe that the family $\curve_{1 / t}$ regenerates two extended circuits, each of which has a relation as given in Section~\ref{sec:exampledeg}. This has the effect of gluing the degenerate vanishing cycles together to obtain two vanishing cycles, for each  branch of the \'etale cover, while parallel transport from one branch to the other yields a regenerated version of the involutions $T_1 T_\infty$ on each regenerated circuit as given in Section \ref{sec:exampledeg}. However, to obtain the correct gluing formulas for these cycles requires a more nuanced control over the boundary framing in the degenerate case. \end{eg}

\section{Homological mirror symmetry applications}
In this subsection we outline a strategy to decompose the Fukaya-Seidel category associated to a pencil of hypersurfaces in a toric stack. After giving a combinatorial description of the decompositions, we discuss applications to the homological mirror symmetry conjecture for Fano toric stacks. The original conjecture has been settled in the case of toric del Pezzo surfaces in \cite{ueda} and weighed projective planes in \cite{ako08}. There are also several variants of the conjecture that have been proven, where the Fukaya-Seidel category is replaced with a different category (see \cite{abouzaid}, \cite{fltz14}). However, our strategy is to consider the original Fukaya-Seidel category as constructed in \cite{SeidelFPL} and produce more detailed information on the structure of the equivalent categories. We conjecture a refined correspondence leading to a variety of equivalences associated to different degenerations of the LG mirrors. In particular, we will observe a finite collection of semi-orthogonal decompositions arising from edge paths in the secondary polytope. To each decomposition we formulate a conjectural homological mirror collection resulting from birational moves in the $B$-model setting.

\subsection{Landau-Ginzburg Degenerations}

We begin by considering the toric stack $\mcx_Q$ associated to the marked polytope $(Q, A)$, the line bundle $\mco_A (1)$ and the linear system $\linsys{A} \subset H^0 (\mcx_Q, \mco_A (1) ) $ from Definition~\ref{defn:polystack}. By the support of a section $s \subset \linsys{A} \cong \C^A$, we mean the subset $A^\prime \subseteq A$ whose monomials have non-zero coefficients as summands of $s$. Referring to Definition~\ref{defn:sections}, $s$ is called a very full section if its support equals  $A$ and a full section if its support contains the vertices of $Q$.  Given any subset $A^\prime \subset A$ and a section $s = \sum_{a \in A} c_a e_a \in \C^A$, we say the restriction of $s$ to $A^\prime$ is $s||_{A^\prime} = \sum_{a \in A^\prime} c_a e_a$. By an $A$-pencil, we mean a pencil in $\linsys{A}$. If it is clear from the context, we will simply write pencil for $A$-pencil. For what follows, we will consider $A$-pencils satisfying a strong, but common, property.

\begin{defn} \label{defn:shp}\gls{lgmod} \begin{itemize} \item[(i)] Given $A \subset \Lambda$ and $A^\prime \subset A$, a pencil $W \subset \linsys{A}$ is $A^\prime$-sharpened if it contains a full section $s$ with $0 \ne s||_{A^\prime} \in W$. 
\item[(ii)] The Landau-Ginzburg, or LG model associated to an $A^\prime$-sharpened pencil $W$ is the induced map $\mathbf{w} : \mcx_Q - D_W \to \C$ where $D_W = \textnormal{Zero} (s ||_{A^\prime})$ is the fiber over infinity of the pencil.
\end{itemize}
\end{defn}

Our motivation to consider such pencils comes from homological mirror symmetry of Fano toric varieties (see \cite{givental94}, \cite[Section~3]{horivafa}). Given a $d$-dimensional Fano toric stack specified by a fan $\fan$, the Batyrev mirror is defined as $\mcx_Q$ (or a partial crepant resolution thereof) with $A$ equal to the union of $0$ and the primitive generators of the one-cones $\fan (1)$. A symplectic structure on the original variety then specifies a superpotential $\mathbf{w}$ on $(\C^*)^d \subset \mcx_Q$. From \cite{givental94}, one observes that $\mathbf{w}$ is the LG model associated to a $\{0\}$-sharpened pencil $W \subset \linsys{A}$ on $\mcx_Q$.  In fact, the case where $A^\prime = \{a\}$ is a single element of $A$ can simplify the discussion because, in such cases, a pencil is $A^\prime$-sharpened if and only if it contains $e_a$. For now, though, we keep the exposition general.

Given an $A^\prime$-sharpened pencil we associate a rank $1$ sublattice $\tilde{\Gamma}_{A^\prime} \subset (\Z^A)^\vee$ generated by the cocharacter $e_{A^\prime}^\vee := \sum_{a \in A^\prime} e_a^\vee$. This induces a one parameter subgroup which we denote by\gls{tilgap} $\tilde{G}_{A^\prime} \subset (\C^*)^A$.
\begin{lem} \label{lem:orb}
	A pencil $W \subset \C^A$ is $A^\prime$-sharpened if and only if it contains a full section and is stable under the action of $\tilde{G}_{A^\prime}$.
\end{lem}
\begin{proof}
Suppose that $W$ is an $A^\prime$-sharpened pencil. It is elementary to check that there exists a full section $s \in W$ for which $s_\infty := s||_{A^\prime} \ne s$. Then the support of $s_0 := s - s||_{A^\prime}$ is non-empty and disjoint from $A^\prime$. As $W$ is a pencil, $W = \lin_{\R} \{s_0, s_\infty\}$. The cocharacter $e^\vee_{A^\prime}$ gives the one-parameter subgroup $\tilde{G}_{A^\prime} \subset  \C^* \otimes (\Z^{{\mathcal{A}}})^\vee$ which acts by 
\begin{align*} \lambda \cdot \left( \sum_{a \in A} c_a e_a \right) = \sum_{a \in A^\prime} \lambda\, c_a e_a + \sum_{a \not\in A^\prime} c_a e_a . \end{align*}
Thus $\lambda \cdot s_\infty = \lambda s_\infty$ and $\lambda \cdot s_0 = s_0$ which implies that $\tilde{G}_{A^\prime} (W) = W$. 

Conversely, if $\tilde{G}_{A^\prime} (W) = W$ and $s \in W$ then $s - \lim_{\lambda \to 0} \lambda \cdot s = s||_{A^\prime} \in W$ which implies that it is an $A^\prime$-sharpened pencil.
\end{proof} 

We now wish to consider $A^\prime$-sharpened pencils up to toric equivalence. This involves passing from closures of $\tilde{G}_{A^\prime}$-orbits in the space of sections $\linsys{A} = \C^A$ to their counterparts in the stack  $\secon{A}$. We first note that the stacky fan for $\secon{A}$ is given in Lemma~\ref{lem:stfansecon} as
\begin{align*}
\widetilde{\sfan}_{\psec{A}} = \left( \Z^{\dul{\psec{A}}}, \Xi_{{\mathcal{A}}}, \tilde{\beta}_{\dul{\psec{A}}}, \fan_{\mcb}  \right).
\end{align*}
The group $\Xi_{{\mathcal{A}}}$ is realized as the colimit of the diagram \eqref{eq:colimdiag} so that there is a map $\tilde{\alpha}_{{\mathcal{A}}}: \left( \Z^{{\mathcal{A}}} \right)^\vee \to \Xi_{{\mathcal{A}}}$ and we take\gls{rhoaprime} $\rho_{A^\prime} = \tilde{\alpha}_{{\mathcal{A}}} (e^\vee_{A^\prime})$ and $G_{A^\prime} = (\tilde{\alpha}_{{\mathcal{A}}} \otimes \C^*) (\tilde{G}_{A^\prime})$. Now, there is a quotient map 
\begin{align}
F : (\mathbb{C}^*)^{{\mathcal{A}}_v} \times \C^{{\mathcal{A}}_{nv}} \to \mcv_A
\end{align} 
from the space of full sections $(\mathbb{C}^*)^{{\mathcal{A}}_v} \times \C^{{\mathcal{A}}_{nv}} \subset \C^{{\mathcal{A}}}$ to its moduli space $\mcv_A$ defined in equation \eqref{eq:fullsec}. It follows from the definition of $\mcv_A$ that $F$ is equivariant with respect to the groups $\tilde{G}_{A^\prime}$ and $G_{A^\prime}$ for any $A^\prime \subset A$. Using the notation of the proof of Lemma~\ref{lem:orb}, if $W$ is an $A^\prime$-sharpened pencil and $A^\prime$ does not contain the vertices $A_v$, then its intersection with the space of full sections is $W - \lin_\R (s_0)$. This implies that  $F (W \cap (\mathbb{C}^*)^{{\mathcal{A}}_v} \times \C^{{\mathcal{A}}_{nv}})$ is $\C \subset \mcv_A$ or a finite quotient thereof and is the closure of a $G_{A^\prime}$-orbit in $\mcv_A$.  By Theorem~\ref{thm:toric1}, there is an open embedding of toric stacks from $\mcv_A$ to $\secon{A}$. Thus we may view $W$, up to toric equivalence, as a the closure of a $G_{A^\prime}$-orbit contained in the substack $\mcv_A$ of $\secon{A}$. 

Were we to consider only those orbits intersecting the maximal torus in $\secon{A}$, its space would be easily described as the quotient of the maximal torus in $\secon{A}$ by $G_{A^\prime}$, namely $\mathbb{G}_{\psec{A}} / G_{A^\prime}$ where $\mathbb{G}_{\psec{A}} = (\Xi_{{\mathcal{A}}} \otimes \C^*) \cong (\C^*)^{|A| - d - 1}$ is the torus acting on $\secon{A}$. To gain a better understanding of this space, we consider a natural compactification. At this point, we simplify by moving to the coarse space of $\secon{A}$ which we denote $X_{\psec{A}}$. Choose $x$ to be a point in the maximal orbit of $X_{\psec{A}}$ and $\phi = \overline{G_{A^\prime} \cdot x}$ to be the closure of its orbit.  Let $CV_{A^\prime}$ be the relative Chow variety of $1$-dimensional cycles of degree $[\phi]$. Then the maximal torus $\mathbb{G}_{\psec{A}}$ acts on $CV_{A^\prime}$ and, following the definition of Chow quotients, we take\gls{lgmoduli} $\mlg{A}{A^\prime}$ to be the closure of the orbit $\mathbb{G}_{\psec{A}} \cdot [\phi]$ in $CV_{A^\prime}$. It is not hard to see that the $\mathbb{G}_{\psec{A}}$ torus action on $X_{A}$  induces an action on $\mlg{A}{A^\prime}$ (which is the trivial action when restricted to $G_{A^\prime}$).
\begin{defn} \label{defn:maxdeg}
A fixed point $\xi \in \mlg{A}{A^\prime}$ under the $\mathbb{G}_{\psec{A}}$ action will be called a maximal degeneration of $W$.
\end{defn}

The first result we need is a combinatorial description of the maximal degenerations. For this, we review some terminology from \cite{BS1}, \cite{BS2} and \cite{KSZ}. Let $P \in \R^n$ be a $n$-dimensional polytope with vertices $\{p_1, \ldots, p_m\}$ and $\gamma: \R^n \to \R$ a linear map. We order the vertices so that if $q_i := \gamma (p_i )$, then $q_i \leq q_j$ if $i < j$ and write $Q = \gamma (P)$. Let $\theta \in (\R^n)^\vee$ be linearly independent from $\gamma$ and\gls{vthfth} $V_\theta$ the subspace spanned by $\gamma$ and $\theta$. We take $\mcf_\theta$ to be the fan in $V_\theta$ whose cones are intersections of cones in the normal fan of $P$ with $V_\theta$. Assume that the half-plane $H_\theta = \R \cdot \gamma  \oplus \R_{>0} \cdot \theta$ intersects the normal fan of $P$ transversely, by which we mean that every $k$-dimensional cone in $\mcf_{\theta}$ lying in $H_\theta$ is the intersection of $H_\theta$ with an $(n - 2 + k)$-dimensional cone in the normal fan of $P$. Ordering the $2$-cones $\mcf_\theta (2) = \{ \sigma_0, \ldots, \sigma_r\}$ clockwise, one obtains the increasing sequence, $p_{i_0} < \cdots < p_{i_r}$ of points on $P$ where $p_{i_j}$ is the vertex dual to $\sigma_{j}$. From the construction, it is clear that $\{p_{i_j}, p_{i_{j + 1}}\}$ lie on an edge of $P$ for any $0 \leq j < r$, $q_{i_0} = q_0$ and $q_{i_r} = q_m$. Any path 
\begin{align} \label{eq:defnparsimpath} \left< p_{i_0}, \ldots , p_{i_r} \right> \end{align}
obtained in this way is known as a \emph{parametric simplex path} relative to $\gamma$.

In \cite{BS1}, these paths were realized as the vertices of the fiber polytope\gls{mppol} $\Sigma_\gamma (P) := \Sigma (P , Q)$ called the monotone path polytope of $P$. Leaving a detailed review of fiber polytopes to the references above, we content ourselves to describe a theorem from \cite{KSZ}. Let $G \approx (\C^*)^n$ be a complex torus acting on a projective toric variety $X_\fan$ with fan $\fan \subset G^\vee_\R$ where\gls{gvgv} $G^\vee = \Hom (\C^* , G)$ and $G^\wedge = \Hom (G , \C^*)$ are the lattice of one parameter subgroups and characters respectively.
We recall from \cite[Section~9.4]{dolgachev} that if $G$ acts on a vector space $V$ and $\textnormal{wt} (V) \in G^\wedge$ are the set of characters which have non-trivial eigenspaces in $V$, then the convex hull of $\textnormal{wt} (V)$ in $G^\wedge_\R$ is called the weight polytope of $V$.
Assume that $L$ is an equivariant ample line bundle on $X_\fan$ and $P \subset G^\wedge_\R$ is the weight polytope for the action on $H^0 (X_\fan , L)$.  Elementary toric geometry gives $\fan$ as the normal fan of $P$. 

Suppose $H \subset G$ is a subgroup and take $E = H\cdot x$ for a non-boundary point $x \in X_\fan$. The Chow quotient $X_\fan // H$ is defined as the closure of the orbit $G \cdot E$ in the relative Chow variety of $\dim (H)$ cycles of degree $[E]$ in $X_\fan$. Write $\pi_H : G^\wedge_\R \to H^\wedge_\R$ for the associated projection and take $Q = \pi_H (P)$. 

\begin{thm}{\cite[Lemma~2.6]{KSZ}} \label{thm:fiberquotient} The Chow quotient $X_\fan // H$ is a projective toric variety with $G$ action and ample line bundle weight polytope equal to the fiber polytope $\Sigma (P, Q)$.
\end{thm}

Indeed, it was shown that $\Sigma (P, Q)$ is the Newton polytope of the Chow form of $E$. We now utilize this theorem.

\begin{cor} \label{prop:mppvertices} Suppose $W$ is an $A^\prime$-sharpened pencil. The maximal degenerations of $W$ are in bijective correspondence with the vertices of the monotone path polytope $\Sigma_{\rho_{A^\prime}} (\psec{A})$. 
\end{cor}

The iterated fiber polytope $\Sigma_{\rho_{A^\prime}} (\psec{A})$ in this proposition was initially examined in \cite{BS2}. 

\begin{proof} Since $\mlg{A}{A^\prime}$ is defined as the Chow quotient of $X_{\psec{A}}$ by $G_{A^\prime}$, we need only apply Theorem \ref{thm:fiberquotient} which gives that $\mlg{A}{A^\prime}$ is equivariantly homeomorphic to the toric variety $X_{\Sigma_{\rho_{A^\prime}} (\psec{A})}$ associated to the monotone path polytope $\Sigma_{\rho_{A^\prime}} (\psec{A})$. This confirms that the fixed points correspond bijectively to the vertices and proves the claim.
\end{proof}

\begin{figure}[b]
\begin{picture}(0,0)%
\includegraphics{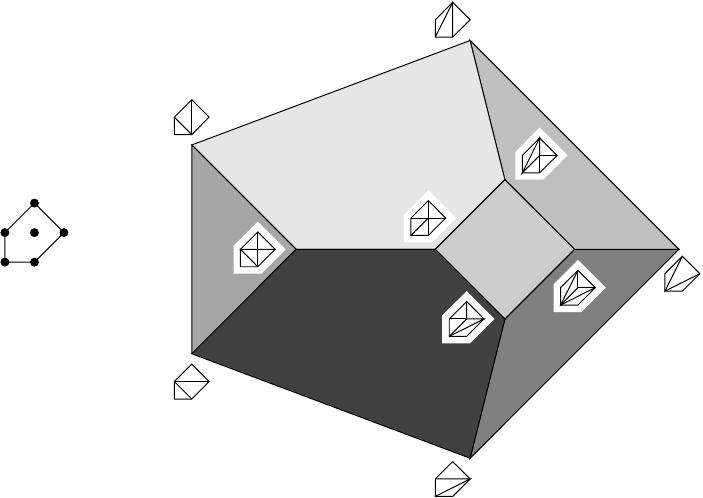}%
\end{picture}%
\setlength{\unitlength}{4144sp}%
\begingroup\makeatletter\ifx\SetFigFont\undefined%
\gdef\SetFigFont#1#2#3#4#5{%
  \reset@font\fontsize{#1}{#2pt}%
  \fontfamily{#3}\fontseries{#4}\fontshape{#5}%
  \selectfont}%
\fi\endgroup%
\begin{picture}(5345,3789)(1553,-3913)
\put(1606,-2521){\makebox(0,0)[lb]{\smash{{\SetFigFont{12}{14.4}{\rmdefault}{\mddefault}{\updefault}{$A$}%
}}}}
\put(5806,-3286){\makebox(0,0)[lb]{\smash{{\SetFigFont{12}{14.4}{\rmdefault}{\mddefault}{\updefault}{$\psec{A}$}%
}}}}
\end{picture}%
\caption{\label{fig:secbp1p1} $A$ and its secondary polytope}
\end{figure}
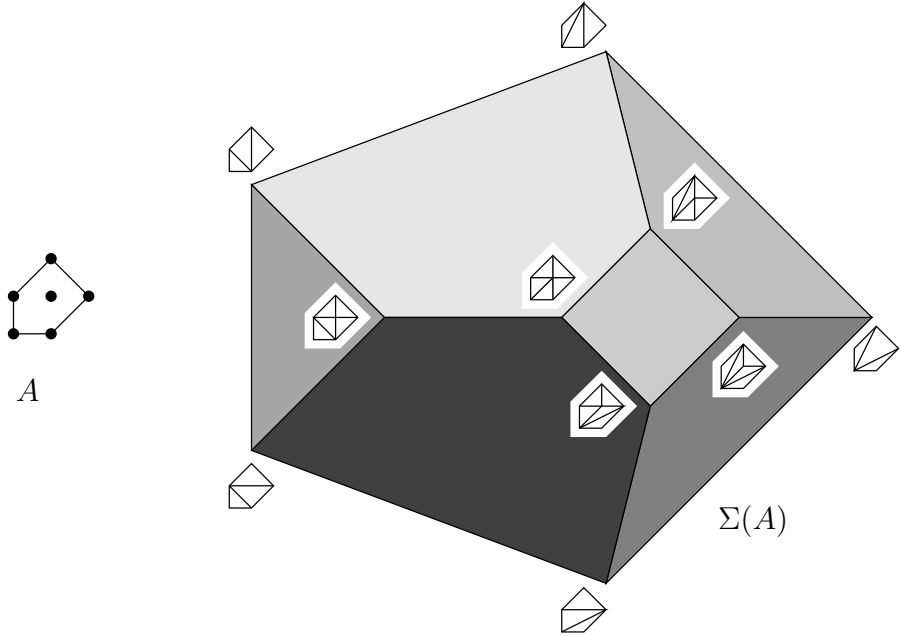

We now  study the fixed points of $\mlg{A}{A^\prime}$. Given a maximal degeneration $\xi \in \mlg{A}{A^\prime}$ associated to the parametric simplex path\gls{txi}  
\begin{align*} T_\xi = \left< t_{i_0}, \ldots , t_{i_r} \right> \end{align*}
defined in equation \eqref{eq:defnparsimpath}, we will write $C_1, \ldots, C_r$ for the irreducible components of the cycle $\xi$ in $\secon{A}$.  We will say that $\xi$ has length $r$ and for each $1 \leq j \leq r$, we will associate the pair of natural numbers\gls{dmj} $(d_j, m_j)$ where $[\xi] = \sum_{j = 1}^r d_j [C_j]$, and $m_j$ is the intersection number $\mce_A \cdot (d_j C_j )$. The total intersection number of $\mce_A$ with $\xi$ is then written as $m_\xi = \sum_{j = 1}^r m_j$. Note that this yields the intersection degree of  $\mce_A$ with any cycle in $\mlg{A}{A^\prime}$.

\begin{defn} \label{defn:decsimppath} Given a parametric simplex path $T_\xi$ associated to the fixed point  $\xi \in \mlg{A}{A^\prime}$, we  call the data\gls{ddatxi} $\ddata{\xi} = (T_\xi , \{(d_j, m_j)\})$ a \emph{decorated simplex path}. \end{defn} 

\begin{figure}[h]
	\begin{picture}(0,0)%
	\includegraphics{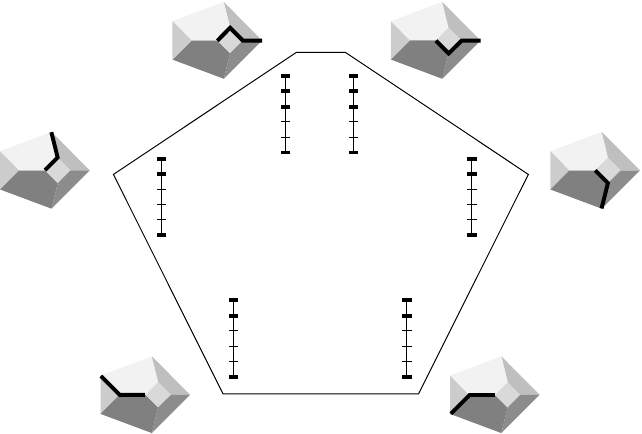}%
	\end{picture}%
	\setlength{\unitlength}{4144sp}%
	\begin{picture}(4794,3264)(1969,-3313)
	\end{picture}%
	\caption{\label{fig:mppbp1p1} The monotone path polytope defined by the $A^\prime$-sharpened pencil} 
\end{figure}

\begin{eg} \label{eg:msexample} As we give our next construction and other results, it will be useful to have an example to reference. We choose a sufficiently rich, but simple one arising as the homological, or Batyrev, mirror of $\p^1 \times \p^1$ blown up at one point. More explicitly, we let $A = \{ (-1, 0), (0, -1), (0, 1), (1, 0), (-1, -1), (0, 0)\}$ and we consider $A^\prime$-sharpened pencils where $A^\prime = \{(0, 0)\}$. Recall from Definition \ref{defn:shp} that an $A^\prime$-sharpened pencil must contain $e_{(0, 0)} \in \C^A$ as a section. The secondary polytope is illustrated in Figure \ref{fig:secbp1p1}. The function $\rho_{A^\prime} : \psec{A} \to \R$ is given by the restriction of $\rho_{A^\prime} : \R^{{\mathcal{A}}} \to \R$ which takes $\sum_{a \in A} c_a e_a$ to $c_{(0,0)}$. This defines the monotone path polytope $\Sigma_{\rho_{A^\prime}} (\psec{A})$ which is a hexagon represented in Figure \ref{fig:mppbp1p1}.  Each vertex of the monotone path polytope corresponds to a distinct parametric simplex path $T_\xi$. They are labeled with their  corresponding coherent tight subdivision of the interval $\rho_{A^\prime} (\psec{A} )$ inside the hexagon and the parametric simplex path on $\psec{A}$ outside of the hexagon. \end{eg}

Having decomposed the cycle $[\xi]$ representing the base of a LG model $\mathbf{w}$, we now will use this decomposition to partition the critical values of $\mathbf{w}$. We construct a decomposition of $\C$ based on the decorated simplex path $\ddata{\xi} = (T_\xi, \{(d_i, m_i)\})$ which will lead to the notion of a \textit{radar screen}.  To align the asymptotics correctly later, we define this decomposition in a fairly flexible fashion. Fix an increasing function $g: \{t_{i_0}, \ldots, t_{i_{r}}\} \to \R \cup \{\infty\}$ with $g(t_{i_0} ) = 0$ and $g (t_{i_r} ) = \infty$. For any $1 \leq j \leq r$ and any $0 \leq k <  d_{j}$ we take 
\begin{equation*} C_{j, k} = \{z \in \C : g(t_{i_j} ) \leq |z| < g( t_{i_{j + 1}} ) , 2 \pi k / d_{j}  \leq \arg (z) <  2 \pi (k + 1) / d_{j}  \} .
\end{equation*}

We totally order the collection $\{C_{j, k}\}$ of regions so that $C_{j, k} < C_{j^\prime , k^\prime}$ if and only if $j < j^\prime$ or $j = j^\prime$ and $k < k^\prime$. We now define a distinguished basis of paths $\mcb_{\ddata{\xi}} = \{ \gamma_1, \ldots, \gamma_{m_\xi} \}$ as in Appendix~\ref{subsec:flp} based at infinity and ordered so that if $\gamma_l (1) \in C_{j, k}$ and $\gamma_{l^\prime} (1) \in  C_{j^\prime , k^\prime}$ with $C_{j, k} < C_{j^\prime , k^\prime}$ then $l < l^\prime$. In order to make this collection precise, we fix a sufficiently small $\varepsilon > 0$ and, for every $j$ take $s_{i_j} := m_{i_j} / d_{i_j}$. For each $0 \leq k < d_j$, choose $s_{i_j}$ ordered points $\{p^{j,k}_1, \ldots, p^{j, k}_{s_{i_j}}\}$ in $C_{j, k}$ which are at least a distance $2\varepsilon$ from the boundary of $C_{j, k}$. Let $P = \cup_{j, k} \{p^{j,k}_1, \ldots, p^{j, k}_{s_{i_j}}\}$ be the ordered set of all such points.

For any $1 \leq j \leq r$, $0 \leq k <  d_{j}$ and any 
\begin{align*} \sum_{i = 1}^{j - 1} m_i + k m_{j} / d_{j}  < l \leq \sum_{i = 1}^{j - 1} m_i + (k + 1) m_{j} / d_{j} \end{align*} we define the path $\gamma_l^\prime$ to be a horizontal line with $\textnormal{Im} (\gamma_l^\prime ) = \frac{l }{m_\xi}\varepsilon $, $\textnormal{Re} (\gamma_l^\prime (0)) = \infty$ and $| \gamma_l^\prime (1) | = g (t_{i_j}) - \varepsilon +  {l \varepsilon / m_\xi}$. We take $\gamma_l^{\prime \prime}: [0, 1] \to \C$ to be a path with $\gamma_l^{\prime \prime} (t) = e^{ 2 \pi (k + \varepsilon)  t  / d_j } \gamma^\prime_l (1)$. Let $\tilde{\gamma_l} : [0, 1] \to \p^1$ be a rescaled concatenation of $\gamma_l^{\prime }$ with $\gamma_l^{\prime \prime}$ and note that, for sufficiently small $\varepsilon$, $\tilde{\gamma}_l (1) \in C_{j, k}$. We may then choose a set of $s_{i_j}$ arbitrary non-intersecting paths $\tilde{\gamma}_l^\prime$ in $C_{j, k}$ from $\tilde{\gamma}_l (1)$ to $p^{j, k}_n$ where $n = l - (\sum_{i = 1}^{j - 1} m_i + k m_{j} / d_{j})$. Finally, 
define $\gamma_l$ to be the concatenation of $\tilde{\gamma}$ with $\tilde{\gamma}^\prime$ to give a distinguished basis of paths from $\infty$ to the set $P$.

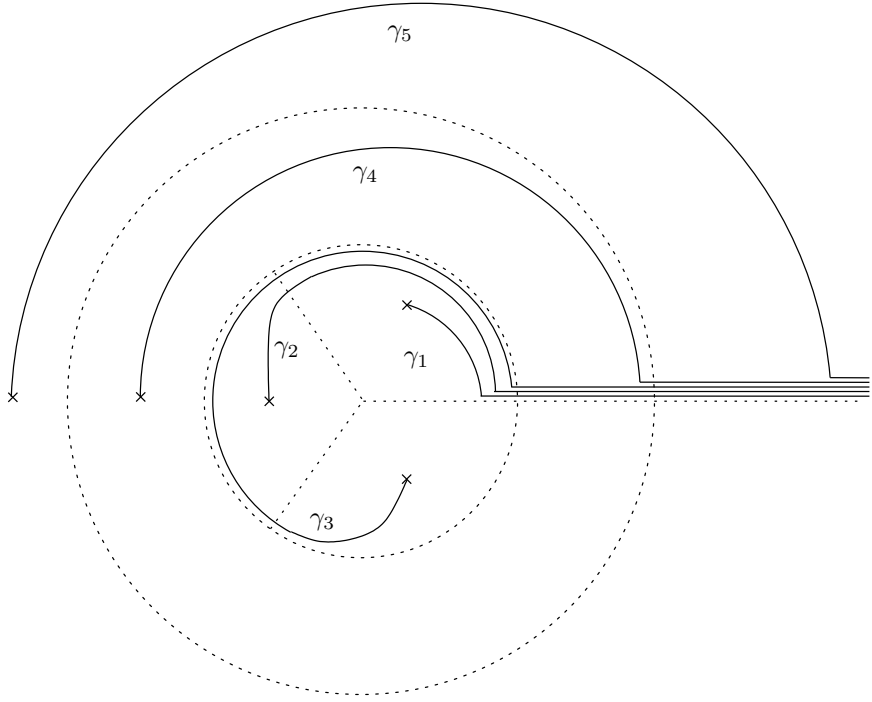
\begin{figure}\begin{picture}(0,0)%
	\includegraphics{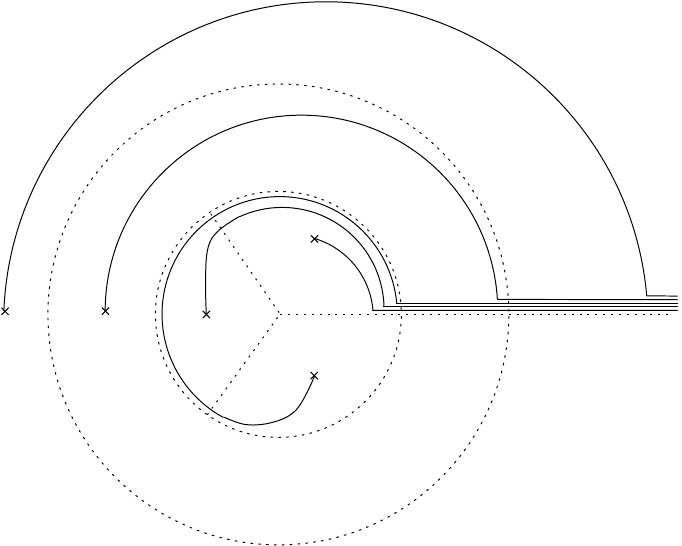}%
	\end{picture}%
	\setlength{\unitlength}{4144sp}%
	\begin{picture}(5176,4155)(1037,-4255)
	\put(2638,-2203){\makebox(0,0)[lb]{\smash{$\gamma_2$}}}
\put(3415,-2261){\makebox(0,0)[lb]{\smash{$\gamma_1$}}}
\put(2853,-3233){\makebox(0,0)[lb]{\smash{$\gamma_3$}}}
\put(3109,-1155){\makebox(0,0)[lb]{\smash{$\gamma_4$}}}
\put(3316,-309){\makebox(0,0)[lb]{\smash{$\gamma_5$}}}
\end{picture}%
\caption{\label{fig:radarscreen} Radar screen for top vertices of monotone path polytope for $A$}
\end{figure}

To apply this construction, we examine a one parameter degeneration in $\mlg{A}{A^\prime}$ to $\xi$. We need only choose a lattice point $\theta \in (\Z^A)^\vee$ which is in the normal cone of the vertex in $\Sigma_{\rho_{A^\prime}} (\psec{A})$ corresponding to $\xi$. Following the discussion after Definition \ref{defn:maxdeg}, this gives a fan $\mcf_\theta$ supported in the half-plane $H_\theta$ which lies in the two dimensional vector space $V_\theta \subset \R^{{\mathcal{A}}}$,  as well as an embedding $i: \mcf_\theta \to \fans{A}$. If $\theta \in \Z^{\mathcal{A}}$, then $\mcf_\theta$ is a rational polyhedral fan and $i$ induces a map of toric stacks $\iota :\mcx_{\mcf_\theta} \to \secon{A}$. Let $\mcx_\theta$ be the stack associated to $\mcf_\theta$.  Quotienting $V_\theta$ by $\lin_{\R} (\rho_{A^\prime})$ gives a map from $V_\theta$ to $\R$ and a map of fans from $\mcf_\theta$ to $\R_{\geq 0}$. This induces a map $F_{\rho_{A^\prime}} :\mcx_{\mcf_\theta} \to \C$ which is a toric degeneration of $\p^1$. It is clear that the zero fiber of $F_{\rho_{A^\prime}}$ is sent to $\xi$ by $\iota$ and that $F_{\rho_{A^\prime}}^{-1} (t)$ is isomorphic to $\p^1$ for $t \ne 0$. 

Now, $\xi$ corresponds to the parametric simplex path $T_\xi = <t_{i_0} , \ldots, t_{i_r} >$ on $\psec{A}$. Letting $s_{j} = \rho_{A^\prime} (t_{i_j})$, $\rho_{A^\prime} (T_\xi)$ is a tight coherent subdivision $\{[s_{j - 1}, s_j] : 1 \leq j \leq r\}$ of the marked interval $[s_0, s_r]$.  In other words, each subinterval $[s_{j - 1}, s_j]$ corresponds to its image, under $\rho_{A^\prime}$ of an edge on $\psec{A}$. 
We may fill in all additional lattice points lying on $\psec{A}$ along the path $T_\xi$ to obtain a modified sequence 
\begin{align*} 
\tilde{T}_\xi = \left<\tilde{t}_{1}, \ldots, \tilde{t}_n \right> .
\end{align*}
Write their images under $\rho_{A^\prime}$ as the sequence
\begin{equation} \label{eq:monpath} \tilde{S} = <\tilde{s}_{1} ,\ldots, \tilde{s}_{n} >
\end{equation}
where $\tilde{s}_j = \rho_{A^\prime} (\tilde{t}_j )$.
It follows directly from  \cite[Section~10.1.G]{GKZ}  that $m_j = s_j - s_{j - 1}$ and $d_j = m_j / e_j$ where $e_j + 1$ is the number of lattice points on the interior of the edge $\{t_{i_j} , t_{i_{j + 1}}\}$. 
For any $1 \leq j \leq n$, we define $b_{j} = \theta (\tilde{t}_{j} )$ and $\mathbf{b} = (b_1, \ldots, b_n)$. Choosing another $\theta$ if necessary, we may assume that $b_1 = \cdots = b_k = 0$ where $\tilde{s}_k = s_1$. Then $\mathbf{b}$ defines the degeneration as in Appendix~\ref{sec:torichd} for the marked polytope $([s_0 , s_r], \tilde{S})$. By this we mean that we consider $\mathbf{b}$ as the function $\mathbf{b} :  \tilde{S} \to \R$ taking $\tilde{s}_i$ to $b_i$ and observe that it induces the convex function $\tilde{\mathbf{b}}$ as in equation \eqref{eq:tildeeta}. 

Working on the level of coarse toric varieties as opposed to stacks, we may parameterize the degeneration using $\mathbf{b}$ as follows. Identify $V_\theta \cap (\Z^A)^\vee$ with $(\Z^2)^\vee$ so that $\mcf_\theta \subset (\Z^2 )^\vee$ is dual to the upper convex hull \begin{align*}
B^u_\theta = \convhull  \{(\tilde{s}_j, b_j + r) : 0 \leq j \leq n , r \in \R_{\geq 0}\}
\end{align*} of $B_\theta = \{(\tilde{s}_j , b_j )\} \subset \Z^2$. For toric varieties, we obtain a map $\beta: \C \times \C^* \to \C \times \p^{n - 1}$ given by
\begin{equation*} \beta (t, z) = (t , [t^{b_0} z^{\tilde{s}_0} : \cdots : t^{b_n} z^{\tilde{s}_n}]) .
\end{equation*}
The coarse variety $X_\theta$ associated to $\mcx_\theta$ is the closure of $\image (\beta )$ with coarse zero fiber $\overline{F_\theta^{-1} (0)} := \overline{X}_\theta (0) = \cup_{j = 1}^r C_j$. Here $C_j$ has moment polytope equal to the line segment from $(\tilde{s}_{k_{j  -1}} , b_{k_{j - 1}} )$ to $(\tilde{s}_{k_{j}} , b_{k_j})$ where $\tilde{s}_{k_j} = s_j$. Let 
\begin{align*} \mu_j = (b_{k_j} - b_{k_{j - 1}} ) / (\tilde{s}_{k_{j}} - \tilde{s}_{k_{j  -1}}) \end{align*} 
be the slope of this line segment and define the map $\alpha_j : \R_{\geq 0} \times \C^* \to \C \times \C^*$ via
\begin{equation*}
\alpha_j (t, z) = (t, t^{-\mu_j} z) .
\end{equation*}
Then we have the following proposition:
\begin{lem} \label{lem:approx} The  parameterization  $(\beta \circ \alpha_j )|_{\{t \} \times \C^* } : \C^* \to \overline{X}_\theta$ of the $\C^*$-orbit $\xi_t$ uniformly converges on compact sets to a $d_j$-fold covering of $C_j$ as $t$ tends to $0$.
\end{lem}
\begin{proof} We simply compute
\begin{eqnarray*}  (\beta \circ \alpha_j ) (t, z) & = & (t , [t^{b_0} (t^{-\mu_j} z)^{\tilde{s}_0} : \cdots : t^{b_n} (t^{-\mu_j} z)^{\tilde{s}_n}]) , \\ 
& = &  (t , [t^{b_0 -\mu_j \tilde{s}_0 } z^{\tilde{s}_0} : \cdots :t^{b_n -\mu_j \tilde{s}_n } z^{\tilde{s}_n}]) , \\ 
& = &  (t , [t^{(b_0 - b_{k_{j  -1}})  -\mu_j (\tilde{s}_0 - \tilde{s}_{k_{j  -1}}) } z^{\tilde{s}_0 - \tilde{s}_{k_{j  -1}}} : \cdots \\ & & \hspace*{.5in} \cdots : t^{(b_n - b_{k_{j  -1}})  -\mu_j (\tilde{s}_n - \tilde{s}_{k_{j  -1}}) } z^{\tilde{s}_n - \tilde{s}_{k_{j  -1}}}]).
\end{eqnarray*}
By convexity, we have that the slope of the line segment connecting $(\tilde{s}_{i} , b_{i})$ to $(\tilde{s}_{k_{j  -1}} , b_{k_{j  -1}})$ is strictly less than $\mu_j$ for all $i < k_{j  -1}$ and strictly greater than $\mu_j$ for all $i > k_j$. This implies that $\kappa_i := (b_i - b_{k_{j  -1}})  -\mu_j (\tilde{s}_i - \tilde{s}_{k_{j  -1}}) \geq 0$ for all $i$ with equality if and only if $k_{j - 1} \leq i \leq k_j$. Utilizing this notation we have that
\begin{equation*} (\beta \circ \alpha_j ) (t, z) = \left(t, [t^{\kappa_0} z^{\tilde{s}_0 - \tilde{s}_{k_{j  -1}}}: \cdots : 1 : \cdots : z^{\tilde{s}_{k_j} - \tilde{s}_{k_{j  -1}}}: \cdots : t^{\kappa_n} z^{\tilde{s}_n - \tilde{s}_{k_{j  -1}}}] \right). \end{equation*}
It is then clear that as $t$ tends to $0$, $(\beta \circ \alpha_j) (t, z)$ converges pointwise to the map sending $z$ to $(0, [ 0: \cdots :0 : 1 : \cdots : z^{\tilde{s}_{k_j} - \tilde{s}_{k_{j - 1}}} : 0 : \cdots 0])$ which is a degree $d_j$ cover of $C_j$. Uniform convergence on compact sets then follows.
\end{proof}
We utilize this in the proof of the following theorem:

\begin{thm} \label{thm:radar} Let $\xi$ be a maximal degeneration of a LG model associated to $A$. If $\xi_t \in \mlg{A}{A^\prime}$ is sufficiently close to $\xi$, there exists a radar screen $\ddata{\xi}$ decomposition of the domain of $\xi_t$ such that the paths of the distinguished basis $\{\gamma_1, \ldots, \gamma_m\}$ end on the critical values of the LG model associated to $\xi_t$. 
\end{thm}

\begin{proof}  For any $\epsilon$ let $\p^1 (\epsilon)$ consist of all points in $\p^1$ that are at least $\epsilon$ distance from $0$ and $\infty$. From \cite[Section~10.1]{GKZ}, we have that $\mce_A \cap C_j$ consists of a single point $\{q_j\}$ for every $j$. It then follows from Lemma~\ref{lem:approx} that for any $\varepsilon$ and $0 <  \kappa < 1$, there exists $\delta > 0$ so that for $t < \delta$ and every $1 \leq j \leq n$ the function $(\beta \circ \alpha_j )|_{\{t \} \times \p^1 ( \kappa) }$ is $\varepsilon$ close to the $d_j$ covering $(\beta \circ \alpha_j)|_{\{0 \} \times \p^1 (\kappa)}$. In particular, from the comment above, we may choose $\varepsilon$ and $\kappa$ small enough so that
\begin{align} \label{eq:convd} \mce_A \cap \beta (t , \C^* ) =  \mce_A \cap \left( \cup_{j = 1}^n (\beta \circ \alpha_j ) |_{\{t \} \times \p^1 ( \kappa  ) } \right) \end{align}
for $t < \delta$. Let $\mcc_{t, j} (\kappa ) =  (\beta \circ \alpha_j ) |_{\{t \} \times \p^1 ( \kappa  ) }$ and $\mcc_t (\kappa ) = \cup_{j = 1}^n (\beta \circ \alpha_j ) |_{\{t \} \times \p^1 ( \kappa  ) }$, then it is clear that we may choose $\varepsilon$ sufficiently small so that the sets $\mcc_{t, j} (\kappa )$ in the union are mutually disjoint. Fix such an  $\varepsilon$ and $\kappa$ so that equation \eqref{eq:convd} holds and let \begin{align*} \delta_0 = \max \left\{ \kappa, \delta^{\frac{\mu_{i} - \mu_{i - 1}}{2}} : 2 \leq i \leq r \right\}. \end{align*} 
Then if $t < \delta$, since $\mu_i > \mu_{i - 1}$ and $\delta^{\mu_i - \mu_{i - 1}} \leq \delta_0^2$ we have 
\begin{align*}
\delta^{\mu_i - \mu_{i - 1}} & \leq \delta_0^2,  \\ & < \delta_0^2 \left( \frac{\delta}{t} \right)^{\mu_i - \mu_{i - 1}} . 
\end{align*}
This implies that
\begin{align*}
\frac{1}{\delta_0} t^{-\mu_{i - 1}} <  \delta_0 t^{- \mu_i} 
\end{align*} 
for every $2 \leq i \leq r - 1$. Choose a collection of continuous real valued functions  $\{g_2 (t), \ldots, g_{r - 1} (t) \}$ for which
\begin{equation*} \frac{1}{\delta_0} t^{-\mu_{i - 1}} \leq g_i (t) \leq \delta_0 t^{- \mu_i} . \end{equation*}
We define $g_{\varepsilon , \kappa , t} : T_\xi \to \R$ via $g_{\varepsilon , \kappa, t} (t_{i_0} ) = 0$, $g_{\varepsilon, \kappa, t} (t_{i_r}) = \infty$ and $g_{\varepsilon , \kappa, t} (t_{i_j} ) = g_j (t)$. We observe that for $0 < t < \delta_0$ and any $z \in \mcc_t (\kappa )$ we have that $z = \mcc_{t, j} (\kappa )$ if and only if $z = t^{-\mu_j} w$ for some $\delta_0 < |w| < 1 / \delta_0$. This implies that $z \in \mcc_{t , j} (\kappa )$ is contained in this image only if $\delta_0 t^{-\mu_j} < |z| < t^{-\mu_j} / \delta_0$. Thus for $z \in \mcc_t ( \kappa )$ we have $z \in \mcc_{t , j} (\kappa )$ if and only if $g_{\varepsilon , \kappa, t} (t_{i_j} ) < |z| < g_{\varepsilon , \kappa, t} (t_{i_{j + 1}} )$. By equation \eqref{eq:convd} and Lemma~\ref{lem:approx}, this implies that the points $\mce_A \cap \mcc_{t, j} (\kappa )$ are, after a rotation, contained in the interior of the components $C_{j, k}$ for $0 \leq k \leq d_j$ of the radar screen for $\ddata{\xi}$ with radial function $g_{\varepsilon , \kappa}$. Indeed, 
because we may choose $\varepsilon$ small enough that $\mcc_{t, j} (\kappa )$ is approximately a $d_j$-fold covering of $C_j$, we have that the $2\pi / d_j$ angular regions each approximately cover $C_j$ once and the intersection of $E^s_A = 0$ with each such map contains $m_j / d_j$ points (the order of $E^s_A$ restricted $C_j$ as given in \cite[Theorem~1.12]{GKZ}), justifying that this radar screen is  associated $\ddata{\xi} = (T_\xi, \{(d_i, m_i)\})$. By definition, the degenerate values of the LG model $\xi_t$ are the intersection points of $\beta (t, \_ )$ with $\mce_A$ and, again, by equation~\eqref{eq:convd}, all such points are accounted for in the interiors of the regions $\mcc_{t , j} (\kappa )$. 
\end{proof}
Note that the proof of Lemma~\ref{lem:approx} gives precise control on a simultaneous regeneration of every circuit in the maximal degeneration $\xi$.

\begin{figure}
\begin{picture}(0,0)%
\includegraphics{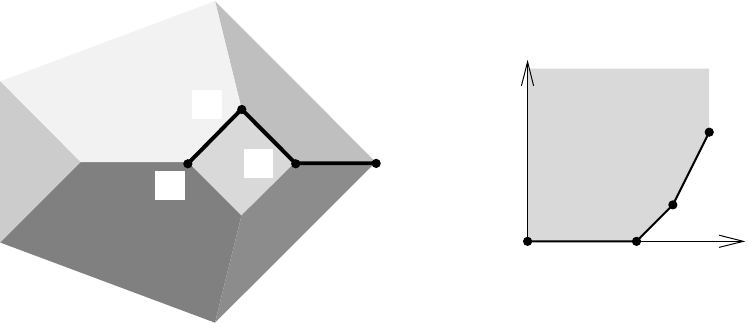}%
\end{picture}%
\setlength{\unitlength}{4144sp}%
\begin{picture}(5693,2453)(900,-2153)
\put(2138,-1164){\makebox(0,0)[lb]{\smash{$t_3$}}}
\put(2433,-545){\makebox(0,0)[lb]{\smash{$t_2$}}}
\put(2813,-995){\makebox(0,0)[lb]{\smash{$t_1$}}}
\put(3789,-882){\makebox(0,0)[lb]{\smash{$t_0$}}}
\put(4726,-1931){\makebox(0,0)[lb]{\smash{$B_\theta = \{(\tilde{s}_j, b_j)\}$ and $B^u_\theta$}}}
\end{picture}%

\caption{\label{fig:radarex}One parameter regeneration of a maximally degenerated LG model.}
\end{figure}
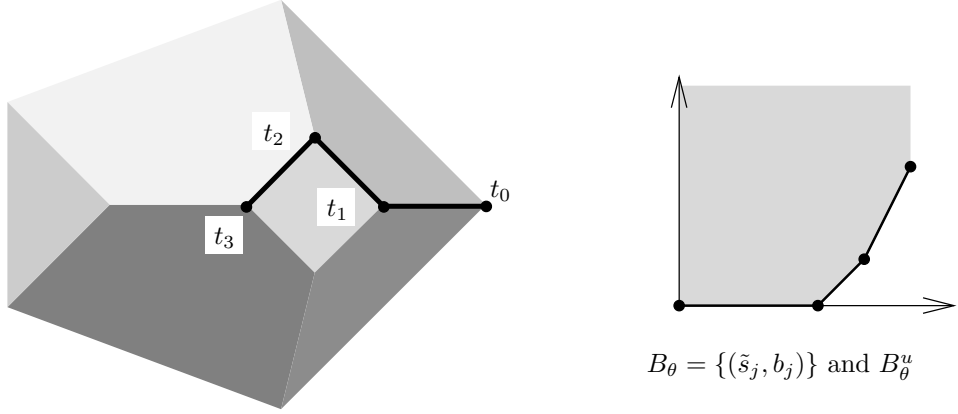
\begin{eg} \label{eg:radardeg}
	Returning to Example~\ref{eg:msexample} and identifying $\Z^A$ and its dual with $\Z^6$ using the ordering of $A$, we have that $\rho_{A^\prime} = (0, 0, 0, 0, 0, 1)$. Consider the path $\left< t_0, t_1, t_2, t_3 \right>$ on $\Sigma (A)$ pictured on the left in  Figure~\ref{fig:radarex}. Using equation~\eqref{eq:secvert}, one computes the coordinates for $t_i$ to be
	\begin{align}\label{eq:vertradar} \begin{split} 
	t_0 & = (1,1,4,4,5,0), \\ t_1 & = (1,1,3,3,4,3), \\ t_2 & = (1,2,2,3,3,4), \\ t_3 & = (2,2,2,2,2,5).
	\end{split}	\end{align} 
	Pairing with $\rho_{A^\prime}$ gives the coherent subdivision $\left< s_0, s_1, s_2, s_3 \right> = \left< 0, 3, 4, 5 \right>$ of $[0,5]$.  Using the coordinates given in equation~\eqref{eq:vertradar}, one sees that there are no additional lattice points on the edges $[t_i, t_{i + 1}]$ of $\Sigma (A)$ so that $\tilde{S} = \left< 0, 3, 4, 5 \right>$ and $e_j = 0$ for $j \in \{1,2,3\}$. This implies $d_j = m_j = s_j - s_{j - 1}$ and $(d_1,d_2,d_3) = ( 3,1,1 )$.
	
	Choosing $\theta = (0,7,0,-8,5,-1) \in \Z^6$, a short computation shows that $\theta$ pairs to zero on $t_0$ and $t_1$, while $\left< \theta , t_2 \right> = 1$ and $\left< \theta, t_3 \right> = 3$. The upper envelope $B^u_\theta$ of the set $B_{\theta}$ is shown on the right hand side of Figure~\ref{fig:secbp1p1}. The normal fan of this polyhedron defines the toric variety $\overline{X}_\theta$ which embeds into $\secon{A}$ and defines a  degeneration of $\p^1$ into the (closure of the) orbits corresponding to the edges $[t_j, t_{j + 1}] \subset \Sigma (A)$. In this case, $\p^1$ degenerates into three projective lines $C_1 \cup C_2 \cup C_3$. For each $j \in \{1,2,3\}$, the map $(\beta \circ \alpha_j)_{\{t\} \times \C^*}: \C^* \to \overline{X}_\theta$ from Lemma~\ref{lem:approx} converges to a $d_j$-covering of $C_j$. As $d_1 = 3$, the degeneration onto the first component gives a $3$-fold covering and the single critical value $\mathcal{E}_A \cap C_1$ yields $3$ critical values in the pullback along $(\beta \circ \alpha_1)_{\{t\} \times \p (\kappa)}$. The other two degenerations do not yield multiple coverings, so the critical values (i.e. their intersections with $\mce_A$) consist of one point each. The radar screen and distinguished basis that arises in this case is pictured in Figure~\ref{fig:radarscreen}.
\end{eg}
 Utilizing Theorem \ref{thm:radar}, for every maximal degeneration of a LG model $\xi$, we may use the radar screen distinguished basis to obtain a semi-orthogonal decomposition of a category which can be thought of as a type of Fukaya-Seidel category (see \cite{SeidelFPL}). However, for a general subset $A^\prime \subset A$ and an $A^\prime$-sharpened pencil $W$, the associated LG model $\mathbf{w}$ has a hypersurface degeneration, as opposed to a Morse singularity, over $0$ and the Fukaya-Seidel category for such a function has not yet been defined in general. Thus we will examine the special case for which an $A^\prime$-sharpened pencil gives rise to the Fukaya-Seidel category of a Lefschetz pencil as defined in \cite[Chapter 18]{SeidelFPL}. 

\begin{prop} Let $A^\prime \subset \textnormal{Int} (Q)$ and $W$ be a generic $A^\prime$-sharpened pencil. Then the LG model $\mathbf{w}$ associated to $W$ has isolated Morse critical points away from $\infty$.
\end{prop}

\begin{proof} Recall from Theorem~\ref{thm:GKZ1} that the principal $A$-determinant has a product decomposition $E_A (f) = \prod_{Q^\prime \leq Q} \Delta_{A \cap Q^\prime} (f)^{i (\Lambda , A) \cdot u ({\textnormal{Lin}_{\mathbb{N}} (\mathcal{A})} / Q^\prime )}$.
An intersection $W \cap \Delta_{A \cap Q^\prime}$  corresponds to stratified Morse critical values of $\mathbf{w}$ (by definition, these are points for which the hypersurface intersects the orbit associated to $Q^\prime$ non-transversely). To see that no such intersection points occur, we first note that, by definition, $\Delta_{A \cap Q^\prime} (f)$ equals $\Delta_{A \cap Q^\prime} (f||_{A \cap Q^\prime})$. Now, by the proof of Lemma~\ref{lem:orb}, we have that $W = \lin_\R \{s_0, s_\infty\}$ where $s_\infty||_{A^\prime} = s_\infty$ and $s_0||_{A^\prime} = 0$. For a generic choice of $W$, we may assume that $\Delta_{Q^\prime \cap A} (s_0) \ne 0$ for all faces $Q^\prime < Q$ (as the zero loci of such discriminants are hypersurfaces in $(\C^*)^{Q^\prime \cap A}$). This implies that, for any $t \in \C$,  
\begin{align*} \Delta_{Q^\prime \cap A} (s_0 - t s_\infty ) & = \Delta_{Q^\prime \cap A} \left( (s_0 - ts_\infty)||_{Q^\prime \cap A}  \right) , \\ & = \Delta_{Q^\prime \cap A} \left(  s_0||_{Q^\prime \cap A} \right) , \\ & \ne 0. 
\end{align*}
Thus all intersections $\{E_A = 0 \} \cap W$ arise as singularities $\Delta_A ( s_0 - t s_\infty ) = 0$. For such a $t$, the hypersurface $Y_t$ in $\mcx_Q$ defined by $s_0 - t s_\infty$ is singular in the interior $Y_t - (Y_t \cap \partial \mcx_Q)$. A generic choice of coefficients ensures that the intersections $\{E_A = 0 \} \cap W$ away from $0$ are transverse and therefore yield Morse singularities of the pencil.
\end{proof}

As was mentioned above, given a LG model $\mathbf{w}$ with Morse singularities and reasonable boundary conditions, i.e. a symplectic Lefschetz pencil, the Fukaya-Seidel category\gls{fukcat} $\textnormal{Fuk}^\rightharpoonup (\mathbf{w} )$ is well defined and studied in \cite{SeidelFPL}. Given an $A^\prime$-sharpened pencil $\xi \in \mlg{A}{A^\prime}$, write $\mathbf{w}_\xi : \mcx_Q - D_{A^\prime} \to \C$ for the associated function off the divisor at infinity $D_{A^\prime} = \{s_{\infty} = 0\}$. Taking the paths $\mcb$ associated to a radar screen to be the generating exceptional collection, Theorem \ref{thm:radar} and the above proposition gives the following corollary. 

\begin{cor} Assume $A^\prime \subset \textnormal{Int} (Q)$. For every maximal degeneration of a LG model in $\mlg{A}{A^\prime}$, there exists a smooth LG model $\xi_t$ and a semi-orthogonal decomposition of the Fukaya-Seidel category:
\begin{equation*} \textnormal{Fuk}^\rightharpoonup \left(\mathbf{w}_{\xi_t} \right) \approx \left< \mct_1 , \ldots, \mct_r \right> \end{equation*}
where $\mct_i$ is the Fukaya-Seidel category of a regenerated circuit corresponding to $\xi|_{C_i}$.
\end{cor}



\subsection{Homological mirror symmetry}

In the final pages of this article, we will detail a conjectural homological mirror to the maximally degenerated LG model and present some supporting evidence for this viewpoint. Aside from the intrinsic interest which many have for the subject of homological mirror symmetry, the perspective obtained from maximal degenerations predicts many results in the $B$-model setting which previously are either unknown or have been approached from a more opaque angle. 

We restrict our conversation to the homological mirrors of nef Fano DM toric stacks. More concretely, we take a simplicial fan $\Sigma$ in $\Z^d$ with a choice of one-cone generators, which we identify with $\Sigma (1)$, and consider its canonical stacky fan $\sfan = (\Z^{\Sigma (1)} , \Z^d , \beta_{\Sigma (1)}, \tilde{\Sigma} )$ where $\tilde{\Sigma}$ is the pullback of $\Sigma$ via $\beta_{\Sigma (1)}$.  The nef condition amounts to the assumption that  $\Sigma (1) \subset \partial (\convhull (\Sigma (1)))$.  This condition is equivalent to $-K_{\mcx_\Sigma}$ being nef. Letting $a_0 = 0 \in \Z^d$, we define the $A$-model mirror of $\mcx_\Sigma$ to be a generic LG model $\mathbf{w}$ associated to a $A^\prime = \{a_0 \}$-sharpened pencil $W$ for the set $A = \Sigma (1) \cup \{a_0 \}$. It is not hard to show that any homological mirror of a toric Fano orbifold as defined by \cite[Section 3]{horivafa} can be obtained in this way. We now formulate a structure associated to  $\mcx_\Sigma$ corresponding to a maximal degeneration $\xi$ of $\mathbf{w}$. 

For any triangulation $T$ of $A$, we define a stacky fan\gls{sigt} $\Sigma_T$ as follows.  Let $\sigma \in T$ be a simplex which contains $a_0$, $\tau$ the minimal face of $\sigma$ containing $a_0$, and $\tau (1)$ the vertices of $\tau$. We write $\Lambda_\tau $  for the finite rank abelian group $\Z^d / \lin_\Z (\tau (1))$ and $\lambda : \Z^d \to \Lambda_\tau$ for the quotient homomorphism. The star $\textnormal{St}_T (\tau )$ of $\tau$ in $T$ is defined to be the collection of simplices in $T$ containing $\tau$ as a face. For each such simplex $\upsilon \in \textnormal{St}_T (\tau)$ we define the cone $S_\upsilon = \textnormal{Lin}_{\R_{\geq 0}} (\{\lambda (v ) : v \in \upsilon (1) \} ) \subset \Lambda_\tau \otimes \R$ with generators $\lambda (v) \in \Lambda_\tau$. The collection of cones $\{ S_\upsilon \}$ along with their intersections defines a stacky fan which we write as $\Sigma_T$.

\begin{defn} \label{def:mmseq} Let $\xi \in \mlg{A}{A^\prime}$ be a maximal degeneration with decorated simplex path $\ddata{\xi} = (T_\xi , \{(d_i, m_i)\})$ where $T_\xi = \langle t_0 , \ldots, t_{r + 1} \rangle$. The sequence of stacks\gls{mirseq}
\begin{equation*} \mathbf{S}_\xi = (\mcx_{\Sigma_{t_{r + 1}}} ,  \ldots, \mcx_{\Sigma_{t_0}} ) \end{equation*}
will be called the mirror sequence to $\xi$.
\end{defn}
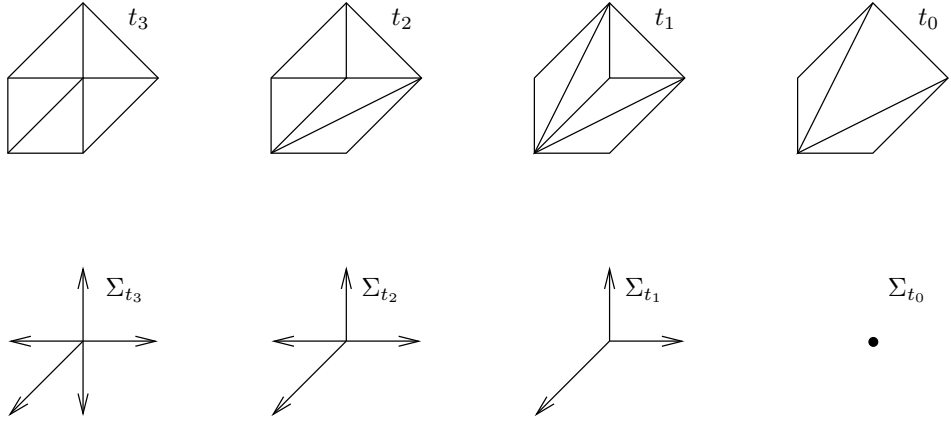
\begin{figure}[t]
\begin{picture}(0,0)%
\includegraphics{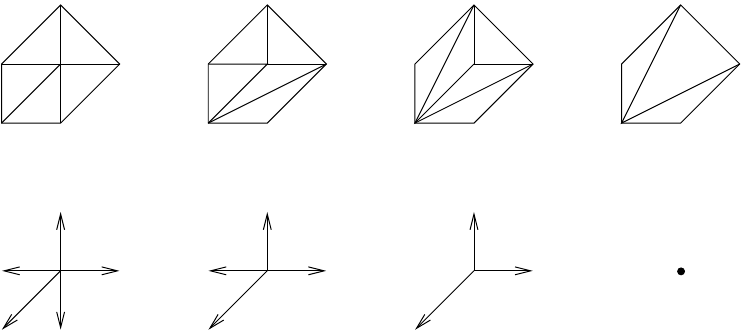}%
\end{picture}%
\setlength{\unitlength}{4144sp}%
\begin{picture}(5649,2511)(-461,12302)
\put(271,14654){\makebox(0,0)[lb]{\smash{$t_3$}}}
\put(1846,14654){\makebox(0,0)[lb]{\smash{$t_2$}}}
\put(3421,14654){\makebox(0,0)[lb]{\smash{$t_1$}}}
\put(4996,14654){\makebox(0,0)[lb]{\smash{$t_0$}}}
\put(136,13034){\makebox(0,0)[lb]{\smash{$\Sigma_{t_3}$}}}
\put(1666,13034){\makebox(0,0)[lb]{\smash{$\Sigma_{t_2}$}}}
\put(3241,13034){\makebox(0,0)[lb]{\smash{$\Sigma_{t_1}$}}}
\put(4816,13034){\makebox(0,0)[lb]{\smash{$\Sigma_{t_0}$}}}
\end{picture}%
\caption{\label{fig:mirseq} The mirror sequence to a maximal degeneration. }
\end{figure}
\begin{eg} Let us write out the mirror sequences for the maximal degenerations of $\{a_0\}$-pencils on the variety $\mcx_Q$ of Example~\ref{eg:msexample}. Referring to Figure \ref{fig:mppbp1p1}, we enumerate the maximal degenerations $\xi_1, \xi_2$ and $\xi_3$ associated to the vertices on the left hand side of the monotone path polytope, starting from the top and ending on the bottom. The mirror sequences of these degenerations are
\begin{align*} \mathbf{S}_{\xi_1} & = ( \mcx_Q^{mir} , F_1 , \p^2 , \{pt\}), \\ \mathbf{S}_{\xi_2} & = (\mcx_Q^{mir} , F_1 , \p^1 ) , \\ \mathbf{S}_{\xi_3} & = (\mcx_Q^{mir} , \p^1 \times \p^1 , \p^1 ). \end{align*}
The sequence of triangulations occurring in the decorated simplex path associated to $\xi_1$ and its mirror fans are illustrated in Figure~\ref{fig:mirseq}. As a degeneration of LG models, this sequence was examined in Example~\ref{eg:radardeg}. Noting that $F_1$ is the projective line bundle of $\mco (-1) \oplus \mco$ over $\p^1$ for the second sequence and that $\p^1 \times \p^1$ is the trivial projective line bundle over $\p^1$ for the third, this example suggests that the mirror sequences to maximal degenerations correspond to runs of the minimal model program for the mirror.
\end{eg}

We briefly recall the minimal model program on toric varieties as presented in \cite[Chapter 15]{cls}, \cite[Chapter 14]{matsuki}, or \cite{Reid}. For the moment, we take $\Sigma$ to be an arbitrary projective, simplicial stacky fan in $\Z^d$, and write $X_\Sigma$ for the corresponding toric orbifold. 

Given a codimension $1$ cone $w = \textnormal{Lin}_{\R_{\geq 0}} \left\{ a_{3} , \cdots, a_{{d + 1}} \right\}$, there exist precisely two maximal cones containing $w$ with the additional vertices denoted $a_{1}$ and $a_{2}$ respectively. The set\gls{circon} $C (w) = \{a_0 , a_{1} , \cdots , a_{{d + 1}} \}$ is an extended circuit and has a fundamental relation
\begin{align*} 
\sum_{j = 0}^{d + 1} c_j a_{j} =  0, \quad
\sum_{i = 0}^{d + 1} c_j = 0 
\end{align*}
as in equation \eqref{eq:affrelation}. We write $C_\pm (w)$ and $C_0 (w)$ for the subsets $(C (w))_\pm$ and $(C(w))_0$ respectively. We assume $(c_1, \ldots, c_{d + 1} ) = 1$ and will orient the circuit so that $c_0 < 0$. 

Denote the full rank sublattice $ \text{Lin}_{\Z} \{a_1, \ldots, a_{d + 1}\}$ of $\Z^d$ by $\Lambda_w$. As was noted in Section \ref{sec:circuit1}, the volume 
\begin{align} \label{eq:vol0} \Vol_0 (C (w)) := \Vol (\convhull (\{a_1, \ldots, a_{d + 1} \} )) \end{align}  is given by $i_w \cdot \sum_{i = 1}^{d + 1} c_i$ where $i_w$ is the index $[\Z^d : \Lambda_w]$.   

Recall from Definition~\ref{defn:core} that the core of an extended circuit is the unique circuit contained within it. For any $a_{j} \in C_+ (w)$, define the cone
\begin{align*} \tau_j = \textnormal{Lin}_{\R_{\geq 0}} \{a_{i}  \in \Core (C (w)) : i \ne j\}. \end{align*} 
Note that $\{\tau_j\}_{a_j \in C_+ (w)}$ are cones over the simplices in the triangulation $T_+$ of $\Core (C(w))$ defined in equation~\eqref{eq:triang}.  We state a proposition which essentially rephrases \cite[Proposition~14.2.1]{matsuki}. 

\begin{prop}{\cite[Proposition~14.2.1]{matsuki}}.  \label{prop:matsuki}\gls{Upspm} The fan consisting of $\Upsilon = \{\tau_j : a_{j} \in C_+ (w)\}$ is contained in $\Sigma$. If the signature of $C(w)$ is $(p,q;r)$, then there exists a collection of $r$-cones $\textnormal{Supp} (w) := \{\sigma_i : 1 \leq i \leq m\}$ in $\Sigma$ 
such that the maximal cones in the star of $\Upsilon$ consists of cones $\overline{\Upsilon} (d ) := \{\tau_j + \sigma_i: a_{j} \in C_+ (w) , 1 \leq i \leq m\}$.
\end{prop}

Associated to the codimension one-cone $w=\langle a_3,\ldots ,a_{d+1}\rangle$ is an extremal contraction in the sense of Mori theory, the structure of which can be phrased combinatorially in terms of the circuit $\core{C (w)}$ as follows. First, consider the collection of simplices\gls{sympf} $\simp{\Sigma} = \{ \convhull (\sigma (1) \cup \{a_0\} ): \sigma \in \Sigma (d)\}$ and write 
\begin{align} \Vol (\Sigma ) = \sum_{ \sigma \in \simp{\Sigma} } \Vol (\convhull (\sigma (1) \cup \{a_0\} )).\end{align} 
Note that if $X_\Sigma$ is projective, then $\simp{\Sigma} $ extends to a regular triangulation $T$ of $\convhull (A)$. This can be seen by choosing an very ample divisor $\sum_{a \in \Sigma (1)} r_a D_a$ on $X_\Sigma$ and observing that the function sending $a$ to $r_a$ will defines an extension $T$ of  $\simp{\Sigma}$.  We call such a triangulation $T$ a convex extension of $\simp{\Sigma}$. In this way, we may consider the two collections of cones which depend on the circuit\gls{Upsneg} 
\begin{align*} \Upsilon^- & = \{\Cone (C (w) - a_j ) : a_j \in C_- (w) , a_j \ne a_0 \} , \\ \overline{\Upsilon}^- & = \{\tau + \sigma : \tau \in \Upsilon^- , \sigma \in \text{Supp} ({w}) \}. \end{align*} 
Assuming $T$ is supported on the circuit $\Core (C(w))$ as in Definition~\ref{defn:cirsupp}, we can replace the maximal cones (and their faces) of $\Sigma$ occurring in $\overline{\Upsilon}$ with those in $\overline{\Upsilon}^-$. This yields a fan $\Sigma^\prime = \Sigma_{m_{C(w)}(T)}$ which implements the circuit modification $m_{C(w)}(T)$ of $T$ by $C(w)$ as defined in . We summarize the corresponding statement in birational geometry as follows. 



\begin{prop}{\cite[Theorem~15.4.1]{cls}}  With notation as above, let $T$ be a convex extension of $\textnormal{Simp} (\Sigma )$ and $w$ a codimension one-cone of $\Sigma$ such that $T$ is supported on $\core{C (w)}$. Then the extremal contraction corresponding to the rational curve determined by $w$ is given by a birational map 
\begin{equation*} f_w : \mcx_\Sigma  \dashrightarrow \mcx_{\Sigma^\prime}\;\; . \end{equation*}
\end{prop}


While we refer to the reference \cite{cls} for the proof of this proposition, we will detail the three essentially different situations that can occur. These are the standard operations of Mori theory: Mori fiber space, divisorial contraction, and flip. They are distinguished by the signature $(p, q; r)$ of $C (w)$. To see this, we need to define three stacks associated to $w$. Define the lattices\gls{feb}
\begin{eqnarray*}
\Lambda_F & = & \frac{\lin_\R (\core{C(w)} ) \cap \Z^d}{\lin_\R ( C_- (w) ) \cap \Z^d} , \\
\Lambda_E & = & \frac{\Z^d}{\lin_\R (C_- (w)) \cap \Z^d} , \\
\Lambda_B & = & \frac{\Z^d}{\lin_\R (\core{C(w)} ) \cap \Z^d}
\end{eqnarray*}
with the natural projections 
\begin{align*} \pi_F : \lin_\R (\core{C(w)} ) \cap \Z^d &\to \Lambda_F,  & \pi_B : \Z^d & \to \Lambda_B , &  \pi_E : \Z^d & \to \Lambda_E , \end{align*} 
Define stacky fans 
\begin{align} \label{eq:stakyfans} 
\begin{split} \Sigma_F & = \{\pi_F (\tau ) : \tau \in \Upsilon \}, \\ \Sigma_E & = \{\pi_E (\sigma \cup \tau ) : \sigma \in \text{Supp} (w) , \tau \in \Upsilon\},  \\ \Sigma_B & = \{\pi_B (\sigma ) : \sigma \in \text{Supp} (w)  \}
\end{split}
\end{align}
and denote their associated toric stacks by $\mcf$, $\mce$ and $\mcb$ respectively, with coarse spaces $F, E$ and $B$. Note that there is an obvious toric fibration $\pi : \mce \to \mcb$ with fiber $i: \mcf \hookrightarrow \mce$. 

We start with the case of $q  = 1$.  In this case the map $\pi_B : \Z^d \to \Lambda_B$ induces a map of stacky fans from $\Sigma = \Sigma_E$ onto $\Sigma_B$ which gives a smooth map $f : \mcx_\Sigma \to \mcb$.  Here $\Sigma_B=\Sigma^\prime$ and $f = \pi$ is a Mori fiber space map with general fiber equal to $\mcf$. 

In the case $q = 2$,  $C_- (w ) = \{a_0 , a_{d + 1} \}$, so $\Lambda_E$ has rank $(d - 1)$ and $\mce$ is a divisor in $\mcx_\Sigma$. The circuit modified fan $\Sigma^\prime$ is obtained by replacing the cones in the star of $\Upsilon$ with $\{ \sigma \cup C_+ (w ) : \sigma \in \text{Supp} (w)  \}$. In other words, we delete the one-cone corresponding to $a_{d + 1}$ which, on the coarse level, gives a divisorial contraction $f : X_\Sigma \to X_{\Sigma^\prime}$ whose exceptional locus is $E$ blown up along $B$. 

The case of $q > 2$ corresponds to a flip. Indeed, as in the case of $q = 2$, the circuit modified stacky fan $\tilde{\Sigma }$ is obtained by replacing the star of $\Upsilon$ by $\{\sigma \cup C_+ (w): \sigma \in \text{Supp} (w)  \}$. The induced map $\tilde{\pi} :\mcx_\Sigma \to \mcx_{\tilde{\Sigma}}$ contracts $\mce$, which in this case has codimension $> 1$ and contains the rational curve corresponding to $w$. As $\mcx_{\tilde{\Sigma}}$ is not $\mathbb{Q}$-factorial, to obtain the flip $\phi : \mcx_{\Sigma^\prime} \to \mcx_{\tilde{\Sigma}}$ one observes that $K_{\mcx_{\Sigma^\prime}}$ is ample relative to $\phi$, as required. 


We are interested in sequences of these birational operations which come from certain runs of the Mori program. 
\begin{defn} Given a toric stack $\mcx = \mcx_r$, a sequence of equivariant birational maps
\begin{equation*} \mcx_r \stackrel{f_r}{\dashrightarrow} \mcx_{r - 1} \dashrightarrow \cdots \stackrel{f_1}{\dashrightarrow} \mcx_0 \end{equation*}
will be called a MMP sequence of $\mcx$ if for every $1 \leq i \leq r - 1$, $f_i$ is a divisorial contraction or flip and $f_1$ is a Mori fiber space. 
\end{defn}

Recall that the effective cone of a projective simplicial toric variety $X$ admits a chamber decomposition whose chambers correspond to those toric varieties obtained from $X$ by the operations of the toric Mori program \cite{hukeel}, for brevity's sake we call the induced fan structure the Mori fan.  A MMP sequence as above gives rise to a piecewise linear path in the Mori fan of $\mcx$ starting at the ample cone and ending at the boundary of the nef cone. If this path can be made to be a linear path, we call the sequence \textit{regular}. These are instances of MMP sequences obtained from the MMP with scaling, in the terminology of \cite{BCHM}.  We may now state a suggestive theorem relating maximal degenerations of LG models to the minimal model program.

\begin{thm} \label{thm:mmps} Given a set of lattice points $A$, the set of regular MMP sequences which begin with a toric stack in \begin{equation*} \{\mcx_\Sigma : \Sigma (1) \cup \{0\} = A ,\; \mcx_\Sigma \text{ is nef Fano}\}\end{equation*} are in bijective correspondence with the set of mirror sequences to maximal degenerations of $\{a_0\}$-sharpened pencils on $\mcx_Q$. Both are in bijective correspondence with vertices of the monotone path polytope $\Sigma_{\rho_{a_0}} (\psec{A})$.
\end{thm}

\begin{proof} Consider the linear projection $\rho_{a_0} : \R^{{\mathcal{A}}} \to \R$ and its restriction to $\psec{A}$. Recall that this projection takes $\sum_{a \in A}  r_a e_a$ to $r_{a_0}$  and thus, by equation \eqref{eq:secvert}, for any triangulation $T = \{(Q_i, A_i): i \in I\}$,  $\rho_{a_0} (\varphi_T) =  \sum_{a_0 \in A_i} \Vol (Q_i)$. In particular, $\rho_{a_0} (\varphi_T) = 0$ for triangulations of $(Q, A - \{a_0\})$ and $\rho_{a_0} (\varphi_T) = \Vol (Q)$ for triangulations in which every simplex contains $a_0$. Thus $\rho_{a_0}$ maps $\Sigma (A)$ onto $[0, \Vol (Q)]$. By Proposition \ref{prop:mppvertices}, the vertices of the monotone path polytope are in bijective correspondence with maximal degenerations. For any such vertex $\xi$, let $\ddata{\xi} = \left( \left<t_0, \ldots, t_{r + 1} \right> ,\{ (d_i, m_i ) \} \right)$ be its decorated simplex path as given in Definition~\ref{defn:decsimppath}. We first observe that the mirror sequence to $\xi$ is a MMP sequence for $\mcx_\Sigma$. Taking $\rho_{a_0} ( t_{r + 1} ) = \Vol (Q)$ to have the maximal value, we must have that $\Sigma_{t_{r + 1}}$ is nef Fano. For every circuit $C_i$ whose modifications give $t_i$ and $t_{i + 1}$, we have that $a_0 \not\in (C_i)_+$ which implies that there is an extremal contraction $f_i : \mcx_{\Sigma_{t_{i + 1}}} \dashrightarrow \mcx_{\Sigma_{t_i}}$ corresponding to the circuit. If $1 \leq i \leq r$ then since $\rho_{a_0} (t_i ) \ne 0$, we have that $a_0$ is a vertex of a simplex in $t_i$. This implies that $\sigma (C_i ) = (p, q; r)$ with $q \geq 1$ so that $f_i$ is a divisorial contraction or a flip. On the other hand, if $i = 0$, then $\rho_{a_0} (t_0 ) = 0$ which implies that $a_0$ is not a vertex of any simplex of $t_0$. This implies that $\sigma (C_0 ) = (p, 1; r)$ and that $f_0$ is a Mori fiber space. Therefore the mirror sequence to $\xi$ corresponds to a MMP sequence for $\mcx_\Sigma$. The converse is obtained by running the above correspondences in reverse.
\end{proof}

From this result, one is led naturally to conjecture that every decomposition of the $A$-model category $\textnormal{Fuk}^\rightharpoonup (\mathbf{w})$ given by a radar screen corresponds to an equivalent decomposition of the $B$-model derived category of $\mcx_\Sigma$ which is associated to the mirror MMP sequence. On the $B$-model side, such a decomposition has been given very explicitly in \cite{kawamataderived}. We write a condensed version of these results here. While we refer the reader to loc. cit. for complete proofs of these statements, we include a partial proof to verify the count given in Theorem~\ref{thm:kawamata}(i). 
\begin{thm}[\cite{kawamataderived}] \label{thm:sodbmodel} \label{thm:kawamata}
(i) Let $C (w) = \{a_0 , \ldots, a_{d + 1} \}$ correspond to a signature $(p, q; r)$ circuit in a rank $\Z^d$ lattice with $a_0 = 0 \in C_- (w)$ and triangulations $T_\pm$. Letting $\mcx = \mcx_{\Sigma_{T_+}}$, the derived category $D^b (\mcx )$ has a strong exceptional collection of $\Vol_0 (C(w))$ line bundles
\begin{equation*} \mathbf{E}_w = \left\{ \mco \left( \sum_{i = 1}^{d + 1} k_i D_i \right) : 0 \geq \sum c_i k_i > - \sum c_i \right\}. \end{equation*}
If $q = r = 1$, then the collection is complete. \\

(ii) Let $\mcx = \mcx_r$ be a toric stack with a MMP sequence
\begin{equation*} \mcx_r \stackrel{f_r}{\dashrightarrow} \mcx_{r - 1} \dashrightarrow \cdots \stackrel{f_1}{\dashrightarrow} \mcx_0 \end{equation*}
%
and associated toric stacks $\mcf_i$, $\mce_i$ and $\mcb_i$ at each stage. Then there is a semi-orthogonal decomposition
\begin{equation*} D^b (\mcx ) \simeq < \mcs_1, \ldots, \mcs_r > \end{equation*}
where each $\mcs_i$ admits a semi-orthogonal decomposition
\begin{equation*}
\mcs_i \simeq <  j_* ( \pi^* (D^b (\mcb_i )) \otimes \mcl ) : \mcl \in \mathbf{E}_w >.
\end{equation*}
\end{thm}

\begin{proof}  These statements are part of Theorems 3.1, 4.3, 5.2 and 6.1 in \cite{kawamataderived}. The only additional point not proven in loc. cit. is the count of exceptional objects being $\Vol_0 (C(w))$ as defined in equation \eqref{eq:vol0}. To prove  this, we observe that $\phi : \Z^{A } \to \Z^d$ given by $\phi (e_i ) = a_i$ has cokernel equal to $\Z^d / \Lambda_w$ and rank $1$ kernel. So the line bundles $\mco (\sum b_i D_i )$ form a subgroup of $\Pic (\mcx )$ isomorphic to $(\Z^d / \Lambda_w )^\vee \oplus \Z$. Thus the line bundles $\mco (\sum k_i D_i )$ satisfying $0 \geq \sum k_i c_i > - \sum_{i \ne 0} c_i$ can be counted up to equivalence as $|\Z^d / \Lambda_w| \cdot (\sum c_i  ) = \Vol_0 (A)$.
\end{proof}

One notational distinction worth noting is that what is called $\mcf$ in \cite{kawamataderived}, we refer to as $\mcb$. Now we recall from \cite[Section~2.6]{fulton} that the multiplicity of a $d$-dimensional cone $\sigma$ in $\Z^d$ is\gls{multip} 
\begin{align*}
\textnormal{Mult} (\sigma ) = [\Z^d : \lin_{\Z} (\sigma (1))].
\end{align*}
We use Theorem~\ref{thm:kawamata} to prove a more elementary result.

\begin{prop} \label{prop:rkkth} Let $\mcx_\Sigma$ be a complete toric stack with simplicial stacky fan $\Sigma$ in $\Z^d$. Then
\begin{equation} \label{eq:kthr} \rank (K_0 (D^b (\mcx_\Sigma ) ) ) = \Vol (\Sigma ) = \sum_{\sigma \in \Sigma (d)} \textnormal{Mult} (\sigma ) \end{equation}
\end{prop}

\begin{proof} We prove this by induction on dimension. Every stacky fan in $\Z$ is given by two primitive points $a_1, a_2 \in \Z$ which give a $(2, 1)$ circuit $A = \{a_0=0, a_1, a_2\}$. Clearly $\Vol_0 (A ) = |a_1 | + |a_2 |$ which equals the two quantities on the right in \eqref{eq:kthr}. By Theorem \ref{thm:sodbmodel} (i),  this is also the number of exceptional objects in a complete exceptional collection, so the proposition holds for this case.

Now assume that the proposition holds for dimensions $< d$ and all $d$-dimensional complete, simplicial stacky fans $\tilde{\Sigma}$ with $\Vol (\tilde{\Sigma}) < V$ for some $V \in \N$. Let $\Sigma$ be a $d$-dimensional complete, simplicial stacky fan with $\Vol (\Sigma ) = V$. Let $f : X_\Sigma \dashrightarrow X_{\Sigma^\prime}$ be a birational map associated to a circuit modification $C(w)$ of signature $(p, q; r)$. Write the corresponding fibration, defined in equation~\eqref{eq:stakyfans}, as $\mcf \to \mce \to \mcb$, where $ \dim (\mcb ) = r < d$. 
By Theorem \ref{thm:sodbmodel} (ii), the additivity of the rank of $K_0$ relative to semi-orthogonal decompositions, and the above assumptions, we have that
\begin{eqnarray*}
\rank (K_0 (D^b (\mcx_\Sigma ) ) ) & = & \rank (K_0 (D^b (\mcx_{\Sigma^\prime} ))) + \Vol_0 (C (w)) \cdot \rank (K_0 (D^b (\mcb ))) , \\ 
& = & \Vol (\Sigma^\prime ) + \Vol_0 (C (w)) \cdot \Vol (\Sigma_B ) .
\end{eqnarray*}

Now, from the definition of $\Sigma_B$ and Proposition~\ref{prop:matsuki}, we have that for every $d$-dimensional cone $\tilde{\sigma} \in \overline{\Upsilon} \subset \Sigma$ which contains $\tau_j$ as a face for some $\tau_j \in \Upsilon$, there is a unique $\sigma \in \Sigma_B$ which is the $\pi_B$-image of $\sigma^\prime  \in \Sigma$ where $\sigma^\prime + \tau_j = \tilde{\sigma}$. The volume of the simplex associated to $\tilde{\sigma}$ is thus  $\Vol (\sigma ) \cdot \Vol (\tau_j )$.
The contribution to $\Vol (\Sigma )$ from $\Upsilon$ is therefore $\sum_{\tau_j \in \Upsilon , \sigma \in \Sigma_B} \Vol (\tau_j ) \cdot \Vol (\sigma)$.

The same statement holds for $\overline{\Upsilon}^-$, so that the following formula holds for the difference
\begin{eqnarray*}
\Vol (\Sigma ) - \Vol (\Sigma^\prime ) & = & \sum_{\tau_j \in \Upsilon , \sigma \in \Sigma_B} \Vol (\tau_j ) \cdot \Vol (\sigma) - \sum_{\tau_i \in \Upsilon^- , \sigma \in \Sigma_B} \Vol (\tau_i ) \cdot \Vol (\sigma) , \\ & = & \sum_{\sigma \in \Sigma_B} \Vol (\sigma ) \left( \sum_{\tau_j \in \Upsilon} \Vol (\tau_j ) - \sum_{\tau_i \in \Upsilon^-} \Vol (\tau_j )  \right), \\ & = & \Vol (\Sigma_B ) \cdot \Vol_0 (\core{C (w)}) = \Vol (\Sigma_B ) \cdot \Vol_0 (C(w)) .
\end{eqnarray*}
But this implies that $\rank (K_0 (D^b (\mcx_\Sigma ))) = \Vol (\Sigma ) = V$ which proves the induction step.
\end{proof}

From this, we obtain an equivalence on the rank of the $K$-theory for the semi-orthogonal pieces arising from both the $A$-model and $B$-model categories.

\begin{cor} \label{prop:kth} Suppose $T_\xi = \left< t_{i_0}, \ldots, t_{i_r} \right>$ is the parametric simplex path corresponding to  the maximal degeneration $\xi \in \mlg{A}{a_0}$, $\xi_t$ is a regeneration of $\xi$ and  $ \{[s_j, s_{j + 1}] : s_j = \rho_{a_0} (t_{i_j} )\}$ is the induced tight coherent subdivision of $[0, \Vol (Q)]$. The associated semi-orthogonal decompositions $\textnormal{Fuk}^\rightharpoonup (\mathbf{w}_{\xi_t}) = < \mct_1, \ldots, \mct_r>$ and $D^b (\mcx_{\Sigma} ) = < \mcs_1, \ldots, \mcs_r >$ have the property
\begin{equation*} \rank (K_0 (\mct_j ))  = s_j - s_{j - 1} = \rank ( K_0 ( \mcs_j)) .\end{equation*}
\end{cor}

\begin{proof} The statement asserting the equality $\rank (K_0 (\mct_j )) = s_j - s_{j - 1}$ follows from the discussion after equation~\eqref{eq:monpath}, where it was observed that the multiplicity $m_i$ of $E_A$ equals  $s_j - s_{j - 1}$. This multiplicity denotes the number of critical points in the $j$-th outer annulus of the radar screen decomposition and thus the number of exceptional objects in the generating collection for $\mct_j$, verifying the first equality.
	
The equality for $\rank (K_0 (\mcs_i ))$ follows from the observation that $s_j - s_{j - 1}$ equals $\rho_{a_0} (t_{i_j}- t_{i_{j - 1}})$ which is the difference of the sum of the volumes  of simplices containing $a_0$ in $t_{i_j}$ and $t_{i_{j - 1}}$. Using the construction of $\Sigma_j = \Sigma_{t_{i_j}}$ preceding Definition~\ref{def:mmseq}, it follows that this equals $ \Vol (\Sigma_j ) - \Vol (\Sigma_{j - 1} )$, which is $\rank (K_0 (\mcs_i ))$ by Proposition~\ref{prop:rkkth}. \end{proof}

We infer from Theorem \ref{thm:mmps} and Proposition \ref{prop:kth} the following natural conjecture.

\begin{conj} Given any maximal degeneration $\xi$ of an $\{a_0\}$-sharpened pencil and  a regeneration $\xi_t$ of $\xi$, let
\begin{eqnarray*} \textnormal{Fuk}^\rightharpoonup (\mathbf{w}_{\xi_t} ) & = & < \mct_1, \ldots, \mct_r> , \\ D^b (\mcx_{\Sigma} ) & = & < \mcs_1, \ldots, \mcs_r >
\end{eqnarray*} be the semi-orthogonal decompositions associated to $\xi$ and its mirror sequence. Then there exists an equivalence of triangulated categories
\begin{equation*} \Phi_\xi : \textnormal{Fuk}^\rightharpoonup (\mathbf{w}_{\xi_t} ) \to D^b (\mcx_\Sigma ) \end{equation*}
which restricts to equivalences $\Phi_\xi : \mct_i \to \mcs_i$ for all $1 \leq i \leq r$.
\end{conj}

In fact, a more detailed conjecture can easily be formulated about the equivalence of the categories $\mct_i$ and $\mcs_i$ associated to degenerate circuits, but we will leave this to a later work. Additional evidence for this conjecture comes from the case of $A$ actually equaling a circuit, which is simply the statement of homological mirror symmetry for a weighted projective stack. Certain classes of $(2, 2)$ circuits were also examined in \cite{kerr} where the equivalence of the circuit regeneration and the semi-orthogonal component associated to a stacky blow-up was proved.

As a final remark, we point out that the edges of the monotone path polytope $\Sigma_{\rho_{a_0}} (\psec{A})$ correspond to minimal transitions between MMP sequences. They also correspond to certain two dimensional faces of $\psec{A}$. Restricting attention to those faces which have an edge on the minimum facet $\rho_{a_0} = 0$, we obtain a transition between two Mori fiber spaces. Such moves, or links, have been well studied in a much more general context and their classification is referred to as the Sarkisov program. As an outgrowth of our perspective, one may pursue a complete structure theorem for all toric Sarkisov links.


\appendix

\section{\label{sec:toric} Toric preliminaries}

In this section, we will give key definitions and constructions for a toric moduli space of hypersurfaces and its compactification. An important point to keep in mind throughout is that our moduli stacks are only of hypersurfaces in toric stacks, and only up to toric isomorphism, not general isomorphisms. The advantage of this approach is that we obtain stacks with extremely explicit representations. 

In the first two subsections we recall and collect notions of the algebraic and symplectic geometry of toric stacks. Many familiar aspects of this subject will be assumed, but all novel constructions will be discussed. In the last two subsections, we recall the constructions of Gelfand, Kapranov and Zelevinsky \cite{GKZ} and Lafforgue \cite{Lafforgue}. We adapt these ideas into the definition of several toric stacks which give the moduli compactification, a universal toric variety lying over it and its universal hypersurface.

\subsection{\label{sec:toricpreliminaries} Basic definitions}

We start this section by recalling the construction of toric stacks through the data of a stacky fan. We utilize the material in \cite{toricstacks} rather than the more classical approach given in \cite{bcs, Cox}. This allows one to work with more general Artin stacks.
\begin{defn} \label{defn:stfan} \gls{stfan} A stacky fan $\sfan $ consists of the data $(\Lambda_1, \Lambda_2, \beta , \fan )$ where:
\begin{itemize}
\item[(i)]  $\Lambda_2$ is a finitely generated abelian group, 
\item[(ii)] $\fan$ is a fan in $\Lambda_1 \otimes \R$ and $\Lambda_1$ is a lattice,
\item[(iii)] $\beta : \Lambda_1 \to \Lambda_2$ is a homomorphism with finite cokernel.
\end{itemize}

\end{defn}
The set of $d$-dimensional cones in $\fan$ will be denoted by $\fan (d)$ and we refer to $\sigma \in \fan (d)$ as a  $d$-cone.  We will frequently abuse notation and identify a one-cone $\R_{\geq 0} \cdot \lambda$ in $\fan (1)$ with its primitive generator $\lambda$. Note that our definition of a stacky fan is called a generically stacky fan in \cite[Definition~2.4]{toricstacks}. Now extend $\beta$ to an exact sequence
\begin{equation}\label{eq:fundamentalsequence}  0 \to L_\sfan \stackrel{\alpha}{\longrightarrow} \Lambda_1 \stackrel{\beta}{\longrightarrow} \Lambda_2 \to K_\sfan \to 0.\end{equation}
Let $\text{cone} (\beta) = [\Lambda_1 \stackrel{\beta}{\longrightarrow} \Lambda_2]$ be the cone of $\beta$ in the category of chain complexes of abelian groups and take 
\begin{align*} \gls{hgroup} \mathbb{H}_{\sfan} := \mathbf{Tor}_1 (\text{cone}(\beta),\mathbb{C}^*) \cong (L_\sfan \otimes \C^*) \oplus \Tor_1 (K_\sfan , \C^* ) \end{align*} to be the first hypertor group of $\text{cone} (\beta)$. The isomorphism above follows from considering the hypertor spectral sequence which collapses on the second page as $L_\sfan$ is a lattice and $K_\sfan$ is finite. Furthermore, the connecting homomorphism in the long exact sequence of hypertor maps $\mathbb{H}_\sfan$ onto $\ker (\beta \otimes_{\mathbb{Z}} 1) \subset \Lambda_1 \otimes_\mathbb{Z} \mathbb{C}^*$. This in turn gives rise to an action of $\mathbb{H}_\sfan$ on the toric variety \gls{prestack}$X_\Sigma$. One notes that if $\Lambda_2$ is not torsion free, then another look at the long exact sequence of hypertor shows that the finite group $\Tor_1 (\Lambda_2 , \mathbb{C}^*)$ naturally includes into $\mathbb{H}_\sfan$ as the subgroup which stabilizes $X_\Sigma$ generically. 
\begin{defn} \label{defn:toricstack} \gls{torstack} Given a stacky fan $\sfan$, the toric stack $\mcx_\sfan$ is defined to be the quotient stack $\mcx_\sfan = [X_\fan / \mathbb{H}_\sfan]$.
\end{defn}
The torus acting on $\mcx_\sfan$ is
\begin{align*} \gls{ggroup}
\mathbb{G}_\sfan = \Lambda_2 \otimes_\mathbb{Z} \mathbb{C}^*.
\end{align*}
Indeed, note that for any $\lambda \in \tgroup{\sfan}$, we may choose $\lambda^\prime \in \Lambda_1 \otimes_\mathbb{Z} \C^*$ with $\beta (\lambda^\prime ) = \lambda$ and define $\lambda \cdot \_ : \mcx_{\sfan} \to \mcx_{\sfan}$ by $\lambda^\prime \cdot z$ for $z \in X_\fan$. This defines the torus action of $\tgroup{\sfan}$ on $\mcx_\sfan$ up to natural isomorphisms. The action can be made strict when $K_\sfan$ is trivial.

Given two stacky fans, $\tilde{\sfan}$ and $\sfan$, we define a map $g :  \tilde{\sfan} \to \sfan$ to be a pair $(g_1 , g_2)$ such that $g_1 :\tilde{\Lambda}_1 \to \Lambda_1 $ induces a map of fans $ g_1 : \tilde{\fan} \to \fan$ and $g_2 :  \tilde{\Lambda}_2 \to \Lambda_2 $ satisfies $\beta \circ g_1 = g_2 \circ \tilde{\beta}$. It is clear that any such map of stacky fans induces a map $\stackmap{g}{} :  \mcx_{\tilde{\sfan}} \to \mcx_\sfan $ along with a homomorphism $g_2 \otimes 1 : \mathbb{G}_{\tilde{\sfan}} \to \mathbb{G}_\sfan$. While $\stackmap{g}{}$ is not strictly equivariant, it is weakly equivariant in the sense that for every $\tilde{\lambda} \in \mathbb{G}_{\tilde{\sfan}}$ and $z \in X_{\tilde{\Sigma}}$, there is a natural isomorphism $h_{\tilde{\lambda}} \in \mathbb{H}_{\sfan}$ for which $h_{\tilde{\lambda}} (\stackmap{g}{} ({\tilde{\lambda}} \cdot z)) = (g_2 \otimes 1)({\tilde{\lambda}} ) \cdot \stackmap{g}{} (z)$. These isomorphisms must satisfy a cocycle condition which is evident from their construction. In particular, if ${\tilde{\lambda}} \in \tgroup{\tilde{\sfan}}$ lifts to act via ${\tilde{\lambda}}^\prime \in \tilde{\Lambda}_1 \otimes_\mathbb{Z} \mathbb{C}^*$ and $(g_2 \otimes 1) ({\tilde{\lambda}}) \in \mathbb{G}_{\sfan}$ lifts to $\lambda^\prime \in \Lambda_1 \otimes_\mathbb{Z} \mathbb{C}^*$, then one defines $h_{\tilde{\lambda}} =  \lambda^\prime \cdot [(g_1 \otimes 1)({\tilde{\lambda}}^\prime)]^{-1}$.

Following \cite[Section~5]{toricstacks}, we call a stacky fan $\sfan$ good if the primitive generators $\fan (1)$ in $\Lambda_1$  are linearly independent and span a saturated sublattice of $\Lambda_2$.  All of the toric stacks defined and worked with in this paper will be good and most will be Deligne-Mumford (DM). It was shown that in loc. cit. that for any toric stack $\mcx$, there is a canonical stack $\mathcal{X}_{\tilde{\sfan}}$ and a map $\mcx_{\tilde{\sfan}} \to \mcx$ where $\tilde{\sfan}$ is a good stacky fan. This map satisfies a universal property and can be thought of as a stacky resolution of $\mcx$.  When $\mcx = \mcx_{\sfan}$ and $\sfan$ is good, the map is an isomorphism. 

For a good DM toric stack $\mcx_\sfan$, one can identify the space of equivariant Cartier divisors $\Divisor_{eq} (\mcx_{\sfan } )$ with $\Lambda_1^\vee$ and the Picard group with $\Pic (\mcx_{\sfan} ) = L_\sfan^\vee \oplus \ext^1 (K_\sfan , \Z)$. Indeed, letting $\fan^\vee \subset \Lambda_1^\vee$ be the dual cone to the cone over $\fan (1)$, the ring 
\begin{align} \label{eq:coxring} R_\sfan = \C [ x_{\sigma} : \sigma \in \Sigma (1) ] \end{align} is the homogeneous coordinate ring for $X_\Sigma$ graded by the character lattice $L_\sfan^\vee \oplus \ext^1 (K_\sfan , \Z)$ of $\mathbb{H}_\sfan$.  Given $\gamma_0 \in \Lambda_1^\vee$, we write $D_{\gamma_0}$ for the associated Cartier divisor and $\mco (D_{\gamma_0} )$ for the line bundle in $\Pic (\mcx_\sfan )$. Utilizing the map $\alpha$ from the exact sequence  \eqref{eq:fundamentalsequence}, for any character $\gamma_0 \in \fan^\vee$ define the set
\begin{equation*} [\gamma_0 ] = \{\gamma \in \fan^\vee : \alpha^\vee (\gamma ) = \alpha^\vee (\gamma_0 ) \} \subset \Lambda_1^\vee . \end{equation*}
This gives the space $H^0 (\mcx_\sfan , \mco (D_{\gamma_0} ))$ an eigenvector decomposition into $(\C^{[\gamma_0 ]})^\vee = \Hom_{set} ([\gamma_0 ] , \C )$ with eigenbasis consisting of the monomials $\{x_{\gamma} : \gamma \in [\gamma_0 ]\} \subset R_\sfan$. When the divisor $D_\gamma$ is chosen, the group $\tgroup{\sfan} \times \C^*$ acts on $H^0 (\mcx_\sfan , \mco (D_\gamma ))$ via
\begin{equation*}
(\lambda , t) \left( \sum_{\gamma \in [\gamma_0]} c_\gamma x_\gamma \right) = t \sum_{\gamma \in [\gamma_0]} (\beta^\vee)^{-1} (\gamma - \gamma_0) (\lambda ) c_\gamma x_\gamma .
\end{equation*}
Here we have identified $\Lambda_2^\vee$ with the group of characters $\Hom (\tgroup{\sfan} , \C^*)$.

Suppose $g: \tilde{\sfan} \to \sfan$ is a map of stacky fans and $\gamma \in \fan^\vee$ an effective divisor on $\mcx_\sfan$, then the map
\begin{equation} \label{eq:stackmap} \stackmap{g}{}^* : H^0 (\mcx_\sfan , \mco (D_\gamma )) \to H^0 (\mcx_{\tilde{\sfan}} , \mco (D_{g_1^\vee (\gamma )}))
\end{equation}
is simply
\begin{equation*}
\stackmap{g}{}^* \left( \sum_{\gamma \in [\gamma_0]} c_\gamma x_\gamma \right) = \sum_{\gamma \in [\gamma_0]} c_\gamma x_{g_1^\vee (\gamma )} .
\end{equation*}

Now assume that $g:  \tilde{\sfan} \to \sfan $ describes a flat morphism of good toric stacks. Recall from \cite[Proposition~2.4]{yau} that such a map has the property that $g_1$ maps $\tilde{\fan}(1)$ onto $\fan (1)$ implying that $g_1 : \tilde{\Lambda}_1 \to \Lambda_1$ has cofinite image $\Gamma_1 := \image (g_1 )$. Let $\Gamma_2$ be the pushout
\begin{equation*}
\begin{CD}
\tilde{\Lambda}_1 @>{\tilde{\beta}}>> \tilde{\Lambda}_2 \\
@V{g_1}VV             @V{h}VV \\
\Gamma_1 @>{\gamma}>> \Gamma_2
\end{CD}
\end{equation*}
used in the following definition.

\begin{defn} \gls{colimit} \label{defn:colimitstack} Given an equivariant flat morphism $g : \mcx_{\tilde{\sfan }} \to \mcx_{\sfan}$ between two good toric stacks, let $\sfan_g^{\to} = (\Gamma_1, \Gamma_2 ,  \gamma , \Sigma )$ and $\mathcal{X}_{g}^{\to} = \mcx_{\sfan_g^{\to}}$ be called the colimit stack relative to $g$, and $g^\to = (g_1 , h) : \mcx_{\tilde{\sfan}} \to \mcx_g^{\to}$ the induced morphism.
\end{defn}

Note that the colimit stack is a good toric stack. The following proposition establishes a universal property for the colimit stack.

\begin{prop}\label{prop:colup} Suppose $\mcx_1$, $\mcx_2$ and $\mcx_3$ are good toric stacks. Let $g :\mcx_1 \to \mcx_3$ be a flat equivariant morphism, factored by $g = h \circ f$.
\begin{equation} \label{diag:univprop}
\begin{tikzpicture}[baseline=(current  bounding  box.center), >=angle 90]
\matrix(a)[matrix of math nodes,
row sep=3em, column sep=3em,
text height=1.5ex, text depth=0.25ex]
{\mcx_1 & \mcx_g^\to  \\
\mcx_3	& \mcx_2 \\};
\path[->,font=\scriptsize] (a-1-1) edge node[left]{$g$} (a-2-1)
edge node[right=.5ex]{$f$} (a-2-2)
edge node[above]{$g^{\to}$} (a-1-2);
\path[->,font=\scriptsize] (a-2-2) edge node[below]{$h$} (a-2-1);
\path[dotted,->,font=\scriptsize] (a-1-2) edge node[right]{$\tilde{h}$} (a-2-2);
\end{tikzpicture}
\end{equation}
	 If $h$ is a bijection on orbits then there exists a unique map $\tilde{h} : \mcx_g^\to \to \mcx_2$ such that $f = \tilde{h} \circ g^\to$.
\end{prop}

\begin{proof}
The proof of this follows from the universal properties of pushout along with the assumption that $h$ is an isomorphism on fans  defining $\mcx_2$ and $\mcx_3$. In particular, suppose $\sfan_i = \left( \Lambda^i_1, \Lambda^i_2, \beta_i, \Sigma_i \right)$ for $i \in \{1,2,3\}$ are stacky fans for $\mcx_i$, and $f, g$ and $h$ are represented by maps of stacky fans $(f_1, f_2)$, $(g_1,g_2)$ and $(h_1,h_2)$ respectively. By \cite[Theorem~IV.6.7]{Ewald}, the condition that $g$ is flat implies that $g_1$ is surjective. Since $h$ is an isomorphism of coarse toric varieties, it follows that $h_1$ is an isomorphism of lattices and induces an isomorphism of fans from $\Sigma_2$ to $\Sigma_3$. This implies that $f_1$ factors through $g_1$ so that the pushout $\Gamma$ of $g_1$ and $\beta_1$ admits a map $\tilde{h}$ to $\Lambda^2_2$ making diagram \eqref{diag:univproppushout} commute. 

\begin{equation} \label{diag:univproppushout}
\begin{tikzpicture}[baseline=(current  bounding  box.center), >=angle 90]
\matrix(a)[matrix of math nodes,
row sep=3em, column sep=2em,
text height=1.5ex, text depth=0.25ex]
{\Lambda^1_1 & & & \Lambda^1_2  \\
	& \Lambda^3_1 & \Gamma & \\
\Lambda^2_1 & & & \Lambda^2_2  \\
\Lambda^3_1 & & & \Lambda^3_2 \\};
\path[->,font=\scriptsize] (a-1-1) edge node[above]{$\beta_1$} (a-1-4)
edge node[right]{$f_1$} (a-3-1)
edge node[right]{$g_1$} (a-2-2)
edge [bend right=30] node[left]{$g_1$} (a-4-1);
\path[->,font=\scriptsize] (a-1-4) edge node[left]{$f_2$} (a-3-4)
edge [bend left=30] node[right]{$g_2$} (a-4-4);
\path[dotted,->,font=\scriptsize] (a-1-4) edge node[left]{$\tilde{g}^\to_2$} (a-2-3);
\path[->,font=\scriptsize] (a-2-2) edge node[right]{$h_1^{-1}$} (a-3-1);
\path[->,font=\scriptsize] (a-2-3) edge node[left]{$\tilde{h}$} (a-3-4);
\path[dotted,->,font=\scriptsize] (a-2-2) edge node[below]{$\beta^\to$} (a-2-3);
\path[->,font=\scriptsize] (a-3-1) edge node[above]{$\beta_2$} (a-3-4)
edge node[right]{$h_1$} (a-4-1);
\path[->,font=\scriptsize] (a-3-4) edge  node[left]{$h_2$} (a-4-4);
\path[->,font=\scriptsize] (a-4-1) edge node[above]{$\beta_3$} (a-4-4);
\end{tikzpicture}
\end{equation}
The stacky fan of the colimit stack $\sfan^\to_g = \left( \Lambda^3_1, \Gamma, \beta^\to , \Sigma_3 \right)$ then admits the stacky fan map $(h_1^{-1}, \tilde{h})$ to $\sfan_2$ whose induced map on toric stacks makes diagram \eqref{diag:univprop} commute.
\end{proof}

The toric stacks relevant for this paper arise from finite sets in a lattice or finitely generated abelian group. We now recall this construction and fix our notation.  Let $B$ be a finite subset of a finitely generated abelian group $\Lambda$ which spans $\Lambda \otimes \Q$. To construct a toric stack associated to $B$, let $\beta_B : \Z^B \to \Lambda$ be the homomorphism given by assigning $e_b$ to $b$ where $\{e_b : b \in B\}$ is the standard basis for $\Z^B$. We call the exact sequence
\begin{align}  \gls{fundseq} \gls{lb} \gls{kb} \label{eq:setfunseq} 0 \to L_B \stackrel{\alpha_B}{\longrightarrow} \Z^B \stackrel{\beta_B}{\longrightarrow} \Lambda \to K_B \to 0 \end{align}
the fundamental sequence associated to $B$. Let $\text{cone} (\beta_B )$ be the cone of $\beta_B$ in the category of chain complexes of abelian groups. Using the hyperext spectral sequence, one can compute the hyperderived dual $\mathbb{R}^* \Hom (\text{cone} (\beta_B ) , \mathbb{Z})$ to see that it is concentrated in degree $1$ and isomorphic to $L_B^\vee \oplus \ext^1 (K_B , \mathbb{Z})$. We will use the notation\gls{lamgdual} $\Lambda_{\gdual{B}}$ for this abelian group. Note that the long exact sequence associated to $\mathbb{R}^* \Hom (- , \mathbb{Z})$ is then
\begin{align} \label{eq:alphastar} 0 \to \Lambda^\vee \stackrel{\beta_B^\vee}{\longrightarrow} (\Z^B)^\vee \stackrel{\alpha_B^\star}{\longrightarrow} \Lambda_{B^\vee} \to 0 ,
\end{align}
where $\alpha_B^\star = \alpha^\vee \oplus \delta$ is the connecting homomorphism.  

Assume $B$ comes naturally equipped with an abstract simplicial complex $\mcb$, i.e. a collection of subsets of $B$ which is closed under intersection. Then we define the fan $\fan_\mcb$ in $\R^B$ to consist of the cones $\textrm{Cone} (\tau ) = \textnormal{Lin}_{\R_{\geq 0}} \{ e_b : b \in \tau \}$ for every $\tau \in \mcb$. We write $\sfan_{B , \mcb} = (\Z^B , \Lambda , \beta_B , \fan_\mcb )$ and $\mcx_{B, \mcb}$ for the associated stack. If $\mcb$ is understood, we may write $\sfan_B$ and $\mcx_B$. Note that all stacky fans in the sense of \cite{bcs} and fantastacks from \cite{toricstacks} are obtained from this construction.

Suppose $\Lambda$ is a rank $d$ lattice. Let $A \subset \Lambda$ be a finite subset which affinely spans $\Lambda \otimes \R$ and $Q \subset \Lambda_\R$ equals the convex hull of $A$ denoted \gls{convhl}$ \convhull (A) $. By a \gls{markpoly} marked polyhedron we mean a pair $(Q, A)$ where $Q$ is a polyhedron, i.e. the intersection of finitely many half spaces in $\Lambda \otimes \R$. We take $\gls{suphyp}\du{Q} \subset \Lambda^\vee$ to be the finite set of primitive generators for supporting hyperplanes of $Q$. More precisely, for every $b \in \Lambda^\vee$ let 
\begin{align} \gls{nsubb}\label{eq:mcoq} n_b = - \min \{ b (v) : v \in Q\}. \end{align}  Then $b \in \du{Q}$ if and only if $b$ is primitive and $\{v \in Q : b (v) = - n_b\}$ is facet of $Q$.  The dual of the face poset of $Q$ then defines an abstract simplicial complex $\mcb_Q$ on $\du{Q}$. In particular, 
\begin{align} \label{eq:simpcompdual} \mcb_Q = \{\du{Q}_{Q^\prime} : Q^\prime \text{ is a face of } Q\},
\end{align}
where the set $\du{Q}_{Q^\prime}$ is defined as $\{b \in \du{Q}: b (v) = - n_b \text{ for every } v \in Q^\prime\}$.

The marked polyhedron $(Q, A)$ provides the stack $\mcx_{\sfan_{\bar{Q}, \mathcal{B}_Q}}$ with a positive line bundle $\mco (D_{\gamma_A} )$ where 
\begin{align}\label{eq:gammaA} \gamma_A = \sum_{b \in \du{Q}} n_b e_b^\vee \in (\Z^{\du{Q}})^\vee, \end{align} and a linear system $(\C^A)^\vee \subset (\C^{[\gamma_A]} )^\vee = H^0 (\mcx_{\du{Q}, \mcb_Q} , \mco (D_\gamma ))$. 
\begin{defn}\label{defn:polystack}\gls{polsf}\gls{polst}\gls{polbound}\gls{pollb}\gls{polls}\gls{polgroup} Given a marked polyhedron $(Q, A)$, let
	\begin{enumerate}
		\item $\sfan_Q := \sfan_{\du{Q}, \mcb_{Q}}$ be the stacky fan associated to $Q$,
		\item  $\mcx_Q := \mcx_{\du{Q} , \mcb_Q}$ be the toric stack associated to $Q$,
		\item  $\partial \mcx_Q = D_{\sum_{b \in \du{Q}} e_b^\vee}$ be the boundary divisor of $\mcx_Q$,
		\item $\mco_A (1) := \mco (D_{\gamma_A} )$, 
		\item  $\linsys{A} := (\C^A)^\vee \subset H^0 (\mcx_Q , \mco_A (1))$ be the linear system of sections of $\mco_A (1)$ with by the equivariant sections in $A$, 
		\item   $\mathbb{G}_Q$ be the group $\mathbb{G}_{\sfan_{Q}} = \Lambda^\vee \otimes \C^*$ acting on $\mcx_Q$.
	\end{enumerate}   
\end{defn}
We illustrate these definitions with two basic examples.
\begin{eg} \label{eg:simplex} Suppose $A \subset \Z^2 = \Lambda$ is the subset $A = \{(1,0), (-1, 0), (0,1)\}$ so that $Q$ is a triangle. One computes $\du{Q} = \{(1,-1), (-1,-1), (0,1)\} $ which implies that the fundamental exact sequence \eqref{eq:setfunseq} for $B = \du{Q}$ is isomorphic to 
\begin{align*}
0 \to \Z \stackrel{\alpha_{\du{Q}}}{\longrightarrow} \Z^{3} \stackrel{\beta_{\du{Q}}}{\longrightarrow} \Z^2 \to 0 ,
\end{align*}
where $\alpha_{\du{Q}} (1) = (1,1,2)$. Thus the stacky fan for $(Q,A)$ is $\sfan = \left( \Z^3, \Z^2 , \beta_{\du{Q}}, \Sigma \right)$ where $\Sigma$ consists of all proper faces of $\R^3_{\geq 0}$.  This implies $X_{\Sigma} = \C^3 - \{0\}$ and $\mathbb{H}_{\sfan} = \C^*$ via the action $\lambda \cdot (x_1, x_2, x_3) = (\lambda x_1, \lambda x_2, \lambda^2 x_3)$ so that 
\begin{equation}
\mcx_{\sfan} = \p (1, 1, 2) .
\end{equation}
One can check to see that $n_{(1,-1)} = 1, n_{(-1,-1)} = 1, n_{(0,1)} = 0$ so that $\mco_A (1) = \mco_{\p (1, 1, 2)} (2)$ where $\mco_{\p (1,1,2)} (n )$ corresponds to the graded module $\C [x_1, x_2, x_3]$ with $1$ in degree $-n$. Furthermore, the linear system $\linsys{A}$ is the span of $\{x_1^2, x_2^2, x_3 \}$. As $(0,0) \in Q$ was not included in $A$, its corresponding section $x_1 x_2$ does not appear in the linear system.
\end{eg}

\begin{eg} \label{eg:circuit1} Suppose $A = \{(0,0), (1,0), (0,1),(-1,-1)\} \subset \Z^2 = \Lambda$. Again we have that $Q$ is a triangle and $\du{Q} = \{(2,-1), (-1, 2), (-1,-1)\}$. However, in this case $K_{\du{Q}}$ is non-trivial in the fundamental sequence
\begin{align*}
	0 \to \Z \stackrel{\alpha_{\du{Q}}}{\longrightarrow} \Z^{3} \stackrel{\beta_{\du{Q}}}{\longrightarrow} \Z^2 \to \left( \Z / 3 \Z \right) \to 0.
\end{align*}
Here $\alpha_{\du{Q}} (1) = (1,1,1)$ and the stacky fan is $\sfan = \left( \Z^3 , \Z^2, \beta_{\du{Q}}, \Sigma \right)$ where $\Sigma$ is the same as in Example \ref{eg:simplex}. Thus, letting $\mu_3$ be the third roots of unity, $X_{\Sigma} = \C^3 - \{0\}$ and $\mathbb{H}_\sfan \cong \C^* \oplus \mu_3$ where the action of $\mathbb{H}_\sfan$ on $X_\Sigma$ is 
$(\lambda, \zeta) \cdot (x_1, x_2, x_3) = (\lambda x_1,  \lambda \zeta^{-1} x_2, \lambda \zeta x_3 )$ (up to a change of coordinates), so that 
\begin{align*}
\mcx_{\sfan} = \left[ \p^2 / \mu_3 \right].
\end{align*}
One checks that $n_b = 1$ for $b \in \du{Q}$ which implies that $\mco_A (1)$ is the pullback of $\mco_{\p^2} (3)$. The generators of $\linsys{A}$ corresponding to $A$ are the invariant sections $\{x_1^3, x_2^3, x_3^3, x_1x_2x_3\}$. 
\end{eg}

The study of toric varieties and stacks from the perspective of marked polytopes places the linear system as a central object. Those sections that have singularities on various orbits of $\mcx_Q$ will be of particular interest. Let \gls{avert}$\ver{A}$ be the set of vertices of $Q$ and \gls{anvert} $\nver{A} = A - \ver{A}$. For any face $Q^\prime$, we will write $\orb{Q^\prime} \subset \mcx_Q$ for the corresponding orbit. For a section $s$ of a line bundle over a stack, we denote its zero locus by $\mcy_s$.
\begin{defn} \label{defn:sections} A section $s \in \linsys{A} \subset H^0 (\mcx_Q , \mco_A (1))$ is degenerate if the scheme theoretic intersection $Y_s \cap \orb{Q^\prime}$ is singular for some face $Q^\prime$ of $Q$. If $s = \sum_{a \in A} c_a e_a$ we say $s$ is full if $c_a \ne 0$ for all $a \in \ver{A}$ and very full if $c_a \ne 0$ for all $a \in A$.
\end{defn}
When $\mcx_Q$ is a smooth stack, a degenerate section is a section which does not transversely intersect the toric boundary. The principal $A$-determinant,
\begin{equation} \label{eq:princadet} \gls{Adet}  E_A : \linsys{A} \to \C \end{equation}
is a polynomial which vanishes on degenerate sections (see  \cite[Chapter~10]{GKZ}). We also recall that the discriminant \gls{Adisc}$\Delta_A : \linsys{A} \to \C$ is a polynomial that vanishes on the closure of the set of sections with a singularity in the maximal torus orbit of $\mcx_Q$. Note that there exist sets $A$ for which the discriminant $\Delta_A$ is constant. These cases yield toric varieties that are called dual defect and are studied in \cite{dicksturm}.

Our next aim is to review the procedure of equipping $\mcx_Q$ with an
invariant symplectic structure. We will follow the usual route of
symplectic reduction \cite{audin}. We take \gls{ucirc}$\mbbT = \{z \in \C^* : |z| = 1\}$ and, given any lattice $\Gamma$, we
write ${\mbbT}_\Gamma$ and $\mft_\Gamma \approx \Gamma_\R$ for the
real torus $\mbbT \otimes \Gamma$ and its Lie algebra. We will utilize the fundamental sequence
\begin{equation} \label{eq:Bex} 0 \to L_{\du{Q}} \stackrel{\alpha_{\du{Q}}}{\longrightarrow} \Z^{\du{Q}} \stackrel{\beta_{\du{Q}}}{\longrightarrow} \Lambda^\vee \to K_{\du{Q}} \to 0 . \end{equation}
We note that the toric variety $X_{\fan_{\du{Q}, \mcb_Q}} \subset \C^{\du{Q}}$ is an open equivariant
subset, so that restricting the standard K\"ahler structure on
$\C^{\du{Q}}$ to $X_{\Sigma_A}$ yields the moment map \gls{moment} $\mu_{\du{Q}} : X_{\Sigma_A} \to \R_{\geq 0}^{\du{Q}}$ given by
\begin{equation} \label{eq:momentstandard} \mu_{\du{Q}} (z_1, \ldots, z_{|\du{Q}|} ) = (|z_1|^2 , \ldots , |z_{|\du{Q}|}|^2) , \end{equation}
where we have chosen the action of $\mbbT_{\Z^{\du{Q}}}$ on $\C^{\du{Q}}$ to be
\begin{equation*} (\theta_1, \ldots , \theta_{|\du{Q}|} ) \cdot (z_1 , \ldots, z_n ) = (e^{-2 i \theta_1 } z_1 , \ldots, e^{-2 i \theta_{|\du{Q}|} } z_{|\du{Q}|} ). \end{equation*}
On the other hand, restricting to the ${\mbbT}_{L_{\du{Q}}}$ action gives the
moment map $\mu_{L_{\du{Q}}} = \mu_\du{Q} \circ \alpha_{\du{Q}}^\vee$ where $\alpha_{\du{Q}}^\vee : \mft^\vee_{\Z^\du{Q}}
\to \mft^\vee_{L_{\du{Q}}}$ is just tensoring with $\R$ and taking the dual.
Choosing a value $\omega $ in the interior of the image of $\mu_{L_{\du{Q}}}$
gives a symplectic form on $\mcx_Q$ via the symplectic reduction
\begin{equation*} (\mcx_Q , \omega ) = [\mu_{L_{\du{Q}}}^{-1} (\omega ) /
\mbbT_{L_{\du{Q}}}] . \end{equation*}
We write \gls{symred} $\rho_\omega : \mu_{L_{\du{Q}}}^{-1} (\omega ) \to \mathcal{X}_Q$ for the symplectic quotient map. If no choice of $\omega$ is mentioned, we set 
\begin{align} \label{eq:stsympform} \omega = \alpha_{\du{Q}}^\vee (D_\gamma ) \end{align} and call this the standard symplectic form on $\mcx_Q$. Such a choice fixes $\mcx_Q$ as a monotone symplectic stack, which can be thought of as a very stringent condition \cite{mcdufftoric}. After having chosen a symplectic form on $\mcx_Q$, we recover the moment
map of $\mbbT_{\Lambda^\vee}$ on $\mcx_Q$ by first considering
$\tilde{\Lambda}^\vee = \beta_{\du{Q}} (\Z^{\du{Q}} )$ and the moment map with respect
to $\mbbT_{\tilde{\Lambda}^\vee }$. We have that, for this group,
there is a splitting $i : \mbbT_{\tilde{\Lambda}^\vee} \to
\mbbT_{\Z^\du{Q}}$ of $\beta$. From the exactness of the sequence,
\begin{equation}\label{eq:Bexm} 0 \to L_{\du{Q}} \stackrel{\alpha_{\du{Q}}}{\longrightarrow} \Z^\du{Q} \stackrel{{\beta}_{\du{Q}}}{\longrightarrow} \tilde{\Lambda}^\vee  \to 0 ,\end{equation}
we have that $\tilde{\mu}_{A} : \mcx_Q \to \mft_{\tilde{\Lambda}^\vee}
\approx \tilde{\Lambda}_\R$ is given by $i^* \circ \mu_{\du{Q}}$. To
recover the actual moment map, we need only compose with the
natural map $\tilde{\Lambda}_\R \to \Lambda_\R$ inverse to the
dual of the inclusion. These moment maps fit into the commutative
diagram
\begin{equation} \label{diag:moment2}
\begin{CD} \mu_{L_{\du{Q}}}^{-1} (\omega ) @>{\rho_\omega}>> \mcx_Q \\
@VV{\mu_\du{Q}}V  @VV{\mu_A}V \\
\R^\du{Q} @<{\beta_{\du{Q}}^\vee + \gamma}<< \Lambda_\R \end{CD} \end{equation}
where $\gamma \in \R^\du{Q}$ satisfies
$\alpha_{\du{Q}}^\vee (\gamma ) = \omega$ (note that a different choice will simply
translate the moment map).

We observe that the image of the moment map on $\mcx_Q$ can be
seen as the intersection of an affine subspace $i (\Lambda_\R ) +
\omega$ with the positive cone $\R^\du{Q}_{\geq 0}$. For the 
case of the standard form, the image of $\mu_A$ is $Q$ itself. This can be seen by taking $\gamma = \gamma_A$ from equation~\eqref{eq:gammaA}.
%
%
%
%

\subsection{\label{sec:torichd}Stable pair degenerations}

We now review the procedure for simultaneous degeneration of a toric stack and its hypersurface (see \cite{gross05, gross, mumford72}). 

\begin{defn} Given a section $s \in \linsys{A} \subset H^0 (\mcx_Q , \mco_A (1))$, write $\mcy_s$ for its zero locus and
call the pair $(\mcx_Q, \mcy_s)$ a stable pair. Two such pairs, $(\mcx_Q, \mcy_s)$ and $(\mcx_{Q^\prime} , \mcx_{s^\prime})$ will be considered equivalent if there exists an equivariant isomorphism from $\mcx_Q$ to $\mcx_{Q^\prime}$ which pulls back $s^\prime$ to $s$.
\end{defn}

We recall the definition of a regular marked subdivision \gls{subdiv} $S = \{( Q_i, A_i)\}_{i \in I}$ of $(Q, A)$ from \cite[Chapter~7.2]{GKZ}. First, we require that for each $i \in I$, $A_i \subset A$, $Q_i = \convhull (A_i)$, the union of the $Q_i$ is $Q$, and the intersection of any two $Q_i$ is a face of each. Note that the union  $\cup_{i \in I} A_i$ is not necessarily the set $A$. The added condition of regularity is formulated in the following way. Let $\eta : A \to \R$ be any function and take
\begin{equation*} \gls{qeta} Q_\eta = \convhull \{(a , t) \in \Lambda_\R \oplus \R : a \in A,  t \geq \eta
(a ) \} \end{equation*}
to be the convex hull of the half lines defined by $\eta$. Let $\tilde{\eta} : Q \to \R$ be the function
\begin{equation} \label{eq:tildeeta} \tilde{\eta} (q) = \min \{ t : (q, t) \in Q_\eta\}. \end{equation}
It follows that $\tilde{\eta}$\gls{eta} is a convex, piecewise affine function on $Q$.
\begin{defn} \label{defn:regsub}
We say that $\eta$ is a defining function for the subdivision $S = \{(Q_i , A_i) : i \in I\}$ if
\begin{enumerate}[label=(\roman*), ref=\thelem(\roman*)]
		\item \label{lem:regsub:1} $\tilde{\eta}|_{Q_i}$  extends to an affine function  $\varsigma_i$  on $\Lambda \otimes \R$ ,
		\item  \label{lem:regsub:2}  $\eta (a ) = \varsigma_i  (a )$  if and only if $a \in A_i$.
\end{enumerate}
 $S$ is a regular subdivision if it has a defining function. If the set $A_i$ is affinely independent for every $i \in I$, the subdivision $S$ is called a regular triangulation and denoted by\gls{triang} $T$.
\end{defn}  
An example of the graph $\{(a , \eta (a )): a \in A \}$ of the function $\eta$ and its associated polyhedron $Q_\eta$ is illustrated in Figure \ref{fig:toricd1}.

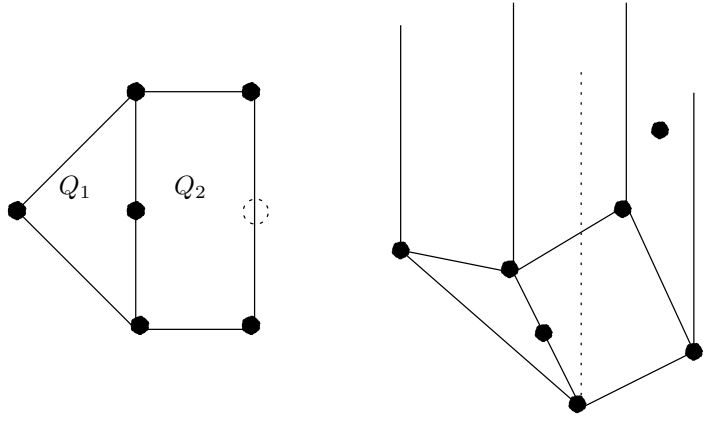
\begin{figure}[h] 
\begin{picture}(0,0)%
\includegraphics{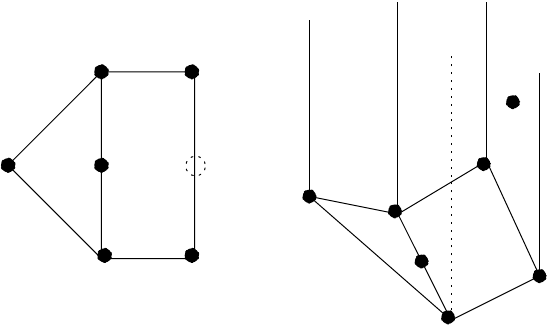}%
\end{picture}%
\setlength{\unitlength}{4144sp}%
\begingroup\makeatletter\ifx\SetFigFont\undefined%
\gdef\SetFigFont#1#2#3#4#5{%
  \reset@font\fontsize{#1}{#2pt}%
  \fontfamily{#3}\fontseries{#4}\fontshape{#5}%
  \selectfont}%
\fi\endgroup%
\begin{picture}(4216,2474)(1268,-2458)
\put(2271,-1144){\makebox(0,0)[lb]{\smash{{\SetFigFont{10}{12.0}{\rmdefault}{\mddefault}{\updefault}{$Q_2$}%
}}}}
\put(1581,-1144){\makebox(0,0)[lb]{\smash{{\SetFigFont{10}{12.0}{\rmdefault}{\mddefault}{\updefault}{$Q_1$}%
}}}}
\end{picture}%
\caption{\label{fig:toricd1} $S = \{(Q_1, A_1), (Q_2, A_2)\}$ and a defining function $\eta$}
\end{figure}

For any regular subdivision $S$, we take $\gls{seccone}C_\R^\circ (S )$ to be the cone of all defining functions for $S$ and $C^\circ_\Z (S ) = (\Z^A)^\vee \cap C_\R^\circ (S )$ the set of integral defining functions. Write $C_\R (S)$ for its closure and $C_\Z (S ) = (\Z^A)^\vee \cap C_\R (S )$. For any $\eta \in C_\Z^\circ (S )$, we define
\begin{equation*} A_\eta = \{(r , t) \in \Lambda \oplus \Z : r \in A
, t \geq \eta (r) \} \end{equation*}
and write $(Q_{\eta, i} , A_{\eta , i} )$ for the marked facet of $(Q_\eta , A_\eta )$ over $Q_i$. 

We will now use integral defining functions to construct and study a degeneration of $\mcx_Q$. This technique follows that of Mumford in \cite{mumford72}. Let $\eta \in C_\Z^\circ (S )$ and
write $\mcx_\eta$ for the toric stack $\mcx_{Q_\eta}$ as
constructed in Definition~\ref{defn:polystack}. Recall that $\du{Q}_\eta$ is in bijection with the facets of the polyhedron $Q_\eta$. Then $\du{Q}_\eta$ can be written as the disjoint union $\du{Q}_\eta^v  \cup \du{Q}_\eta^h$ of two types of facets where $v$ and $h$ refer to vertical and horizontal divisors. The first type, $b \in \du{Q}_\eta^v$, is a facet on the lower boundary of $Q_\eta$. These are in one to one correspondence with the polytopes $\{(Q_i,
A_i) : i \in I\}$ of $S$. The second type, $b \in \du{Q}_\eta^h$ is a facet of $Q_\eta$ which is invariant under positive translations by $(0, t)$ for $t \geq 0$. These are in one to one correspondence with the facets of $Q$ itself. 

We notice that the combinatorics of the polyhedron $Q_\eta$ and thus those of $\mcb_{Q_\eta}$ and $\fan_{\du{Q}_\eta , \mcb_{Q_{\eta}}}$ are dictated by $S$ and not $\eta$. The role that $\eta$ plays in the definition of $\mcx_\eta$ is in the function $\beta_{\du{Q}_\eta}: \Z^{\du{Q}_\eta} \to (\Lambda \oplus \Z)^\vee$.
The sub-fan $\Sigma_{A_\eta}$ consisting of one-cones in $\du{Q}_\eta^v$ projects to a fan $\beta_{\du{Q}_\eta} (\Sigma_{A_\eta} ) \subset  
(\Lambda_\R \oplus \R)^\vee$ with one-cones given by $\beta_{\du{Q}_\eta} ( \du{Q}_\eta^v) = \{ f - d
\varsigma_i : i \in I \}$ where $f = (0, 1) \in (\Lambda
\oplus \Z)^\vee$ and $d \varsigma_i$ is the derivative (or linear part) of the affine function $\varsigma_i$ appearing in Definition~\ref{lem:regsub:1}. 
A subtle point about this formula is that, when $A_i$ affinely spans a proper  sublattice of $\Lambda$, the element $f - d \varsigma_i $ is not necessarily in $(\Lambda \oplus \Z)^\vee$. In this case, it is necessary to take a multiple to obtain a primitive generator. We write $c_{\eta, i} \in \Z_{> 0}$ for the  denominator of $d \varsigma_i$. In other words, for any $i \in I$, we define the constant $c_{\eta, i}$ as 
\begin{align} \label{eq:defceta} \gls{ceta}
c_{\eta, i} := \min \{n \in \Z_{> 0} : n \cdot d \varsigma_i \in \Lambda^\vee \}.
\end{align} 
It is not hard to see that $c_{\eta , i}$ divides the index $[\Lambda : \aff_\Z (A_i)]$ where $\aff_{\Z} (A_i)$ is the affine hull of $A_i$. So, in general, there are only a finite number of possible constants $c_{\eta, i}$ that can occur amongst all $\eta \in C_\Z^\circ (S)$. 

As is always the case with toric stacks defined by polyhedra, the stack $\mcx_\eta$ comes equipped with a line bundle
$\mco_\eta (1)$ such that the vector space $\C^{A_\eta}$ is canonically
identified with a linear system. The map $\eta$ induces a
natural inclusion $ \iota_\eta : \C^A \to \C^{A_\eta} $ given by
\begin{equation*} \iota_\eta \left( \sum_{a \in A} c_a e_a \right) = \sum_{a \in A} c_a e_{(a, \eta (a ))}. \end{equation*}

\begin{defn} \gls{degfam} \label{defn:degfamily} A degenerating family of $(\mcx_Q, \mcy_s)$ is
a stable pair $(\mcx , \mcy)$ equivalent to a pair
$(\mcx_\eta , \mcy_{\iota_\eta (s^\prime)})$ for some  defining function $\eta$ of a regular subdivision
$S$ of $(Q, A)$ and a very full section $s^\prime$.
\end{defn}

We note that the stack $\mcx_\eta$ admits a morphism\gls{tordeg} $F_\eta : \mcx_\eta \to
\C$. Taking $\C$ to be the stacky fan given by $(\Z, \Z, 1_\Z,  \R_{\geq 0} )$ where $\R_{\geq 0}$ is thought of as the fan consisting of itself and $\{0\}$, we may describe $F_\eta$ as a map $(f_1, f_2)$ of stacky fans
\begin{equation*}
\begin{CD} \Z^{\du{Q}_\eta} @>{\beta_{\du{Q}_\eta}}>> (\Lambda \oplus \Z)^\vee \\
@V{f_1}VV  @V{f_2}VV \\
\Z @>{Id}>> \Z .
\end{CD}
\end{equation*}
Here, for every $b \in \du{Q}_\eta^h$,  $f_1 (e_b ) = 0$ while for $b_i \in \du{Q}_\eta^v$ corresponding to $(Q_i, A_i)$, $f_1 (e_{b_i}) = c_{\eta , i}$. The homomorphism $f_2$ is simply projection to the $\Z$ factor.
 It is not hard to see then that the fiber of $(\mcx_\eta ,
\mcy_{\iota_{\eta} (s)})$ over $1 \in \C^*$ is equivalent to $(\mcx_Q, \mcy_s)$. On the other hand, the fiber over zero is the union $\left( \cup_{i \in I} \mcx_{Q_i}, \cup_{i \in I} \mcy_{s |_{A_i}} \right)$ whose irreducible components are equivalent to the toric pairs $(\mcx_{Q_i}, \mcy_{s |_{A_i}})$. 

It is useful to view the morphism $F_\eta$ from the moment map perspective as well. Here we have that $\mu_{L_{\du{Q}_\eta}}^{-1} (\omega ) \subset \C^{\du{Q}_\eta}$ defines the stack $\mcx_\eta$ after taking the quotient by $\mbbT_{L_{\du{Q}_\eta}}$. Observe that the map $F_\eta$ then can be defined on $\C^{\du{Q}_\eta}$ as 
\begin{equation} \label{eq:toricd} \tilde{F}_\eta \left(z_1, \cdots, z_{|\du{Q}_\eta|} \right) = \prod_{i \in \du{Q}_\eta^v} z_i^{c_{\eta, i}} . \end{equation}
In other words, $\tilde{F}_\eta$ is invariant with respect to the $\mbbT_{L_{\du{Q}_\eta}}$ action and descends to $F_\eta$ on the quotient $\mcx_\eta = [\mu_{L_{\du{Q}_\eta}}^{-1} (\omega ) / \mbbT_{L_{\du{Q}_\eta}}]$.

In general, the marking $A$ should be thought of as a set
specifying the non-zero coefficients of a given section. Let $A_v \subset A$ be the set of vertices of $Q$ and call any stable pair $(\mcx_Q , \mcy_s)$ full if $s \in (\C^*)^{A_v} \times \C^{A -  A_v}$ and very full if $s \in (\C^*)^A$.

\begin{defn} \label{defn:degfiber} Let $s \in H^0 (\mcx_Q, \mco_A (1))$ be a full section and $F_\eta : \mcx \to \C$ is the projection associated to a  $\eta \in \Z^A$.
\begin{itemize} \item[(i)] A toric degeneration of $\mcx_Q$ is the fiber $F_\eta^{-1} (0)$. 
\item[(ii)] A hypersurface degeneration of $\mcy_s$ is the fiber $ F_\eta^{-1} (0) \cap \mcy$. 
\item[(iii)] A stable pair degeneration of $(\mcx_Q, \mcy_s)$ is the pair $(F_\eta^{-1} (0),  F_\eta^{-1} (0) \cap \mcy)$.
\end{itemize}
If $t \in \C$, we write $\fib{\eta}{t}$ for the fiber $F_\eta^{-1} (t) \cap \mcy_s$. \end{defn}

\subsection{\label{sec:toricstack}Secondary and Lafforgue stacks}

In this section we give an explicit formulation of several moduli stacks related to $A$. One stack we obtain is closely related to those defined in \cite{alexeev02} and \cite{Lafforgue}.

We start by setting up more notation and recalling several general results from \cite{GKZ}. Given a monoid $M$ acting on an abelian group $\Lambda$ and a subset $A \subset \Lambda$, we write $\gls{lin}_M (A)$ for the set of linear combinations of $A$ with coefficients in $M$. Again we assume $A \subset \Lambda$ is a finite set which affinely spans $\Lambda \otimes \R$ and promote it to the subset
\begin{align}
\gls{gicone} := \{ (a , 1) : a \in A \} \subset \Lambda \oplus \Z.
\end{align}
This spans a semigroup $\textnormal{Lin}_\N ({\mathcal{A}} )$ with convex hull $\textnormal{Lin}_{\R_{\geq 0}} ({\mathcal{A}} )$. We note that the supporting hyperplane functions $\dul{\textnormal{Lin}_{\R_{\geq 0}} ({\mathcal{A}} )} = \{(b , n_b) : b \in \du{Q}\}$ and $\mcx_{\dul{\textnormal{Lin}_{\R_{\geq 0}} ({\mathcal{A}} )}}$ is the affine cone of $\mcx_{Q}$ where constants $n_b$ were defined in equation \eqref{eq:mcoq}. Recall from equation \eqref{eq:setfunseq} that the fundamental sequence associated to $\mathcal{A}$ is
\begin{equation} \label{eq:Aex} 0 \to L_{\mathcal{A}} \stackrel{\alpha_{\mathcal{A}}}{\longrightarrow} \Z^{\mathcal{A}} \stackrel{\beta_{\mathcal{A}}}{\longrightarrow} \Lambda \oplus \Z \to K_{\mathcal{A}} \to 0 . \end{equation}
We will return to the extension ${\mathcal{A}}$ of $A$ and the sequence \eqref{eq:Aex} several times throughout this section.

A marked polytope $(Q, A)$ will be referred to as a simplex if $Q$ is a simplex and $A$ is its set of vertices. Recalling Definition~\ref{defn:regsub}, a regular triangulation of $A$ is a regular subdivision $S = \{(Q_i, A_i ) : i \in I\}$ such that every $(Q_i , A_i)$ is a simplex. Such triangulations correspond to vertices of the secondary polytope $\psec{A}$ as defined in \cite[Chapter~7]{GKZ}. More concretely, for a regular triangulation $T =  \{(Q_i, A_i): i \in I\}$, define the element 
\begin{equation} \gls{spv} \label{eq:secvert} \varphi_T = \sum_{a \in \cup A_i} \left(\sum_{a \in A_i} \Vol (Q_i) \right) e_a \in \Z^A . \end{equation}
In this formula, $\Vol (Q_i)$ is normalized so that the standard simplex has volume $1$.
\begin{figure}[t]
\begin{picture}(0,0)%
\includegraphics{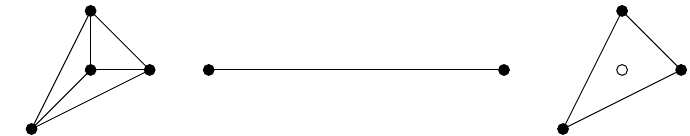}%
\end{picture}%
\setlength{\unitlength}{4144sp}%
\begin{picture}(5239,1018)(-689,-1460)
\put(-634,-601){\makebox(0,0)[lb]{\smash{$T_-$}}}
\put(4456,-601){\makebox(0,0)[lb]{\smash{$T_+$}}}
\put(1846,-736){\makebox(0,0)[lb]{\smash{$\Sigma (A)$}}}
\put(846,-1276){\makebox(0,0)[lb]{\smash{$\varphi_{T_-}$}}}
\put(3016,-1276){\makebox(0,0)[lb]{\smash{$\varphi_{T_+}$}}}
\end{picture}%
\caption{\label{fig:circuitfirst} Regular triangulations and the secondary polytope for $A = \{(0,0), (1,0), (0,1),(-1,-1)\} \subset \Z^2$.}
\end{figure}
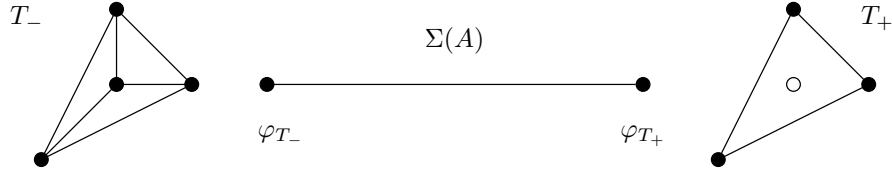

The secondary polytope of $A$ is then the convex hull
\begin{equation} \gls{sp} \label{eq:secondef} \psec{A} = \convhull \{ \varphi_T : T \text{ a regular triangulation of } A \} \subset \R^A . \end{equation}
While the vertices of $\Sigma (A)$ have a particularly nice  formula in $\R^A$, we will see in Theorem \ref{thm:GKZ2} that the dimension of $\Sigma (A)$ is always $(|A| - d - 1)$ where $d$ is the rank of $\Lambda$. 

\begin{eg}  Suppose $A \subset \Z^2 = \Lambda$ is the subset $A = \{(1,0), (-1, 0), (0,1)\}$ from Example \ref{eg:simplex} consisting of the vertices of a simplex. Then there is only one regular triangulation $T = \{(Q,A)\}$ and the secondary polytope $\Sigma (A)$ consists of a single point $\varphi_T = 2e_{(1,0)} + 2 e_{(-1,0)} + 2 e_{(0,1)}$. 
\end{eg}

\begin{eg}
	Suppose $A = \{(0,0), (1,0), (0,1),(-1,-1)\} \subset \Z^2 = \Lambda$ as in Example~\ref{eg:circuit1} and observe that there are precisely two regular triangulations $T_-$ and $T_+$ of $(Q,A)$ illustrated in Figure \ref{fig:circuitfirst}. Thus the secondary polytope in this case is an interval between the points
	\begin{align} \begin{split} \label{eq:exampleverts}
	\varphi_{T_-} & = 3e_{(0,0)} + 2e_{(1,0)} + 2 e_{(0,1)} + 2 e_{(-1,-1)}, \\
	\varphi_{T_+} & = 3e_{(1,0)} + 3 e_{(0,1)} + 3 e_{(-1,-1)}.
	\end{split}
	\end{align}
\end{eg}

The next cited theorem connects the secondary polytope to the linear system $\linsys{A}$. In order to state it, we review more of the notation from \cite[Section~5.3]{GKZ}. We write $Q^\prime \leq Q$ if $Q^\prime$ is a face of $Q$. For any face $Q^\prime \leq Q$, take ${\mathcal{A}}^\prime = \{(a , 1): a \in Q^\prime \cap A\}$ and let $\textnormal{Lin}_\R (\mca^\prime )$ and $\textnormal{Lin}_\Z (\mca^\prime )$ be the $\R$-linear and $\Z$-linear span of ${\mathcal{A}}^\prime$ respectively. Then the index $i (Q^\prime , A)$ is set to equal $[\Lambda \oplus \Z \cap \textnormal{Lin}_\R ( \mca^\prime ) : \textnormal{Lin}_\Z (\mca^\prime )]$. Given an additive monoid $M$ contained in a lattice, the notation $u(M)$ denotes its subdiagram volume. This is defined by letting $\Lambda$ be the group completion of $M$, $K (M)$ ($K_+ (M)$) the convex hull of $M$ ($M - \{0\}$) in $\Lambda_\R$, and $K_- (M)$ equal to the closure of $K (M) - K_+ (M)$. With this notation, the subdiagram volume is given by $u(M) = \Vol_{\Lambda} (K_- (M))$. The notation $u ({\textnormal{Lin}_{\mathbb{N}} (\mathcal{A})} / Q^\prime )$ denotes the subdiagram volume of the semigroup ${\textnormal{Lin}_{\mathbb{N}} (\mathcal{A})} / Q^\prime$ defined as the image of ${\textnormal{Lin}_{\mathbb{N}} (\mathcal{A})}$ in $\Lambda \oplus \Z /  (\Lambda \oplus \Z \cap \textnormal{Lin}_\R ( \mca^\prime ) )$. 
\begin{thm}{\cite[Theorem~10.1.2]{GKZ}} \label{thm:GKZ1} \mbox{ }
\begin{itemize}
\item[(i)]  The Newton polytope of $E_A$ is $\psec{A}$. 
\item[(ii)] $E_A (f) = \prod_{Q^\prime \leq Q} \Delta_{A \cap Q^\prime} (f)^{i (\Lambda , A) \cdot u ({\textnormal{Lin}_{\mathbb{N}} (\mathcal{A})} / Q^\prime )}$.
\end{itemize} \end{thm}
The exponent $i (Q^\prime , A) \cdot u ({\textnormal{Lin}_{\mathbb{N}} (\mathcal{A})} / Q^\prime )$ equals the multiplicity of any point on the orbit associated to $Q^\prime$ on the possibly non-normal toric variety associated to $A$. We prefer the formulation above over simply writing the multiplicity since our definition of a toric stack associated to a polytope does not coincide with the one given in \cite{GKZ}. However, there is always a dominant map from our definition of $\mcx_Q$ to theirs, namely, the map associated to the linear system given by $A$.

The secondary fan is a construction more in the spirit of Appendix~\ref{sec:torichd} than the secondary polytope. This fan consists of the cones of defining functions $C_\R (S )$ given in Definition~\ref{defn:regsub} for all regular subdivisions $S$. We write \gls{secfan} $\fans{A}$ as the secondary fan with support $(\R^A)^\vee$ and cite the following theorem.
\begin{thm}{\cite[Chapter~7.1]{GKZ}} \label{thm:GKZ2} \mbox{ }
\begin{itemize}
\item[(i)] The secondary polytope $\psec{A}$ has a single point as its image under $\beta_{\mathcal{A}}$. 
\item[(ii)] The fan $\fans{A}$ is the normal fan of $\psec{A}$.
\end{itemize}  
\end{thm}
In more detail, it follows from \cite[Proposition~7.1.11]{GKZ} that 
\begin{align} \label{eq:secpolhomogenious} \beta_{\mathcal{A}} (\psec{A}) = (\delta_Q , (d + 1) \Vol (Q)) \end{align} where $\delta_Q = (d + 1) \int_Q x \hbox{ d}x$ is the dilated centroid of $Q$ and that $\psec{A}$ affinely spans the fiber $\beta_{\mathcal{A}}^{-1} (\delta_Q , (d + 1) \Vol (Q)) $. Consequently, $\Sigma (A)$ is an $(|A| - d - 1)$-dimensional polytope inside an $|A|$-dimensional vector space.  We will define several stacks associated to $\psec{A}$  utilizing techniques from Appendix~\ref{sec:toricpreliminaries}. Since $\psec{A}$ does not affinely span $\R^A$, but rather an affine plane parallel to $L_{\mathcal{A}} \otimes \R$, we cannot define $\mcx_{\psec{A}}$ as before. Instead, choose any $v \in \Z^A$ for which $\beta_{\mathcal{A}} (v) = \delta_Q$ and let
\begin{equation*} \gls{tsp} \psecv{A}{v} = \left\{ w \in L_{\mathcal{A}} \otimes \R : \alpha_A (w) + v \in \psec{A} \right\} \subset L_{\mathcal{A}} \otimes \R . \end{equation*}
Thus $\psecv{A}{v}$ is the translation of $\psec{A}$ to a full dimensional integral polytope in a linear, instead of affine, subspace. As a different choice of $v$ will simply translate $\psec{A}$ in $L_{\mathcal{A}} \otimes \R$, the stack $\mcx_{\psecv{A}{v}}$ is independent of this choice.  We will denote it by $\mcx^r_{\psec{A}}$ where the exponent $r$ is a  notational convenience to distinguish it from a finer stack $\mcx_{\psec{A}}$ which will be defined later in this section. 

Let us detail the stacky fan associated to $\psecv{A}{v}$. First observe that Theorem \ref{thm:GKZ2} gives a bijective correspondence between faces of $\psec{A}$ (or equivalently, the translated polytope $\psecv{A}{v}$) and regular subdivisions of $A$. This bijection is order reversing in the sense that a face inclusion corresponds to a refinement of a subdivision.  Recall that $\dul{\psecv{A}{v}} \subset L_A^\vee$ denotes the supporting hyperplane primitives for $\psecv{A}{v}$. By \cite[Section~7.2]{GKZ}, the set of supporting hyperplanes is $\{b_S : S \text{ a coarse subdivision}\}$. By definition, a coarse subdivision is a regular subdivision that is not a refinement of any non-trivial regular subdivision. Given $b \in \dul{\psecv{A}{v}}$, we take $S_b$ to be the corresponding coarse subdivision and $F_b$ the facet of $\Sigma (A)$ supported by $b$.  A collection $J \subset \dul{\psecv{A}{v}}$ is in the abstract simplicial complex $\mcb$ associated to $\psecv{A}{v}$ if and only if there is a regular subdivision $S$ refining the coarse subdivisions $\{S_b : b \in J\}$. Indeed, we recall from equation \eqref{eq:simpcompdual} that this simplicial complex, viewed as a poset inside the power set of its vertices, is dual to the face poset of $\Sigma (A)$. So if $J = \{b_1, \ldots, b_{k}\}$, $J$ will be a member if and only if the intersection of the facets $F_{b_1}, \ldots, F_{b_k}$ is a non-empty face of $\Sigma(A)$. This is equivalent to there existing a regular refinement $S$, corresponding to the face $F_{b_1} \cap \cdots \cap F_{b_k}$, of $S_{b_1}, \ldots, S_{b_k}$. Assembling these structures gives the stacky fan 
\begin{align} \label{eq:secstkfan}  \sfan_{\psecv{A}{v}} = \left( \Z^{\dul{\psecv{A}{v}}} , L_A^\vee , \beta_{\dul{\psecv{A}{v}}}, \fan_{\mcb} \right)  \end{align} 
for $\mcx^r_{\psec{A}} $.
\begin{eg} \label{eg:simplex2}
We continue to explore Examples \ref{eg:simplex} and \ref{eg:circuit1}. For Example \ref{eg:simplex}, one observes that since the secondary polytope is a point, the stacky fan  $\sfan_{\psecv{A}{v}} = (0, 0, 0, \{0\})$ is completely trivial and defines only a point. For Example \ref{eg:circuit1}, notice that there are no lattice points on the relative interior of $\psec{A}$, so that $\psecv{A}{v}$ is a unit interval in $L_A^\vee\otimes \R \cong \R$. Thus $\mcx_{\psec{A}}^r$ is isomorphic to $\p^1$ and the line bundle determined by $\psec{A}$ is $\mco (1)$.
\end{eg}

To obtain more control over the hypersurfaces in $\mcx_Q$ and their degenerations, we will need a more nuanced secondary stack than $\mcx^r_{\psec{A}}$. Instead of working around the fact that $\psec{A}$ does not span $\R^A$, we extend the polytope $\psec{A}$ to a polyhedron $\plaf{A}$ and apply constructions from Appendix~\ref{sec:toricpreliminaries}. This yields a stack $\laf{A}$ which we call the Lafforgue stack of $A$ due to the fact that its coarse toric variety equals the Lafforgue variety as defined in \cite{Lafforgue}.
\begin{defn}\gls{lp} \gls{lph} Let $\simplex{A}{t} = \left\{\sum_{a \in A} c_a e_a : c_a \geq 0 ,  \sum c_a = t \right\}$ be a simplex in $\R^A$ and $\simplex{A}{\geq t} = \cup_{s \geq t} \simplex{A}{s}$.
\begin{itemize}
\item[(i)] The Lafforgue polytope $\ptlaf{A}$ of $A$ is the Minkowski sum $\psec{A} + \simplex{A}{1}$,
\item[(ii)] The Lafforgue polyhedron $\plaf{A}$ of $A$ is defined as the Minkowski sum $ \psec{A} + \simplex{A}{\geq 1}$. 
\end{itemize} 
\end{defn}
%

To justify the name of these polyhedra, we recall the construction by Lafforgue (\cite{hacking}, \cite[Chapter~2.1]{Lafforgue}) of a fan \gls{lfan}$\fanl{A}$ which refines the secondary fan $\fans{A}$. Given a regular subdivision $S = \{ (Q_i, A_i ) : i \in I \}$ and a non-empty marked face $(Q_p , A_p)$  of one of the subdividing polytopes $(Q_i, A_i)$ satisfying $A_p = Q_p \cap A_i$, we define the closed cone
\begin{align} \gls{lafcone} \label{eq:lafcones} C_\R (S, A_p) = \{\eta \in {C_\R (S )} : \eta ( a ) \leq \eta (a^\prime ) \text{ for all } a \in A_p, a^\prime \in A\} . \end{align}
We call the pair\gls{psubdiv} $(S , A_p)$ a pointed subdivision and when $A_p = \{a\}$, we simply write $C_\R (S, a )$. It is clear that $C_\R (S, A_p) \subset C_\R (S^\prime , A_p^\prime )$ if and only if $S^\prime$ refines $S$ and $A_p \supset A_p^\prime$. In this case we write $(S^\prime , A_p^\prime) \preceq (S , A_p)$.
By definition, the fan $\fanlt{A}$ consists of the cones $\{C_\R (S , A_p) : (S, A_p) \text{ a pointed subdivision of }(Q,A)\}$. For certain classes of sets $A$, Lafforgue has shown that the toric variety associated to this fan yields a parameter space for toric degenerations of the variety $X_A$. However, this paper is concerned primarily with degenerations of hypersurfaces in a toric stack, so in order to relate this work to ours, we require a line bundle on the associated variety. Furthermore, to preserve information on toric isomorphisms, we wish to consider the toric stack construction along the lines of Appendix~\ref{sec:toricpreliminaries}.  For this, we prove the following lemma.
\begin{lem} \label{prop:lafpolytope} The fan $\fanlt{A}$ is the normal fan of to the polytope $\ptlaf{A}$. 
\end{lem}
\begin{proof} Let $R \subseteq \ptlaf{A}$ be any subset containing the vertices of $\ptlaf{A}$. Given any element $\phi \in \ptlaf{A}$, write 
	\begin{align*} N_{\phi} (\ptlaf{A}) & = \{\psi \in (\R^A)^\vee : (\psi , \phi) \leq (\psi , \phi^\prime) \text{ for all }\phi^\prime \in \ptlaf{A}\}, \\ & = \{\psi \in (\R^A)^\vee : (\psi , \phi) \leq (\psi , \phi^\prime) \text{ for all }\phi^\prime \in R \} 
	\end{align*}
for the normal cone of $\phi$. Here, in accordance with the description of defining functions in $C_\R (S)$, we view elements $\psi \in (\R^A)^\vee$ as functions from $A$ to $\R$ and the contraction is given by $(\psi , e_a) := \psi (a)$. 
	
It follows from the definition that $\fanlt{A}$ is a refinement of $\fans{A}$. In particular, the cones $C_\R (S , A_p )$ can be described as intersections of $C_\R (T_j , a)$ where $T_j$ is a regular triangulation refining $S$, and $a$ is both a member of $A_p$ and a vertex in a simplex of $T$.  Thus the cones 
\begin{align*} \{C_\R (T, a) : T \text{ a regular triangulation}, a \text{ a vertex in a simplex of T} \} \end{align*} form the set of maximal cones $\fanlt{A} (|A|)$. We now show that each one of these maximal cones is a normal cone to an element in $\ptlaf{A}$. 

For any regular triangulation $T$ and $a \in A$ let
\begin{equation} \label{eq:verticeslaf} \varphi_{(T, a)} := \varphi_T + e_a . \end{equation}
Note that since $\ptlaf{A}$ is the Minkowski sum of $\Sigma (A)$ and $\Delta^A_1$, the set 
\begin{align*} R := \{\varphi_{(T, a)} : T \text{ a regular triangulation} , a \in A \} \end{align*} contains the set of vertices of $\ptlaf{A}$. Fixing one triangulation $T$, suppose $a$ is a vertex of a simplex of $T$ so that $C_\R (T, a)$ is in $\fanlt{A} (|A|)$. Let $\psi \in C_\R (T, a)$, and $\varphi_{(T^\prime, b )} \in R$. By the definition of $C_\R (T, a )$, we have $\psi (a ) \leq \psi (a^\prime )$ for any $a^\prime \in A$. Using the result that the secondary fan is dual to the secondary polytope, and in particular that $C_\R (T) = N_{\varphi_T} (\psec{A})$,  we have
\begin{eqnarray*} (\psi , \varphi_{(T, a )} ) & = & (\psi , \varphi_T) +  \psi (a ) , \\ & \leq & (\psi , \varphi_T ) +  \psi (b ) , \\ & \leq & (\psi , \varphi_{T^\prime}) +  \psi (b ) , \\ & = & (\psi , \varphi_{(T^\prime, b )}).  \end{eqnarray*}
Thus $C_\R (T, a ) \subseteq N_{\varphi_{(T, a)}} (\ptlaf{A})$. For the converse, one simply observes that both the normal fan to $\ptlaf{A}$ and the fan $\fanlt{A}$ are complete fans supported in $\R^A$ with $C (T, a )$, and thus also $N_{\varphi_{(T, a)}} (\ptlaf{A})$, both $|A|$-dimensional cones. The inclusions $C_\R (T, a ) \subseteq N_{\varphi_{(T, a)}} (\ptlaf{A})$ thus imply that $\fanlt{A}$ is a refinement of the normal fan to $\ptlaf{A}$. However, as the number of vertices of $\ptlaf{A}$ is greater than or equal to  the number of maximal cones in $\fanlt{A}$, we must have that $C_\R (T, a ) = N_{\varphi_{(T, a)}} (\ptlaf{A})$. Returning to the initial observation that every cone in $\ptlaf{A}$ can be described as a non-trivial intersection of the maximal cones $C_\R(T, a)$, and observing that the same is true for normal fans of polytopes, we obtain the result.
\end{proof}


This proposition gives us a polarization for the variety associated to the Lafforgue fan. However, if we wanted to obtain a polytope spanning $\R^{\mathcal{A}}$, we have missed the mark by one dimension. As in the case of the secondary polytope, we could restrict to the subspace spanned by $\ptlaf{A}$. However, it is more natural to consider the polyhedron $\plaf{A} \subset \R^{\mathcal{A}}$ and a variant of its associated stacky fan as defined in Appendix~\ref{sec:toricpreliminaries}. Before introducing this stacky fan, we examine the combinatorics and geometry of the polyhedron $\plaf{A}$.

\begin{lem}
	\label{lem:lafhyppart} \gls{lafhv} The primitives of the supporting hyperplanes, $\dul{\plaf{A}}$, can be partitioned into a disjoint union \begin{align*} \{\varrho_A\} \cup \dul{\plaf{A}}^h \cup \dul{\plaf{A}}^v \subset (\mathbb{Z}^A)^\vee \end{align*} where :
	\begin{itemize}
		\item[(i)]  $\varrho_A = \sum_{a \in A} e_a^\vee$ defines the supporting hyperplane of $\ptlaf{A}$.
		\item[(ii)] The set $\dul{\plaf{A}}^h$ bijectively corresponds to pointed subdivisions $(S, A_p)$ where $S = \{(Q, A)\}$ and $A_p$ are the elements of $A$ on a facet of $Q$.
		\item[(iii)] The set $\dul{\plaf{A}}^v$ bijectively corresponds to pointed subdivisions $(S, A_p)$ where $S = \{(Q_i, A_i) : i \in I\}$ is a coarse subdivision and $A_p = A_i$ for some $i \in I$.
	\end{itemize} 
\end{lem}
\begin{proof}
First we observe that the polyhedron $\plaf{A}$ is combinatorially equivalent to the cone  $\mathbb{R}_{\geq 1} \times \ptlaf{A}$ over $\ptlaf{A}$. This can been seen by recalling equation \eqref{eq:secpolhomogenious} which gives us that $\varrho_A ( \psec{A} ) = (d + 1) \Vol (Q)$ and, by definition, $\varrho_A (\Delta^A_t) = t$. As $\plaf{A} = \Delta_{\geq 1}^A + \psec{A} = \cup_{t \geq 1} (\Delta^A_t + \psec{A})$, we have that  
\[\varrho_A : \plaf{A}  \to \left[1 + (d + 1)\Vol (Q), \infty \right) \] combinatorially trivializes $\plaf{A}$ as the product of a ray and $\ptlaf{A}$. Since $\ptlaf{A}$ has the $(|A|-1)$-dimensional simplex as a Minkowski summand, it is $(|A|-1)$-dimensional. In particular, $\ptlaf{A}$ is the facet $\varrho_A^{-1} (1 + (d + 1) \Vol (Q))$ of $\plaf{A}$ defined by the primitive $\varrho_A \in \dul{\plaf{A}}$. 

Since $\plaf{A}$ is combinatorially a product of $\ptlaf{A}$ and a ray, the remaining facets of $\plaf{A}$ arise as products $ \mathbb{R}_{\geq 1} \times F$ where $F$ is a facet of $\ptlaf{A}$. By Lemma \ref{prop:lafpolytope}, these are in bijection with the minimal non-trivial cones in $\fanlt{A}$. Here the trivial cone is the one dimensional space spanned by $\varrho_A$ (as it is the normal cone to points in the relative interior of $\ptlaf{A}$). Such cones correspond to pointed subdivisions $(S, A_p)$ which are minimal among non-trivial pointed subdivisions with respect to the partial order $\preceq$ discussed after the definition of $C(S, A_p)$ in equation \eqref{eq:lafcones}. In particular, they are pointed subdivisions $(S, A_p)$ such that $(\{(Q, A)\}, A) \prec (S, A_p)$, but no other pointed subdivision $(S^\prime, A_p^\prime)$ satisfies $(\{(Q, A)\}, A) \prec  (S^\prime, A_p^\prime) \prec (S, A_p)$. It follows from the definition of $\prec$ that either $S = \{(Q, A)\}$ or $S$ is a coarse subdivision. In the former case, $A_p$ must be the set of points in $A$ lying on a facet of $Q$ (again, by the definition of $\prec$). We let the collection of the dual primitives of such facets make up the subset $\dul{\plaf{A}}^h$. In the latter case, we have $S = \{(Q_i, A_i) : i \in I\}$ is a coarse subdivision and, if $A_p$ lies on a proper face of $Q_i$ for some $i \in I$, then $(\{(Q, A)\}, A) \prec (S, A_i) \prec (S,A_p)$ contradicting minimality of $(S,A_p)$. Thus $A_p = A_i$ for some $i \in I$ and we denote the collection of the dual primitives to these facets $\dul{\plaf{A}}^v$.
\end{proof}
Having classified elements of $\dul{\plaf{A}}$ combinatorially, we now consider their linear forms. For elements of $\dul{\ptlaf{A}}^h$, we return to the exact sequence \eqref{eq:Aex}. If $A_p = F \cap A$ for a facet $F$ of $Q$, then there is a unique primitive $b_{A_p} \in (\Lambda \oplus \Z)^\vee$ which is a supporting hyperplane for the cone $\textnormal{Lin}_{\R_{\geq 0}} ({\mathcal{A}})$ and vanishes on $A_p \oplus \{1\} \subset \Lambda \oplus \Z$. 
\begin{lem} \label{lem:lafnormal}  The elements of $\dul{\plaf{A}} \subset (\mathbb{Z}^A)^\vee$ not equal to $\varrho_A$ are uniquely characterized by:
	\begin{enumerate}[label=(\roman*), ref=\thelem(\roman*)]
		\item \label{lem:lafnormal:1} If $b \in \dul{\plaf{A}}^h$ then it is contained in $C_\R (\{(Q, A)\}, A_p)$ where $A_p$ consists of all elements of $A$ in a facet of $Q$. There exists $c_b \in \N$ such that \[b = c_b^{-1} \beta_{\mathcal{A}}^\vee (b_{A_p}). \]
		\item \label{lem:lafnormal:2} If $b \in \dul{\plaf{A}}^v$ corresponds to $(S, A_p)$, then  $b = \eta_{(S, A_i)} \in C_\Z^\circ (S )$ is the primitive defining function for $S$ satisfying $\eta_{(S , A_i)} |_{A_i} = 0$.
	\end{enumerate} 
\end{lem}

\begin{proof}
	First observe that the proof of Lemma \ref{lem:lafhyppart} classifies the dual facets to a given $b \in \dul{\plaf{A}}$. In particular, if $b \ne \varrho_A$, then there is a pointed subdivision $(S, A_p)$ for which $b \in C_\R (S, A_p)$. Take $F_S \subseteq \Sigma (A)$ to be the face of $\Sigma (A)$ whose normal cone is $C_\R (S)$ in $\mcf_{\psec{A}}$. Note that when $b \in \dul{\plaf{A}}^h$ we have $F_S = \Sigma (A)$, while if $b \in \dul{\plaf{A}}^v$,  $F_S$ is a facet of $\Sigma (A)$. In either case, the facet of $\plaf{A}$, defined by $b$ is then the polyhedron
	\begin{align*}
	F_{(S,A_p)} := \bigcup_{t \geq 1} \left( F_S + t \cdot \convhull \left\{ e_a :  a \in A_p \right\} \right).
	\end{align*} 
	Since $b$ is constant along the facet $F_{(S, A_p)}$ and $e_a$ is parallel to $F_{(S,A_p)}$ for every $a \in A_p$, we have that $b  |_{A_p} = 0$. This, along with the fact that $b$ is a primitive element of  $C_\Z (S, A_p)$, uniquely characterizes $b$ and proves \ref{lem:lafnormal:2}.
	
	To prove \ref{lem:lafnormal:1}, assume $b \in C_\R (\{(Q, A)\}, A_p)$ so that $b \in C_\R (\{(Q,A)\})$ implying $b$ is the restriction of an affine function on $\Lambda_\R$ to $A$ (for otherwise, it defines a non-trivial subdivision). The set of such functions is precisely the image of $\beta_{\mathcal{A}}^\vee$. In particular, if  $b (a) = \psi (a) $ for every $a \in A$, where $\psi (u) = \tilde{\psi} (u) + c$ for a linear function $\tilde{\psi} \in \Lambda_\R^\vee$ and $c \in \R$ then $b = \beta_{\mathcal{A}}^\vee (\tilde{\psi}, c)$.	Since $b$ achieves its minimum  strictly on $A_p$, for every $a \in A_p$ and $a^\prime \in A$, we have that $\tilde{\psi} (a ) \leq \tilde{\psi } (a^\prime )$ with equality if and only if  $a^\prime \in A_p$. Thus $\tilde{\psi}$ is a supporting hyperplane of the convex hull of $A_p$. Furthermore, since $(\tilde{\psi}|_{A_p} + c) = b |_{A_p} = 0$, we have that $(\tilde{\psi} , c) |_{A_p \oplus \{1\}} = 0$. Thus $(\tilde{\psi}, c)$ also equals zero on the cone $\textnormal{Lin}_{\R_{\geq 0}} (A_p \oplus \{1\})$, which is a facet of $\textnormal{Lin}_{\R_{\geq 0}} (A \oplus \{1\})$. Thus $(\tilde{\psi} , c)$ can be expressed uniquely as $r \cdot b_{A_p}$ with $r > 0$. As both $b$ and $b_{A_p}$ are primitive and $\beta_{\mathcal{A}}^\vee : \textnormal{Lin}_{\Z} (b_{A_p}) \to \textnormal{Lin}_{\Z} (b)$, we have that $r = c_b^{-1}$ for a unique $c_b \in \mathbb{N}$.
\end{proof}
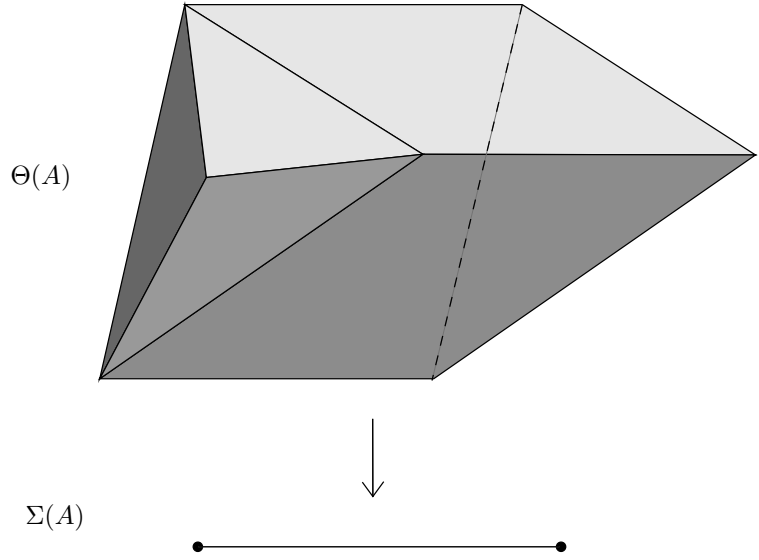
\begin{figure}
\begin{picture}(0,0)%
\includegraphics{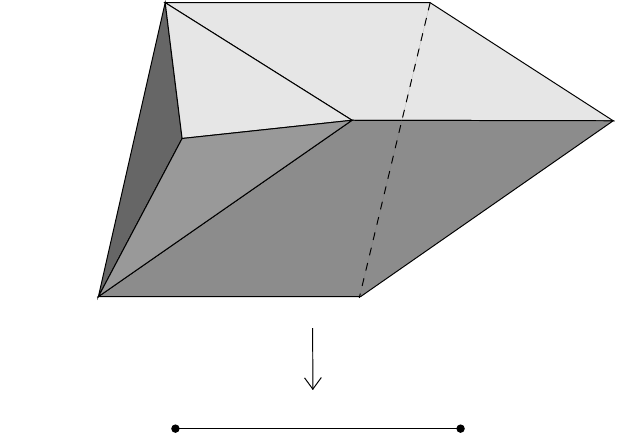}%
\end{picture}%
\setlength{\unitlength}{4144sp}%
\begin{picture}(4685,3293)(458,-1767)
\put(763,-1597){\makebox(0,0)[lb]{\smash{$\psec{A}$}}}
\put(673,442){\makebox(0,0)[lb]{\smash{$\ptlaf{A}$}}}
\end{picture}%
\caption{\label{fig:circuitsecond} The Lafforgue polytope relative to the secondary polytope for $A = \{(0,0), (1,0), (0,1),(-1,-1)\} $. }
\end{figure}

\begin{eg} \label{eg:circuit2}
	The Lafforgue polytope $\ptlaf{A}$ for Example \ref{eg:circuit1} illustrates the geometry seen in general. Since $\psec{A}$ is an interval and $\Delta^A_1$ is a tetrahedron in $\R^4$ parallel to $\psec{A}$, we can place their Minkowski sum in a three dimensional hyperplane. This is illustrated in Figure \ref{fig:circuitsecond}. Note that the facets parallel to $\psec{A}$ correspond to the horizontal boundary components $\dul{\ptlaf{A}}^h$ of $\ptlaf{A}$ and are in natural bijection with the facets of $Q$. Meanwhile, the vertical facets in $\dul{\ptlaf{A}}^v$ lie over the boundary of $\psec{A}$. Each of them correspond to one of the two triangulations $T_\pm$ along with a choice of subdividing polytope in $T_\pm$ (which, in the case of $T_+$, must be all of $Q$). These subdividing polytopes $Q_i$ determine the pointing sets $A_p = A \cap Q_i$. 
	
	Using equation \eqref{eq:verticeslaf}, we have explicit coordinates for the vertices of $\ptlaf{A}$. Recall that equations \eqref{eq:exampleverts} gave formulas for the vertices $\varphi_{T_\pm}$ of $\psec{A}$ corresponding to the triangulations $T_\pm$. Building off of this, the four vertices of $\ptlaf{A}$ on the left in Figure \ref{fig:circuitsecond}  are $\{\varphi_{(T_-, a)} : a \in A\}$ while the three vertices on the right are $\{\varphi_{(T_+, a)}: a \in A - \{(0,0)\}\}$. In general, it is a consequence of Lemma \ref{lem:lafhyppart} that the vertices of $\ptlaf{A}$ are $\{\varphi_{(T, a)} : (T,a) \text{ a pointed triangulation of }(Q,A)\}$.
\end{eg}

In case the cokernel $K_{{\mathcal{A}}}$ of $\beta_{{\mathcal{A}}}$ is non-zero, we will need to consider a more refined version of a primitive supporting hyperplane. Recall from properties~\ref{lem:regsub:1} and \ref{lem:regsub:2} that if $\eta_S \in (\R^{{\mathcal{A}}})$ defines the subdivision $S = \{(Q_i, A_i): i \in I\}$, then its restriction to each $Q_i$ equals that of an affine function $\varsigma_i \in (\Lambda_\R \oplus \R)^\vee$. We will say that $\eta_S$ is a $\Lambda$-defining function for $S$ if 
\begin{align} \label{eq:lambdadef} \varsigma_i \in (\Lambda \oplus \Z)^\vee & \text{ for every }i \in I.
\end{align}
It is clear that the set of $\Lambda$-defining functions forms a semigroup in $\Z^{\mathcal{A}}$ and, if $\eta_S$ is a primitive element of this semigroup, we call $\eta_S$ a primitive $\Lambda$-defining function for $S$. Generally, a $\Lambda$-primitive function for a given subdivision is not unique. However, for a coarse pointed subdivision, $\{(S, A_p)\}$, Lemma \ref{lem:lafnormal:2} implies that there is a one dimensional ray $\R_{\geq 0} \cdot \eta_{(S,A_p)}$ in $(\R^{\mathcal{A}})^\vee$ of defining functions for $S$ which vanish on $A_p$. As $\eta_{(S,A_p)} \in (\Z^A)^\vee$, there is a positive integer multiple of it that is the unique primitive $\Lambda$-defining function in this ray. We write $\du{\eta}_{(S, A_p)}$ for this defining function.

\begin{eg}
Let $A = \{-2, 0,2\} \subset \Z = \Lambda$ so that $K_A \approx \Z / (2)$. Then $e_{-2}^\vee \in (\Z^A)^\vee$  defines the subdivision $S = \{([-2,0], \{-2,0\} ) , ([0,2], \{0,2\}) \}$. While it is primitive, it is not a primitive $\Lambda$-defining function for $S$. Rather, the multiple $2e_{-2}^\vee$ is and gives the unique function $\du{\eta}_{(S, \{0,2\})}$.
\end{eg}

We use the definition of primitive $\Lambda$-defining functions and the constants $c_b$ occurring in Lemma \ref{lem:lafnormal:1} to define the homomorphism $\tilde{\beta}_{\dul{\plaf{A}}} : \Z^{\dul{\plaf{A}}} \to (\Z^A)^\vee$ via 
\begin{align} \label{eq:deftildbeta} \tilde{\beta}_{\dul{\plaf{A}}} (e_b) & = \begin{cases} \varrho_A & \text{ if } b = \varrho_A, \\ \du{\eta}_{(S,A_p)} &  \text{ if } b = \eta_{(S,A_p)} \in \dul{\plaf{A}}^v , \\ c_b b & \text{ if } b \in \dul{\plaf{A}}^h . \end{cases} \end{align} 
Define the stacky fan \gls{sflaf}
\begin{align}  \label{eq:lafsfan} \widetilde{\sfan}_\plaf{A} = \left( \Z^{\dul{\plaf{A}}} , (\Z^A)^\vee , \tilde{\beta}_{\dul{\plaf{A}}} , \fan_{\plaf{A}} \right) \end{align}
where the lattices and fan are equal to those for the stacky fan of $\plaf{A}$ as in Definition \ref{defn:polystack}, but $\tilde{\beta}_{\dul{\plaf{A}}}$ differs from the prescribed homomorphism $\beta_{{\plaf{A}}}$. In particular, even in the case where $K_{{\mathcal{A}}} = 0$ (implying every defining function is $\Lambda$-defining), there will generally be elements $b \in \dul{\plaf{A}}$ for which the scaling constants $c_b \in \N$ are not equal to $1$. We will glean a bit more detailed information about $\tilde{\beta}_{\dul{\plaf{A}}}$ later in this section, but first we consider an assortment of structures on the stack associated to $\widetilde{\sfan}_\plaf{A}$.

By defining the polyhedron $\plaf{A}$ as a Minkowski sum $\psec{A} + \Delta^A_{\geq 1}$, we ensure that its normal fan not only refines $\fanlt{A}$, but also the normal fan $\mathcal{F}_\Delta$ of $\Delta^A_{\geq 1}$. The toric variety associated to this fan is the total space of the tautological bundle $\mco (-1)$ over $\mathbb{P}^{|A| - 1}$. Indeed, $\mathcal{F}_\Delta$ is a refinement of the cone $\textnormal{Lin}_{\R_{\geq 0}} \{e_a^\vee : a \in A\}$ obtained by adding the ray $\textnormal{Lin}_{\R_{\geq 0}} (\varrho_A )$ and subdividing. This is the toric construction for blowing up the origin in $\mathbb{C}^{|A|}$. One can check to see that $e_{a}^\vee$ is the primitive corresponding to a pointed subdivision $(S, A_p)$ where $S$ is a coarse subdivision. Thus $e_a^\vee \in \dul{\plaf{A}}^v$ and there is a morphism of stacks 
\begin{align} \label{eq:morphtoproj} 
\tilde{G} : \mcx_{\widetilde{\sfan}_\plaf{A}} \to \mco_{\p^{|A| - 1}} (-1)
\end{align}
which, after projection, gives a morphism $G : \mcx_{\widetilde{\sfan}_\plaf{A}}  \to \p^{|A| - 1}$.
In fact, $\mcx_{\widetilde{\sfan}_\plaf{A}}$ is itself a line bundle over the divisor $D_{\varrho_A}$ defined by $\varrho_A$ and $\tilde{G}$ is a map of line bundles over proper stacks. We will not use this fact, but we will consider the restriction $G : D_{\varrho_A} \to \mathbb{P}^{|A| - 1}$.
%
%
%
\begin{defn} \mbox{ } \label{defn:moduli}
\begin{itemize} 
\item[(i)]\gls{tlfst} The total Lafforgue stack of $A$ is $\tlaf{A} := \mcx_{\widetilde{\sfan}_\plaf{A}}$.  
\item[(ii)] \gls{uuniv}The universal line bundle $\univ{A}$ on $\tlaf{A}$ is $G^* (\mco_{\p^{|A| - 1}} (1))$.
\item[(iii)]  \gls{unsec}The universal section $s_A \in H^0 (\tlaf{A} , \univ{A})$ is the pullback  $G^*(\sum_{a \in A} Z_a)$.
\item[(iv)] \gls{tothyp}The total universal hypersurface  is the zero locus $\tilde{\mcy}_A$ of $s_A$.
\item[(v)] \gls{lafst}The Lafforgue stack of $A$ is $\laf{A} := D_{\varrho_A}$.
\item[(vi)] \gls{unhyp}The universal hypersurface $\hyp{A} \subset \laf{A}$ is $\tilde{\mcy}_A \cap \laf{A}$.
\end{itemize} 
\end{defn}

The toric stack $\laf{A}$ can also be described by taking the star of $\varrho_A$ in $\Sigma_{\plaf{A}}$ which yields a fan $\Sigma_{\ptlaf{A}}$ in $\R^B$ combinatorially equivalent to the Lafforgue fan where $B = \dul{\plaf{A}} - \{\varrho_A\}$. The map $\tilde{\beta}_{\ptlaf{A}} : \Z^B \to (\Z^A)^\vee / \textnormal{Lin}_\Z  ( \varrho_A )$ obtained by restricting $\tilde{\beta}_{\dul{\plaf{A}}}$ to $\Z^B$ and then quotienting by $\textnormal{Lin}_\Z (\varrho_A )$ defines the stacky fan  \begin{align} \label{eq:ptlafstfan}  \sfan_{\ptlaf{A}} := \left( \Z^B , (\Z^A)^\vee / \textnormal{Lin}_\Z (\varrho_A), \tilde{\beta}_{\ptlaf{A}} , \Sigma_{\ptlaf{A}} \right) .\end{align} This gives an alternative description of $\laf{A}$. The advantage of detailing the total Lafforgue stack is to give a natural context in which to define the universal line bundle and the universal section. Let us describe this stack for our two examples.

\begin{eg} For $A = \{(1,0), (-1, 0), (0,1)\} \subset \Lambda$ as in Example \ref{eg:simplex}, we have seen that $\psec{A}$ is a point and thus $\plaf{A}$ is a translation of $\Delta^A_{\geq 1}$. One can check that
\begin{align*}
\dul{\plaf{A}} = \{ e_{(1,0)}^\vee, e_{(-1,0)}^\vee, e_{(0,1)}^\vee,  e_{(1,0)}^\vee + e_{(-1,0)}^\vee + e_{(0,1)}^\vee = \varrho_A  \} \subset (\Z^{{\mathcal{A}}})^\vee.
\end{align*} 
Here, in indexing the basis, we identify elements of $A$ with their counterparts in ${\mathcal{A}}$. Were we to have taken the usual toric stack defined by the normal fan of this polyhedron, we would obtain the total space $\mco_{\p^{|A| - 1}} (-1)$. However, having altered $\beta_{\dul{\plaf{A}}}$ to $\tilde{\beta}_{\dul{\plaf{A}}}$, we have to check to see if this has modified the stacky fan in this case. Since there are no coarse subdivisions of $(Q,A)$, we consider only $b \in \dul{\plaf{A}}^h$. 

From Example \ref{eg:simplex} we have that $\du{Q} = \{(1,-1), (-1,-1), (0,1)\} \subset \Lambda^\vee$ and the associated primitive normal rays to  $\textnormal{Lin}_{\R_{\geq 0}} ({\mathcal{A}})$ are $\{(1,-1,1), (-1,-1,1), (0,1,0) \} \subset (\Lambda \oplus \Z)^\vee$. The tautological map $\beta_{{\mathcal{A}}} : \Z^{{\mathcal{A}}} \to \Lambda \oplus \Z$ sends $e_{(a,b)}$ to $(a, b, 1)$ and one can compute 
\begin{align*}
\beta_{{\mathcal{A}}}^\vee (1,-1,1) & = 2 e_{(1,0)}^\vee, \\
\beta_{{\mathcal{A}}}^\vee (-1,-1,1) & = 2 e_{(-1, 0)}^\vee, \\
\beta_{{\mathcal{A}}}^\vee (0,-1,0) & =  e_{(0,1)}^\vee.
\end{align*}
By equation \eqref{eq:deftildbeta}, we see that $\tilde{\beta}_{\dul{\plaf{A}}}$ takes $e_{b}$ to $b$, for $b \in \{\varrho_A , e_{(0,-1)}^\vee \}$ and $e_{b}$ to $2b$ for $b \in \{e_{(1,0)}^\vee, e_{(1,0)}^\vee\}$. The cokernel of this map is $\Z / 2\Z$ and one observes that the stacky fan $\sfan_{\ptlaf{A}}$ yields the stack 
\begin{align*}
\laf{A} \cong \left[ \p (1,1,2) / \left( \Z / 2 \Z \right) \right].
\end{align*}
To explain the appearance of the group $\Z / 2 \Z$, we note that there is a $\Z / 2\Z$ subgroup of $\mathbb{G}_{Q}$ which fixes $\linsys{A}$ and therefore is an automorphism of any hypersurface defined by a section in $\linsys{A}$. 
\end{eg}

\begin{eg} \label{eg:circuit3}  Let us consider the Lafforgue stack for Example~\ref{eg:circuit1}, where $A = \{(0,0), (1,0), (0,1),(-1,-1)\} \subset \Z^2 = \Lambda$. We computed both $\du{Q}$ and $n_b$ in Example~\ref{eg:circuit1}, and putting these together gives $\{(2,-1, 1), (-1, 2, 1), (-1,-1, 1)\}$ as the set of supporting hyperplane primitives to $\textnormal{Lin}_{\R_{\geq 0}} ({\mathcal{A}})$. Now, applying $\beta_{{\mathcal{A}}}^\vee$ gives 
	\begin{align*} \beta_{{\mathcal{A}}}^\vee (2,-1, 1) & = e_{(0,0)}^\vee + 3 e_{(1,0)}^\vee , \\
	\beta_{{\mathcal{A}}}^\vee (-1,2, 1) & =
	e_{(0,0)}^\vee + 3 e_{(0,1)}^\vee , \\
	\beta_{{\mathcal{A}}}^\vee (-1,-1, 1) & = e_{(0,0)}^\vee + 3 e_{(-1,-1)}^\vee. \end{align*} As each of these is primitive, we have that the constants $c_b = 1$ for each $b \in \dul{\plaf{A}}^h$. Turning to the vertical facets, we note that $\dul{\plaf{A}}^\vee = \{e_{(0,0)}^\vee , e_{(1,0)}^\vee, e_{(0,1)}^\vee, e_{(-1,-1)}^\vee \}$. This follows from the fact that $\Delta^A_{\geq 1}$ is a Minkowski summand of $\plaf{A}$ and, from Example~\ref{eg:circuit2}, there are only four remaining facets. One checks that 
	\begin{align}
	\label{eq:circuit3} \begin{split} e_{(0,0)}^\vee & = \eta_{(T_+, A - \{(0,0)\})}, \\
	e_{(1,0)}^\vee & = \eta_{(T_-, A - \{(1,0)\})}, \\ e_{(0,1)}^\vee & = \eta_{(T_-, A - \{(0,1)\})}, \\ e_{(-1,-1)}^\vee & = \eta_{(T_-, A - \{(-1,-1)\})}. 
	\end{split}
	\end{align} Furthermore, each of these are $\Lambda$-defining functions for the respective triangulations. These facts imply that $\tilde{\beta}_{\dul{\plaf{A}}}  = \beta_{\dul{\plaf{A}}}$ in this case. One can apply Proposition~\ref{prop:circuitlaf} to obtain a description of this Lafforgue fan and stack. In particular, $\laf{A}$ is shown to be a weighted blowup of $\p^3$ over three lines. 
\end{eg}

To obtain the last construction of this section, we first make a modification of the stacky fan defining the toric stack $\mcx_{\psecv{A}{v}}^r$ associated to the secondary polytope. We note that there are alternatives to this approach if $K_{\mathcal{A}} \ne 0$. Given a coarse subdivision $S = \{(Q_i, A_i): i \in I\}$ of $(Q,A)$, write $b_S \in \dul{\psecv{A}{v}}$ for the primitive hyperplane supporting the facet corresponding to $S$. Let $\Gamma_S \subseteq (\Z^{{\mathcal{A}}})^\vee$ be the $\Z$-linear span of all $\Lambda$-defining functions for $S$. As $S$ is a coarse subdivision, the image $\alpha_{{\mathcal{A}}}^\vee (\Gamma_S)$ is contained in $\textnormal{Lin}_\Z (b_S ) \approx \Z$ and we  choose a primitive $c_S \in \Gamma_S$ such that 
\begin{align} \label{eq:defrs1} \alpha_{{\mathcal{A}}}^\vee (c_S ) & = r_S b_S \end{align} for $r_S \in \N$ and $\alpha_{{\mathcal{A}}}^\vee (\Gamma_S) = \textnormal{Lin}_\Z (r_S b_S)$. This uniquely defines $r_S$ and we modify the homomorphism $\beta_{\dul{\psecv{A}{v}}}$ in equation \eqref{eq:secstkfan} for the stacky fan for $\mcx_{\psecv{A}{v}}^r$ by defining $\tilde{\beta}_{\dul{\psecv{A}{v}}} : \Z^{\dul{\psecv{A}{v}}} \to L_A^\vee$ via
\begin{align} \label{eq:defrs2}
\gls{psvsf} \tilde{\beta}_{\dul{\psecv{A}{v}}} (e_{b_S}) & = r_S b_S .
\end{align}
Keeping the rest of the data in equation \eqref{eq:secstkfan} the same, we take
\begin{align} \label{eq:tildsecstkfan}  \widetilde{\sfan}_{\psecv{A}{v}} = \left( \Z^{\dul{\psecv{A}{v}}} , L_A^\vee , \tilde{\beta}_{\dul{\psecv{A}{v}}}, \fan_{\mcb} \right)  \end{align} 
to be the modified stacky fan. 

Define $p_1 : \Z^{\dul{\plaf{A}}} \to \Z^{\dul{\psecv{A}{v}}}$ to be the homomorphism
\begin{align} \label{eq:defp1}
p_1 (e_b) & = \begin{cases}
 e_{b_S} & \text{ if } b = \eta_{(S,A_p )} \in \dul{\plaf{A}}^v , \\
0 & \text{ otherwise}. 
\end{cases}
\end{align} 
\begin{lem} \label{lem:stackfanmap}
	The diagram 
	\begin{equation} \label{diag:secondef}
	\begin{CD}
	\Z^{\dul{\plaf{A}}} @>{\tilde{\beta}_{\dul{\plaf{A}}}}>> (\Z^A)^\vee \\
	@V{p_1}VV  @V{\alpha^\vee_A}VV \\
	\Z^{\dul{\psecv{A}{v}}} @>{\tilde{\beta}_{\dul{\psecv{A}{v}}}}>> L_A^\vee .
	\end{CD}
	\end{equation}
	commutes and defines a map of stacky fans
	\begin{align}
\tilde{p} = (p_1, \alpha^\vee_{{\mathcal{A}}} ): \widetilde{\sfan}_{\plaf{A}} \to \widetilde{\sfan}_{\psecv{A}{v}}.
	\end{align}
\end{lem}
\begin{proof}
To prove that diagram \eqref{diag:secondef} commutes, one must show that for any coarse subdivision $S = \{(Q_i, A_i): i \in I \}$ and any $i \in I$,  $\alpha_A^\vee (\du{\eta}_{(S,A_i)} ) = r_S b_S$. Equivalently, one checks that the $\alpha_A^\vee$ image of $\du{\eta}_{(S,A_i)}$ generates $\alpha_A^\vee (\Gamma_S)$. To see this, suppose $c_S \in \Gamma_S$ maps to such a generator. Then since $c_S$ is a $\Lambda$-defining function for $S$, restricting $c_S$ to $A_i$ one obtains the $\Lambda$-affine function $\varsigma_i \in (\Lambda \oplus \Z)^\vee$. Thus $c^\prime_S := c_S - \beta^\vee_{{\mathcal{A}}} (\varsigma_i) \in \Gamma_S$ and, since $\text{im} (\beta^\vee_{{\mathcal{A}}}) \subset \ker (\alpha^\vee_{{\mathcal{A}}})$, $\alpha^\vee_{{\mathcal{A}}} (c^\prime_S )$ also generates the image of $\Gamma_S$. But since $c^\prime_S$ vanishes on $A_i$ and is a $\Lambda$-defining function for $S$, it is in $\lin_{\N} (\du{\eta}_{(S,A_i)})$ implying it must equal $\du{\eta}_{(S,A_i)}$. This verifies the commutativity of diagram \eqref{diag:secondef}.

The assertion that $\tilde{p}$ induces a map of stacky fans then follows from the fact that $\fanlt{A}$ is a refinement of $\fans{A}$.
\end{proof}
Quotienting by $\varrho_A$ factors $\tilde{p}$ to give a morphism $p : \sfan_{\ptlaf{A}} \to \widetilde{\sfan}_{\psecv{A}{v}}$. Moreover, an application of \cite[Theorem~IV.6.7]{Ewald} shows that this is a toric fibration meaning that it is a flat, surjective morphism of normal toric stacks. Thus we can apply Definition~\ref{defn:colimitstack} of a colimit stack and expect the universal property in Proposition~\ref{prop:colup} to hold.

\begin{defn} \label{defn:secstack}\gls{secstack} The secondary stack is $\secon{A} := \mcx_p^\to$ and the map $p^\to$ will be written as $\pi: \laf{A} \to \secon{A}$. Given $q \in \secon{A}$, write\gls{fibtot} $\fib{A}{q}$ for the fiber $\pi^{-1} (q) \cap \hyp{A}$.  Using the coefficients of the $A$-determinant $E_A$, write\gls{Adetsect} $E_A^s \in \mco_{\psec{A}} (1)$ for the section and\gls{Adethyp} $\mce_A \subset \secon{A}$ for its zero locus.
\end{defn}

Observe that the map $p$ in this definition can be replaced with $\tilde{p}$ to give an isomorphic stack as the two associated diagrams yield the same pushout. 

We conclude this section by describing the stacky fan for $\secon{A}$. To do this, we recall the notation for the hyperext group $\Lambda_{{\mathcal{A}}^\vee} = \mathbb{R} \Hom^* (\text{cone} (\beta_{{\mathcal{A}}}), \Z)$. Here, the set $B = {\mathcal{A}}$ and the long exact sequence \eqref{eq:alphastar} is   
\begin{align} \label{eq:extAex}
0 \to (\Lambda \oplus \Z)^\vee \stackrel{\beta_{{\mathcal{A}}}^\vee}{\longrightarrow} (\Z^{{\mathcal{A}}})^\vee \stackrel{\alpha_{{\mathcal{A}}}^\star}{\longrightarrow} \Lambda_{{{\mathcal{A}}}^\vee} \to 0 .
\end{align}
To describe a stacky fan for $\secon{A}$ in complete generality, we will require a finite extension of $\Lambda_{{\mathcal{A}}^\vee}$. We will say that $\eta \in (\Lambda \oplus \Z)^\vee$ defines a wall in $A$ if it is constant on a subset ${\mathcal{A}}^\prime \subset {\mathcal{A}}$ which spans a codimension $1$ subspace of $\Lambda_\R \oplus \R$. Note that this definition implies that a constant affine function $(0,n)$ on $\Lambda$, defines a wall.
\begin{defn} \label{defn:walllattice}\gls{walat}\gls{xilat}
	The wall lattice $(\Lambda \oplus \Z)_{wall}^\vee$ of $A \subset \Lambda$ is the sublattice of $(\Lambda \oplus \Z)^\vee$ generated by elements that define a wall in $A$. If $(\Lambda \oplus \Z)_{wall}^\vee = (\Lambda \oplus \Z)^\vee$, we say that $A$ is wall complete. We write $\Xi_{{\mathcal{A}}}$ for the cokernel of $\beta_{{\mathcal{A}}}^\vee$ restricted to $(\Lambda \oplus \Z)_{wall}^\vee$. 
\end{defn}
In most examples that we consider, $A$ will be wall complete implying that $\Lambda_{{\mathcal{A}}^\vee} = \Xi_A$. In particular, if $A$ contains a standard simplex then this equality will occur. A more general criterion is given in the following lemma.

\begin{lem} \label{lem:KtoXi}
If $K_{{\mathcal{A}}} = 0$ then $A$ is wall complete. 
\end{lem}
\begin{proof}
	Given any simplex $B = \{b_0, \ldots, b_d\} \subset A$, the set $\{(b_i , 1) : 0 \leq i \leq d\}$ forms a basis for $\Lambda_\Q \oplus \Q$. Take $b_{i, B}^\vee$ to be the dual basis in $(\Lambda_\Z \oplus \Z)^\vee$ and observe that $\textnormal{Lin}_\Z \{b_{i,B}^\vee : 0 \leq i \leq d, B \text{ a simplex in }A \} = (\Lambda \oplus \Z)^\vee_{wall}$. 
	
	Now, choosing a basis $\{e_0, \ldots, e_{d}\}$ for $\Lambda \oplus \Z$, consider the isomorphism $\phi : \wedge^{d} (\Lambda \oplus \Z) \to (\Lambda \oplus \Z)^\vee$ given by $\phi (v_0 \wedge \cdots \wedge v_{d - 1}) (u) = \left< e_0^\vee \wedge \cdots \wedge e_d^\vee , v_0 \wedge \cdots \wedge v_{d - 1} \wedge u\right>$. Observe that, for any simplex $B= \{b_0, \ldots, b_d\}$ there are constants $r_i \in \Z$ for which 
	\begin{align*} r_i b_{i, B}^\vee = \phi (b_{0}\wedge \cdots b_{i - 1} \wedge b_{i + 1} \wedge \cdots \wedge b_d ) . \end{align*} 
	Since $K_{{\mathcal{A}}} = 0$, it follows that ${\mathcal{A}}$ spans $\Lambda \oplus \Z$. Thus $\left\{ \bar{a}_0 \wedge \cdots \wedge \bar{a}_{d - 1} : \bar{a}_i = ( a_i, 1) \in {\mathcal{A}} \right\} $ spans $\wedge^d (\Lambda \oplus \Z)$ which implies its image $\{r_i b_{i, B}\}$ under $\phi$ spans $(\Lambda \oplus \Z)^\vee$ yielding $(\Lambda \oplus \Z)^\vee = (\Lambda \oplus \Z)^\vee_{wall}$. 
\end{proof}
Using $\Xi_A$, we now describe a stacky fan for $\secon{A}$. 

\begin{lem} \label{lem:stfansecon}\gls{secstfan} There is a map $\tilde{\beta}_{\dul{\psec{A}}}$ for which 
	\begin{align}
	\widetilde{\sfan}_{\psec{A}} = \left( \Z^{\dul{\psec{A}}}, \Xi_{{\mathcal{A}}}, \tilde{\beta}_{\dul{\psec{A}}}, \fan_{\mcb}  \right)
	\end{align}
is a stacky fan for $\secon{A}$. If $K_{{\mathcal{A}}} = 0$, then $\tilde{\beta}_{\dul{\psec{A}}} = \tilde{\beta}_{\dul{\psecv{A}{v}}}$.
\end{lem}

\begin{proof}
By Definition \ref{defn:colimitstack} of the colimit stack, it suffices to prove that $\Xi_{{\mathcal{A}}}$ is isomorphic to the pushout of the diagram:
\begin{align} \label{eq:colimdiag}
\begin{CD}
 \Z^{\dul{\plaf{A}}} @>{\tilde{\beta}_{\dul{\plaf{A}}}}>> 
(\Z^{{\mathcal{A}}})^\vee \\
 @V{p_1}VV  @.\\ 
\Z^{\dul{\psecv{A}{v}}} @. 
\end{CD}
\end{align}
Since $p_1$ is onto, the pushout is isomorphic to the cokernel of $\tilde{\beta}_{\dul{\plaf{A}}}$ restricted to $\ker (p_1 )$.  Thus, using the definition of $\Xi_{{\mathcal{A}}}$, it suffices to show that 
\begin{align}  \tilde{\beta}_{\dul{\plaf{A}}} ( \ker (p_1)) & = \beta_{{\mathcal{A}}}^\vee ( (\Lambda \oplus \Z)^\vee_{wall} ). \end{align}
We first prove that $\tilde{\beta}_{\dul{\plaf{A}}} ( \ker (p_1)) \subseteq \beta_{{\mathcal{A}}}^\vee ( (\Lambda \oplus \Z)^\vee_{wall} )$. For any coarse subdivision $S = \{(Q_i, A_i): i \in I\}$ of $(Q,A)$ we will say that $A_i$ and $A_j$ are adjacent if $Q_i \cap Q_j$ is a facet of both $Q_i$ and $Q_j$. Define the set of differences
\begin{align*} B_S := \{e_{\eta_{(S,A_i)}} - e_{\eta_{(S,A_j)}} : i, j \in I, A_i \text{ adjacent to } A_j \} \subset \Z^{\dul{\plaf{A}}}. \end{align*}
Then the lattice $\ker (p_1)$ is easily seen to be generated by 
\begin{align*} \{e_{\varrho_A}\} \cup \left\{e_b : b \in \dul{\plaf{A}}^h \right\} \cup \left( \cup_{S} B_S \right),
\end{align*}
where the last union is over all coarse subdivisions of $(Q,A)$. One computes that $\tilde{\beta}_{\dul{\plaf{A}}} (e_{\varrho_A}) = \varrho_A = \beta_{{\mathcal{A}}}^\vee(0,1)$. Also, by the definition of $\tilde{\beta}_{\dul{\plaf{A}}}$ in equation \eqref{eq:deftildbeta} and Lemma \ref{lem:lafnormal:1}, if $b \in \dul{\plaf{A}}^h$ corresponds to the facet $F$ of $Q$ with $A_p = A \cap F$, then $\tilde{\beta}_{\dul{\plaf{A}}} (e_b ) = \beta_{{\mathcal{A}}}^\vee (b_{A_p})$. By definition, $b_{A_p}$ is constant on $A_p$ which implies it is in $ (\Lambda \oplus \Z)^\vee_{wall}$. Finally, by  Lemma \ref{lem:stackfanmap}, if $e_{\eta_{(S,A_i)}} - e_{\eta_{(S,A_j)}} \in B_S$ then 
\begin{align*}
\alpha_{{\mathcal{A}}}^\vee \left(\tilde{\beta}_{\dul{\plaf{A}}} \, ( e_{\eta_{(S,A_i)}} - e_{\eta_{(S,A_j)}}) \right) = \tilde{\beta}_{\dul{\psecv{A}{v}}}\, (e_{b_S} - e_{b_S}) = 0 .
\end{align*}
This implies that $\tilde{\beta}_{\dul{\plaf{A}}} \, ( e_{\eta_{(S,A_i)}} - e_{\eta_{(S,A_j)}}) = \du{\eta}_{(S,A_i)} - \du{\eta}_{(S,A_j)}$ restricts to an affine function (as the kernel of $\alpha^\vee_{{\mathcal{A}}}$ is $\left(\text{im} (\beta_{{\mathcal{A}}})\right)^\vee$). Since $\du{\eta}_{(S,A_i)}, \du{\eta}_{(S,A_j)} \in \Gamma_S$ are $\Lambda$-defining functions for $S$, their restriction to $A_i$ is in $(\Lambda \oplus \Z)^\vee$. This implies their difference is an affine function on $\Lambda$ and thus equals $\beta_{{\mathcal{A}}}^\vee ( \lambda )$ for some $\lambda \in (\Lambda \oplus \Z)^\vee$. Furthermore, since both $\du{\eta}_{(S,A_i)}$ and $\du{\eta}_{(S,A_j)}$ are zero on $A_i \cap A_j$, $\lambda$ is as well, and as $A_i$ is adjacent to $A_j$, $\lambda$ defines a wall in $A$. Therefore, we have shown $\tilde{\beta}_{\dul{\plaf{A}}} \, ( e_{\eta_{(S,A_i)}} - e_{\eta_{(S,A_j)}}) \in \beta_{{\mathcal{A}}}^\vee ( (\Lambda \oplus \Z)^\vee_{wall} )$.

We now turn to the inclusion $ \beta_{{\mathcal{A}}}^\vee ( (\Lambda \oplus \Z)^\vee_{wall} ) \subseteq \tilde{\beta}_{\dul{\plaf{A}}} ( \ker (p_1))$. To verify this, suppose $\nu \in (\Lambda \oplus \Z)^\vee$ defines a wall in $A$ and observe that restricting $\nu$ to ${{\mathcal{A}}}$ will give us one of three possible scenarios. First, $\nu |_{{\mathcal{A}}}$ could be constant, in which case $\nu$ is as well and $\beta^\vee_{{\mathcal{A}}} (\nu ) =  \tilde{\beta}_{\dul{\plaf{A}}} (n \varrho_A)$ for some $n \in \Z$. Second, it could be the case that $\nu$ is constant on a facet $F$ of $Q$ and non-constant on $Q$. In this case, it is a multiple of the primitive $b_{A_p}$ where $A_p = F \cap A$. By Lemma \ref{lem:lafnormal:1} and the definition of $\tilde{\beta}_{\dul{\plaf{A}}}$, $\beta_{{\mathcal{A}}}^\vee (\nu ) \in \tilde{\beta}_{\dul{\plaf{A}}} ( \ker (p_1))$. Third,  $\nu$ could be constant on subset ${{\mathcal{A}}}^\prime \subset {{\mathcal{A}}}$ whose affine span in $\Lambda_\R$ divides $Q$ into two subpolytopes. More precisely, letting $c = \nu (a, 1 )$ for some $(a, 1) \in {{\mathcal{A}}}^\prime$, take $A_- = \{(a,1) \in {{\mathcal{A}}} : \nu (a,1) \leq c \}$ and $A_+  = \{(a,1) \in {{\mathcal{A}}} : \nu (a) \geq c \}$. Then $A_-$ and $A_+$ affinely span $\Lambda_\R$ and one can define elements $\nu_\pm \in (\Z^A)^\vee$  as 
\begin{align*} \nu_+ & = \sum_{a \in A_+} \nu (a) e_a^\vee + \sum_{a \in A - A_+} c e_a^\vee - c \varrho_A
\end{align*}
and $\nu_- = \nu_+ - \beta_{{\mathcal{A}}}^\vee (\nu )$. We claim that $\nu_\pm$ lie in $C_\Z (S)$ for a coarse subdivision $S = \{(Q_i, A_i): i \in I\}$ of $(Q,A)$. As they differ by an element of $\text{im} (\beta_{{\mathcal{A}}}^\vee)$, they define the same subdivision. By definition, $S$ is a coarse subdivision if and only if the cone of defining functions $C_\R^\circ (S) \subset (\R^{{\mathcal{A}}})^\vee$ for $S$ is $(d + 2)$-dimensional (recall that $\dim (\Lambda_\R) = d$).  Let $L_{{{\mathcal{A}}}^\prime} \subset [(\Lambda_\R \oplus \R)^\vee]^2$ be the subspace of all pairs of affine functions $(b_1, b_2)$ for which $b_1|_{{{\mathcal{A}}}^\prime} = b_2 |_{{{\mathcal{A}}}^\prime}$ and observe that, since ${{\mathcal{A}}}^\prime$ spans a codimension $1$ subspace,  $\dim (L_{{{\mathcal{A}}}^\prime}) = d + 2$. Now, letting $Q_\pm = \convhull (A_\pm)$, it follows from the construction of $\nu_+$ that  $(Q_\pm , A_\pm ) \in S$. Thus there is a homomorphism $F : C_\R^\circ (S) \to L_{{{\mathcal{A}}}^\prime}$ defined by $F ( \eta ) = ( \varsigma_+ , \varsigma_-)$ where $\varsigma_\pm$ are the affine functions which restrict to $A_\pm$. One can check that the image of $F$ is $(d + 2)$-dimensional and, since $A = A_+ \cup A_-$, $F$ is also injective. This verifies the claim that $\nu_\pm$ define a coarse subdivision. 

Finally, by construction we have that $\nu_\pm |_{A_\pm} = 0$ so $\nu_\pm$ lie in the cones $C_\Z (S , A_\pm)$ of the Lafforgue fan and are multiples of $\eta_{(S, A_\pm)}$, respectively. Again, by construction, they are both $\Lambda$-defining functions for $S$, so they are also multiples of $\du{\eta}_{(S, A_\pm)}$. Indeed, it follows from Lemma \ref{lem:stackfanmap} that there exists a single constant $C$ such that $\nu_\pm = C \du{\eta}_{(S,A_\pm)}$ which implies that $\beta_{{\mathcal{A}}}^\vee (\nu) = \nu_+ - \nu_- = C (\du{\eta}_{(S,A_+)} - \du{\eta}_{(S,A_-)}) \in \tilde{\beta}_{\dul{\plaf{A}}} ( \ker (p_1))$. This completes the proof of the first statement. 

For the second statement, applying Lemma \ref{lem:KtoXi} gives us that $(\Lambda \oplus \Z)^\vee_{wall} = (\Lambda \oplus \Z)^\vee$. This in turn implies that $\Xi_{{\mathcal{A}}} = L_{{\mathcal{A}}}^\vee$ and that diagram \eqref{diag:secondef} is a colimit diagram.
\end{proof}
We conclude this subsection with a description of this stacky fan for our two main examples.
\begin{eg}
For Example \ref{eg:simplex}, we had that $A = \{(1,0),(-1,0),(0,1)\}$ was a simplex and observed that $\mcx_{\psec{A}}^r$ was equal to a point. The primitives hyperplane support functions for $Q$ are 
\begin{align*} \du{Q} = \{(0,1), (1,1), (-1,1)\} .\end{align*} 
By including $\Lambda^\vee$ into $(\Lambda \oplus \Z)^\vee$ and adding any constant affine of the form $(0,0,n)$, one observes that $(\Lambda \oplus \Z)^\vee_{wall} = (\Lambda \oplus \Z)^\vee$ in this case. One computes then that $\Xi_{{\mathcal{A}}} = \Lambda_{{\mathcal{A}}^\vee} \cong \left( \Z / 2 \Z \right)$. Applying Lemma~\ref{lem:stfansecon} and results from Example~\ref{eg:simplex2}, we have that the stacky fan for $\secon{A}$ is 
\begin{align*}
\widetilde{\sfan}_{\psec{A}} = \left( 0, \left(\Z / 2 \Z \right), 0 , \{0\} \right).
\end{align*}
This implies that 
\begin{align*}
\secon{A} \cong B \left(\Z / 2 \Z \right).
\end{align*}
More generally, for any set $A$ which consists solely of lattice vertices of a $d$-dimensional simplex in $\Lambda_\R$,  one can show that $\secon{A}$ is isomorphic to the classifying stack $ B \, \Xi_{{\mathcal{A}}} = [\text{pt} / \Xi_{{\mathcal{A}}} ]$ (note that, even in this set of examples, it is not always the case that $\Xi_{{\mathcal{A}}} = \Lambda_{{{\mathcal{A}}}^\vee}$ ). 
In the next subsection, we will interpret this as the moduli stack for hypersurfaces in $\mcx_Q$ defined by sections in $\linsys{A}$.
\end{eg}

\begin{eg}
We conclude this subsection by describing the stack $\secon{A}$ for $A = \{(0,0), (1,0), (0,1),(-1,-1)\}$. Since $A$ contains a simplex which affinely spans $\Lambda$, we have that $(\Lambda \oplus \Z)^\vee_{wall} = (\Lambda \oplus \Z)^\vee$. This implies $\Xi_{{\mathcal{A}}} = \Lambda_{{\mathcal{A}}^\vee}$ and as the fundamental sequence for ${\mathcal{A}}$ is equivalent to
\begin{align*}
0 \to \Z \stackrel{\alpha_{{\mathcal{A}}}}{\longrightarrow} \Z^{4} \stackrel{\beta_{{\mathcal{A}}}}{\longrightarrow} \Z^3  \to 0.
\end{align*}
we have $K_{{\mathcal{A}}} = 0$. Here we compute  $\alpha_{{\mathcal{A}}}(1) = (3,-1,-1,-1)$ as this yields the generating relation $3 (0,0,1) - (1,0,1) - (0,1,1) - (-1,-1,1) = 0$ of elements in ${\mathcal{A}}$.  Thus $\Lambda_{{\mathcal{A}}^\vee}$ is isomorphic to $L_{{\mathcal{A}}}^\vee \cong \Z$. In particular, the commutative diagram~\eqref{diag:secondef} in Lemma~\ref{lem:stackfanmap} is a pushout and the homomorphism $\tilde{\beta}_{\psecv{A}{v}}$ is equivalent to $\tilde{\beta}_{\psec{A}}$.

Now, $\dul{\psec{A}} = \{b_{T_-}, b_{T_+}\} \in L_{{\mathcal{A}}}^\vee \cong \Z$. In Example~\ref{eg:circuit3} we saw that $\tilde{\beta}_{\dul{\plaf{A}}} = \beta_{\dul{\plaf{A}}}$. Using the commutativity of diagram~\eqref{diag:secondef} and equations~\eqref{eq:circuit3}, this implies 
\begin{align*} \tilde{\beta}_{\psec{A}} (b_{T_-}) &  = \alpha_{{\mathcal{A}}}^\vee (\eta_{(T_-, A - \{(1,0)\})}) = \alpha_{{\mathcal{A}}}^\vee (0,1,0,0) = -1, \\  
\tilde{\beta}_{\psec{A}} (b_{T_+}) &  = \alpha_{{\mathcal{A}}}^\vee (\eta_{(T_+, A - \{(0,0)\})}) = \alpha_{{\mathcal{A}}}^\vee (1,0,0,0) = 3.
\end{align*} 
Thus the stacky fan is isomorphic to 
\begin{equation*}
\widehat{\sfan}_{\psec{A}} = \left( \Z^2, \Z, 3e_1^\vee - e_2^\vee, \Sigma \right),
\end{equation*}
where $\Sigma$ is the fan consisting of all proper faces of $\R_{\geq 0}^2$. Thus $X_\Sigma = \C^2 - \{0\}$ and the secondary stack is the weighted projective line
\begin{equation*}
\secon{A} \cong \p (3, 1).
\end{equation*}
\end{eg}


\subsection{\label{sec:moduli}Stable pair moduli}

In this section we relate the moduli space of hypersurfaces defined by full sections in $\linsys{A}$ introduced in Definition \ref{defn:sections} to a dense open subset in $\secon{A}$. While this will be done with the standing assumption that $A$ is wall complete, we note that if this is not the case, one can replace $\mcx_Q$ by an \'etale cover $\mcx_{Q}^\prime$ to obtain a similar interpretation of $\secon{A}$. We emphasize here that by moduli space, we mean hypersurfaces in a toric stack up to toric equivalence, not up to isomorphism. This toric moduli space, which we will denote by $\mcv_A$, will be proven to be an affine DM stack and is therefore much easier to control. We then show that the pullback $\hyp{A} \subset \laf{A}$ along the inclusion yields a universal hypersurface over $\mcv_A$. Finally, we prove that any toric degeneration $F_\eta : \mcx_\eta \to \C$ obtained by a $\Lambda$-defining function $\eta$ can be realized by pulling back $\laf{A}$ along a map $\rho_\eta : \C \to \secon{A}$ where $0$ is sent to the compactifying divisor $\secon{A} - \mcv_A$. Restricting this to the universal hypersurface gives meaning to the notion of $\secon{A}$ as a moduli for hypersurface degenerations. 

Our first goal is to describe the space of sections $\linsys{A}$ modulo toric isomorphisms. For every $a \in A$ there is an equivariant  divisor  \begin{align*} D_a = \sum_{b \in \du{Q}} \left( \left< b, a \right> + n_b \right)e_b^\vee \in (\Z^{\du{Q}})^\vee = \Divisor_{eq} (\mcx_Q ). \end{align*} 
The section vanishing on $D_a$ will be denoted $x_{D_a} \in \linsys{A}$. Recall from Definition \ref{defn:polystack} that the torus $\mathbb{G}_Q$ acting on $\mcx_Q$ is $\Lambda^\vee \otimes \C^*$. By fixing the set $A \subset \Lambda$, we may identify the maximal torus orbit $U \subset \mcx_Q$ as $\mathbb{G}_Q$ and trivialize $\mco_Q (1)$ over $U$ so that the section $x_{D_a}$ is identified with the monomial $a \in \Lambda \cong \Hom (\Lambda^\vee \otimes \C^* , \C^*)$. Using these identifications, the action of the torus $\mathbb{G}_{Q}$ on $\mcx_Q$ extends to one on $\linsys{A} \cong \C^A$ by tensoring the homomorphism $\alpha_A^\vee : \Lambda^\vee  \to (\Z^A)^\vee$ by $\C^*$. In other words, taking the dual of the evaluation map $\alpha_A$ and tensoring with $\C^*$ realizes $\mathbb{G}_Q$ inside of $(\mathbb{C}^*)^A$ which acts diagonally on $\linsys{A} = \C^A$. We also wish to quotient by the $\C^*$ scaling action on sections giving the group $\mathbb{G}_A \times \C^* \cong (\Lambda \oplus \Z)^\vee \otimes \C^*$. The action of this group on the space $\linsys{A} = \C^A$ can be realized by the tensoring $\alpha_{{\mathcal{A}}}^\vee : (\Lambda \oplus \Z )^\vee \to (\Z^{{\mathcal{A}}})^\vee$ with $\C^*$. By equation \eqref{eq:extAex}, this leads to the following definition.
\begin{defn} The $A$-linear system quotient stack is the toric stack $\mcx_\linsys{A}$ given by the stacky fan:
\begin{equation*}  \sfan_{\linsys{A}} = \left( (\Z^{{\mathcal{A}}})^\vee , \Lambda_{\gdual{{\mathcal{A}}}},  \alpha_{{\mathcal{A}}}^\star , \fan \right) \end{equation*} where $\fan$ is the fan with unique maximal cone $(\R_{\geq 0}^{{\mathcal{A}}})^\vee$.
\end{defn}
From the arguments preceding the definition, this gives the Artin stack $[\linsys{A} / (\Lambda \oplus \Z)^\vee \otimes \C^*]$ corresponding to sections in $\linsys{A}$ up to toric equivalence. There are many substacks of $\mcx_{\linsys{A}}$ that are  Deligne-Mumford (or DM), but our focus will be on the substack of hypersurfaces defined by full sections. Recall that ${\mathcal{A}}_v$ denotes the set of vertices of $\convhull ({\mathcal{A}})$ and ${{\mathcal{A}}}_{nv} = {\mathcal{A}} - {{\mathcal{A}}}_v$ denotes the remaining elements of ${\mathcal{A}}$. The substack of full sections is  obtained by taking the subfan $\Sigma^\prime$ of $\Sigma$ which has $(\R^{{\mathcal{A}}_{nv}}_{\geq 0} )^\vee$ as its maximal cone. The stacky fan 
\begin{align} \label{eq:fullsec}  \sfan^\prime := \left( (\Z^{{\mathcal{A}}})^\vee , \Lambda_{\gdual{{\mathcal{A}}}} , \alpha_{{\mathcal{A}}}^\star, \fan^\prime \right) \end{align}
is otherwise the same as for $\sfan_{\linsys{A}}$ and we denote its toric stack by\gls{mcvA} $\mcv_A$. We now verify the claim that this is a DM stack.
\begin{prop} \label{prop:etale}
The dense open substack $\mcv_A \subset \mcx_\linsys{A}$ is an affine DM stack.
\end{prop}
\begin{proof} To prove this, we will define an affine DM stack $\mathcal{X}$ with stacky fan 
\begin{align} \label{eq:doubleprime} \sfan^{\prime \prime} = \left( \Gamma , \Lambda_{{\mathcal{A}}^\vee} , \gamma , \Sigma^{\prime \prime} \right) \end{align} and an isomorphism $g : \mcx \to \mcv_A$ induced by a map of stacky fans $(g_1, g_2) :  \sfan^{\prime \prime} \to \sfan^\prime $. 

First we define $\sfan^{\prime \prime}$. Recall that $A$ affinely spans $\Lambda_\Q$ implying that $A_v$ does as well. In turn, this implies ${\mathcal{A}}_v$ linearly spans $(\Lambda_\Q \oplus \Q)$ and we choose $C \subseteq {\mathcal{A}}_v$ to be any $(d + 1)$ element subset which is a basis (recalling that $\rank (\Lambda) = d$). Let $\Gamma = \textnormal{Lin}_{\Z} \{e_a^\vee : a \in {\mathcal{A}} - C \} \subset \left( \Z^{{\mathcal{A}}} \right)^\vee$ and $\gamma : \Gamma \to \Lambda_{{\mathcal{A}}^\vee} $ be the restriction of $\alpha^\star_{{\mathcal{A}}}$ to $\Gamma$. To complete the definition of $\sfan^{\prime \prime}$, take $\Sigma^{\prime \prime}$ to again be the fan with unique maximal cone $(\R_{\geq 0}^{{\mathcal{A}}_{nv}})^\vee$.

To verify that $\sfan^{\prime \prime}$ is a stacky fan, we must show that $\gamma$ has finite cokernel. Recall that the exact sequence \eqref{eq:extAex} is 
\begin{align*}
0 \to (\Lambda \oplus \Z)^\vee \stackrel{\beta_{{\mathcal{A}}}^\vee}{\longrightarrow} (\Z^{{\mathcal{A}}})^\vee \stackrel{\alpha_{{\mathcal{A}}}^\star}{\longrightarrow} \Lambda_{{{\mathcal{A}}}^\vee} \to 0 .
\end{align*}
Here $\beta_{{\mathcal{A}}} : \Z^{{\mathcal{A}}} \to \Lambda \oplus \Z$ was the tautological map $\beta_{{\mathcal{A}}} (e_a) = a$. Now if $\mathbf{a} \in \Gamma \cap \text{im} (\beta_{{\mathcal{A}}}^\vee )$ then there exists $f \in (\Lambda \oplus \Z)^\vee$ such that 
\[\sum_{a \in {\mathcal{A}} - C} c_a e_a^\vee = \mathbf{a} = \beta_{{\mathcal{A}}}^\vee (f) = \sum_{a \in {\mathcal{A}}} f(a) e_a^\vee. \]
But then $f(a) = 0$ for all $a \in C$ and, as $C$ was chosen to be a basis for $\Lambda_\Q \oplus \Q$, $f = 0$. Since the sequence is exact, this implies $\ker (\gamma )$ is zero and, as the rank of $\Gamma$ equals that of $\Lambda_{{\mathcal{A}}^\vee}$, $\gamma$ must have finite cokernel and $\sfan^{\prime \prime}$ is a stacky fan. Furthermore, since $\gamma$ is injective, $\mathbb{H}_{\sfan^{\prime \prime}}$ is a finite group and we have that
\[\mcx_{\sfan^{\prime \prime}} = \left[ \left( (\C^*)^{|A_v| - d - 1}  \times \C^{|A_{nv}|} \right) / \, \mathbb{H}_{\sfan^{\prime \prime}} \right] \]
is an affine DM stack.

Letting $g_1 : \Gamma \to \left( \Z^{{\mathcal{A}}} \right)^\vee$ be the inclusion and $g_2 : \Lambda_{{\mathcal{A}}^\vee} \to \Lambda_{{\mathcal{A}}^\vee} $ the identity, we have the commutative diagram 
\begin{equation*}
\begin{CD}
	 \Gamma @>{\gamma}>> \Lambda_{{{\mathcal{A}}^\vee}} \\
	@V{g_1}VV @V{=}V{g_2}V \\
	(\Z^{{\mathcal{A}}})^\vee @>{\alpha^\star_{{\mathcal{A}}}}>> \Lambda_{{{\mathcal{A}}^\vee}} 
\end{CD}
\end{equation*}
and since $g_1$ takes $\Sigma^{\prime \prime}$ to $\Sigma^\prime$, $(g_1, g_2)$ is a map of stacky fans. 

Finally, $g_1 : | \Sigma^{\prime \prime} | \to |\Sigma^\prime|$ is an isomorphism on the support of the fans and this restricts to an isomorphism of monoids  $g_1 : | \Sigma^{\prime \prime} | \cap \Gamma \to |\Sigma^\prime| \cap (\Z^{{\mathcal{A}}})^\vee$. As $g_2$ is the identity, we have verified the hypothesis of \cite[Theorem~B.3]{toricstacks} to ensure that $g$ induces an isomorphism of toric stacks.
\end{proof}

By definition, the stack $\mcv_A$ is a quotient of the affine toric variety of full sections $X_{\Sigma^\prime} \cong (\C^*)^{{{\mathcal{A}}}_v} \times \C^{{{\mathcal{A}}}_{nv}}$ by $\mathbb{H}_{\sfan^{\prime}} = \left(\Lambda \oplus \Z \right)^\vee \otimes \C^* $. Observe that the group $\mathbb{H}_{\sfan^\prime}$ is naturally isomorphic to $\mathbb{G}_Q \times \C^*$ where $\mathbb{G}_Q$ is the torus acting on $\mcx_Q$ and the additional factor of $\C^*$ rescales the sections. This group also acts naturally on the total space $\mco_A (-1)$. Taking the dual action of $\mathbb{H}_{\sfan^\prime}$ on $X_{\Sigma^\prime}$ (i.e. $\lambda \cdot x = \lambda^{-1} x$), we obtain a diagonal action of $\mathbb{H}_{\sfan^\prime}$ on the product  $X_{\Sigma^\prime} \times \mco_A (-1)$ and define the toric stack $\mcu_A$ over $\mcv_A$ to be the quotient
\begin{equation*} \mcu_A := \left[(X_{\Sigma^\prime} \times \mco_A (-1) ) / \, \mathbb{H}_{\sfan^\prime} \right].  \end{equation*}
Write $E (1)$ for the line bundle on $\mco_A (-1)$ which is the pullback of $\mco_A (1)$ along the projection $\mco_A (-1) \to \mcx_Q$ and examine the tautological section $\tilde{s}$ of $\mco \, \boxtimes \, E (1)$ over $X_{\Sigma^\prime} \times \mco_A (-1)$. This is given by taking $(t, (q,v))$ with $t \in X_{\Sigma^\prime} \subset H^0 (\mcx_Q , \mco_A (1))$, $q \in \mcx_Q$ and $v  \in \mco_A (-1)$ lying over $q \in \mcx_Q$ to $t (q ) \in \mco \, \boxtimes \, E (1)$. As we have  $\mathbb{H}_{\sfan^\prime}$ acting with an inverse on $X_{\Sigma^\prime}$, $\tilde{s}$ is invariant under the diagonal action and defines a section of the line bundle $\mce(1)$ on the quotient $\mcu_A$. Its zero locus is the incidence variety  $[\tilde{s}^{-1} (0) / \mathbb{H}_{\sfan^\prime} ]$ which we denote by\gls{mcwA} $\mcw_A$. We consider the pair $\mcw_A \subset \mcu_A$ to be the universal hypersurface $\{\tilde{s} = 0\}$ over $\mcv_A$.

\begin{prop} \label{prop:openemb1}\gls{mcuA}
	There is an open inclusion $\iota : \mcu_A \to \tlaf{A}$.
\end{prop}
\begin{proof} To prove the proposition, we first provide a stacky fan description for $\mcu_A$. For this, note that the total space $\mco_A (-1)$ has a stacky fan dual to the polyhedron which is the cone  $\textnormal{Lin}_{\R_{\geq 1}} (Q \oplus \{1\}) \subset \Lambda^\vee_\R \oplus \R$. As this is a cone over the polytope $Q$, a facet is either $Q$, defined by $\varrho_B = (0, 1)$, or a facet of the cone $\textnormal{Lin}_{\R_{\geq 0}} ({\mathcal{A}})$. So the set of primitive hyperplanes can be identified with $B = \du{Q} \cup \{\varrho_B\}$ with the simplicial set $\mcb$ corresponding to the normal fan of $\textnormal{Lin}_{\R_{\geq 1}} (Q \oplus \{1\})$. The fundamental exact sequence \eqref{eq:setfunseq} for $\mco_A (-1)$ is then 
\begin{align} \label{eq:mco-1}
0 \to L_{\du{Q}} \stackrel{\alpha_B}{\longrightarrow} \Z^\du{Q} \oplus \Z \stackrel{\beta_B}{\longrightarrow} (\Lambda \oplus \Z)^\vee \to K_{\du{Q}} \to 0.
\end{align}
Thus 
\begin{align*}
\sfan_{\mco_A (-1)} := \left( \Z^\du{Q} \oplus \Z , (\Lambda \oplus \Z)^\vee , \beta_B , \Sigma_{B, \mcb} \right)
\end{align*} gives a stacky fan for $\mco_A (-1)$. Now, as $\mathbb{H}_{\sfan^\prime} = \mathbb{G}_{\sfan_{\mco_A (-1)}} = ( \C^*)^B / \mathbb{H}_{\sfan_{\mco_A (-1)}}$ we let $\mathbb{H}_{\sfan_{\mco_A (-1)}}$ act trivially on $X_{\Sigma^\prime}$ and obtain
\begin{align*}
\mcu_A & = \left[(X_{\Sigma^\prime} \times \mco_A (-1) ) / \, \mathbb{H}_{\sfan^\prime} \right] , \\ & = \left[\left(X_{\Sigma^\prime} \times \left( X_{\Sigma_{B, \mcb}} / \, \mathbb{H}_{\sfan_{\mco_A (-1)}} \right)\right) / \, \left(  ( \C^*)^B / \mathbb{H}_{\sfan_{\mco_A (-1)}} \right) \right], \\ & = \left[\left(  (X_{\Sigma^\prime} \times X_{\Sigma_{B, \mcb}}) / \, \mathbb{H}_{\sfan_{\mco_A (-1)}} \right) / \, \left(  ( \C^*)^B / \mathbb{H}_{\sfan_{\mco_A (-1)}} \right) \right] , \\ & = \left[  (X_{\Sigma^\prime} \times X_{\Sigma_{B, \mcb}})   / \, ( \C^*)^B  \right].
\end{align*}
The action of $(\C^*)^B$ on $X_{\Sigma^\prime}$ is obtained by tensoring the negative of the homomorphism 
\begin{align} \label{eq:chi} \tilde{\chi}_A :=   \beta_{{\mathcal{A}}}^\vee \circ \beta_B : \Z^{\du{Q}} \oplus \Z  \to (\Z^{{\mathcal{A}}})^\vee \end{align} 
with $\C^*$, where $\beta_B$ and $\beta_{{\mathcal{A}}}^\vee$ occur in equations \eqref{eq:mco-1} and \eqref{eq:Aex} respectively.  On the other hand, the action of $(\C^*)^B$ on $X_{\Sigma_{B, \mcb}} \subset \C^B$ is just the restriction of the torus action. Thus the diagonal action is given by  $(- \tilde{\chi}_A , \text{Id} )  : \Z^{\du{Q}} \oplus \Z  \to (\Z^{{\mathcal{A}}})^\vee \oplus \Z^{\du{Q}} \oplus \Z$ which has cokernel $\chi_A = (\text{Id}, \tilde{\chi}_A )$.
\begin{align*} \chi_A : (\Z^{{\mathcal{A}}})^\vee \oplus \Z^{\du{Q}} \oplus \Z \to (\Z^{{\mathcal{A}}})^\vee .\end{align*}
Taking the product fan $\Sigma_{U,A} :=\Sigma^{\prime } \times \Sigma_{B, \mcb}$, one observes that $\mcu_A$ can be obtained from the stacky fan 
\begin{align}
\label{eq:openU}
\sfan_{U, A} = \left(  (\Z^{{\mathcal{A}}})^\vee \oplus \Z^{\du{Q}} \oplus \Z , (\Z^{{\mathcal{A}}})^\vee , \chi_A, \Sigma_{U,A} \right).
\end{align}

From Lemma \ref{lem:lafhyppart} we have that $\dul{\plaf{A}}$ is the disjoint union $\{\varrho_A\} \cup \dul{\plaf{A}}^h \cup \dul{\plaf{A}}^v$. Recall that elements of $\dul{\plaf{A}}^h$ are indexed by pointed subdivisions $(S, A_p)$ where $S$ is the trivial subdivision and $A_p$ are points an a facet of $Q$ whereas $\dul{\plaf{A}}^v$ correspond to $(S, A_p)$ where $S = \{(Q_i, A_i) : i \in I\}$ is a coarse subdivision and $A_p = A_i$ for some $i \in I$. By Lemma \ref{lem:lafnormal}, the primitive corresponding to the latter type is $\eta_{(S, A_p)}$ and is the unique primitive defining function for $S$ satisfying $\eta_{(S, A_p)}|_{A_p} = 0$. Amongst all of the elements in $\dul{\plaf{A}}^v$ are those whose coarse subdivisions are of the type $S = \{(Q, A - \{a\})\}$ with $a \in A_{nv}$.   One can check that in this case $\eta_{(S, A - \{a\})} = e_a^\vee$.  We define the subset $B_{\plaf{A}}  = \{e_a^\vee : a \in \mca_{nv}\} \cup \{\varrho_A\} \cup \dul{\plaf{A}}^h \subset \dul{\plaf{A}}$ of supporting hyperplanes of $\plaf{A}$ and
let $\Sigma_{\ptlaf{A}}^\prime$ be the subfans of $\Sigma_{\ptlaf{A}}$, appearing in equation \eqref{eq:ptlafstfan}, consisting of all cones whose boundary one-cones are generated by elements in $B_{\ptlaf{A}}$. Expanding to the union of $B_{\ptlaf{A}}$ with $A_v$, we define the open substack $\mcx_{\plaf{A}}^\circ \subset \tlaf{A}$ by taking the stacky fan
\begin{align*} 
\sfan_{\plaf{A}}^\prime & = \left( \Z^{A_v \cup B_{\plaf{A}}}  , (\Z^{A})^\vee ,  \beta^\prime , \Sigma_{\plaf{A}}^\prime \right).  
\end{align*}
Here $\beta^\prime$ is the restriction of $\tilde{\beta}_{\dul{\plaf{A}}}$, defined in equation \eqref{eq:lafsfan}, to $\Z^{A_v \cup B_{\plaf{A}}}$. 
We mention that while  the inclusion $i$ of $\Z^{A_v \cup B_{\plaf{A}}}$  into $\Z^{\dul{\plaf{A}}}$ is not an isomorphism, the induced map on the stacks $(i , \text{Id}): \sfan^\prime_{\plaf{A}} \to \widetilde{\sfan}_{\plaf{A}}$  defines the open inclusion whose image is the complement of the divisors associated to $\dul{\plaf{A}} - B_{\plaf{A}}$. This is an application of  \cite[Theorem~B.3]{toricstacks}.
	
We now relate $\sfan_{\plaf{A}}^\prime$ to $\sfan_{U, A}$. By Lemma \ref{lem:lafhyppart}, there is a bijection between $\du{Q} \cup \{\varrho_B\}$ and $\dul{\plaf{A}}^h \cup \{\varrho_A\}$. Extending this bijection to basis vectors, and taking the identity on basis vectors indexed by ${\mathcal{A}}$, we obtain an isomorphism 
\begin{align} g_1 :   (\Z^{{\mathcal{A}}})^\vee \oplus \Z^{\du{Q}} \oplus \Z  \to \Z^{A_v \cup B_{\plaf{A}}}. \end{align}
Taking $g_2$ to be the identity, and consulting the definition of $\tilde{\beta}_{\dul{\plaf{A}}}$ obtained from Lemma \ref{lem:lafnormal:1}, we observe that this gives a commutative diagram
\begin{equation} \label{eq:isomstack1}
\begin{CD}
(\Z^{{\mathcal{A}}})^\vee \oplus \Z^{\du{Q}} \oplus \Z  @>{\chi_A}>> (\Z^{{\mathcal{A}}})^\vee \\
@V{g_1}VV @V{g_2}VV \\ 
\Z^{A_v \cup B_{\plaf{A}}}  @>{\beta^\prime }>> (\Z^A)^\vee \\.
\end{CD}
\end{equation}

To prove that $(g_1, g_2)$ this induces an isomorphism of stacks, we check that $g_1$ realizes an isomorphism of fans from  $\Sigma_{U, A}$ to $\Sigma^\prime_{\plaf{A}}$. Take $\Sigma_1$ to be the fan supported in $\R^{{\mathcal{A}}_{nv}} \subset \R^{B_{\plaf{A}}}$ with unique maximal cone $\R_{\geq 0}^{{{\mathcal{A}}}_{nv}}$ and let $\Sigma_2$ be the subfan of $\Sigma_{\plaf{A}}^\prime$ supported in $\R^{\dul{\plaf{A}}^h \cup \{\varrho_A\}}$. Lemma \ref{lem:lafnormal:2} along with the definition of the Lafforgue fan implies that the bijection between $\du{Q} \cup \{\varrho_B\}$ with $\dul{\plaf{A}}^h \cup \{\varrho_A\}$ induces an isomorphism between $\Sigma_{B, \mcb}$ and $\Sigma_2$. Moreover, consulting the inclusion relation $\prec$ for the Lafforgue fan following equation \eqref{eq:lafcones}, we check that $\Sigma^\prime_{\plaf{A}}$ is the product fan $\Sigma_1 \times \Sigma_2$. Indeed, any cone in $\Sigma^\prime_{\plaf{A}}$ corresponds to a pointed subdivision $(S, A_p)$ of $(Q,A)$ where $S$ only refines subdivisions of the type $S^\prime = \{(Q, A^\prime)\}$ (otherwise $C_\R (S, A_p)$ contains a one-cone which is not generated by an element of $B_{\plaf{A}}$). By the same reasoning, $S$ itself must also be such a subdivision, i.e. $S = \{(Q, A^\prime)\}$. Now, suppose $\sigma \times \tau $ is a cone in $\Sigma_1 \times \Sigma_2$. Then $\sigma$ is the span of basis vectors corresponding to a subset $C$ of $A_{nv}$. As $\Sigma_2$ is isomorphic to $\Sigma_{B, \mcb}$, there is a face $Q^\prime$ of $Q$ such that $\tau$ is the span of the basis vectors whose corresponding pointed subdivisions have pointing set that contain $Q^\prime \cap A$. Then the cone in $\Sigma^\prime_{\plaf{A}}$ corresponding to $(\{(Q, A - C)\}, Q^\prime \cap A)$ is the isomorphic image of $\sigma \times \tau$. As these exhaust all possible cones of each fan, we have shown that $\Sigma^\prime_{\plaf{A}}$ and $\Sigma_1 \times \Sigma_2$ are isomorphic on their support and thus $\mcu_A$ and $\mcx_{\plaf{A}}^\circ$ are isomorphic stacks. 
\end{proof}

We continue by relating the universal line bundle and hypersurface in the total Lafforgue stack, introduced in Definition \ref{defn:moduli}, to the incidence variety $\mcw_A \subset \mcu_A$.

\begin{prop}
	The morphism $\iota$ satisfies $\iota^* (\univ{A}) = \mce(1)$ and $\iota^* (\tilde{\mcy}_A) = \mcw_A$.
\end{prop}

\begin{proof}
To prove the proposition, we write explicit formulas for the sections $\tilde{s}$ and $s_A$.  First recall from equation \eqref{eq:mcoq} that for $b \in \du{Q}$, $-n_b$ is the minimum value of $b$ on $Q$. For a toric stack $\mcx$ and an equivariant divisor $D \in \Divisor_{eq} (\mcx )$, we write $x_D$ for the section of $\mco ([D])$ defining $D$.

Now, the universal section $s_A \in H^0 (\tlaf{A}, \univ{A})$ is defined as the pullback of the section  $\sum_{a \in A} x_{D_a}$ in $\mco_{\p^{|A| - 1}} (1)$. To understand this pullback in terms of the stacky fan $\sfan_{\plaf{A}}^\prime$, first consider the map $\mcx_{\ptlaf{A}}^\circ \to \p^{|A| - 1}$ which is the restriction of $G: \tlaf{A} \to \p^{|A| - 1}$  appearing before Definition \ref{defn:moduli}. A stacky fan for $\p^{|A| -1}$ is $\sfan_\p = \left( (\Z^A)^\vee , (\Z^A)^\vee / (\varrho_A), \beta_\p , \Sigma_{\p^{|A| - 1}} \right)$ where $\beta_\p$ is the quotient homomorphism and $\Sigma_{\p^{|A| - 1}}$ is the fan of all proper subcones of $(\R^A_{\geq 0})^\vee$. We define a function $\rho : \Z^{A_v \cup B_{\plaf{A}}} \to (\Z^A)^\vee$ so that the function $G$ is induced from the map of stacky fans given by $(\rho , \beta_\p)$. For $(\rho , \beta_\p)$ to be a map of stacky fans, the diagram 
\begin{equation*}
\begin{CD}
\Z^{A_v \cup B_{\plaf{A}}}  @>{\beta^\prime}>> (\Z^{A})^\vee  \\
 @V{\rho}VV @V{\beta_{\p}}VV \\ 
(\Z^{A})^\vee @>{\beta_{\p}}>> (\Z^{A})^\vee / (\varrho_A) 
\end{CD}
\end{equation*} 
must commute. For this to occur, we must have $\rho = \beta^\prime + \rho^\prime \cdot \varrho_A$ where $\rho^\prime \in (\Z^{A_v \cup B_{\plaf{A}}})^\vee$. The map $\rho^\prime$ is then uniquely defined so that the support of the fan $\Sigma^\prime_{\plaf{A}}$ maps to that of $\Sigma_{\p^{|A| -1}}$ (as translating by $\varrho_A$ displaces the $|\Sigma_{\p^{|A| -1}}|$ from itself). Now, $\beta^\prime (e_a) = e_a^\vee \in |\Sigma_{\p^{|A| -1}}|$, so $\rho^\prime (e_a) = 0$ while $\beta^\prime (e_{\varrho_A}) = \varrho_A$ implying $\rho^\prime (\varrho_A ) = -1$. For $b \in \dul{\plaf{A}}^h$ corresponding to a facet $F$ of $Q$, we apply the definition of $\beta^\prime$ arising from Lemma  \ref{lem:lafnormal:1} to recall $\beta^\prime (b) =  \beta_{{\mathcal{A}}}^\vee (b_{A_p})$ where $A_p = A \cap F$. 
But $b_{A_p}$ is the support function for the cone over $F$ in $\textnormal{Lin}_{R_{\geq 0}} ({\mathcal{A}})$ and is thus zero on $A_p$ and positive on $A - A_p$ implying $\beta^\prime (b) \in |\Sigma_{\p^{|A| -1}}|$ and $\rho^\prime (b) = 0$. Therefore, $\rho^\prime = - e_{\varrho_A}^\vee$ and 
\begin{align*}
\rho = \beta^\prime - e^\vee_{\varrho_A} \cdot \varrho_A.
\end{align*}

The pullback of $\sum_{a \in A} x_{D_a}$, where $D_a$ is identified with the basis element $e_a$ in $\Divisor_{eq} (\p^{|A| - 1}) = ((\Z^A)^\vee )^\vee = \Z^A$, is the universal section $s_A = \sum_{a \in A} x_{\rho^\vee (e_a)}$. Since $\rho (e_{\varrho_A}) = 0$, the explicit form of $\rho$ gives 
\begin{align*}  \rho^\vee (e_a) = e_a^\vee + \sum_{b \in \dul{\plaf{A}}^h} \left<\beta^\prime (e_b), e_a \right> e_b^\vee  \in  (\Z^{A_v \cup B_{\plaf{A}}})^\vee =  \Divisor_{eq} ( \mcx^\circ_{\plaf{A}}) , \end{align*}
and $s_A = \sum_{a \in A} x_{\rho^\vee (e_a)}$. If $b$ corresponds to the facet $F$ with $A_p = F \cap A_p$, let $b_F \in \du{Q}$ be the supporting hyperplane function of $F$ in $(\Z^A)^\vee$. Then, using the definition of $n_{b_F}$ we have $b_{A_p} = (b_F , n_{b_F})$. As $\beta^\prime (b) = \beta^\vee_{{\mathcal{A}}} (b_{A_p}) $ we compute 
\begin{align*} \left<\beta^\prime (e_b), e_a \right> & = \left< \beta^\vee_{{\mathcal{A}}} (b_{A_p}), e_a \right>,  \\ & = \left< b_{A_p}, \beta_{{\mathcal{A}}} e_a \right>, \\ & = 
\left< (b_F, n_{b_F}), (a, 1) \right>, \\ & = 
\left< b_F ,a \right> + n_{b_F} . 
\end{align*}

Turning to $\mcu_A$, recall that the tautological section $\tilde{s}$ was defined on $X_{\Sigma^\prime} \times \mco_A (-1)$ before Proposition \ref{prop:openemb1}.  Let $D_a \in \Divisor_{eq} (\mcx_Q )$ be the divisor associated to $a \in A$.  Take $r_a \in H^0 (\mco_A (-1), E(1))$ as the pullback of $x_{D_a} \in H^0 (\mcx_Q , \mco_A (1))$ and $t_a :X_{\Sigma^\prime} \to \C$ the projection to the $a$-th coordinate. Then, by definition, $\tilde{s} = \sum_{a \in A} t_a \otimes r_a$.  We lift this to an equivariant function $\mathbf{s}$ on the affine toric variety $X_{\Sigma_{U,A}} = X_{\Sigma^\prime} \times X_{\Sigma_{\mco_A (-1)}}$ defined in equation \eqref{eq:openU}.  For every equivariant divisor $D \in \Divisor_{eq} ( X_{\Sigma_{U,A}} ) = \Z^{{\mathcal{A}}} \oplus ( \Z^\du{Q} \oplus \Z)^\vee$ we write $x_D$ for its defining function. The lift of the divisor associated to the monomial $t_a \otimes r_a$ is
\begin{align}
D_a = e_a +  \sum_{b \in \du{Q}} \left(n_b + \left< b, a \right> \right) e_b^\vee   , 
\end{align}
so that $\mathbf{s} = \sum_{a \in A} x_{D_a}$. As the isomorphism $(g_1, g_2)$ in equation \eqref{eq:isomstack1} pulls back $\rho^\vee (e_a)$ to $D_a$, the result has been shown.
\end{proof}

Having related the universal hypersurface in the total Lafforgue stack with the incidence variety in $\mcu_A$, the following theorem shows that the secondary stack from the previous section is a compactification of the moduli stack $\mcv_A$ of full sections. In particular, the discussion immediately following Theorem \ref{thm:GKZ2} described the facets of $\psec{A}$ in terms of coarse subdivisions $S = \{(Q_i, A_i) : i \in I\}$ of the marked polytope $(Q, A)$. Among such subdivisions are those which contain only one marked polytope  $\{(Q, A - \{a\})\}$ where $a \in A_{nv}$. The pointed subdivisions $(\{(Q, A- \{a\}) \}, A - \{a\})$ of this type formed the vertical boundary of $\mcx_{\plaf{A}}^\circ$.
Including the components of the toric boundary in $\secon{A}$ which correspond to such subdivisions, and taking the complement of the remaining ones, yields a stack isomorphic to $\mcv_A$. The compactifying strata then correspond to reasonable degenerations of $\mcx_Q$. This is in analogy to the moduli space of curves and their stable compactifications which served as motivation for the definition.


\begin{thm} \label{thm:toric1} There is an open embedding $i : \mcv_A \to \secon{A}$. If $p \not\in i (\mcv_A)$, then $p$ is in a boundary divisor $D_S$ where $S = \{(Q_i, A_i ) : i \in I\}$ is a coarse subdivision and $|I| > 1$. 
\end{thm}

\begin{proof}  
To start we define an open substack of $\mcx_{\psec{A}}^r$ given by the stacky fan
\begin{align}   \sfan_{\psecv{A}{v}} = \left( \Z^{\dul{\psecv{A}{v}}} , L_A^\vee , \beta_{\dul{\psecv{A}{v}}}, \fan_{\mcb} \right)  \end{align} 
defined in equation \eqref{eq:secstkfan}. Note that if $a \in A_{nv}$, then there is a unique pointed subdivision $(\{(Q, A - \{a\})\},A - \{a\})$ corresponding to a facet of $\ptlaf{A}$. The set $A_{nv}$ also labels a subset of supporting primitives in $\dul{\psecv{A}{v}}$.  Define the subset
\begin{align*} B_{\psec{A}} & = \{\eta_S \in \dul{\psecv{A}{v}}   \text{ a primitive dual to } S = \{(Q, A - \{a\})\} : a \in A_{nv} \}.
\end{align*} 
Let $\Sigma_{\mcb}^\prime$ be the subfan of $\Sigma_{\mcb}$ consisting of all cones whose boundary one-cones are generated by elements in $B_\psec{A}$. Define the open substack $(\secon{A}^r)^\circ \subset \secon{A}^r$ to be that associated to the taking the stacky subfan
\begin{align*}  \sfan_{\psecv{A}{v}}^\prime & = \left( \Z^{A_v \cup B_{\Sigma (A)}} , L_A^\vee , \beta_{\dul{\psecv{A}{v}}}, \fan^\prime_{\mcb} \right).  \end{align*}
Recall from Proposition \ref{prop:openemb1} that $\mcx_{\plaf{A}}^\circ$ was defined from the stacky subfan 
\begin{align*}
\sfan_{\plaf{A}}^\prime & = \left( \Z^{A_v \cup B_{\plaf{A}}}  , (\Z^{A})^\vee ,  \beta^\prime , \Sigma_{\plaf{A}}^\prime \right).  
\end{align*}
Restricting the map $\tilde{p} : \widetilde{\sfan}_{\plaf{A}} \to \sfan_{\psecv{A}{v}}$ from Definition \ref{defn:secstack} to these subfans gives the map $\tilde{p}^\prime : \widetilde{\sfan}_{\plaf{A}}^\prime \to \sfan_{\psecv{A}{v}}^\prime$ described by the commutative diagram 
\begin{equation*}
\begin{CD}
\Z^{A_v \cup B_{\plaf{A}}}   @>{\beta^\prime}>> (\Z^A)^\vee \\
@V{p_1^\prime}VV @V{\alpha^\vee_{A}}VV \\
\Z^{A_v \cup B_{\Sigma (A)}} @>{\beta_{\dul{\psecv{A}{v}}}}>>  L_A^\vee .
\end{CD}
\end{equation*}
We claim that the colimit stack $\mcx^\to_{\tilde{p}^\prime}$ of $\tilde{p}^\prime$ has stacky fan 
\begin{align*} \sfan^{\to}_{\tilde{p}^\prime} =  \sfan_{\psecv{A}{v}}^\prime & = \left( \Z^{A_v \cup B_{\Sigma (A)}} , \Lambda_{\gdual{A}}, \alpha_A^\star, \fan^\prime_{\mcb} \right).  \end{align*}
This follows at once from the diagram
\begin{equation*}
\begin{CD}
\Z^{A_v \cup B_{\plaf{A}}}   @>{\beta^\prime}>> (\Z^A)^\vee \\
@V{p_1^\prime}VV @V{\kappa}VV \\
\Z^{B_{\Sigma (A)}} @>{\alpha_A^\star}>> \Lambda_{\gdual{A}}.
\end{CD}
\end{equation*}
being a pushout. To see that this is the case, partition the basis vectors of $A_v \cup B_{\plaf{A}}$ into $A \approx A_v \cup \{e_a^\vee : a \in A_{nv} \}$ and $\dul{\plaf{A}}^h \cup \{\varrho_A\}$. Likewise, as the elements of $B_{\psec{A}}$ are indexed by $A$, we identify $\Z^{B_\psec{A}}$ with $\Z^A$. Then the map $p_1^\prime : \Z^A \oplus \Z^{\dul{\plaf{A}}^h} \oplus  \Z^{\{\varrho_A\}} \to \Z^A$ is simply projection and $\beta^\prime|_{\Z^A}$ is clearly injective, implying the pushout $\kappa$ is the cokernel of $\beta^\prime |_{ \Z^{\dul{\plaf{A}}^h} \oplus  \Z^{\{\varrho_A\}}}$
But by Lemma \ref{lem:lafnormal}, we have that the image of $\beta^\prime$ restricted to $\Z^{{\du{\plaf{A}}} \cup \{\varrho_A\}}$ is the image of $\beta_A^\vee$ which has the indicated cokernel from the dual fundamental sequence for ${\mathcal{A}}$. 

The described stacky fan data obtained on the bottom of the diagram defines the colimit stack and is identical to the stacky fan defining $\mcv_A$. This proves the claim.

\end{proof} 

Finally, we describe the points on the compactifying divisor.

\begin{thm} \label{thm:toric2} Suppose $(\mcx , \mcy )$ is a degenerating family of a hypersurface $(\mcx_Q , \mcy_s )$ defined by a $\Lambda$-defining function. Then $(\mcx, \mcy)$ is represented by a map $\upsilon : \C \to \secon{A}$.
\end{thm}

\begin{proof} Let $\eta \in \left(\Z^{{\mathcal{A}}} \right)^\vee$ be a $\Lambda$-defining function for the family $(\mcx , \mcy)$ corresponding to the subdivision $S = \{(Q_i, A_i) : i \in I\}$ with $|I| > 1$. Define a map $\upsilon_\eta : \N \to \Lambda_{{\mathcal{A}}^\vee}$ by taking $\upsilon (1) := \alpha^\star (\eta )$. The stacky fan $\sfan = (\Lambda_1, \Lambda_2 ,  \beta, \fan)$ occurring in the fiber product $\C  \text{ }_{\stackmap{\upsilon}{\eta}} \times_{\pi} \laf{A}$ has $\Lambda_2 \approx \Lambda^\vee \oplus \Z$ from the Cartesian diagram
\begin{equation} \label{diag:bettersec}
\begin{CD}
\Lambda^\vee \oplus \Z \cdot \eta @>{\psi}>> (\Z^{{\mathcal{A}}})^\vee / (\sum_{\alpha \in A} e_\alpha^\vee ) \\
@V{p_2}VV  @V{\alpha^\star}VV \\
\Z @>{\upsilon_\eta}>> \Lambda_{\gdual{A}},
\end{CD}
\end{equation}
Here the map $\psi$ is the composition of the quotient $\text{proj} : (\Z^{{\mathcal{A}}})^\vee \to  (\Z^{{\mathcal{A}}})^\vee / (\sum_{\alpha \in A} e_\alpha^\vee ) $ and the direct sum $ \beta_{{\mathcal{A}}}^\vee|_{\Lambda^\vee} \oplus \text{inc}$ where $\text{inc}: \Z \cdot \eta \to (\Z^{{\mathcal{A}}})^\vee$ is the inclusion.

To find $\Lambda_1$ and $\fan$, we let $\fan_\eta$ be the fan obtained by intersecting $\beta_A^\vee \oplus \text{inc} ((\R^d)^\vee \oplus \R_{\geq 0})$ with the Lafforgue fan and $T_\eta$ the generators of its  one-cones. The map $\beta_\eta : \Z^{T_\eta} \to \Lambda_2$ by evaluation of primitives gives the stack $\sfan_\eta = (\Z^{T_\eta}, \Lambda_2 , \beta_\eta , \fan_\eta )$. For every $f \in T_\eta$, we have that
\begin{equation*} f = \sum_{b \in \sigma (1) \subset \dul{\plaf{A}}} c_b b \end{equation*} 
for some maximal cone $\sigma$ in the Lafforgue fan. Since $\eta$ is a $\Lambda$-defining function and $\alpha^\star (\psi (f)) = \alpha^\star ( \eta )$, it follows that $c_b b$ is in the image of $\tilde{\beta}_{\dul{\plaf{A}}}$ for all $b \in \sigma (1)$. Let  $g_1 (e_\tau ) = \sum_{b \in \sigma (1)} c_b e_b$. It is not hard to see that the map $g = (g_1 , 1 )$ then induces an equivalence between $\sfan_\eta$ and the pullback $\sfan$. 

To see that $\sfan_\eta$ is the normal stacky fan to $(Q_\eta , A_\eta)$, we need only show that $T_\eta = \du{Q}_\eta \subset \Lambda \oplus \Z$. By Lemma~\ref{lem:lafnormal}, $\tau \in T_\eta$ if and only if it defines the subdivision $S$ and constant on $Q_i$ for some $i$. So the one-cones of $T_\eta$ equal those of $\du{Q}_\eta$. But both sets consist of primitives of their one-cones on vertical divisors, implying the equality. 

To show that the pull back is isomorphic to $(\mcx_\eta , \mcy_{\iota_\eta (s)})$, we prove any section of the form  $\iota_\eta (s)$ can be represented by a pullback of the universal section $s_A$. For this, we simply observe that the pullback of $s_A$ to $H^0 (\mcx_\eta , \mco_\eta (1))$ is $\stackmap{\upsilon}{\eta}^* (s_A ) = \sum_\alpha x_{(\alpha, \eta (\alpha ))}$ from equation~\eqref{eq:stackmap}. The group $\tgroup{\psec{A}}$ acts transitively on the pullback of the space of very full sections of $H^0 ( \mcx_\eta , \mco_\eta (1) )$ up to equivalence. Indeed, from the fundamental exact sequence for $A$, it is easy to see that there exists a $\lambda \in \tgroup{\psec{A}}$ such that $\stackmap{\upsilon}{\eta}^* ( \lambda \cdot s_A) = \sum_\alpha c_\alpha x_{(\alpha , \eta (\alpha ))}$ for any $\{c_\alpha \} $ satisfying $\prod c_\alpha^{m_\alpha} = 1$ with $\sum_\alpha m_\alpha \alpha = 0$. Any full section has a representative in this class, yielding the claim.
\end{proof}

\section{\label{sec:pfsymp}$\partial$-framed symplectomorphisms}

We begin this section by defining certain subgroups of symplectomorphism groups which we refer to as $\partial$-framed groups. The symplectic orbifolds we consider have boundary divisors that are preserved by the symplectomorphisms under consideration. Moreover, we would like to distinguish between subgroups that fix the boundary tangentially and those that do not. This aim would be easily achieved were our boundary divisor smooth and the symplectomorphisms fixed the boundary divisor pointwise. However, neither of these requirements are satisfied in our setting, so we must introduce a more elastic notion of framing.

After defining the notion of a $\partial$-framed group, we proceed to examine the geometry of various symplectomorphisms contained in them. Up to Hamiltonian isotopy, the generators of our groups arise as monodromy maps around a singular symplectic orbifold. The type of permissible singularities that we will study fall into two classes.
The first will be a stable pair degeneration of the symplectic
orbifold into irreducible orbifolds glued along normal crossing
divisors akin to the situation in complex geometry. The study of
maximal degenerations of this type in the toric case was
thoroughly analyzed in \cite{abouzaid}. 

The second type of singularity we see is a stratified Morse
singularity. This is studied in \cite{goresky}, but only the
non-stratified case has been understood in the symplectic setting
\cite{seideldehn}. We will examine the general case and give a
geometric description of the monodromy.

\subsection{Definitions}

Let $(\mcy, \omega )$ be a symplectic orbifold of real dimension $2n$ with atlas\gls{atlas} $\mcu = (U_\beta, G_\beta , \pi_\beta)_{\beta \in \mcb}$. Most of the familiar constructions in symplectic geometry can be defined through the invariant manifold analogs in an atlas when working with symplectic orbifolds. For example, a Hamiltonian will mean a smooth function on $\mcy$ or, equivalently, a collection of smooth, compatible, invariant functions on $(U_i, G_i)$.  Likewise, its flow can be computed in $\mcy$ or, for short time on a relatively compact subset, in each chart of the atlas. Types of submanifolds (Lagrangian, isotropic, symplectic), almost complex structures, Poisson brackets, symplectomorphisms are all defined locally and can be given precise meaning in the symplectic stack setting. We omit using the adjective ``orbifold" for all of these terms throughout the paper. We leave the definitions of these structures to the existing literature \cite{orbifolds}, \cite{mcdufftoric}, but will give details for structures that are less familiar.

Take $\mcj$ to be the space of compatible almost complex structures on $\mcy$ and $D = D_1 + \cdots + D_k$
a symplectic divisor, i.e. each $D_i$ is a smooth symplectic suborbifold of real codimension $2$. If there is an integrable $J \in \mcj$ and $\mcy$ is a manifold, it makes sense to say that $D$ is a divisor with normal crossing singularities. We extend this to symplectic orbifolds in the following fashion.  For every $D_i$ and $\beta \in \mcb$, take $D_i (\beta ) = \pi_\beta^{-1} (D_i)$. 

\begin{defn} Let $J \in \mcj$. \begin{itemize}
\item[(i)] \label{defn:integrable} A symplectic divisor $D$ will be called $J$-integrable if, for every $D_i$ and every $\beta \in \mcb$ there are symplectic neighborhoods $V_i$ of $D_i (\beta )$ such that $J$ is integrable on $V_i$ and $D_i (\beta )$ is a complex divisor in $V_i$ relative to $J$. 
\item[(ii)] \label{defn:nc} A symplectic divisor $D$ is a $J$-normal crossing divisor if, given $p \in D$ and $I = \{i : 1 \leq i \leq k , p \in D_i\}$,  there exists $U \subset \C^n$ and a $J$-holomorphic chart $\psi :U \to \cup_{i \in I} V_i$ near $p \in V$ such that $\psi (0) = p$ and $D \cap V = \psi (\{(z_1, \ldots, z_n) : z_{i_1} \cdots z_{i_k} = 0 \})$.
\item[(iii)] \label{defn:onc} A normal crossing divisor is $J$-standard if for every point $p \in D$, there exists a $J$-holomorphic chart $\psi$ such that $\psi^* \omega = \omega_{st}$ where $\omega_{st}$ denotes the standard symplectic form on $U \subset \C^n$.
\item[(iv)] We say that a divisor is integrable, normal crossing or standard if there exists some $J \in \mcj$ for which it is $J$-integrable, $J$-normal crossing or $J$-standard.
\end{itemize}
 \end{defn}

A consequence of having a $J$-standard normal crossing divisor is that the distance squared functions $h_i : \mcy \to \R$  from $D_i$ (via the metric induced by $\omega$ and $J$) Poisson commute in neighborhoods of $D$. In other words, there exists an $\varepsilon_J > 0$ for which $\{h_i , h_j \} = 0$ on $U_i \cap U_j$ where $U_i = h_i^{-1} ([0, \varepsilon_J ))$. We call any $\varepsilon < \varepsilon_J$  commuting. For any commuting $\varepsilon$, we define $\rho^\varepsilon_i = \lambda^\varepsilon \circ h_i$ where
$\lambda^\varepsilon : \R_{\geq 0 } \to \R_{\geq 0}$ is a smooth monotonic function satisfying
\begin{equation*} \lambda^\varepsilon (r) = \left\{ \begin{matrix} r & r <
\varepsilon / 2 \\ \varepsilon & r \geq \varepsilon \end{matrix}
\right. . \end{equation*}
It is easy to see that $\{\rho^\varepsilon_i , \rho^\varepsilon_j
\} = 0$ on $\mcy$. Given any $\mathbf{x} = (x_1, \ldots, x_k ) \in \R^k$, we define $\boundflow (\mathbf{x})$ to be the flow of $\sum_{i = 1}^k x_i \rho^\varepsilon_i$. The fact that the $\rho_i$ Poisson commute implies that $\boundflow (\mathbf{x}_1 + \mathbf{x}_2) = \boundflow (\mathbf{x}_1) \circ \boundflow (\mathbf{x}_2 )$. It is best to think of these maps as rotations, or twists, about the components of the divisor.

We take\gls{symporb} $\Symp (\mcy)$ to denote the topological group of symplectomorphisms with the $C^\infty$-topology and $\Symp_0 (\mcy)$ the identity component. For a Hermitian line bundle $L$ over $\mcy$, let $\Symp (L / \mcy)$ be the group of unitary line bundle automorphisms of $L$ over symplectomorphisms of $\mcy$ and $\Symp_0 (L / \mcy)$ its identity component (not to be confused with those maps of $L$ lying over $\Symp_0 (\mcy)$).

Given a standard normal crossing divisor $D \subset \mcy$, we fix a commuting $\varepsilon > 0$ and define\gls{sympdv} $\Symp (\mcy, D)$ to consist of symplectomorphisms of $\mcy$ which preserve the distance to each $D_i$ in an $\varepsilon$ tubular neighborhood of $D$. Here we mean that for any $\phi \in \Symp (\mcy, D)$, we have $\phi^* (h_i |_{U_\varepsilon} ) = h_i |_{U_\varepsilon}$ for all $i$ where $U_\varepsilon$ is an $\varepsilon$-neighborhood of $D_i$. Equivalently, we can
consider $\Symp (\mcy, D)$ to be the group of symplectomorphisms which
commute with $\boundflow (\mathbf{x} )$ for every $\mathbf{x} \in \R^k$. From this definition, it is clear the subgroup\gls{bolT}
\begin{equation} \label{eq:frm} \frmee := \{\boundflow (\mathbf{x} ) : \mathbf{x} \in \R^k \} 
\end{equation}
is contained in the center $Z (\Symp (\mcy, D))$.  

For a normal crossing divisor such as $D \subset \mcy$, we write $\Symp (D)$ for the subgroup of $ \times_{i = 1}^k \Symp (D_i, D_i \cap (\cup_{j \ne i} D_j))$ consisting of $\{\phi_i\}$ with $\phi_i |_{D_i \cap D_j} = \phi_j |_{D_i \cap D_j}$. Let $L = \{L_i : 1 \leq i \leq k \} $ be a collection of line bundles where $L_i$ is a line bundle over $D_i$. 
\begin{defn} The collection $L$ of line bundles is compatible if there exist isomorphisms
\begin{align} \gamma_{i,j} : L_i |_{D_i \cap D_j}
\stackrel{\cong}{\longrightarrow} N_{ D_i \cap D_j} D_j \end{align} for every $1 \leq i, j \leq k$. A set of isomorphisms $\mathbf{g} = \{ \gamma_{i, j} \}$ will be called gluing data. 
\end{defn}
Given such data, we take $\Symp_{\mathbf{g}} (L / D )$ to consist of symplectic line bundle automorphisms $\{\psi_i \}$ of $L_i / D_i$
which lie over some $\{\phi_i \} \in \Symp (D)$ and are compatible in the sense that 
\begin{align} \label{eq:compcond} d \phi_j |_{D_j \cap D_i} = \gamma_{i, j} (\psi_i ) \end{align} for every $1 \leq i,j \leq k$.  We will simply write $\Symp (L / D)$ when the gluing data is evident. For example, in our context of a normal crossing divisor $D \subset \mcy$, we take $N_D \mcy$ for the collection of normal bundles $\{ N_{D_i} \mcy\}$ with the induced gluing data.


\begin{defn} Given $\mcy$ with a standard normal crossing divisor $D$, we say
that a compactly generated, closed subgroup\gls{bframe} $\mbF \subseteq  \Symp (N_D \mcy / D)$ is a
$\partial$-frame group of $(\mcy, D)$.
\end{defn}

Let $j: D \to \mcy$ be the inclusion map and\gls{jshrp} $j^\# : \Symp (\mcy, D) \to
 \Symp ( N_D \mcy / D)$ the restriction of the derivative. Given a
$\partial$-frame group $\mbF$, we will say $\phi \in \Symp (\mcy, D)$ is an
$\mbF$-framed, or framed, symplectomorphism if $j^\# (\phi ) \in
\mbF$. Denote the group of $\mbF$-framed symplectomorphisms by\gls{frsymp}
$\Symp^{\mbF} (\mcy, D)$. Letting $i : \mbF \to \Symp (N_D \mcy / D)$ be the inclusion, this group is defined by the Cartesian diagram
\begin{align}
\begin{CD}
\Symp^{\mbF} (\mcy, D) @>>> \mbF \\ @VVV @V{i}VV \\ 
\Symp (\mcy, D) @>{j^\#}>> \Symp ( N_D \mcy / D).
\end{CD}
\end{align}


It may be the case that symplectomorphisms of $N_D \mcy$ do not
extend to those on $\mcy$. Including such maps into the $\partial$-frame group has
no effect on the framed symplectomorphism group. To take care of
this redundancy, we define a reduced framing as follows.

\begin{defn} A $\partial$-frame group $\mbF$ will be called \emph{reduced} if for every $\phi \in \mbF$ there
exists a $\psi \in \Symp (\mcy, D)$ such that $j^\# (\psi ) = \phi$. The
maximal reduced subgroup\gls{bframered} 
\begin{align*}
\reduce{\mbF} = \mbF \cap \text{im} \left( j^\# \right) 
\end{align*} of a $\partial$-frame group $\mbF$ will be called the $(\mcy, D)$ reduction of $\mbF$.
\end{defn}

Of course, the closure of the image $j^\# (\mbG )$ of any subgroup $\mbG \subset \Symp (\mcy, D)$ is a reduced $\partial$-frame group. An important class of these groups occurs in the following definition:
\begin{defn} A $\partial$-gauge group is a $\partial$-frame group contained in $j^\# (\frmee )$.
\end{defn}
The motivation for defining $\partial$-gauge groups stems from the desire to
exert control over a group similar to the group $(S^1)^k$ of complex multiplications on $
\oplus_{i = 1}^k N_{D_i} \mcy$. Such a group would keep track of rotations around the boundary divisor $D$ of $\mcy$.
Unfortunately, for $\dim \mcy > 2$, this group is not contained in $\Symp (N_D \mcy / D)$
as the compatibility condition in equation \eqref{eq:compcond} is violated. The $\partial$-gauge group and its subgroups can
be thought of as an approximation to such a rotation group.

One of the central points of $\partial$-frame groups is to allow more flexibility
than fixing the boundary and a normal bundle on it. In fact, this
more restrictive case occurs as the framed group $\Symp^\one (\mcy,
D)$ with the trivial framing $\one = \{1 \}$. This fits nicely
into the more general framework as follows.

\begin{prop} \label{prop:fiberbundle} For any reduced $\partial$-frame group $\mbF$, the map:
\begin{equation*}  \Symp^\mbF (\mcy, D) \stackrel{j^\#}{\longrightarrow} \mbF
\end{equation*}
defines a topological fiber bundle with fiber $\Symp^{\one} (\mcy,
D)$.
\end{prop}

\begin{proof}
It follows from the definition of reduced framings that $j^\#$ is
the quotient of $\Symp^\mbF (\mcy, D)$ by the closed, normal subgroup
$\Symp^{\one } (\mcy, D)$. Thus, to prove the claim, one need only show the existence of a local section of $j^\#$ in a neighborhood $U \subset \mbF$ around the identity. To see this, note that the tangent space of $\Symp (N_D \mcy / D)$ is a closed subspace of that of $\times_i \Symp (N_{D_i} / D_i)$. In turn, the tangent space of each $\Symp (N_{D_i}/ D_i)$ consists of closed one-forms $\Omega^1 (N_{D_i})$ which are multiples of $d h_i$ when restricted to the tangent space of any fiber. Denoting this space by $V_i$ and the tangent space of $\Symp (N_D \mcy / D)$ by $V$, we have $V \cong \oplus_i V_i$. Any element $\phi$ in a neighborhood can be realized the integral of a path $\delta_\phi : [0,1] \to V$ of such forms. Furthermore, as $\mbF$ is reduced, we may choose the path $\delta_\phi$ to be contained in the image of the derivative $D j^\# : T_{Id} (\Symp^{\mbF} (\mcy, D)) \to V$.  Note that $D j^\#$ is simply the pullback of the closed one-form associated to a tangent vector in $T_{Id} (\Symp^{\mbF} (\mcy, D))$ along the inclusion $D \hookrightarrow \mcy$. As $D j^\#$ is a linear map, we may choose a section of $\tilde{s} : V \to T_{Id} (\Symp^{\mbF} (\mcy, D))$. Using this, we define the desired local section $s : U \to j^\# (U)$ by taking $s (\phi)$ as the integral of the path $\tilde{s} \circ \delta_\phi$.
\end{proof}

This gives an important corollary.
\begin{cor} \label{cor:fibbund} Suppose $\mbF_1 \subseteq \mbF_2$ are reduced $\partial$-frame groups.
Then there is a homotopy  fiber sequence:
\begin{equation*} \Symp^{\mbF_1} (\mcy, D) \to \Symp^{\mbF_2} (\mcy, D) \to
\frac{\mbF_2 }{\mbF_1}
\end{equation*}
\end{cor}
\begin{proof} Let $i : \Symp^{\mbF_1} (\mcy, D) \to \Symp^{\mbF_1} (\mcy, D)$ be the inclusion and $\text{C} (i)$ its homotopy cofiber.  By Proposition \ref{prop:fiberbundle}, the rows and final column of diagram~\eqref{eq:octodiag} are homotopy fiber sequences.
\begin{align} \label{eq:octodiag}
\begin{CD}
\Symp^{\one} (\mcy, D) @>>> \Symp^{\mbF_1} (\mcy, D) @>{j^\#_1}>> \mbF_1 \\
@| @V{i}VV @VVV \\
\Symp^{\one} (\mcy, D) @>>> \Symp^{\mbF_2} (\mcy, D) @>{j^\#_1}>> \mbF_2 \\
@. @VVV @VVV \\
@. \text{C} (i) @>{\psi}>> \frac{\mbF_2 }{\mbF_1}.
\end{CD}	
\end{align}	
The fact that the induced map $\psi$ from $\text{C} (i)$ to $\frac{\mbF_2 }{\mbF_1}$ is a weak equivalence essentially follows from the octahedral axiom. If one wishes to avoid this argument, take the long exact sequence associated to the homotopy fibrations on each row. Using the fact that the two columns are also homotopy fiber sequences and applying an induction argument, one observes that $\psi$ induces an isomorphism on homotopy groups. Applying Whitehead's theorem then gives that $\psi$ is a homotopy equivalence.
\end{proof} 

For a reduced $\partial$-frame group $\mbF$, we define\gls{relfg} $\relaxed{\mbF}$ to be the group generated by $\mbF$ and $\frme$. As all elements in $\Symp (\mcy, D)$ are required to commute with $\boundflow (\mathbf{x} )$ and $\mbF$ is reduced, $\relaxed{\mbF}$ is a central extension of $\mbF$. The following proposition is then an elementary application of Corollary \ref{cor:fibbund}.

\begin{prop} \label{prop:frame} Assume $D = \cup_{i = 1}^k D_i $ is a standard divisor in $\mcy$. For any reduced $\partial$-frame group $\mbF$, there exists $r \leq k$ and a homotopy exact sequence
\begin{equation*}
\Symp^{\mbF} (\mcy, D) \to \Symp^{\relaxed{\mbF}} (\mcy,
D) \to (S^1)^r .
\end{equation*}
\end{prop}

\begin{proof} We first observe that $\relaxed{\mbF}$ is a finite dimensional central extension of $\mbF$. Since $\frme \cap \mbF$ is closed in $\frme \cong \R^k$, it must be isomorphic $\Z^{r} \oplus \R^s$ for $r + s \leq k$. As it is also closed in $\relaxed{\mbF}$, we have that $\relaxed{\mbF} / \mbF \cong \frme / (\frme \cap \mbF ) \cong (S^1)^r \times \R^{k - s}$. So, by Corollary \ref{cor:fibbund}, we have the result.
\end{proof}

Our primary examples of symplectic orbifolds arise as  hypersurfaces in $\mcx_Q$. Generally,  $Q$ is not assumed to be a simple polytope and so the hypersurfaces will generally be singular along a complex codimension $2$ subspace $\mcy_{sing} \subset D$ of $\mcy$. To deal with these cases, we extend our notion of $\partial$-framing. 

\begin{defn} \label{defn:resolvcol}
	Suppose $D \subset \mathcal{Y}$ is a $J$-integrable divisor and $(\mathcal{Y} - D, \omega)$ is a symplectic orbifold. A set 
	\begin{align} \mcr = \{\phi_\varepsilon :(\tilde{\mcy}, \tilde{D} ) \to (\mcy, D) \} \end{align} of normal crossing resolutions of $(\mcy, D)$ will be called a resolving collection if:
	\begin{enumerate}
		\item each $(\tilde{\mcy} , \tilde{D} )$ is a smooth symplectic orbifold with $\tilde{J}$-standard normal crossing divisors,
		\item  $\phi_\varepsilon^* (\omega ) = \tilde{\omega}$ off of an $\varepsilon$ neighborhood of $\mcy_{sing}$, 
		\item $\phi_\varepsilon$ is $(\tilde{J}, J)$-holomorphic in a neighborhood of $\tilde{D}$.
	\end{enumerate}
We say that $(\mcy, D)$ is a standard symplectic stack if there exists a nonempty collection $\mcr $. 	
\end{defn}
Generally, when $(\mcy , D)$ has a resolution of singularities $(\tilde{\mcy}, \tilde{D})$, it is not clear that one may force the resolution to satisfy the conditions in Definition \ref{defn:resolvcol}. However, when $(\mcy, D)$ is a standard symplectic stack with resolving collection $\mcr$, we may consider a proper subgroup of symplectomorphisms that extend to all resolutions in $\mcr$. 
\begin{defn} 
	Suppose $\mcr$ is a resolving collection for $(\mcy, D)$. Let $\Symp_{\mcr} (\mcy, D)$ be the group of symplectomorphisms $\psi  \in \Symp (\mcy - \mcy_{sing} , D - \mcy_{sing} )$ that are restrictions of  symplectomorphisms $\tilde{\psi} \in \Symp (\tilde{\mcy}, \tilde{D} )$ for all $(\tilde{\mcy} , \tilde{D} )$. 
\end{defn}
We note that this is the coarsest group that could be defined relative to $\mcr$, ignoring any of the subtleties of the combinatorics of the distinct resolutions in $\mcr$. 
Indeed, the benefit of considering resolving collections $\mcr$ instead of a single resolution is that we may define the group $\Symp_{\mcr} (\mcy, D)$, which is independent of the choice of resolution in $\mcr$.   

A $\partial$-frame group $\mbF$ for a standard symplectic stack $(\mcy, D)$ is a subgroup of $\Symp (N_{D - \mcy_{sing}} (\mcy - \mcy_{sing} / D - \mcy_{sing} )$ that has a lift to $\Symp (N_{\tilde{D}} \tilde{\mcy} , \tilde{D} )$ for every $(\tilde{\mcy} , \tilde{D} )$. The definition of the framed group $\Symp^\mbF (\mcy, D)$ and the results above all hold in this case for obvious reasons.

Our primary example of standard symplectic stacks arise in the toric setting. Call a complete intersection non-degenerate if its scheme theoretic intersection with every toric orbit is smooth.

\begin{prop} Suppose $(\mcx , \partial \mcx)$ is a K\"ahler DM toric stack, where $\partial \mcx$ is the toric boundary. If $\mcy \subset \mcx$ is a non-degenerate complete intersection, taking $D = \partial \mcx \cap \mcy$, the pair $(\mcy, D)$ is a standard symplectic stack.
\end{prop}
\begin{proof} This follows immediately from the fact that $\mcx$ has standard resolutions and the non-degeneracy assumption for $\mcy$. \end{proof}

\subsection{Stable pair degeneration monodromy}

In this subsection, we obtain the local model for monodromy around a stable pair degeneration. 
Assume $(\mcx , \omega )$ is a symplectic orbifold of
dimension $n$ with an $r$-dimensional Hamiltonian torus action. We
write $\mbbT^r$ for the torus, $\mft_r$ (or $\mft$) for its Lie
algebra and, for $v \in \mft$, take $X_v \in \textnormal{Vect} (\mcx)$ to be the infinitesimal action in the direction of $v$. Let $J$ be a compatible almost complex structure on $\mcx$
which is invariant with respect to the action and $\mu : \mcx \to \mft^\vee$ the moment map. 

Let $p \in \mcx$ with $\mu (p) = u \in \mft^\vee$ and $v \in \mft$. We define the map
\begin{equation} \label{eq:momentkappa}\gls{kppa} \kappa : \mu (\mcx) \to \Hom_\R ( \mft , \mft^\vee ) \end{equation} 
by taking $\kappa_u ( v) = d \mu_p (J X_v) \in T_u \mft^\vee = \mft^\vee$. Note that this is well defined only under the assumption that $J$ is $\mbbT^r$ invariant. Alternatively, we may think of the map $\kappa$ as giving the metric  restricted to the infinitesimal action vector fields  $g|_\mft  \in \mft^\vee \otimes \mft^\vee$. Given two vectors $v, w \in \mft$, we will write
\begin{equation*} \left< v , w \right>_{\kappa_u} := [\kappa_u (v) ](w) = g_p (X_v , X_w ) .
\end{equation*}

Suppose we have the commutative diagram
\begin{equation} \label{diag:moment1} \begin{CD} \mcx  @>{\mu}>> \mft^\vee \\
@VVFV @VVfV \\
\C @>{\mu_\C}>> \R \end{CD}
\end{equation}
where $F$ is a non-constant, holomorphic function and $\mu_\C = |\hbox{ }|^2$. We assume that $F$ has no critical values outside $0$ and let $\mcx^\circ = \mcx - F^{-1} (0)$. The most common example of diagram \eqref{diag:moment1} is that of a normal crossing degeneration. 

\begin{eg} \label{eg:commutingmoment}
Consider the torus $\mbbT^n = \{(z_1, \ldots, z_n) \in \C^n: |z_i| = 1, \text{ for all } 1 \leq i \leq n\}$ acting on $(\C^n , \omega_{st})$ where $\omega_{st}$ is the standard symplectic form. Identify $\mft$ with $\R^n$ so that if $\mathbf{r} = (r_1, \ldots, r_n) \in \mft$, we may exponentiate to obtain $\exp (\mathbf{r}) = (e^{-2 ir_1}, \ldots, e^{-2 i r_n} )$. Using the dual of the standard basis, we identify $\mft^\vee$ with $\R^n$ as well. Then the moment map for this action is 
\begin{align} \mu (z_1, \ldots, z_n) = \left(  |z_1|^2 , \ldots, |z_n|^2 \right).
\end{align}
If $(a_1, \ldots, a_n) \in \N^n$, consider the map $F : \C^n \to \C$ given by $F (z_1, \ldots, z_n) = z_1^{a_1} \ldots, z_n^{a_n}$ defining a normal crossing singularity over $0$. Then taking $f(r_1, \ldots, r_n) = r_1^{a_1} \cdots r_n^{a_n}$ yields the commutative diagram \eqref{diag:moment1}.
\end{eg}
Recall that $\omega$ defines a Hamiltonian connection on the smooth map $F : \mcx^\circ \to \C^*$ by taking the horizontal distribution to be the symplectic orthogonal to the tangent space of the fiber. As usual, this allows us to lift any vector field on $\C^*$ to $\mcx^\circ$ via
\begin{equation*} \xi : \textnormal{Vect} (\C^* ) \to \textnormal{Vect} (\mcx^\circ ) . \end{equation*}

Recall that the map $\mu_{\C}  : \C \to \R$  is the moment map for circle action of $\mbbT = \{z \in \C : |z| = 1\}$ on $\C$ given by multiplication. Here, as in Example \ref{eg:commutingmoment}, we identify $\R$ as the Lie algebra dual to $\mft = \R$ where $\mbbT$ is parametrized by $\exp (r) = e^{-2r i}$ for $r \in \mft$. We take $\rotat = - 2 i z \partial_z$ to denote the infinitesimal vector field of $\partial_r \in \mft$ on $\C^*$. Note also that derivative of $f$ at a point $p$ gives a natural function $df : \mft^\vee \to \mft$.

\begin{lem} \label{lem:moment} Let $p \in \mcx^\circ$ and $q = \mu (p)$. The horizontal lift $\xi (\rotat )$ of $\rotat $ at $p$ is dependent only on $q$ in the sense that it equals infinitesimal vector field $X_{\delta_q}$ where $\delta_q \in \mft$ is given by
\begin{equation*} \delta_q = \frac{4 f(q)}{||df_{q} ||_{\kappa_q}^2} df_{q} .\end{equation*}
\end{lem}

\begin{proof}  We recall that the defining property of the moment map $\mu : \mcx \to \mft^\vee$ is that, for every $v \in \mft$,  \begin{align} \label{eq:defnmommap} \iota_{X_{v}} \omega = d \, \left< \mu , v \right> = \left< d \mu , v \right>. \end{align} 
Here, $\iota_X \eta$ is the interior product of a differential form $\eta$ with a vector field $X$ and $\left< w, v \right>$ is the canonical pairing taking $w \in \mft^\vee$, $v \in \mft$ to $w (v)$. Thus, letting  $Y \in T_p \mcx^\circ$, by the definition of the moment map and the commutative diagram \eqref{diag:moment1}, we have
\begin{eqnarray*} \omega (X_{df_{ q}}, Y) & = & \left< d \mu (Y), df_{\mu (p)} \right> , \\ &
=& d (f \circ \mu )_{p} (Y) , \\ &  = & d (\mu_\C \circ F )_p (Y) .
\end{eqnarray*}
In particular, if $F(p) = p^\prime$ we see that $X_{df_{q}} (p) \in (T_p F^{-1} (p^\prime))^{\perp_{\omega}}$ and that \begin{align*} d \mu_\C [d F (X_{df_{\mu (p)}} )] = 0. \end{align*} The latter equality gives us that $\rotat \wedge dF (X_{df_{q}}) = 0$ so that
$X_{df_{q}}$ is a real multiple of $\xi (\rotat_{p^\prime} )$. To evaluate this constant, let $X_{df_{q}} = \gamma_p$ and define $r_p$ as
\begin{align} \label{eq:definer} d F ( \gamma_p ) = r_p \rotat_{p^\prime} . \end{align}
Now note that 
\begin{align} \label{eq:definer1}
\left< \rho_{p^\prime} , \rho_{p^\prime} \right> = \left< -2p^\prime \partial_z , -2 p^\prime \partial_z \right> = 4 \mu_\C (F (p) ) = 4 f (\mu (p)).
\end{align}
Also, using the defining property of moment maps in equation~\eqref{eq:defnmommap} for $\mu_\C$ one observes that the Hamiltonian vector field of $\mu_\C$ is $X_{\partial_r}$ which we have denoted $\rotat$. When identifying $\C$ with the real tangent space $T_{p^\prime} \C$ the inner product satisfies $\left< a , b \right> = \omega_{st} (a , i b )$ so that
\begin{align} \label{eq:definer2}
 \begin{split}
\left< dF (\gamma_p) , \rotat_{p^\prime} \right> & = \omega_{st} (d F (\gamma_p ), i \rotat_{p^\prime}), \\ & = \omega_{st} (\rotat_{p^\prime}, id F (\gamma_p )), \\ &  = \omega_{st} (\rotat_{p^\prime}, d F ( J \gamma_p ) ), \\ & = d\mu_{\C} \left(  d F ( J \gamma_p ) \right).
\end{split}
\end{align}
In the second to last equality, we used the fact that $F$ is holomorphic.

To evaluate $r_p$, we take the inner product with $\rotat$ on both sides of equation \eqref{eq:definer} and employ equations \eqref{eq:definer1} and \eqref{eq:definer2} to obtain
\begin{subequations} 
\begin{align*} r_p & = \frac{\left< dF (\gamma_p) , \rotat_{p^\prime} \right>}{\left< \rotat_{p^\prime} , \rotat_{p^\prime} \right>} ,  \\ 
& =  \frac{ d \mu_\C ( dF ( J \gamma_p))}{4 f (\mu (p))} ,  \\ 
& =  \frac{ d (\mu_\C \circ F) (J \gamma_p)}{ 4 f (\mu (p))} ,  \\ 
& =  \frac{ df_{\mu (p)} (d \mu_q (J X_{df_{q}}))}{4 f (q)} ,  \\ 
&=  \frac{ \left< df_{q } , df_{q} \right>_{\kappa_q}}{ 4f (q)} .
\end{align*}
\end{subequations} 
Letting $\delta_q = r_p^{-1} df_q$ then gives $dF (X_{\delta_q} ) = \rotat_{p^\prime}$, yielding the claim.
\end{proof}

Given any smooth function $\tilde{f} : \mu (\mcx) \to \mft$, the vector field
$X_{\tilde{f} (\mu (p))} (p)$ is easily integrated to $\phi^{\tilde{f}}_t : \mcx
\to \mcx$ where $\phi^{\tilde{f}}_t ( p ) = \exp (t \tilde{f} (\mu (p)))\cdot p$. Thus
the previous lemma gives an explicit description of the symplectic monodromy map of $F$. Namely, take $\tilde{f} = \left( \frac{4 f(q)}{|| df_{q} ||_{\kappa_q}^2} \right) df_q$, and for any
$\varepsilon >0$ the monodromy map is
\begin{equation*} \phi^{\tilde{f}}_1 : F^{-1} (\varepsilon ) \to F^{-1} (\varepsilon ). \end{equation*}
We utilize this to study the monodromy around a stable pair degeneration by first examining the monodromy with respect to the ambient toric variety and then perturbing this map slightly near the critical points of the degeneration to obtain a characterization of the monodromy on the hypersurface. Recall from Appendix~\ref{sec:torichd} that $A \subset \Lambda$ gives a subset of equivariant linear sections of a line bundle $\mco_A (1)$ on a toric stack $\mcx_Q$ specified by $(Q, A)$. Suppose $S = \{( Q_i, A_i ) \}_{i \in I}$ is a regular subdivision of $(Q, A)$ and $\eta : A \to \Z$ is an integral defining function of $S$. In Definitions \ref{defn:degfamily} we introduced the degenerating family $(\mcx_\eta, \mcy_s)$ which came equipped with a holomorphic function $F_\eta : \mcx_\eta \to \C$. We will explore two aspects of this definition, the symplectic structure as defined by the moment map and the holomorphic function $F_\eta$. 

To perform symplectic parallel transport around the critical value of a degenerating family, one must first choose a symplectic form on $\mcx_\eta := \mcx_{Q_\eta}$. We take the standard symplectic form $\omega$ on $\mcx_\eta$ defined in equation \eqref{eq:stsympform} for an arbitrary polyhedron. In the case of $Q_\eta$, we utilize the divisor associated to $\gamma_\eta = \sum_{b \in \du{Q}_\eta} n_b e_b \in \Z^{\du{Q}_\eta}$ where $n_b$ was defined in equation \eqref{eq:mcoq}. By definition, we have $\omega = \alpha_{\du{Q}_\eta}^\vee (\gamma_\eta )$. To shorten notation, we define\gls{nueta} the affine function
\begin{align} \label{eq:defnups} \nu_\eta := \beta_{\du{Q}_\eta}^\vee + \gamma_\eta .
\end{align}
The moment map $\mu_\eta : \mcx_{\eta} \to \Lambda_\R \oplus \R$ then can be found using diagram \eqref{diag:moment2} for the polyhedron $Q_\eta$ which is
\begin{equation} \label{diag:moment3}
\begin{CD} \C^{\du{Q}_\eta} @<{\text{inc}}<< \mu_H^{-1} (\omega ) @>{\rho_{A_\eta}}>> \mcx_\eta @= \mcx_\eta \\
@V{\mu_{\du{Q}_\eta}}VV @V{\mu_{\du{Q}_\eta}}VV @. @VV{\mu_\eta}V \\
\R^{\du{Q}_\eta} @=  \R^{\du{Q}_\eta} @<{\nu_\eta}<< \Lambda_\R \oplus \R @<{\text{inc}}<< Q_\eta .\end{CD} \end{equation}

Turning to the holomorphic function $F_\eta$, we review equation \eqref{eq:toricd} which defines $\tilde{F}_\eta (z_1,\ldots, z_n) : \C^{\du{Q}_\eta} \to \C$ as
\begin{align*}
\tilde{F}_\eta \left(z_1, \ldots, z_{|\du{Q}_\eta|} \right) = \prod_{i \in \du{Q}_\eta^v} z_i^{c_{\eta , i}}.
\end{align*}
Here we have made two implicit identifications. First, we identified the indexing set $I$ in the subdivision $S$ with $\du{Q}_\eta^v$. Second we identified $\du{Q}_\eta = \du{Q}_\eta^v \cup \du{Q}_\eta^h$ with $\{1, \ldots, |\du{Q}_\eta| \}$. We recall from equation \eqref{eq:defceta} that $c_{\eta, i}$ is the denominator of $d\varsigma_i$ where $\varsigma_i = \beta_{\du{Q}_\eta} (e_i)$ is the affine function which restricts to $\eta$ along $A_i$ as in Definition~\ref{lem:regsub:1}.

Using $\tilde{F}_\eta$, the function $F_\eta : \mcx_\eta \to \C$ associated to $\eta$ was defined by the diagram
\begin{equation} \label{diag:moment4}
\begin{CD} \mcx_\eta @<{\rho_H}<< \mu_H^{-1} (\omega ) \\
 @V{F_\eta}VV  @VV{\text{inc}}V \\ 
\C @<{\tilde{F}_\eta}<< \C^{\du{Q}_\eta}
\end{CD}
\end{equation}

Note that $\tilde{F}_\eta$ was explored in Example~\ref{eg:commutingmoment} and one can fill in diagram \eqref{diag:moment1} as
\begin{equation} \label{diag:moment5} \begin{CD} \C^{\du{Q}_\eta}  @>{\mu_{\du{Q}_\eta}}>> \R^{\du{Q}_\eta} \\
@V{\tilde{F}_\eta}VV @VV{\tilde{f}_\eta}V \\
\C @>{\mu_\C}>> \R \end{CD}
\end{equation}
where $\tilde{f}_\eta (r_1, \ldots, r_{|\du{Q}_\eta|}) = \prod_{i \in \du{Q}_\eta^v} r_i^{c_{\eta, i}}$. Letting $Y = \mu^{-1}_{L_{\du{Q}_\eta}} (\omega )$, we assemble the commutative diagrams \eqref{diag:moment3}, \eqref{diag:moment4} and \eqref{diag:moment5} into the diagram \eqref{diag:moment6} which defines\gls{feta} $f_\eta$.

\begin{equation} \label{diag:moment6}
\begin{tikzpicture}[baseline=(current  bounding  box.center), >=angle 90]
\matrix(a)[matrix of math nodes,
row sep=3.3em, column sep=2.75em,
text height=1.8ex, text depth=0.25ex]
{\mcx_\eta & & & \Lambda_\R \oplus \R \\
& Y & \R^{\du{Q}_\eta} & \\
& \C^{\du{Q}_\eta} & \R^{\du{Q}_\eta} & \\
\C & & & \R \\};
\path[->,font=\scriptsize] (a-1-1) edge node[above]{$\mu_\eta$} (a-1-4)
				edge node[left]{$F_\eta$} (a-4-1);
\path[->,font=\scriptsize] (a-2-2) edge node[below=1ex]{$\rho_\eta$} (a-1-1);
\path[right hook->,font=\scriptsize] (a-2-2) edge node[auto]{inc} (a-3-2);
\path[->] (a-2-2) edge (a-2-3);
\path[->,font=\scriptsize] (a-3-2) edge node[above]{$\mu_{\du{Q}_\eta}$} (a-3-3)
				edge node[above=.5ex]{$\tilde{F}_\eta$} (a-4-1);
\path[->,font=\scriptsize] (a-4-1) edge node[above]{$\mu_\C$} (a-4-4);
\path[left hook->,font=\scriptsize] (a-1-4) edge node[below=1ex]{$\nu_\eta$} (a-2-3);
\path[dotted,->,font=\scriptsize] (a-1-4) edge node[right]{$f_\eta$} (a-4-4);
\path[->,font=\scriptsize] (a-2-3) edge node[left]{$=$} (a-3-3);
\path[->,font=\scriptsize] (a-3-3) edge node[above=.5ex]{$\tilde{f}_\eta$} (a-4-4);
\end{tikzpicture}
\end{equation}

\begin{eg}
The most basic example of this setup is the degeneration of $\p^1$ into two projective lines intersecting in a node. To describe diagram \eqref{diag:moment6} for this case, we take $A = \{-1, 0, 1\} \subset \Z = \Lambda$ so that $Q = [-1, 1]$, $\mcx_Q = \p^1$ and $\mco_A (1) = \mco_{\p^1} (2)$. Define the degeneration by taking $\eta : A \to \Z$ to be the defining function given by $e_{(1,1)}^\vee \in \left( \Z^{{\mathcal{A}}} \right)^\vee$ (i.e. the function taking $-1$ and $0$ to $0$ and $1$ to $1$). Then $Q_\eta$ is illustrated in Figure \ref{fig:p1deg} and one obtains $\du{Q}_\eta = \du{Q}_\eta^v \cup \du{Q}_\eta^h =  \{(0,1), (-1,1)\} \cup \{(1, 0), (-1,0)\} $. It is not hard to check that $\mcx_\eta$ is isomorphic to the blowup of $\p^1 \times \C$ at $([1:0], 0)$. 

While one can work out the inner square of diagram \eqref{diag:moment6}, the details of the computation do not reveal much insight into the geometry. However, one obtains that $F_\eta$ is simply the blow down map composed with the projection $([a,b], z) \mapsto z$. Moreover, using the coordinates $(s,t)$ in Figure \ref{fig:p1deg}, one computes $f_\eta( s, t  ) = t (s - t)$. Observe that this is identically zero on the lower boundary of $Q_\eta$ which is the image of the degenerate fiber, namely the total transform of $\p^1 \times \{0\}$. The level sets $f_\eta^{-1} (\varepsilon ) \cap Q_\eta$ are the moment map images of the fibers of $F_\eta$ lying over a radius $\sqrt{\varepsilon}$ circle. Moreover, Lemma \ref{lem:moment} shows that the symplectic monodromy about such a circle can be described in terms of the torus action on $\mcx_{Q_\eta}$ using the infinitesimal vector fields associated to the derivative of $f_\eta$.
\end{eg}
\begin{figure}[h]
\begin{picture}(0,0)%
\includegraphics{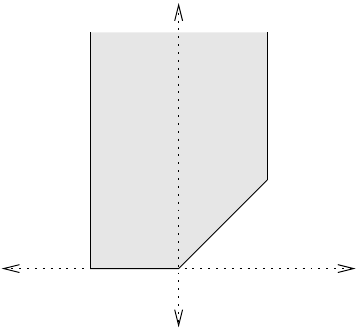}%
\end{picture}%
\setlength{\unitlength}{4144sp}%
\begin{picture}(2724,2499)(1564,-2773)
\put(3376,-2086){\makebox(0,0)[lb]{$\Lambda_\R \oplus \R$}%
}
\put(2476,-1186){\makebox(0,0)[lb]{$Q_\eta$}%
}
\put(4276,-2486){\makebox(0,0)[lb]{$s$}%
}
\put(3076,-386){\makebox(0,0)[lb]{$t$}%
}
\end{picture}%
\caption{\label{fig:p1deg} The polyhedron $Q_\eta$ for a degeneration of $\p^1$.}
\end{figure}
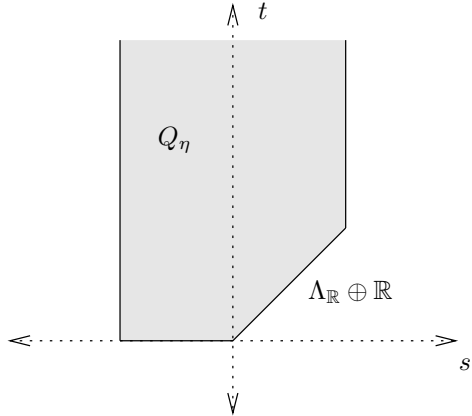
The following proposition gives the general description of $f_\eta$ as well as the monodromy of $F_\eta$ around $0$.

\begin{prop} \label{prop:toricmonodromy} Let $(\mathbf{r} , t)$ be coordinates for $\Lambda_\R \oplus \R$. Then $f_\eta$ can be written as
\begin{equation*} f_\eta (\mathbf{r} , t) = \prod_{i \in I} [c_{\eta, i} (t - \varsigma_i (\mathbf{r} ))]^{c_{\eta , i}} . \end{equation*}
The normalized derivative 
\begin{align} \label{eq:tormono2} \frac{4 f(\mathbf{r} , t)}{||df_{(\mathbf{r}, t)}||_{\kappa (\mathbf{r} , t)}^2} df_{(\mathbf{r} , t)}  \end{align} converges uniformly to $(dt - d\varsigma_i )$ on compactly supported subsets of the interior of $Q_{\eta , i}$.
\end{prop}
\begin{proof} For the first part of the claim, we re-examine the map $\nu_\eta$ defined in equation \eqref{eq:defnups}. Recall that $\beta_{\du{Q}_\eta} : \Z^{\du{Q}_\eta} \to \Lambda^\vee \oplus \Z^\vee$ was the tautological map  $\beta_{\du{Q}_\eta} (e_b ) = b $ for every minimal supporting hyperplane of a facet of $Q_\eta$. Now, for every $i \in I$, let $b_i$ be the supporting hyperplane for the lower boundary facet $Q_{\eta, i}$ which is the marked facet of $(Q_\eta , A_\eta )$ over $Q_i$ and write $e_{b_i} \in \Z^{\du{Q}_\eta}$ as $e_i$. By definition, $\tilde{\eta}$ restricts to the affine function $\varsigma_i$ which is the sum $d \varsigma_i - m_i$ where $d \varsigma_i \in \Lambda_\Q^\vee$ is the derivative, or linear part, of $\varsigma_i$ and $m_i \in \Q$. 
By the construction of $Q_\eta$ and property~\ref{lem:regsub:2} of $\varsigma_i$, we have $\eta ( a ) \geq \varsigma_i (a  )$ for all $a \in A$ with equality if and only if $a \in A_i$.  Taking $h^\vee = (0,1)^\vee \in \Lambda^\vee \oplus \Z^\vee$, for any $a \in A$, $r \in \R_{\geq 0}$ we have 
\begin{align*}
(h^\vee - d \varsigma_i ) (a, \eta (a) + r) & = \eta (a) + r - ( d \varsigma_i (a) - m_i) -  m_i ,\\ 
& = \eta (a) - \varsigma_i (a) + r -  m_i ,  \\ & \geq  r - m_i \geq - m_i .
\end{align*}
Equality is achieved if and only if $r = 0$ and $a \in A_i$. This implies that $h^\vee - d \varsigma_i$ is a supporting hyperplane for $Q_{\eta, i}$. However, only after multiplying by $c_{\eta, i}$ can we ensure that it is contained in $\Lambda^\vee \oplus \Z^\vee$, so that $\beta_{\du{Q}_\eta} (e_i ) =  c_{\eta, i} (h^\vee - d \varsigma_i)$.  We also see from this argument that $n_{b_i} = c_{\eta, i} m_i$ for every $i \in I$. In turn, this gives
\begin{align*}
\gamma_\eta = \sum_{i \in I} c_{\eta, i} m_i e_i + \sum_{j \in \du{Q}_\eta^h} n_j e_j \in \Z^{\du{Q}_\eta}.
\end{align*}
Therefore, for any $i \in I$, we have 
\begin{align*} e_i^\vee \circ \nu_\eta  & =  e_i^\vee \circ \beta_{\du{Q}_\eta}^\vee + e_i^\vee(\gamma) , \\ & =  \beta_{\du{Q}}^\vee (e_i )  + c_{\eta , i} m_i , \\ & =  c_{\eta , i} [(h^\vee - d \varsigma_i)]  + c_{\eta , i} m_i ,\\ & = c_{\eta , i} ( h^\vee - \varsigma_i ). \end{align*}
But the function $r_i : \R^{\du{Q}_\eta} \to \R$ is induced from $e_i^\vee$ so that the composition $r_i \circ \nu_\eta = e_i^\vee \circ \nu_\eta = c_{\eta , i} ( h^\vee - \varsigma_i )$ which, as a function on $\Lambda_\R \oplus \R$, we simply write as $c_{\eta, i} (t - \varsigma_i (\mathbf{r}))$. The formula for $f_\eta := \tilde{f}_\eta \circ \nu_\eta$ then follows from that for $\tilde{f}_\eta$ following diagram \eqref{diag:moment5}. 

We use this and the convexity of $\tilde{\eta}$, defined in equation \eqref{eq:tildeeta}, to for the second claim. Before proving this though, we take
$\tilde{\kappa}$ and $\kappa$ to be the pairings from equation \eqref{eq:momentkappa} for the actions of the tori $\mbbT_{\Z^{\du{Q}_\eta}}$ and $\mbbT_{\Lambda \oplus \Z}$ on $\C^{\du{Q}_\eta}$ and $\mcx_{Q_\eta}$ respectively. Observe that, letting $\mathbf{R} = (r_1 , \ldots, r_{|\du{Q}_\eta |})  \in \R^{\du{Q}_\eta}$, the map $\tilde{\kappa}$ for $\C^{\du{Q}_\eta}$ is
\begin{equation*} \tilde{\kappa}_{\mathbf{R}} = \left[ \begin{matrix} 4 r_1 & 0  &  \cdots & 0 \\ 0 & \ddots & \ddots & \vdots \\ \vdots & & \ddots  & 0 \\ 
0 & \cdots & 0 & 4 r_{|\du{Q}_\eta|} \end{matrix} \right] . \end{equation*}
More succinctly, we have that $\tilde{\kappa}_\mathbf{R} (dr_i \otimes dr_j) = \delta_{ij} 4r_i$. To see how this induces $\kappa$, we first note that the map $\nu_\eta : \Lambda_\R \oplus \R \to \R^{\du{Q}_\eta}$ gives identification $r_i = c_{\eta, i} (t - \varsigma_i)$. 
If $p \in \mu_\eta (\mcx_\eta)$ and $s_1, s_2 \in (\Lambda_\R \oplus \R)^\vee$ then $\kappa_p (s_1 \otimes s_2) = \tilde{\kappa}_p (\tilde{s}_1 \otimes \tilde{s}_2)$ where $\tilde{s}_i \in (\R^{\du{Q}_\eta})^\vee$ for satisfy $\tilde{s}_i \circ \nu_\eta = s_i$ and $\tilde{\kappa}_p (\tilde{s}_i \otimes \tilde{s}) = 0$ for all $\tilde{s} \in \ker (\beta_{\du{Q}_\eta})$. In other words, over every $p \in \mu_\eta$, one can find a linear splitting $\tau_p : (\Lambda \oplus \R)^\vee \to (\R^{\du{Q}_\eta})^\vee$ of $\beta_{\du{Q}_\eta}^\vee$ for which 
\begin{align} \label{eq:kaptotildkap} \kappa_p (s_1 \otimes s_2) = \tilde{\kappa}_p ( \tau_p(s_1) \otimes \tau_p (s_2)). \end{align}

For any polytope $Q_{\eta, i}$, the image $\nu_\eta (Q_{\eta, i})$ lies on the boundary $r_i = 0$ of $\mu_{\du{Q}_\eta} (\C^{\du{Q}_\eta})$. Now, let $\tilde{U} \subset Q_{\eta, i}$ be any neighborhood away from the boundary and  $U \subset Q_\eta$ an normal $\varepsilon$ tubular neighborhood of $\tilde{U}$. We choose a continuous splitting function $\tau : U \to \Hom ((\Lambda \oplus \R)^\vee, (\R^{\du{Q}_\eta})^\vee )$ so that
\begin{align} \label{eq:tausplit} \tau_p \left( c_{\eta, j} d(t - \varsigma_j) \right) = \sum_{k = 1}^{|{\du{Q}_\eta}|} g_{j,k} dr_k .
\end{align}
Since $\tilde{\kappa}|_{r_i = 0}$ contains $dr_i$ in its null space, we may choose  $g_{j, k}$ to be continuous functions on $U$ so that $g_{j, i}|_{Q_{\eta, i}} = 0$ for every $j \ne i$ and $g_{i, i} |_{Q_{\eta, i}} = 1$. 

Using formulas \eqref{eq:kaptotildkap} and \eqref{eq:tausplit}, we compute
\begin{align} \label{eq:kapform}
\kappa_p ( c_{\eta, j} d(t - \varsigma_j) \otimes  c_{\eta, k} d(t - \varsigma_k) ) = \sum_{l \in I} 4 (c_{\eta,l} (t - \varsigma_l)) g_{j,l} g_{k, l}. 
\end{align}

We now calculate
\begin{align*} d \tilde{f}_\eta (\mathbf{R}) = \tilde{f}_\eta (\mathbf{R} )  \cdot \sum_{j \in I}  \frac{c_{\eta , j}}{r_j} d r_j . \end{align*}
Using the fact that $f_\eta = \tilde{f}_\eta \circ \nu_\eta$, we have
\begin{equation*} d f_\eta (\mathbf{r},t) = f_\eta (\mathbf{r},t) \cdot \sum_{j \in I} \frac{1}{ t - \varsigma_j} d \left(c_{\eta, j} (t - \varsigma_j) \right)  .\end{equation*}
Using \eqref{eq:kapform}, we compute the following norm on $U$ as a meromorphic function in $(t- \varsigma_i)$ as
\begin{align*}
\left\| \sum_{j \in I} \frac{1}{ t - \varsigma_j} d \left(c_{\eta, j} (t - \varsigma_j) \right) \right\|^2_{\kappa (\mathbf{r}, t)} & = 4 \sum_{j,k,l \in I} \frac{c_{\eta,_l} (t - \varsigma_l)}{(t - \varsigma_j)(t - \varsigma_k)} g_{j,l} g_{k,l}, \\ & = 
4  \frac{c_{\eta,_i} }{t - \varsigma_i} + O ((t - \varsigma_i)^0).
\end{align*}
Note that while there are poles of this function on the other boundary facets $Q_{\eta, j}$, we have chosen $U$ to be disjoint from these so that the only pole is the first order pole at $t = \varsigma_i$. We utilize this to compute
\begin{align*} \frac{ 4 f_\eta (\mathbf{r} , t)}{||df_\eta (\mathbf{r}, t)||_{\kappa (\mathbf{r}, t)}^2} df_\eta {(\mathbf{r} , t)} & = \frac{4}{\left\| \sum_{j \in I} \frac{1}{ t - \varsigma_j} d \left(c_{\eta, j} (t - \varsigma_j) \right) \right\|^2_{\kappa (\mathbf{r}, t)}} \sum_{j \in I} \frac{c_{\eta , j}}{t - \varsigma_j } d (t - \varsigma_j ) , \\ & =  \frac{1}{\frac{c_{\eta,_i} }{t - \varsigma_i} + O ((t - \varsigma_i)^0)} \sum_{j \in I} \frac{c_{\eta , j}}{t - \varsigma_j } d (t - \varsigma_j ) , \\
& = \sum_{j \in I} \frac{c_{\eta, j}(t - \varsigma_i)}{c_{\eta,i}(t - \varsigma_j)( 1  + O ((t - \varsigma_i))}   d (t - \varsigma_j ) \end{align*}
From our previous observations that   $(t - \varsigma_j)|_{U} \ne 0$ for $j \ne i$, we have that the limit of the normalizing derivative on the interior of $Q_{\eta,i}$ is
\begin{align*} \lim_{\varsigma_i \to t} \frac{ 4 f_\eta (\mathbf{r} , t)}{||df_\eta (\mathbf{r}, t)||_{\kappa (\mathbf{r}, t)}^2} df_\eta {(\mathbf{r} , t)} & = d (t - \varsigma_i ) . \end{align*}
\end{proof}

We now explain the meaning of this proposition in the form of a corollary. Combined with Lemma \ref{lem:moment}, the first part of the proposition gives an explicit formula for monodromy of $\mcx_Q$ about a toric degeneration. To understand the statement, we decompose the fiber $F_\eta^{-1} (\varepsilon) \cong \mcx_Q$ by taking symplectic parallel transport along $F_\eta$ to $0$. The degenerate fiber $F_\eta^{-1} (0)$ is the union $\cup_{i \in I} \mcx_{Q_i}$ of its irreducible components. We write $U_i \subset F_\eta^{-1} (\varepsilon )$ for the set which converges to $\mcx_{Q_i}$ under parallel transport along the positive real axis.

\begin{cor} \label{cor:subdivision} For a regular subdivision $\eta$ and any sufficiently small $\delta > 0$, there is an induced decomposition $\mcx_Q = \cup_{i \in I} U_i$ such that degeneration monodromy relative to $\eta$ is an interpolation of toric multiplications $ \exp (-  d \eta_{Q_i})$ on each $U_i$ along  $\delta$ neighborhoods of their intersections.
\end{cor}
 
\begin{proof}
Choosing $\varepsilon > 0$, the fiber $F_\eta^{-1} (\varepsilon )$ is isomorphic to  $\mcx_Q$ and the inverse image of the circle $F_\eta^{-1} (\varepsilon S^1 )$ is preserved under flowing with respect to $\xi (\rotat )$. The time $- \pi$ flow for $\xi (\rotat )$ sends $F_\eta^{-1} (\varepsilon )$ to itself and yields the symplectic monodromy map (as it lifts the time one flow of $\rotat = -2 i z \partial_z$ ). Proposition \ref{prop:toricmonodromy} gives an explicit expression for $\xi (\rotat )$ as the normalized derivative in equation \eqref{eq:tormono2}. This is a map from $F_\eta^{-1} (\varepsilon S^1 ) \to \mft_{\Lambda \oplus \Z}$ (then composed with the map taking $v \in \mft_{\Lambda \oplus \Z}$ to its infinitesimal vector field $X_v$).  Exponentiating and evaluating at time $-\pi$ gives a map $\exp_\varepsilon : F_\eta^{-1} (\varepsilon ) \cong \mcx_Q \to \mbbT_{\Lambda}$ (since it preserves $F_\eta^{-1} (\varepsilon)$, the additional circle action is constant). Symplectic monodromy around $0$ is then given by the action $x \mapsto \exp_{\varepsilon} (x) \cdot x$.  

Fixing a small $\delta > 0$, take $\tilde{V}_i$ be the open set in $Q_i$ consisting of points which have distance greater than $\delta$ from $ \partial Q_i$. Take set $V_i$ of points in $F^{-1}_\eta (\varepsilon )$ which flow to $\tilde{V}_i$. It is clear that $V_i \subset U_i$ and that $V_i$ can be identified with the complement of a neighborhood of the boundary in $\mcx_{Q_i}$. By Proposition \ref{prop:toricmonodromy}, the monodromy $\exp_{\varepsilon} (x)$ uniformly converges to $\exp  ( {\eta, i}^{-1} (dt - d \varsigma_i))$ on $Q_{\eta, i}$ as $\varepsilon$ tends to $0$. As $\exp (dt)$ acts as the identity, this multiplication converges to $\exp (-  d \varsigma_i) = \exp (- d \eta_{Q_i}) $ which is multiplication by a constant in the maximal torus acting on $V_i \subset \mcx_{Q_i}$. Thus, conjugating by symplectic flow from $F_\eta^{-1} (\varepsilon )$ to $F_\eta^{-1} (0)$ on the domains $V_i$, we obtain a representation of symplectic monodromy as in the corollary. The fact that these interpolate over their boundaries follows from the representation of monodromy as $\exp_\varepsilon (x) \cdot x$ and the continuity of $\exp_\varepsilon (x)$ in equation \eqref{eq:tormono2}.
\end{proof}

Thus we find that the symplectic operation of parallel transport, which is very far from being holomorphic, limits to a holomorphic map on the components of the degeneration.

\begin{figure}[t]
\begin{picture}(0,0)%
\includegraphics{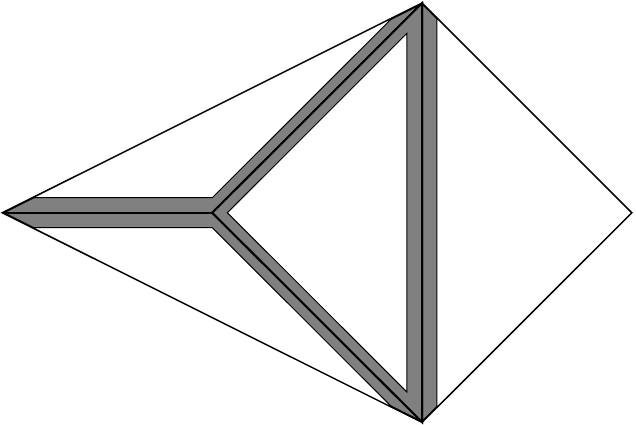}%
\end{picture}%
\setlength{\unitlength}{4144sp}%
\begingroup\makeatletter\ifx\SetFigFont\undefined%
\gdef\SetFigFont#1#2#3#4#5{%
  \reset@font\fontsize{#1}{#2pt}%
  \fontfamily{#3}\fontseries{#4}\fontshape{#5}%
  \selectfont}%
\fi\endgroup%
\begin{picture}(4834,3235)(3122,-3717)
\put(4171,-1846){\makebox(0,0)[lb]{\smash{{\SetFigFont{12}{14.4}{\rmdefault}{\mddefault}{\updefault}{$U_1$}%
}}}}
\put(4156,-2536){\makebox(0,0)[lb]{\smash{{\SetFigFont{12}{14.4}{\rmdefault}{\mddefault}{\updefault}{$U_2$}%
}}}}
\put(5386,-2161){\makebox(0,0)[lb]{\smash{{\SetFigFont{12}{14.4}{\rmdefault}{\mddefault}{\updefault}{$U_3$}%
}}}}
\put(6721,-2176){\makebox(0,0)[lb]{\smash{{\SetFigFont{12}{14.4}{\rmdefault}{\mddefault}{\updefault}{$U_4$}%
}}}}
\end{picture}%
\caption{\label{fig:monodromy} Regions of finite order monodromy}
\end{figure}
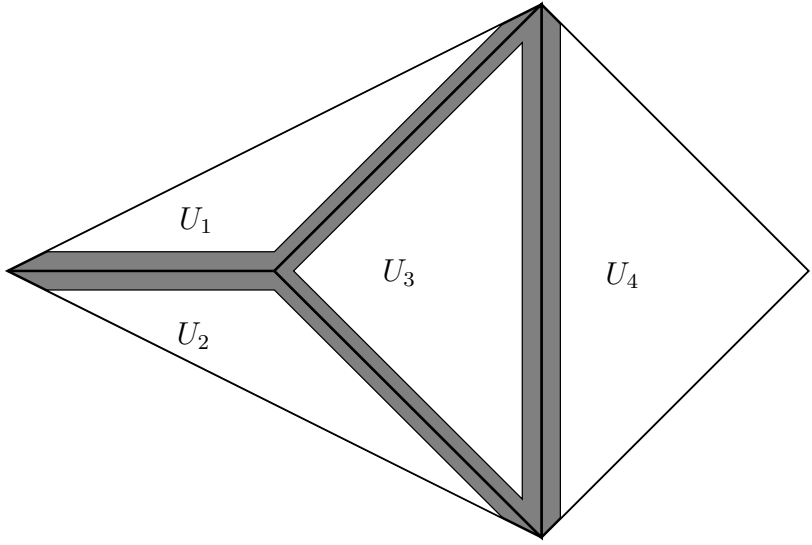

To obtain the structure of parallel transport on the hypersurface, we simply define an appropriate perturbation of these maps which preserve the hypersurface. This is a less elegant approach than the straightening method of \cite[Appendix~A]{abouzaid}, but one which works for arbitrary stable pair degenerations and yields a description that is Hamiltonian isotopic in the case of a degeneration resulting from a triangulation of $(Q,A)$. We only need to assume that the defining section $s$ is in the complement of the principal $A$-determinant. As the hypersurfaces are fixed by the limit of the monodromy maps $\exp_0 (x)$ in the degenerate fiber $F_\eta^{-1} (0)$, the ambient toric monodromy approximates the hypersurface map up to a negligible factor along their intersections.

To describe the hypersurface monodromy map, let $\fib{\eta}{0} = \cup_{i \in I} \fib{i}{0}$ be the components of the degenerated hypersurface from Definition \ref{defn:degfiber} and $g_i : \fib{i}{t} \to \fib{i}{t}$ the K\"ahler automorphism corresponding to $\exp (-d \eta_{Q_i})$. 
\begin{prop} \label{prop:hypdeg} There exists a decomposition of $\fib{\eta}{t} = \cup_{i \in I} \overline{V_i} $ such that $Z_i \approx \fib{i}{0} - \partial \fib{i}{0}$ and the monodromy map $\phi_\eta : \fib{\eta}{t} \to \fib{\eta}{t}$ equals $g_i$ on $Z_i$ off a $\varepsilon$ neighborhood $Z_i (\varepsilon)$ of $\partial Z_i$, and is interpolated smoothly over $Z_i (\varepsilon)$ by a Hamiltonian flow.
\end{prop}
\begin{proof}
Given Corollary~\ref{cor:subdivision}, we need only show that the action $\exp (-d \eta_{Q_i})$ preserves the hypersurface $ \fib{i}{0}$ for every $i \in I$. This follows from the observation that $\eta_{Q_i}$ is an affine function on $\lin_\Z (A_i)$ and $\fib{i}{0}$ is defined by sections in $A_i$. Multiplication of the section $z^a$ by $\exp (-d \eta_{Q_i})$ is given by $\exp (-d \eta_{Q_i} (2 \pi a)) z^a = z^a$ so that the section defining $\fib{i}{0}$ is fixed implying that the hypersurface is preserved as well.
\end{proof}

\subsection{\label{subsec:stm}Stratified Morse singularities}

In \cite{seideldehn}, it was seen that symplectic monodromy around a Morse singularity has infinite order in the symplectic mapping class group for any dimension.
In this paper, these types of singularities are encountered as a non-degenerate case. For the degenerate case, we need a different model whose critical fiber is in fact smooth, but fails to transversely intersect the boundary divisor. Restricting to the boundary divisor, we see a Morse singularity and expect that the monodromy on the ambient space extends the monodromy of the restriction.

We have one essential obstruction to pursuing this naively.
Namely, if our parallel transport map preserves a boundary divisor $D$ in $\mcy$, then $D$ must be horizontal relative to the symplectic orthogonal connection. On the other hand, if a smooth fiber does not intersect $D$ transversely at $p$, then the symplectic orthogonal will be normal, or vertical, to $D$. This holds for all symplectic connections in $\Omega^2 (\mcy )$. We resolve this difficulty by considering a singular connection on $D$ and show that the parallel transport vector field extends over the singularities and preserves the symplectic form of the fiber up to a negligible factor.

Let $U \subset \C^n$ to be a neighborhood of zero, $D (m) = \{z_1 z_2 \cdots z_m = 0 \} \cap U$ and $\maxint = \cap_{i = 1}^m D_i \cap U$. Given a linear function $L : \C^m \to \C$ with nonzero restrictions to each coordinate line, we let $f_L : U \to \C$ be the function
\begin{equation*} f_L (z_1, \ldots , z_n) = L (z_1, \ldots, z_m ) + \frac{1}{2}\left(
z_{m + 1}^2 + \cdots + z_n^2 \right) .\end{equation*}
While this map is smooth, the fiber over zero does not transversely intersect the divisor $D(m)$ along $\maxint$ at $0$. 

For the following definition, let $G$ be a subgroup of the unit circle $\mbbT \subset \C$ and $[\C / G]$ the quotient orbifold. We take $\psi_G : \C \to [\C / G]$ to be the orbifold chart of $[\C / G]$. 
\begin{defn} Let $\mcy$ be a symplectic orbifold with normal crossing divisor $D = \cup_{i = 1}^m D_i$, $p \in \cap_{i = 1}^m D_i$ and  $f: \mcy \to [\C / G]$ a map of orbifolds with $f(p) = 0$. We say that $f$ is a stratified Morse function at $p$ relative to $D$ if there exists a holomorphic orbifold chart $\phi :((U, D(m)), G_U) \to (\mcy, D)$ centered at $p$, a homomorphism $g : G_U \to G$, a $(G_U, G)$-equivariant linear function $L : \C^m \to \C$ such that $f$ lifts to the $(G_U, G)$-equivariant function $ f_L : U \to \C$. In this case we say that $f$ is stratified with codimension $m$, $p$ is a degenerate point of $f$ and that $f(p)$ is a degenerate value of $f$.
\end{defn}
We will concentrate on the case where $G$ and $G_U$ are trivial, as the general orbifold case will be an equivariant quotient thereof.  In the non-stratified setting we have a useful criterion for deciding when a function is Morse. A similar tool, whose proof is straightforward, is available in the stratified case.
\begin{prop} \label{prop:str} Let $U \subset \C^n$, $D = \cup_{i = 1 }^m D_i$,  $\maxint = \cap_{i = 1}^m D_i \cap U$ and $f : U \to \C$ be a holomorphic function. Then $f$ is a stratified Morse function at $0$ relative to $D$ with codimension $m$ if and only if the following conditions are satisfied 
	\begin{enumerate} 
		\item $df_0 |_{C} \ne 0$ on any coordinate subspace $C = \cap_{i \in J} D_i$ for which $J \subsetneq \{1, \ldots, m\}$, 
		\item $df_0 (T_0 \maxint ) = 0$, 
		\item $\textnormal{Hess}_0 (f)$ is non-degenerate on $T_0 \maxint$.
	\end{enumerate}
\end{prop}
\begin{proof}
	Applying the complex Morse Lemma to $f|_{D_{[m]}}$ and, inductively applying the Implicit Function Theorem for  coordinate planes containing $D_{[m]}$ yields the result.
\end{proof}
Let $\disc_{\epsilon}$ be an $\epsilon$ radius disc about the origin in $\C$ and $\tilde{U} = f_L^{-1} (\disc_{\epsilon } ) \subset \C$. As was pointed out above, the symplectic orthogonal connection on $f_L : \tilde{U} \to \disc_\epsilon$ needs to be corrected in order to preserve the boundary $D (m)$. We implement a form of Moser's trick  by integrating a path of equivalent symplectic forms, perform parallel transport relative to the corrected form and then flow back to the standard form.

\begin{figure}[h]
\begin{picture}(0,0)%
\includegraphics{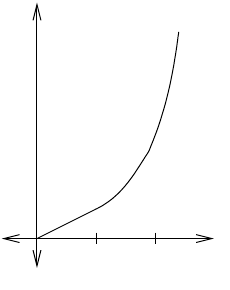}%
\end{picture}%
\setlength{\unitlength}{4144sp}%
\begin{picture}(1737,2110)(2644,-2609)
\put(3286,-1366){\makebox(0,0)[lb]{\smash{$\rho_\varepsilon (r)$}}}
\put(4366,-2401){\makebox(0,0)[lb]{\smash{$r$}}}
\put(3336,-2536){\makebox(0,0)[lb]{\smash{$\varepsilon$}}}
\put(3781,-2536){\makebox(0,0)[lb]{\smash{$2\varepsilon$}}}
\end{picture}%
\caption{\label{fig:rhoeps} The function $\rho_\varepsilon$.}
\end{figure}
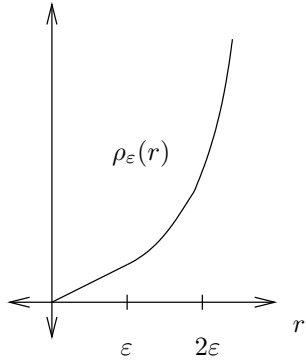

We define a smoothly varying collection $\{\rho_\varepsilon\}_{1 \geq \varepsilon > 0}$ of functions where $\rho_\varepsilon : \R_{\geq 0} \to \R_{\geq 0 }$ is a smooth convex function which satisfies
\begin{equation*}  \rho_\varepsilon (r ) = \left\{ \begin{matrix} \frac{\varepsilon^2}{4}  r & \text{ for } r < \varepsilon, \\ \frac{1}{4}r^2 & r > 2 \varepsilon .\end{matrix} \right. \end{equation*}
The function $\rho_\varepsilon$ is illustrated in Figure \ref{fig:rhoeps}. Then define $\omega_\varepsilon$ on $\C^n$ to be the symplectic form obtained by the K\"ahler potential
\begin{equation} \label{eq:rhoeps} p_\varepsilon (z_1, \ldots, z_n ) = \sum_{ i = 1}^m \rho_\varepsilon (|z_i| ) + \frac{1}{4} \sum_{i = 1}^{m + 1} |z_i|^2 . \end{equation}
It is clear that $\omega_\varepsilon$ is a smooth symplectic form away from $D(m)$ and singular on $D(m)$. While singular, an application of Stokes' Theorem shows that for any disc $\Sigma$ with boundary outside of a radius $2$-neighborhood of $D(m)$ and any $\varepsilon \leq 1$, the integral $\int_\Sigma \omega_\varepsilon$ is finite and independent of $\varepsilon$. Indeed, if $\Sigma$ is such a disc, then we may perturb its interior so that it transversely intersects $D (m)$ implying that each intersection point is in $D_i$ for some $1 \leq i \leq m$. Again, after perturbing, we may assume that $\Sigma$ is orthogonal to $D_i$ which reduces the computation to the one dimensional case. Note that, as $\omega_\varepsilon$ is an exact symplectic form off of $D(m)$ and we have kept the outer boundary fixed, these perturbations do not affect $\int_{\Sigma} \omega_\varepsilon$. To check the assertion in the one dimensional case, assume $\Sigma \in \C$ contains the origin and, for $\delta < \varepsilon$ let $\Sigma_\delta = \Sigma - \disc_\delta$. The boundary of $\Sigma_\delta$ then consists of an outer closed curve $\mathcal{C}_o$ which we assume to be outside a disc of radius $2$, and the inner closed curve $\mathcal{C}_i$ where $\mathcal{C}_i$ is a circle of radius $\delta$ about the origin. Take $\lambda_\varepsilon = -d^c \rho_\varepsilon (|z|)$ to be the Liouville form. Note that in an $\varepsilon$-neighborhood of the origin,
\begin{align}
\omega_\varepsilon =  d  \lambda_\varepsilon = \frac{\varepsilon^2}{4} d \left( \frac{x dy - y dx}{\sqrt{x^2 + y^2}} \right) .
\end{align}
Thus if $\mathcal{C}_i$ is parameterized by $(\delta \cos (\theta ), \delta \sin (\theta ))$ then  $\lambda_\varepsilon |_{\mathcal{C}_i} = (\varepsilon^2 \delta)/4 d \theta$. On the other hand, since $\mathcal{C}_o$ lies outside of the radius $2$ neighborhood of $0$, we have that $\lambda_\varepsilon |_{\mathcal{C}_o}$ is independent of $\varepsilon$. Thus,
\begin{align*}
\int_{\Sigma} \omega_\varepsilon & = \lim_{\delta \to 0} \int_{\Sigma_\delta} \omega_\varepsilon ,  \\ 
& =  \lim_{\delta \to 0} \int_{\partial \Sigma_\delta} \lambda_\varepsilon , \\
& = \lim_{\delta \to 0} \left( \int_{\partial \mathcal{C}_o} \lambda_\varepsilon  - \int_{\mathcal{C}_i} \lambda_\varepsilon \right), \\
& = \int_{\partial \mathcal{C}_o} \lambda_\varepsilon  -  \lim_{\delta \to 0} \left( \frac{\varepsilon^2 \delta}{4}  \int_0^{2\pi} d \theta \right), \\ & = \int_{\partial \mathcal{C}_o} \lambda_\varepsilon .
\end{align*}
%
%
This verifies that the relative cohomology class of $\omega_\varepsilon$ is constant in a radius $2$-neighborhood of $D(m)$ as $\varepsilon$ varies. It is this fact that hints towards a Moser argument relating the standard symplectic structure with $\omega_\varepsilon$. In particular, Let $X_\varepsilon$ be the vector field which is the $f_L$ fiberwise $\omega_\varepsilon$ dual of the $1$-form $\frac{d}{d t} \left( d^c p_t \right)_{t = \varepsilon}$. The vector field $X_\varepsilon$ is smooth off of $D (m)$ and extends continuously to $\C^n$ by letting it equal zero on $D(m)$. Furthermore, the derivative of $X_\varepsilon$ is bounded on $\C^n$. This implies that the time varying vector field $X_t$ may be integrated on $\C^n$ for time $t  \in \R_{\geq 0}$. Recalling that $\tilde{U} = f_L^{-1} (B_\epsilon )$, and noting that, by definition, $X_t$ is tangent to the fibers of $f_L$, we may restrict the $X_\varepsilon$ flow to  $\tilde{U}$. We denote by $\Phi_{s} : \tilde{U} \to \tilde{U}$ the singular symplectomorphism obtained through integrating $\{X_t\}$ to time $t = s$. 

\begin{defn} For any stratified Morse function $f : \mcy \to \C$ with a local model $U \subset \C^n$ and $f_L : U \to \C$, we will call conjugation of parallel transport around $0$ relative to $\omega_1$ by $\Phi_1$ \emph{modified symplectic parallel transport} relative to $U$.
\end{defn}

We now investigate the local behavior of modified symplectic parallel transport for $f_L : \C^n \to \C$. As the computation is local, we may extend the symplectic form $\omega_1$ near $0$ to one over all of $\C^n$, differing from $\omega_1$ only outside of a neighborhood of $0$. From the definition of $\omega_\varepsilon$ using the potential  $p_\varepsilon$ in equation \eqref{eq:rhoeps}, we note that near $D (m) = \cup_{i = 1}^m$, the symplectic form for $\varepsilon = 1$ is 
\begin{equation*} \omega = \frac{i}{4}  \sum_{i = 1}^m
\frac{dz_i \wedge d \bar{z}_i}{ |z_i|}  + \frac{i}{2} \sum_{i = m + 1}^n dz_i \wedge d \bar{z}_i . \end{equation*}
Sufficiently far away from $D (m)$, $\omega_1$ is the standard symplectic form and it interpolates between $\omega$ and $\omega_{st}$. We may thus use $\omega$ for the local model.

Let $L (z_1, \ldots , z_m ) = c_1 z_1 + \cdots c_m z_m$ and note that we may change coordinates by multiplying $z_i$ with $e^{- \text{arg} (c_i) i}$ without affecting the map $\Phi$, so that we may assume $c_i \in \R_{+}$. In the following computation,  we will examine only the case where $c_i = 1$ for every $m + 1 \leq i \leq n$ and write $f$ for $f_L$. The case of a more general linear function $L$ only affects the isotopy class of the modified parallel transport map if a small $\partial$-frame group $\mbF$ is considered for which rotations about the boundary in $\Symp^\mbF (\mcy)$ do not exist. In less restrictive frame groups, for example if $\mbF$ is the image of all symplectomorphisms preserving the boundary divisor under $j^\#$, we may isotope to this case.

Our goal will be to understand parallel transport around $0$. As a
first step, let $\gamma_{vc} : \R_{> 0} \to \C$ be the path $\gamma_{vc} (t) = t$ and examine the flow of the parallel transport vector field which lifts $-\partial_z$. Let
$F_q$ be the fiber of $f$ over $q \in \C$ and $\phi_t : F_q \to
F_{q - t}$ be the parallel transport map for $q > t$. Define
\begin{align} \label{eq:defnvt} \vt^\circ & = \left\{z \in f^{-1} (\R_{> 0} ) : \lim_{ t \to f (z)}
\phi_t (z ) = 0 \right\}, \\ \label{eq:defnvc} \vc^\circ & = \left\{ z \in F_1 : \lim_{t \to 1}
\phi_t (z ) = 0 \right\}, \end{align} 
to be the open vanishing thimble and cycle, respectively, of $f$. A priori,  parallel transport can only be defined where $\omega$ is non-singular, so $\vt^\circ \subset \C^n - D (m)$ and $\vc^\circ \subset F_1 - \left(F_1 \cap D(m)\right)$. The vanishing thimble $\vt$ and cycle $\vc$ will then be defined as the closure of these in $\C^n$ and $F_1$ respectively.

Let $\varphi: \C^n \to \C^n$ be given by the map
\begin{equation*} \varphi (w_1, \ldots, w_n) = \left( \frac{1}{2} w_1^2 , \ldots, \frac{1}{2}w_m^2 , w_{m + 1}, \ldots, w_n \right). \end{equation*}
Observe that $\tilde{f} := f \circ \varphi (w_1, \ldots , w_n) = 1/2 \sum w_i^2$ and that $\varphi^* \omega_1 = \omega_{st}$ so that the diagram
\begin{equation*} \begin{CD} (\C^n , \omega_{st} ) @>{\varphi}>> (\C^n , \omega_1) \\ @V{\tilde{f}}VV @V{f}VV \\ \C @= \C \end{CD} \end{equation*}
commutes. This immediately implies that, off of $D(m)$, the parallel transport vector fields are mapped to each other via $\varphi$. 

In fact, we show that this description extends over $D(m)$. To see this note that, given any K\"ahler form $\tilde{\omega}$ on $\C^n$, a holomorphic function $F : \C^n \to \C$, a regular point $p \in \C^n$ of $F$ with $q = F(p)$ and a tangent vector $z \in T_q \C$, one has the following formula
\begin{align}
\label{eq:partransformula} \xi_{F, \tilde{\omega}} (z) =  z \cdot \frac{\text{grad}_{\tilde{\omega}} (F)}{\| \text{grad}_{\tilde{\omega}} (F)\|_{\tilde{\omega}}^2}
\end{align}
for the symplectic connection lift $\xi_{F, \tilde{\omega}} (z)$ of $z$ to $T_p \C^n$. Here the gradient and norm is with respect to the Hermitian form defined by $\tilde{\omega}$. Letting $p = (w_1, \ldots, w_n) \in \C^n$ with the standard metric, one computes that the lift of $z$ for $\tilde{f}$ 
\begin{align*}
\xi_{\tilde{f}, \omega_{st}} (z)|_p = z \cdot \frac{(\bar{w}_1, \ldots, \bar{w}_n)}{\| (\bar{w}_1,\ldots, \bar{w}_n)\|_{\omega_{st}}^2} .
\end{align*}
On the other hand, taking $\varphi_* (\xi_{\tilde{f}, \omega_{st}}(z)|_p)$ and using equation \eqref{eq:partransformula} for $F = f$ and $\tilde{\omega} = \omega$ at $\varphi (p) = (1/2 w^2_1, \ldots, 1/2 w^2_m, w_{m + 1}, \ldots, w_n)$  gives
\begin{align} \label{eq:partransf}
\begin{split}
\varphi_* (\xi_{\tilde{f}, \omega_{st}}(z)|_p) & = z \cdot \frac{( |w_1|^2, \ldots,  |w_m|^2 , \bar{w}_{m + 1}, \ldots, \bar{w}_n )}{\| (\bar{w}_1,\ldots, \bar{w}_n)\|_{\omega_{st}}^2} , \\ & = z \cdot \frac{( |w_1|^2, \ldots,   |w_m|^2 , \bar{w}_{m + 1}, \ldots, \bar{w}_n )}{\| ( |w_1|^2, \ldots,  |w_m|^2 , \bar{w}_{m + 1}, \ldots, \bar{w}_n )\|_{\omega}^2}, \\ & = z \cdot \frac{\text{grad}_{\omega} (f)}{\| \text{grad}_{\omega} (f)\|^2_{\omega}} \\ & = \xi_{f, \omega} (z)|_{\varphi (p)}.
\end{split}
\end{align} 
Off of $D(m)$, this equality follows from the fact that $\varphi^* (\omega_1) = \omega_{st}$. The upshot of the computation is the realization that any parallel transport vector field with respect to $f$ extends to $D(m)$. Moreover, a closer look at the equations in \eqref{eq:partransf} shows the divisors $\{D_i\}_{1 \leq i \leq m}$ are horizontal with respect to parallel transport. In fact, the vector field $\xi_{f, \omega}$ restricts to the parallel transport field associated to $f|_{D(m)}$.

\begin{prop} \label{prop:vcvt} The vanishing thimble and cycle of $f$ are
\begin{equation*} \vt = \R_{\geq 0}^m \times \R^{n - m}, 
\end{equation*}
\begin{equation*} \vc = F_1 \cap \left( \R^m_{\geq 0} \times \R^{n - m} \right) \approx  \Delta_{m - 1} \star S^{n - m
 - 1} ,
\end{equation*}
where $\Delta_{m - 1}$ is the $(m - 1)$-dimensional simplex, $S^{n
- m - 1 }$ is the sphere and $\star$ is the join.
\end{prop}

\begin{proof} By the definition of vanishing thimble in equation \eqref{eq:defnvt}, $T$ is the union of all integral curves of $\xi_{f, \omega} (-1)$ which limit to $0 \in \C^n$ (here $-1$ represents the vector field $-\partial_z$ on $\C$). It is known that $\R^n$ is the vanishing thimble of $\tilde{f}$ along $\gamma_{vc}$ (see \cite[Example~16.5]{SeidelFPL}).  By equation \eqref{eq:partransf},  $\varphi_* (\xi_{\tilde{f}, \omega_{st}}(-1)|_p) = \xi_{f, \omega} (-1)|_{\varphi (p)}$ which implies that if $\delta : [a,b] \to \C^n$ is an integral curve for $\xi_{\tilde{f}, \omega_{st}}(-1)$ then $\varphi \circ \delta$ is an integral curve for $\xi_{f, \omega} (-1)$. As $\varphi$ is surjective, we see by the uniqueness of integral curves (up to reparameterization) that every integral curve of $\xi_{f, \omega} (-1)$ is the image of one for $\xi_{\tilde{f}, \omega_{st}}(-1)$. As $\varphi^{-1} (0) = 0$, this implies that $T$ is the image $\varphi (\R^n ) = \R^m_{\geq 0} \times \R^{n - m}$. 
	
The description of $\vc$ follows from intersecting $\vt$ with the fiber $F_1$. In particular, $(r_1, \ldots, r_m, z_1, \ldots, z_{n - m}) \in F_1 \cap \left( \R^m_{\geq 0} \times \R^{n - m} \right)$ if and only if $r_i \geq 0$ for all $i$ and 
\begin{align} \label{eq:descvc} \sum_{j = 1}^{n - m} z_i^2 = 2 \left( 1 - \sum_{i = 1}^m r_i \right).
\end{align}
Taking the standard simplex $\Delta_{m - 1} = \{s = (s_1, \ldots, s_m) \in \R_{\geq 0}^{m} : \sum_{i = 1}^m s_i = 1\}$ and $S^{n - m - 1} = \{u \in \R^{n - m}: \|u\|^2 = 1\}$ we map $G : \Delta_{m - 1} \times S^{n - m - 1} \times [0,1] \to F_1 \cap \left( \R^m_{\geq 0} \times \R^{n - m} \right)$ via 
\begin{align*}
G(s , u , t) = \left( ts , \sqrt{2(1 -t)}\,u \right).
\end{align*}
Recall that the join $\Delta_{m - 1} \star S^{n - m - 1}$ is equal to $\Delta_{m - 1} \times S^{n - m - 1} \times [0,1] / \sim$ where $(s, u, 0) \sim (s^\prime, u, 0)$ and $(s, u, 1) \sim (s, u^\prime, 1)$. 
It follows then from equation \eqref{eq:descvc} that $G$ induces a homeomorphism of $\Delta_{m - 1} \star S^{n - m - 1}$ onto $\vc$.
\end{proof}

We illustrate a few examples of these vanishing cycles in Figure \ref{fig:vc}.  In general, one would hope that these cycles could appear as natural objects in a Fukaya-Seidel category, perhaps with a partial wrapping around the stratifying divisors. 

We conclude this section with a description of the monodromy map around the stratified Morse critical value. 
\begin{figure}[t]
\begin{picture}(0,0)%
\includegraphics{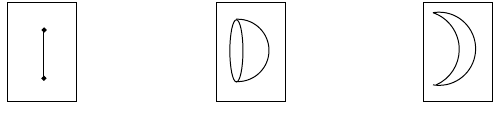}%
\end{picture}%
\setlength{\unitlength}{4144sp}%
\begingroup\makeatletter\ifx\SetFigFont\undefined%
\gdef\SetFigFont#1#2#3#4#5{%
  \reset@font\fontsize{#1}{#2pt}%
  \fontfamily{#3}\fontseries{#4}\fontshape{#5}%
  \selectfont}%
\fi\endgroup%
\begin{picture}(3762,958)(391,-1097)
\put(406,-1051){\makebox(0,0)[lb]{\smash{{\SetFigFont{7}{8.4}{\rmdefault}{\mddefault}{\updefault}{\color[rgb]{0,0,0}$n = 2, m = 1$}%
}}}}
\put(2026,-1051){\makebox(0,0)[lb]{\smash{{\SetFigFont{7}{8.4}{\rmdefault}{\mddefault}{\updefault}{\color[rgb]{0,0,0}$n=3, m=1$}%
}}}}
\put(3601,-1051){\makebox(0,0)[lb]{\smash{{\SetFigFont{7}{8.4}{\rmdefault}{\mddefault}{\updefault}{\color[rgb]{0,0,0}$n=3, m=2$}%
}}}}
\end{picture}%
\caption{\label{fig:vc} The vanishing cycle as a join $\vc \approx \Delta_{m - 1} \star S^{n - m - 1}.$}
\end{figure}
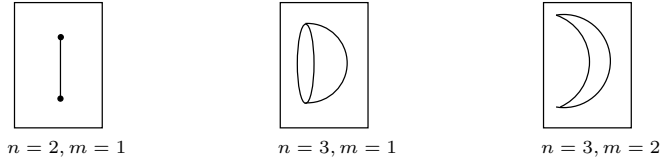
\begin{prop} \label{prop:stratmon}
For any $\varepsilon > 0$ exists a radius $\varepsilon$ neighborhood $U$ of $L$ and a symplectomorphism $\phi$ supported on $U$ isotopic to symplectic monodromy of $f$ around $0$. Furthermore, $L$ is a deformation retract of $U$ with retraction $\rho : U \to L$ and for $x \in L - \partial L$, the fiber $F_x = \{u \in U : \rho (u) = x\}$ satisfies:
\begin{enumerate}
\item $F_x$ is a topological disc,
\item $\phi (F_x) \cap L = \{x\}$,
\item $\rho \circ \phi$ is generically a $2^m$-cover of $L$.
\end{enumerate}
\end{prop}
\begin{proof}
The symplectic monodromy around $0$ relative to the function $\tilde{f}$ describes a spherical twist as introduced in \cite{seideldehn} and surveyed in \cite[Section~6.3]{mcduffsolomon}. We let $G = (\Z / 2 \Z )^m$ act on $\C^n$ by multiplying the $i$-th coordinate with $\pm 1$. As $\tilde{f}(w_1, \ldots, w_n) = 1/2 \sum w_i^2$, we see that $\tilde{f}$ and  $\varphi$ are invariant with respect to this action. Since the action preserves the standard symplectic form, the lift of the parallel transport vector field and the symplectic monodromy are equivariant with respect to the action. Noting that $F_q$ is simply the quotient of $\tilde{f}^{-1} (q)$ by $G$, we aim to understand the symplectic monodromy around zero on $F_1$ as a quotient of that on $\tilde{f}^{-1} (q)$ by $G$.

To accomplish this aim, we take a moment to recall the formulation of the spherical twist with respect to $\tilde{f}$. It is known that $\tilde{f}^{-1} (1) \approx T^* S^{n - 1}$ and the vanishing cycle is $Z = \{(w_1, \ldots, w_n) \in \R^n : \sum w_i^2 = 2\}$. Taking $g$ to be the constant curvature $1$
metric on $Z \approx S^{n -1}$ we have the dual metric $g^*$ on $T^* S^{n - 1}$ and consider $H : T^* S^{n - 1} \to \R$ to be
the Hamiltonian $H (w, v) = \frac{1}{2} \|v\|_{g^*}^2$ generating geodesic flow \cite[Example~1.22]{mcduffsolomon} and
$X_H$ its vector field. For an arbitrary choice of $\varepsilon > 0$, one may rescale $X_H$ off of an $\varepsilon / 2$ neighborhood of the zero section $Z \subset T^* S^{n - 1}$ to obtain an exact vector field $\tilde{X}_H$ supported on a radius $\varepsilon$ neighborhood $\tilde{U}$ of $Z$ whose time $1$ flow the antipode map on $Z$. The resulting monodromy map $\tilde{\phi}$ has support on the neighborhood $\tilde{U}$ of $Z$ and is Hamiltonian isotopic to the spherical twist. This is a generalization of the Dehn twist in $2$-dimensions.

We would like to utilize this description to understand the stratified case. Let $U = \tilde{U} / G$ and $\phi : F_1 \to F_1$ be the monodromy map induced by symplectic parallel transport $\tilde{\phi}$ along a loop about the origin. Note that, as the bundle projection from $\tilde{U}$ to $Z$ is $G$-equivariant, it defines a retract $ \rho : U \to L$ in $F_1$. Decompose $Z$ into $2^m$ regions $Z = \cup_{g \in G} Z_g$ defined as
\begin{equation*} Z_0 = \{(w_1, \ldots, w_n) \in S^{n - 1} : w_i
\geq 0 \text{ for all } 1 \leq i \leq m \} ,\end{equation*}
and $Z_g = g \cdot Z_0$. We let $T^* Z_0$ consist of pairs $(w, v)$ such that, if $w \in \partial Z_0$, then $v (\nu ) \geq 0$ for all inward pointing tangent vectors $\nu \in T_w Z_0$. Observe that $T^* Z_0$ forms a fundamental domain for the $G$ action in $\tilde{f}^{-1} (1)$ ramified over the boundary $\partial Z_0$. Thus the points in  $F_1 \approx T^* S^{n - 1} / G$, can be identified with those in $T^* Z_0$.

Now, by restricting $\tilde{\phi}$ to any cotangent fiber
$T^*_p S^{n - 1}$ and projecting to $Z$, we obtain a decomposition
of each such fiber $T^*_p S^{n -1 } = \cup_{g \in G} Z_{p, g}$. Here 
\begin{align*} Z_{p, g} = \{ (p, v) \in T^*_p S^{n - 1} :  \pi (\phi^{-1} (p, v)) \in Z_g \} \end{align*} 
where $\pi : T^* S^{n - 1} \to Z$ is the cotangent bundle projection. Identifying $T^* Z_0$ with $F_1$, the monodromy map $\phi$ takes $(p, v) \in Z_{p, g}$ to $g^{-1} \varphi (p, v)$. Qualitatively, we observe that a given fiber of $T^*_p Z_0 \approx T^*_p \vc$, identified with $\rho^{-1} (p)$, wraps around the zero section $2^m$ times with one crossing. The map on the vanishing cycle is seen to be the join of the identity on the simplex with the antipode map on the sphere.
\end{proof}

\subsection{\label{subsec:flp} $\partial$-framed Lefschetz pencils}

In this section we address certain transitions in framings for
symplectomorphisms arising in monodromy calculations. We assume $(\mcy , D )$ is a K\"ahler orbifold with standard normal crossing divisor $D = D_1 + \cdots + D_k$, i.e. it is a symplectic orbifold with a specified $J \in \mcj$ that is integrable everywhere. 

Take $\curve$ to be a $1$-dimensional DM stack with coarse space $\p^1$. Let $\pi : \mcy \to \curve$ be a map with determinant values $\detl{\pi}$. These are defined to be values of $\pi$ for which $\pi$ is either singular, or where $\pi |_{\cap_{i \in I} D_i}$ is singular.

\begin{defn} \label{defn:framedpencil} We will say that $\pi$ defines a $\partial$-framed Lefschetz pencil
if $\omega \in \Omega^2 (\mcy)$ is isotopic to some $\tilde{\omega}$ for which $D$ is horizontal and such that there is a covering $\{U_i\}$ of $\curve$ such that $\pi : \pi^{-1} (U_i ) \to U_i$ is either a smooth proper fibration, a normal crossing
degeneration or a stratified Morse function for every $i$.
If $(\mcy , D)$ is a standard K\"ahler stack with resolving collection $\mcr$, we say that $\pi$ is a $\partial$-framed Lefschetz pencil if $\pi \circ \psi_\varepsilon : \tilde{\mcy} \to \curve$ is a $\partial$-framed Lefschetz pencil for every $(\tilde{\mcy}, \tilde{D} ) \in \mcr$. 
\end{defn}

We note that the definition of Lefschetz pencil given in \cite{Donaldson1} is generalized by the definition above. The notion of a partial Lefschetz fibration given in \cite{kerr} can also be put in this framework. However, our principal example of a framed pencil is obtained from considering stacky curves in $\secon{A}$ where $A$ satisfies some basic conditions. 

Before we state the following theorem we review some notation from Appendix~\ref{sec:toric}. Recall that  $A$ is a finite set in a lattice with convex hull $Q$. The toric stack defined by $Q$ was denoted $\mcx_Q$ and, for any face $F \subseteq Q$, the orbit of the maximal torus acting on $\mcx_Q$ corresponding to $F$ was denoted $\orb{F}$. Coupled with $\mcx_Q$ was a line bundle $\mco_A (1)$ defined by $Q$ and a subspace of sections $\linsys{A}$ defined by $A$. In equation \eqref{eq:princadet} we defined the principal $A$-determinant $E_A : \linsys{A} \to \C$ which vanished on degenerate sections in $\linsys{A}$. These were elements of $\linsys{A}$ whose zero locus intersected an orbit $\orb{F}$ non-transversely for some face $F$ of $Q$. In Definition \ref{defn:secstack}, we then extended $E_A$ to a section $E^s_A$ of a line bundle over the secondary stack $\secon{A}$ with zero loci $\mce_A$. The discriminant $\Delta_A : \linsys{A} \to \C$ is a polynomial vanishing only on those sections which defined hypersurfaces with singularities in the maximal orbit $\orb{Q}$. When $\Delta_A$ is constant, we call $(Q,A)$ dual defect.

\begin{thm} \label{thm:fpsec} Suppose $A \subset \Z^d$ defines the marked polytope $(Q, A)$ such that for every face $F$ of $Q$ either $\orb{F}$ 
 has a smooth neighborhood, or $(F , A \cap F )$ is dual defect. 
Let $(\hyp{A} , \partial \hyp{A}) \subset \laf{A}$ be the universal toric
hypersurface with boundary. Suppose $\curve$ is $1$-dimensional and $i : \curve \to \secon{A}$ is an embedding which transversely intersects $\mce_A$ and $\partial \secon{A}$. Then $\pi : i^* (\hyp{A} , \partial \hyp{A} ) \to
\curve$ is a $\partial$-framed pencil. \end{thm}

\begin{proof} From Theorem \ref{thm:GKZ1}, there is a product decomposition  
	\[E_A (f) = \prod_{Q^\prime \leq Q} \Delta_{A \cap Q^\prime} (f)^{i (\Lambda , A) \cdot u ({\textnormal{Lin}_{\mathbb{N}} (\mathcal{A})} / Q^\prime )}\] of $E_A$ where, in this case, $\Lambda = \Z^d$. Recall that this is indexed by faces $Q^\prime = \convhull (A^\prime ) \subseteq Q$. Under the conditions above, if the orbit $\orb{Q^\prime}$ does not admit a smooth neighborhood then $\Delta_{A^\prime}$ is constant. So we may assume that for every set of lattice points $A^\prime$ in a face $Q^\prime$ of $Q$ which has a non-constant discriminant,  $X_{Q^\prime}$ admits a smooth neighborhood. 

Since every intersection point  $ p \in i (\curve ) \cap \{\Delta_{A^\prime} = 0\}$ is transverse, we have that the point $p$ is a smooth point of $\Delta_{A^\prime} = 0$ and not in $\{\Delta_{A^{\prime \prime}} = 0\}$ for any other subset $A^{\prime \prime}$ on a face of $Q$. By \cite[Theorem~1.5.1]{GKZ}, this implies that the Hessian of $\pi: i^* (\hyp{A}) \cap \orb{Q^\prime} \to \curve$ at $p$ is non-degenerate and that $\pi :i^* (\hyp{A}) \cap \orb{Q^{\prime \prime}} \to \curve$ is non-singular at $p$ for all faces $Q^{\prime \prime}$ containing $Q^\prime$. Thus, by Proposition \ref{prop:str},  $\pi: i^* (\hyp{A} , \partial \hyp{A} ) \to U$ is a stratified Morse singularity in a neighborhood $U$ of $p$. 

For every $p \in i (\curve ) \cap \partial \secon{A}$, Theorems \ref{thm:toric1} and \ref{thm:toric2} imply that there is a neighborhood $U$ of $p$ such that $\pi : i^* (\hyp{A} , \partial \hyp{A} ) \to U$ is a hypersurface degeneration of $\fib{A}{q} = \pi^{-1} (q)$  for $q \in U - \{p \}$. 
\end{proof}

All of the results on symplectomorphisms will be obtained by parallel transport in a $\partial$-framed Lefschetz pencil. However, the parallel transport map occurs naturally as a functor in higher dimensional settings. We take a moment to fix notation for the general set up, and quickly return to the $1$-dimensional case afterwards.

Given a stack $\mcx$ with atlas $(U_\beta, G_\beta , \pi_\beta)_{\beta \in \mcb}$, let $\Pi (\mcx )$ be the path category of $\mcx$ defined by taking elements $p \in \cup U_\beta$ to be objects and morphisms $\Hom (p, q) = \{\gamma : [0, 1] \to \mcx : \gamma (0) = p , \gamma (1) = q \}$. We can think of this category as an $(\infty, 1)$-category, as morphisms do not compose associatively.

Given a bundle $\pi : (\mcy , \partial \mcy ) \to \mcx$ of standard symplectic stacks over $\mcx$ and a symplectic connection which preserves their boundaries, we write parallel transport as a functor
\begin{equation*} \partrans : \Pi (\mcx ) \to \mathbf{Symp} \end{equation*}
where $\mathbf{Symp}$ is the category of standard symplectic stacks. This map takes $p \in \mcx$ to $\pi^{-1} (p )$ and a morphism to the map obtained by parallel transport. We will abuse notation and also write $\partrans : \Omega_p (\mcx ) \to \Symp (\pi^{-1} (p) , \partial \pi^{-1} (p) )$ for the restriction to based loops. As indicated by Theorem \ref{thm:fpsec}, the primary example we consider is $\mcx = (\secon{A} - \mce_A)$.  Using this theorem and the general parallel transport map, we define:
\begin{defn}\gls{mongroup}
For any point $p \in (\secon{A} - \mce_A)$ let 
\begin{align*} \agroup_p \subset \pi_0 (\Symp ( \fib{A}{p} , \partial \fib{A}{p} )) \end{align*} be the group of components of the image $\partrans (\Omega_p (\secon{A} - \mce_A ))$.
\end{defn}

For any $\partial$-framed Lefschetz pencil $\pi : \mcy \to \curve$ and $q \in \curve - \detl{\pi}$
and take $\mcz_q = \pi^{-1} (q)$ to be a fiber with $\partial \mcz_q = \mcz_q \cap D$. If the $q$ is a chosen base point, we simply write $\mcz$ and $\partial \mcz$.  Note that the definition above ensures that every fiber outside $\detl{\pi}$ transversely intersects $D$, so $(\mcz, \partial \mcz)$ is a symplectic orbifold with standard normal crossing divisor. 

The connection given by the modified symplectic form $\tilde{\omega}$ yields a parallel transport map that preserves the boundary, which we write as\gls{partrns}
\begin{equation*} \partrans : \Omega_q (\curve - \detl{\pi} ) \to \Symp
(\mcz, \partial \mcz) , \end{equation*}
where $\Omega_q$ is notation for piecewise smooth based loops at $q$.

A key point which will be made precise in Proposition~\ref{prop:flp} is that when we examine local monodromy, we may
utilize the local model descriptions to analyze the
symplectomorphisms as framed maps with respect to a reasonable
$\partial$-frame group. However, as we extend to the global pencil, these maps loose their framing in the holonomy.  Another way of saying this is that if we omit a point $q_\infty \in \curve - \detl{\pi}$, we may define a $\partial$-frame group $\mbF$ which is tightly controlled on parts of $\partial \mcz$ and obtain, up to homotopy, a lift\gls{prtrns}
\begin{equation*} \partran : \Omega_q (\curve - \detl{\pi} -
\{q_\infty\} ) \to \Symp^{\mbF} (\mcz, \partial \mcz) . \end{equation*}
However, to extend the map to $\Omega_q (\curve - \detl{\pi} )$, we need to consider the $\partial$-frame group $\relaxed{\mbF}$.

To make this idea precise, we must define a $\partial$-frame group of a $\partial$-framed Lefschetz fibration. For this, recall the normal crossing divisor $D = \sum_{i = 1}^k D_i$ is contained in $\mcy$ and let $\partial \mcz = \sum_{i = 1}^k \divi_i$ where $\divi_i = D_i \cap \mcz$. In general, the symplectomorphisms in $\Symp (\mcz, \partial \mcz)$ arising from parallel transport will have non-trivial restrictions to the symplectomorphism group of the boundary divisors $\Symp (\divi_i)$. The following definition gives conditions that allow us to control this additional complexity.
\begin{defn} A boundary component $D_i$ is called rigid if there exists a trivialization $D_i - \partial D_i \approx (\divi_i - \partial \divi_i ) \times \curve$ over $\mcc$ where $\pi$ is projection to the second factor. 
\end{defn}
We say that a face $F \subset Q$ is a simplicial face if $F$ is a face of $Q$ and $F \cap A$ is an affinely independent set. The following proposition may be seen from the fact that the orbits corresponding to simplicial faces of $Q$ occur in trivial families as substacks of $\pi : \hyp{A} \to \secon{A}$.
\begin{prop} \label{prop:rigid} If $i: \curve \to \secon{A}$ pulls back $\pi : \hyp{A} \to \secon{A}$ to a $\partial$-framed Lefschetz fibration and $D$ is a divisor associated to a simplicial facet of $(Q, A)$, then $D$ is rigid.
\end{prop}
\begin{proof}
If $(Q^\prime, A^\prime)$ is a simplicial facet of $(Q,A)$ then by Theorem~\ref{thm:GKZ2}, $\Sigma (A^\prime)$ is zero dimensional. Thus the moduli of hypersurfaces in the toric stack $\mcx_{Q^\prime}$ is also zero dimensional. Let $\mathbf{D} \subset \laf{A}$ be the horizontal divisor corresponding to the pointed subdivision $(S,A^\prime)$ where $S$ is the trivial subdivision $\{(Q,A)\}$. By Proposition~\ref{prop:lafpolytope} and Lemma~\ref{lem:lafnormal}, the facet $F_\mathbf{D}$ of the Lafforgue polytope $\ptlaf{A}$ corresponding to the boundary divisor $\mathbf{D}$ is a Minkowski sum $P + \Sigma (A)$ of two polytopes $P := \convhull \{e_a : a \in A^\prime\}$ and $\psec{A}$. By equation~\eqref{eq:secpolhomogenious}, the polytopes $P$ and $\psec{A}$ lie on independent affine spaces in $\R^{A}$ implying that the boundary is $P \times \Sigma (A)$ and $\mathbf{D} \cong \mcx_P \times \secon{A}$. From the definition of $\pi$, one has that $\pi |_{\mathbf{D}}$ is the projection of the second factor. Also, as the hypersurface in $\hyp{A} \cap \mathbf{D}$ forms a trivial family over $\secon{A}$ (since $\Sigma (A^\prime)$ is a point), we have that $\hyp{A} \cap \mathbf{D}$ also splits as a product. Pulling back  along $\iota : \mcc \to \secon{A}$ gives the desired splitting over $\mcc$.
\end{proof}
Now, let $\detl{\pi}  = \{q_1 , \ldots, q_N \}$ and take $B_\varepsilon (p)$ to be a disc of radius $\varepsilon$ around $p$. We take $\mcb = \{\gamma_1, \ldots, \gamma_N \}$ to be a set of embedded paths from $[0, 1]$ to $\curve$ such that $\gamma_i (0 ) = q$, $\gamma_i (1) = q_i$
and, for $i \ne j$, $\gamma_i (t) = \gamma_j (s)$ if and only if $t = s = 0$.
We also assume that $\gamma_i^\prime (0)$ is ordered
clockwise. Such a collection is known as a distinguished
basis of paths \cite{brieskorn}. For any such distinguished basis and any $\gamma_i$, we define the loop $\gamma^\varepsilon_i$ by following
$\gamma_i$ until reaching a distance of $\varepsilon$, circling
around the boundary of $B_\varepsilon (q_i )$ clockwise
and following $\gamma_i$ back to $q$. By Definition~\ref{defn:framedpencil}, there is  an $\varepsilon$ 
sufficiently small so that $\phi_i = \partran (\gamma_i^\varepsilon )$ is a degeneration monodromy map or stratified Morse monodromy map as presented in the previous sections. We divide\gls{idim} $\{1, \ldots, N\} = I_d \cup I_m$ into those points of hypersurface degeneration monodromy and stratified Morse values respectively. 

For $i \in I_m$, we let $\vc_i \subset (\mcz, \partial \mcz)$ be the vanishing cycle pulled back along $\gamma_i$ and write $S_i = \{j : \vc_i \cap \divi_j \ne \emptyset \}$ for the set of divisors that intersect the vanishing cycle of $\gamma_i$. Let $K_i$ be a relatively compact neighborhood of $\partial \vc_i$ and $K = \cup_i K_i$. By the discussion following Proposition \ref{prop:vcvt}, we have that $\phi_i$ can be viewed as a symplectomorphism with support in $K_i$. 

For $i \in I_d$, first recall that the facets $\{\partial Q_j\}$ of $Q$ are indexed by $\{1, \ldots, k\}$ and to each facet $\partial Q_J$ there corresponds a divisor $\divi_j$ of $\mathcal{Z}$. Let $\eta_i : Q \to \R$ be the defining function for the stable pair degeneration at $q_i$ and take $S_i = \{j : \eta_i \text{ is not affine on } \partial Q_j\}$. In other words, the degeneration of $\mcz$ at $q_i$ also degenerates $\divi_j$. We take 
\begin{align} R = \{1, \ldots, k \} - \cup_i S_i \end{align} 
and observe that every boundary divisor $\divi_j$ is rigid if $j \in R$, as in Proposition \ref{prop:rigid}. Let\gls{rketa}
\begin{align} \label{eq:rkpi} \R^k_{\eta_i} = \{(r_1, \ldots, r_k) : r_j \in \R,  r_j \in \mathbb{Z} \text{ for } j \not\in S_i \} , 
\end{align}
 and $\frme_{\eta_i} = \{ \boundflow ( \mathbf{r} ) : \mathbf{r} \in \R^k_{\eta_i}\}$. We take\gls{frmpi} $\frme_\pi$ to be the group generated by the subgroups $\frme_{\eta_i}$ over all $i \in I_d$.
 
\begin{eg}
Consider the set $A = \{ (0, 0), (1, 0), (-1, 0), (0, 1)\}$, the universal hypersurface $\hyp{A} \subset \laf{A}$ and the restriction $\pi|_{\hyp{A}} : \hyp{A} \to \secon{A}$. Theorem~\ref{thm:GKZ2} implies that $\secon{A}$ is one dimensional so that its coarse moduli space is $\p^1$. The horizontal boundary of $\hyp{A}$ is the intersection of $\hyp{A}$ with the horizontal boundary of $\laf{A}$ which, by Lemma~\ref{lem:lafhyppart} correspond to the boundary of $Q$. Index the three horizontal boundary divisors of $\hyp{A}$ as $D_1$ for the line segment between $(-1, 0)$ and $(1,0)$, $D_2$ for line segment between $(- 1, 0)$ and $(0,1)$, and $D_3$ for line segment between $( 1, 0)$ and $(0,1)$. Note that $D_2$ and $D_3$ are rigid by Proposition~\ref{prop:rigid}. For a regular value $t \in \secon{A}$, the fiber $(\fib{A}{t}, \partial \fib{A}{t})$ is isomorphic to $\p^1$ with four marked points where $\tilde{D}_2$ and $\tilde{D}_3$ are each a single point and $\tilde{D}_1$ consists of two points.

As is shown in Section~\ref{sec:exampledeg},  $\textnormal{Det} (\pi |_{\fib{A}{t}}) = \{q_0, q_1, q_2\}$ where $q_0, q_2$ are hypersurface degenerations arising from the triangulations 
\begin{align*} T_0 &= \{(\convhull{A - (0,0)}, A - (0,0))\} , \\ T_2 & = \{ (\convhull{A - (1,0)}, A - (1,0)),(\convhull{A - (-1,0)}, A - (-1,0)) \}. \end{align*}
The point $q_1$ is a stratified Morse singularity relative to $\tilde{D}_1$ of codimension $1$. 

Thus $I_d = \{0,2\}$ and $I_m = \{1\}$. The set $S_1$ is simply $\{1\}$ as the vanishing thimble only intersects $D_1$. The set $S_0$ is empty since $T_0$ does not subdivide any boundary component. However, $S_2 = \{1\}$ since $T_2$ defines a degeneration of the divisor $\tilde{D}_1$ into multiple components. Thus $R = \{2,3\}$ in this case and $\R^3_{\eta_i} = \R \oplus \Z^2$ for $i \in \{0,2\}$. The group $\mathbf{T}_\pi$ then consists of all (simultaneous) angular rotations around the points in  $\tilde{D}_1 \subset \fib{A}{t}$ but only full $2\pi$ rotations around $\tilde{D}_2$ and $\tilde{D}_3$. 
\end{eg}
\begin{defn} Let $\pi : \mcy \to \curve$ be a $\partial$-framed Lefschetz pencil. The $\partial$-frame group $\mbF \in \Symp (N_{\partial \mcz} \mcz / \partial \mcz )$ associated to $\pi$ is given by the collection of maps whose restrictions to $\cup_{j \in R} \divi_j$ are contained in the restriction of $\frme_\pi$.
\end{defn}

We note that if every facet on the boundary of $Q$ is corresponds to a degenerated divisor or one with a stratified Morse singularity over some $q_i \in \detl{\pi}$, then $\mbF$ equals $\Symp (N_{\partial \mcz} \mcz / \partial \mcz)$. On the other hand, if $Q$ is simplicial, then $\mbF$ is a discrete subgroup of $\frme$. Ideally, one would like to obtain more control over the $\partial$-framing for the non-rigid boundary components and incorporate this into a formula such as the one in Proposition \ref{prop:frmechnge}, but this is currently not within our sight. However, we may use the results of the previous sections to prove the following proposition.

\begin{prop} \label{prop:flp} If $\pi : \mcy \to D$ is a $\partial$-framed Lefschetz pencil and $q_\infty$ is chosen as above, then there exists a symplectic connection for which the parallel transport map $\partrans$ lifts to
\begin{equation*} \partran : \Omega_q (\curve - \textnormal{Det} (\pi ) - \{q_\infty \} ) \to \Symp^{\mbF} (\mcz , \partial \mcz ) , \end{equation*}
where $\mbF$ is the $\partial$-frame group associated to $\pi$.
\end{prop}
\begin{proof} For every $i \in I_m$, we have, by definition, that the divisors supporting the degenerate point are horizontal with respect to $\tilde{\omega}$. By Proposition~\ref{prop:stratmon}, the monodromy around $q_i$ is Hamiltonian isotopic to a map supported on the relatively compact neighborhood $K_i$. By the definition of $R$, the vanishing cycle $L_i$ associated to $i \in I_m$ is disjoint from $\divi_j$ for all $j \in R$ so we may choose a neighborhood $K_i$ which also is disjoint. Thus the restriction of the map to the framing group $\mbF$ is well defined. Indeed, the monodromy map is the identity on any rigid $\divi_i$.


\end{proof}

We observe that for $\partial$-frame groups associated to $\partial$-framed Lefschetz pencils, the exact sequence in Proposition \ref{prop:frame} yields the fiber sequence
\begin{equation*} \Symp^{\mbF} (\mcz,\partial \mcz ) \to \Symp^{\relaxed{\mbF}} (\mcz, \partial \mcz ) \to \R^k / \R^k_{\pi}, \end{equation*}
where $\R^k_\pi$ was defined in equation~\eqref{eq:rkpi}.  
Note that the last group is homotopic to $(S^1)^t$ where $t \leq |I_d|$.

Now, write\gls{gampi} $\gamma$ for the concatenation $\gamma_N \circ \cdots
\circ \gamma_1$ which is independent of the distinguished basis.
Let $N (\gamma )$ be the normalizer of $\gamma$ in the group
$\pi_1 ( \curve - \detl{\pi} - \{q_\infty\} )$. We write $F_N$ for
the free group on $N$ letters and
obtain the commutative diagram
\begin{equation*} \begin{CD}  N (\gamma ) @>>> F_N @>>> F_{N - 1} @>>> 1 \\
@V{\varrho}VV @V{[\partran ]}VV @V{[\partrans]}VV @. \\
 \pi_1 ((S^1)^t ) @>{\delta}>> \pi_0
(\Symp^\mbF (\mcz, \partial \mcz)) @>>> \pi_0 ( \Symp^{\relaxed{\mbF}} (\mcz , \partial \mcz)) @>>> 1 \end{CD} .
\end{equation*}
The bottom row of this diagram arises as the long exact sequence of homotopy groups associated to a fiber exact sequence. The top row is the short exact sequence associated to the quotient group. The homomorphism $\varrho$ is uniquely constructed from the commutativity of the remaining portion of the diagram. 

The image of $\gamma$ under $\varrho$ will be of particular importance. One may interpret this image as the amount of rotation around the boundary needed to isotope $\gamma_N \circ \cdots \circ \gamma_1$ to the identity in $\Symp^\mbF (\mcz, \partial \mcz)$. Under the restrictions laid out above, there is an explicit formula for this map.
\begin{prop} \label{prop:frmechnge} If $\pi : \mcy \to \curve$ is a $\partial$-framed Lefschetz pencil, 
then, for every $i \in T$, there exists a section $s_i : \curve_i \to D_i$ such that
\begin{equation*} \varrho (\gamma ) = \sum_{i \in T} a_i e_i
\end{equation*} where $a_i = \int_{\curve_i} s_i^* (c_1 (N_\mcy \divi_i ))$ and $e_i$ is the
loop in $\Symp (N_{\partial \mcz} \mcz / \partial \mcz )$ corresponding to rotation around
$\divi_i$.
\end{prop}
\begin{proof} Fix $q \in \curve$ to be the base point and $\mcz$ its fiber. By definition of rigid boundary divisor, for every $i \in T$,  the restriction of $\pi$ to $D_i$ is trivial over $\curve$ so that there exists an isomorphism $\psi : D_i \cong \tilde{D}_i \times \mcc$ where $\tilde{D}_i = D_i \cap \mcz$. Over the contractible subset $U_0 = \curve - \{q_\infty\}$, we may extend this to an isomorphism of normal bundles $\tilde{\psi}_i : N_{\tilde{D}_i} \mcz \times U_0 \to N_{\tilde{D}_i \cap \pi^{-1} (U_0)} \pi^{-1} (U_0)$.  Likewise, in an open neighborhood $U_1$ of $q_\infty$, we may trivialize $\tilde{\phi}_i : N_{\tilde{D}_i} \mcz \times U_1 \to N_{\tilde{D}_i \cap \pi^{-1} (U_1)} \pi^{-1} (U_1)$. Taking a circle $\delta$ in the intersection $U_0 \cap U_1$, there is a fiberwise transition function between these trivializations. The multiplicative factor of the transition function on the normal bundle restricted to $\delta$ is given by the transition function on $N_{D_i} \mcy$ restricted to $\curve \times \{p\} \subset D_i$. The winding number is given by the Chern number of $s_i^* (N_{D_i} \mcy)$ where $j : \curve \to D_i$ is a section. On a normal neighborhood of $\divi_i$ in $\mcz$, this is the restriction of $\boundflow (\mathbf{x} )$ to $\divi_i$ where $\mathbf{x} = (0, \ldots , 0 , a_i , 0, \ldots , 0)$. Adding these together for each rigid component yields the claim.
\end{proof}
We end this section by defining a subgroup of the framed symplectomorphism group of a hypersurface in a toric stack.

\begin{defn} \label{defn:curvesubgroup} Let $A \subset \Z^d$ satisfy the hypothesis of Theorem \ref{thm:fpsec} and $i: \curve \to \secon{A}$ be an embedded curve. The group\gls{grcurve} $\mathbf{G}_{\curve} = \partran (i_* (\Omega_q (\curve - \curve \cap \mce_A ))) \subset \fib{A}{i (q)}$ will be called the $\curve$ subgroup of $\Symp (\fib{A}{i(q)}, \partial \fib{A}{i(q)})$.
\end{defn}

One of our stated goals is to understand generators and relations for the group $\mcg_A := \partrans (\Omega_q (\secon{A} - \mce_A ) )$. We may reduce the of complexity of this problem by examining $\partial$-framed Lefschetz pencils and their monodromy.

\begin{prop} Assume that $\secon{A}$ does not have generic isotropy. For any embedded $i : \curve \to \secon{A}$ for which the cycle $i_* [\curve]$ is Poincar\'e dual to a very ample divisor, the group $\pi_0 (\mathbf{G}_\curve )$ surjects onto $\pi_0 (\mcg_A )$. \end{prop}

\begin{proof} For a very ample line bundle $\mcl$ with equivariant linear system $V$ we have an embedding on the coarse space $j: X_{\psec{A}} \to \p (V)$. The Lefschetz Hyperplane Theorem gives a surjection from the fundamental group of the curve arising from a linear section of $j (X_{\psec{A}} - \mce_A)$ and that of $\mcx_Q - \mce_A$. But if  $\mcb$ denotes the points with non-trivial isotropy on $\secon{A}$ and $B$ its coarse space, then $\pi_1 (X_{\psec{A}} - \mce_A - B)$ is a surjection onto $\pi_1 (X_{\psec{A}} - \mce_A)$ yielding the result. 
\end{proof}


\bibliographystyle{plain}
\bibliography{basebib8}

\end{document}